%% file: weinberger_diss.tex
\documentclass[11pt,openright,twoside]{memoir}

\usepackage{newpxtext}

\usepackage{amsmath,amsthm,latexsym,amssymb,dsfont}
\usepackage{thmtools,mathtools}
\usepackage{tikz-cd}
\usepackage{quiver}
\usepackage[hidelinks,bookmarksdepth=2,pdfencoding=unicode]{hyperref}
\usepackage{todonotes}
\usepackage{extpfeil}
\usepackage{mymacros_phd}
\usepackage{cleveref}
\usepackage[all,cmtip]{xy}
\usepackage{extarrows}
\usepackage{float}
\usepackage[mathscr]{euscript}
\usepackage[ngerman]{isodate}

\DeclareSymbolFont{stmry}{U}{stmry}{m}{n}
\DeclareMathDelimiter\llbracket{\mathopen}{stmry}{"4A}{stmry}{"71}
\DeclareMathDelimiter\rrbracket{\mathclose}{stmry}{"4B}{stmry}{"79}

\nouppercaseheads
\usepackage[style=alphabetic,maxnames=10,maxalphanames=10,
backend=biber,sorting=nyt,
useprefix=true,hyperref=true]{biblatex}

\usepackage{mathpartir}

\theoremstyle{definition}
\newtheorem{definition}{Definition}[section]
\newtheorem{example}[definition]{Example}
\newtheorem{remark}[definition]{Remark}
\newtheorem{defn}[definition]{Definition}
\newtheorem{expl}[definition]{Example}
\newtheorem{rem}[definition]{Remark}
\newtheorem{ax}[definition]{Axiom}

\theoremstyle{theorem}
\newtheorem{theorem}[definition]{Theorem}
\newtheorem{lemma}[definition]{Lemma}
\newtheorem{proposition}[definition]{Proposition}
\newtheorem{corollary}[definition]{Corollary}
\newtheorem{lem}[definition]{Lemma}
\newtheorem{prop}[definition]{Proposition}
\newtheorem{cor}[definition]{Corollary}
\newtheorem{obs}[definition]{Observation}

\usepackage{subfig}
\usepackage{geometry}
\setsecnumdepth{subsubsection}

\title{\textbf{\HUGE A Synthetic Perspective on $\inftyone$-Category Theory}}

\begin{document}

			
	\frontmatter
	
	\newgeometry{letterpaper,hmargin=1in,vmargin=1in} 
	\pagenumbering{Roman}

	\begin{titlingpage}
		\date{}
		\maketitle
		{
				\thispagestyle{empty}
			\noindent
			\emph{A Synthetic Perspective on $\inftyone$-Category Theory: Fibrational and Semantic Aspects}\\ ~ \\
			Accepted doctoral thesis by Jonathan Maximilian Lajos Weinberger, MSc. \\ ~ \\
		\begin{tabular}{ll}
				\textit{Referee:} & Prof.~Dr.~Thomas Streicher\\
				\textit{1st~co-referee:} & Prof.~Emily~Riehl,~Ph.D.\\
				\textit{2nd~co-referee:} & Prof.~Benno~van~den~Berg,~Ph.D.\\
			\end{tabular} \\ ~ \\
			Darmstadt, Technische Universität Darmstadt \\  ~ \\
			Date of submission: October 20, 2021  \\
			Date of thesis defense: December 9, 2021 \\ ~ \\
			Darmstadt -- D17 \\ 
			
			\vspace{4ex}
			\noindent Bitte zitieren Sie dieses Dokument als / Please cite this document as: \\
			URN: \href{https://nbn-resolving.org/html/urn:nbn:de:tuda-tuprints-207163}{\texttt{urn:nbn:de:tuda-tuprints-207163}} \\
			URL: \url{http://tuprints.ulb.tu-darmstadt.de/20716} \\
			\vspace{2ex}

			\noindent Dieses Dokument wird bereitgestellt von tuprints, E-Publishing-Service der TU Darmstadt / \\
			This document is provided by tuprints, e-publishing service of TU Darmstadt \\
			\url{http://tuprints.ulb.tu-darmstadt.de} \\
			\href{mailto:tuprints@ulb.tu-darmstadt.de}{\texttt{tuprints@ulb.tu-darmstadt.de}} \\
			
			\vspace{4ex}
			\noindent Veröffentlicht von tuprints im Jahr 2022 \\
			Published by tuprints in 2022 \\
			
			\vspace{4ex}
			\noindent Die Veröffentlichung steht unter folgender Creative Commons Lizenz: \\ 
			Namensnennung 4.0 International \\
			CC BY 4.0 \\
			\url{https://creativecommons.org/licenses/by/4.0/deed.de} \\
			
			\vspace{2ex}
			\noindent This work is licensed under a Creative Commons License: \\
			Attribution 4.0 International \\
			CC BY 4.0 \\
			\url{https://creativecommons.org/licenses/by/4.0/}
		}
	\end{titlingpage} 

	\restoregeometry
	\pagenumbering{roman}

	\cleardoublepage
	\begingroup
	\pagestyle{empty}
	\null
	\newpage
	\null
	\newpage
	\endgroup
	
	\thispagestyle{empty}
	\par
	\vspace*{.35\textheight}{\begin{flushright}{\Large \emph{In Liebe meinen Eltern Andrea und Anton.}}\par\end{flushright}}
	
	\cleardoublepage

	\settocdepth{section}
	\tableofcontents*

	\chapter{Abstract}
	\label{ch:abstract}
	\input{abstract}
	
	\chapter{Zusammenfassung}
	\label{ch:abstract-ger}
	\input{abstract-ger}

	\chapter{Acknowledgments}
	\label{ch:ack}
	\input{acknowledgments}

	\chapter{Summary and overview}
	\label{ch:summ}
	\input{summ}

	\mainmatter
		
	\chapter{Introduction}
	\label{ch:intro}
	\input{intro}

	\chapter[Preliminaries on synthetic \texorpdfstring{$\inftyone$}{(∞,1)}-categories]{Preliminaries on synthetic \texorpdfstring{$\inftyone$}{(∞,1)}-categories}
	\label{ch:prelim}
	\input{prelims}

	\chapter{Cocartesian families of synthetic \texorpdfstring{$\inftyone$}{(∞,1)}-categories}
	\label{ch:cocart}		
			\input{cocart-intro}
				
			\section{Cocartesian arrows}\label{ssec:cocart-arr}
			
			\input{cocart-arr}

			\section{Cocartesian families}

\input{cocart-fib}

			\section{Cocartesian functors}\label{sec:cocart-fun}
			
			\input{cocart-fun}

			\section{Cartesian families}
			\label{ch:cartfin}
			\input{cart-fib}

			\chapter{Bicartesian families of synthetic \texorpdfstring{$\inftyone$}{(∞,1)}-categories}
			\label{ch:bicart}
			\input{bicart-fib}

			\input{fl-fib}		
			\chapter{Two-sided cartesian families of synthetic \texorpdfstring{$\inftyone$}{(∞,1)}-categories}
			\label{ch:2scart}
			\input{two-sid-intro}
				\section{Sliced cocartesian families}
			\input{relcocart-fib}
		
			\section{Two-sided cartesian families}
			\input{two-sided-fibs}

			\section{Two-sided cartesian functors and closure properties}
			\input{two-sided-cartesian-functors}
						
			\section{Two-sided Yoneda Lemma}
			\input{two-sided-yon}

			\section{Discrete two-sided families}
			\input{two-sided-disc-fam}

	\chapter{On the semantics of simplicial homotopy type theory}
	\label{ch:sem}
			\section{Simplicial diagram models}\label{sec:sdiag-models}
			
			\input{univ-ttmt}

			\section{Strictly stable extension types}\label{sec:coh-ext}

\input{coh-ext}

	\chapter{Conclusion and Outlook}\label{ch:concl-outlook}
	
			\section{Conclusion}\label{sec:concl}
			\input{conclusion}
			\section{Outlook}\label{sec:outlook}
			\input{outlook}

	\appendix
		
	\chapter{Relative adjunctions}
	\label{app:reladj}
	
	\input{rel-adj}

	\chapter{Fibered and sliced constructions}
	\label{app:fibconstr}
	
	\input{fib-adj}

	\backmatter
	
	\phantomsection%
	\nocite{Ras18model,AFfib,GHT17,RV2cat,RVexp,RVscratch,RVyoneda,rasekh2021cartesian,RSSmod,LiBint,clementino2020lax,hermida1992fibred,rezk2017stuff,kock2013local,BorHandb2,BarwickShahFib,JoyQcat,JoyNotesQcat,B19,bergner_2018,joyal2007quasi,StrYon,MasarykFormal}
	
	\printbibliography[heading=bibintoc]

	\chapter*{Academic Curriculum Vitae}
	\label{sec:cv}
	\addcontentsline{toc}{chapter}{\nameref{sec:cv}}
	\input{cv}

\end{document}

%% file: abstract.tex
Reasoning about weak higher categorical structures constitutes a challenging task, even to the experts. One principal reason is that the language of set theory is not invariant under the weaker notions of equivalence at play, such as homotopy equivalence. From this point of view, it is natural to ask for a different foundational setting which more natively supports these notions.

A possible approach along these lines has been given by Riehl--Shulman in 2017 where they have developed a theory of synthetic $\inftyone$-categories. For this purpose they have introduced an extension of \emph{homotopy type theory/univalent foundations (HoTT/UF)}. Pioneered by Voevodsky, this logical system is designed to develop homotopy theory in a \emph{synthetic} way, meaning that its basic entities can essentially be understood as topological spaces, or more precisely \emph{homotopy types}. As per Voevodsky's \emph{Univalence Axiom}, homotopy equivalence between types coincides with logical equivalence. As a consequence, a lot of the conceptual ideas from classical homotopy theory can be imported to reason more intrinsically about homotopical structures.

In fact, in 2019 Shulman achieved to prove the long-standing conjecture that any higher topos (in the sense of Grothendieck--Rezk--Lurie) gives rise to a model of homotopy type theory. This gives a precise technical sense in which HoTT can be regarded as a kind of internal language of $\inftyone$-toposes $\mathscr E$. By analogy, Riehl--Shulman's extension, called \emph{simplicial homotopy type theory (sHoTT)} can be interpreted in diagram $\inftyone$-toposes of the form $\mathscr E^{\Simplex^{\Op}}$. Syntactic additions make it possible to reason type-theoretically about internal $\inftyone$-categories implemented as (complete) Segal objects.

Based on Riehl--Shulman's work about synthetic $\inftyone$-categories and discrete covariant fibrations, we present a theory of co-/cartesian fibrations including sliced and two-sided versions. The study is informed by Riehl--Verity's work on model-independent $\infty$-category theory, and transfers results of their $\infty$-cosmoses to the type-theoretic setting. We prove characterization theorems for cocartesian fibrations and their generalizations, given as existence conditions of certain adjoint functors~(\emph{Chevalley criteria}). Extending the author's previous joint work with Buchholtz, we prove several closure properties of two-sided cartesian fibrations as well as a two-sided Yoneda Lemma. Furthermore, we discuss so-called Beck--Chevalley (bi)fibrations and prove a synthetic $\inftyone$-categorical version of Moens' Theorem, after Streicher.

Finally, we show how to interpret Riehl--Shulman's strict extension types in the intended higher topos model, in such a way that makes them strictly stable under substitution. This uses a method, originally due to Voevodsky, which has previously been employed in the works of~\eg~Kapulkin, Lumsdaine, Warren, Awodey, Streicher, and Shulman.

Our work takes up on suggestions in the original article by Riehl--Shulman to further develop synthetic $\inftyone$-category theory in simplicial HoTT, including in particular the study of cocartesian fibrations. Together with a collection of analytic results, notably due to Riehl--Verity and Rasekh, it follows that our type-theoretic account constitutes a synthetic theory of fibrations of internal $\inftyone$-categories.

%% file: abstract-ger.tex
Der Umgang mit schwachen unendlichdimensionalen Kategorien stellt selbst für Expert*innen eine Herausforderung dar. Dies begründet sich schon auf grundsätzlicher Ebene damit, dass die Sprache der Mengenlehre nicht invariant unter den Äquivalenzbegriffen dieser schwachen, höherdimensionalen Strukturen ist, wie z.B.~Homotopieäquivalenz. Unter diesem Gesichtspunkt ist es naheliegend, nach einer alternativen Grundlagentheorie zu suchen, die mit den homotopietheoretischen Begriffen besser kompatibel ist.

Ein möglicher Ansatz wurde 2017 von Riehl--Shulman in Form einer synthetischen Theorie von $\inftyone$-Kategorien vorgeschlagen. Zu diesem Zweck haben sie eine Erweiterung der sogenannten \emph{Homotopie-Typentheorie} bzw.~\emph{univalenten Grundlagen} (engl.~\emph{homotopy type theory/univalent foundations (HoTT/UF)}) eingeführt. Hierbei handelt es sich um ein logisches System, entscheidend geprägt durch Wojewodski, das eine \emph{synthetische} Grundlage für Homotopietheorie bieten soll: Die grundlegenden Objekte tragen bereits eine topologische Struktur. Genauer gesagt, handelt es sich um \emph{Homotopietypen}. Wojewodskis \emph{Univalenz-Axiom} bewirkt, dass in diesem System Homotopieäquivalenz mit logischer Äquivalenz übereinstimmt. Dies hat zur Folge, dass viele konzeptionelle Ideen aus der klassischen Homotopietheorie benutzt werden können, um eine synthetische Theorie aufzubauen.

Ein entscheidender Beitrag wurde 2019 von Shulman geleistet, der bewies, dass jeder höherdimensionale Topos (im Sinne von Grothendieck--Rezk--Lurie) Anlass zu einem Modell von Homotopie-Typentheorie gibt. Diese Vermutung war lange Zeit offen. Nach diesem Resultat kann man HoTT in einem genauen technischen Sinne als eine Art interne Sprache beliebiger $\inftyone$-Topoi $\mathscr E$ verstehen. Analog dazu lässt sich Riehl--Shulmans Erweiterung, die sog.~\emph{simpliziale Homotopietypentheorie (sHoTT)}, in Diagramm-$\inftyone$-Topoi der Form $\mathscr E^{\Simplex^{\Op}}$ interpretieren. Vermöge syntaktischer Erweiterungen erlaubt die Theorie die Behandlung interner $\inftyone$-Kategorien, modelliert durch sogenannte (vollständige) Segal-Objekte.

Ausgehend von Riehl--Shulmans Arbeit über synthetische $\inftyone$-Kategorien und diskrete kovariante Fibrationen entwickeln wir eine Theorie ko-/kartesischer Fibrationen, die auch gefaserte und zweiseitige Verallgemeinerungen erfasst. Dabei dienen Riehl--Veritys Arbeiten über modellunabhängige $\inftyone$-Kategorientheorie als wichtige Grundlage, deren Resultate für $\infty$-Kosmoi wir auf den typentheoretischen Kontext übertragen. Wir geben Charakterisierungssätze für kokartesische Fibrationen und ihre Verallgemeinerungen an, die als Existenzsätze bestimmter adjungierter Funktoren formuliert sind \emph{(Chevalley-Kriterien)}. Gemeinsame Vorarbeiten mit Buchholtz verallgemeinernd zeigen wir eine Auswahl von Abschlusseigenschaften zweiseitiger kartesischer Fibrationen sowie ein zweiseitiges Yoneda-Lemma. Desweiteren diskutieren wir sogenannte Beck--Chevalley-(Bi-)Fibrationen und beweisen eine synthetische $\inftyone$-kategorielle Version des Satzes von Moens, basierend auf einem Beweis von Streicher.

Abschließend führen wir eine Kohärenzkonstruktion für Riehl--Shulmans strikte Erweiterungstypen durch, sodass diese auch in den homotopietheoretischen Topos-Modellen strikt substitutionsstabil interpretiert werden können. Diese Methode geht auf Wojewodski zurück und fand Anwendung in Arbeiten von z.B.~Kapulkin, Lumsdaine, Warren, Awodey, Streicher und Shulman.

Unsere Arbeit folgt Vorschlägen aus dem ursprünglichen Artikel von Riehl--Shulman, die synthetische $\inftyone$-Kategorientheorie in simplizialer Homotopie-Typentheorie weiterzuentwickeln, insbesondere bezüglich kokartesischer Fibrationen. Zusammen mit einer Reihe analytischer Resultate, v.a.~von Riehl--Verity und Rasekh, folgt aus den genannten Betrachtungen, dass unser typentheoretischer Zugang eine synthetische Theorie von Fibrationen interner $\inftyone$-Kategorien darstellt.

%% file: acknowledgments.tex
First and foremost, I am deeply indebted to my doctoral advisor Thomas Streicher. Ever since the early days of my undergraduate studies you have been an outstandingly supportive and passionate mentor to me. I am most grateful that you have been sharing your unique perspectives and intuitions with me. You have taught me to cut through the fog. Thank You, Thomas, for all your guidance, inspiration, compassion, patience, and for always being there for me in every way throughout the past twelve years.

Next, I owe the most profound gratitude to Ulrik Buchholtz. I express my highest gratefulness to you for acting as an unofficial doctoral advisor with all the more dedication, and for engaging in our comprehensive joint project. Thank You, Ulrik, for all your effort, energy, and commitment in working with and supporting me. I am greatly indebted to you for sharing your outstandingly vast horizon and perspectives with me, and for your steady and tireless involvement in our project for the past five years. Thank you for proof-reading a draft of this thesis.

Furthermore, I am exceedingly grateful to Emily Riehl. Thank You, Emily, for all the steady guidance, extremely insightful and inspiring discussions, and enthusiastic support in many ways. This thesis is built on your work, and I am immensely thankful for you sharing your ideas, suggestions, and feedback. Your engagement has been decisive for this thesis and our project in simplicial HoTT. I deeply appreciate you hosting me so welcomingly and dedicatedly on several visits.

Without the three of you this thesis would not have been possible.

I am furthermore substantially thankful to Mike Shulman and Dominic Verity, whose joint work with Emily Riehl constitutes the foundations of this thesis in the first place.

Special gratitude goes to Emily Riehl and Benno van den Berg for acting as referees to this thesis, and to Torsten Wedhorn and Steffen Roch for being available as examiners.

Distinguished acknowledgement also goes to Steve Awodey. Thank You, Steve, for lots of highly engaging discussions, eager support---both mathematical and non-mathematical---and for hosting my visit at CMU.

I am furthermore particularly greatful to Mathieu Anel. Thank You, Mathieu, for your great effort, support and numerous enlightening conversations. I highly appreciate to learn your perspectives on higher topos/logos and category theory.

Special thanks is owed as well to Paolo Capriotti. Thank You, Paolo, for all the enlightening discussions which always have been fruitful to me, and for great times here in Darmstadt and in Nottingham.

With deep gratitude I acknowledge the generous financial support of the Centre~for~Advanced~Study~(CAS) at the Norwegian~Academy~of~Science~and~Letters in Oslo, 
Norway, which funded and hosted the research project Homotopy Type Theory and Univalent Foundations during the academic~year~2018/19. I am indebted to Bj{\o}rn~Ian~Dundas and Marc Bezem for their hospitality on the occasion of several guest visits of the HoTT-UF project during which integral parts of the work in simplicial HoTT jointly with Ulrik Buchholtz has been carried out.

I also greatly appreciate the generous support of Thorsten Altenkirch which allowed me to visit the FP Lab at the University of Nottingham. Thank You, Thorsten, for your great hospitality, the many fruitful discussions, and the opportunity to give presentations of my work.

Great acknowledgment is owed to the mentors of the MIT Talbot Workshop 2018, Emily Riehl and Dominic Verity, to the organizers Eva Belmont, Calista Bernard, Inbar Klang, Morgan Opie, and Sean Poherence, and to all the participants. I am thankful to the MIT and Northwestern for providing the frame of this series, and for ample financial support. This highly memorable workshop was marked by a unique excellent and engaging atmosphere, and was of major importance in my work towards this thesis.

I am very thankful to Jonathon Funk and Noson Yanofsky for inviting to me to speak at the NYC Category Theory Seminar at CUNY, and for fruitful discussions.

Furthermore, I greatly thank Nima Rasekh for inviting me to speak at the EPFL Topology Seminar to present work of this thesis, and the Laboratory for Topology and Neuroscience for generous financial support.

I am highly grateful to Dimitri Ara and Andrea Gagna for letting me speak at the Aix--Marseille logic seminar, and for compelling ensuing discussions.

Further acknowledgment goes to Mathieu Anel and Jonas Frey for inviting me to speak about my work on sHoTT at the CMU HoTT Seminar and Graduate Student Workshop.

Through the stages of my doctoral studies, I was very lucky to enjoy many fruitful mathematical conversations from which I have benefited a lot. For those, I am additionally very thankful to Anthony Agwu, Benedikt Ahrens, Carlo Angiuli, Peter Arndt, David Ayala, Reid Barton, Gershom Bazerman, Benno van den Berg, Martin Bidlingmaier, Auke Booij, Pierre Cagne, Alexander Campbell, Tim Campion, Evan Cavallo, R\'{e}my Cerda, Felix Cherubini, Denis-Charles Cisinski, tslil clingman, Bastiaan Cnossen, Johan Commelin, Thierry Coquand, Shanna Dobson, Daniel Fuentes-Keuthan, Jonas Frey, Nicola Gambino, Daniel Gratzer, Sina Hazratpour, Simon Henry, Simon Huber, Andr\'{e} Joyal, Chris Kapulkin, Nicolai Kraus, Nikolai Kudasov, Edoardo Lanari, Dan Licata, Fosco Loregian, Catrin Mair, Aaron Mazel-Gee, Anders Mörtberg, Lyne Moser, David Jaz Myers, Paige Randall North, Andreas Nuyts, Ian Orton, Anja Petkovi\'{c} Komel, Gun Pinyo, Andrew Pitts, Moritz Rahn, Nima Rasekh, Jakob von Raumer, Timo Richarz, Egbert Rijke, Mitchell Riley, Martina Rovelli, Christian Sattler, Jay Shah, Mike Shulman, Bas Spitters, Raffael Stenzel, Jonathan Sterling, Chaitanya Leena Subramaniam, Andrew Swan, Dominic Verity, Matthew Weaver, Torsten Wedhorn, Chuangjie Xu, and Colin Zwanziger.

For further non-mathematical support I am highly grateful to Benedikt Ahrens, Kord Eickmeyer, Anton Freund, Simon Henry, Martin Otto, Viktoriya Ozornova, Thomas Powell, Martina Rovelli, Sam Sanders, and Mike Shulman.

I will forever remember my time as a member of the logic group at TU~Darmstadt. I owe essential gratitude to Betina Schubotz. Thank You, Betina, for your kind, warm and excellent help and support in all the ways.  

Times would not have been the same without my dear colleagues, especially Julian Bitterlich, Ulrik Buchholtz, Felix Canavoi, Paolo Capriotti, Anton Freund, Angeliki Koutsouko-Argyraki, Pedro Pinto, Thomas Powell, Sam Sanders, Matthias Schröder, Andrei Sipoș, and Florian Steinberg. Thank You all for unforgettable days of logic and beers!

Furthermore, I thank all the participants of the logic group's Doktorandentreff for your participation in whatever way.

My deepest gratitude is owed to my dear friends and family. Thank You for being who you are, and for standing by my side. Without you, I would have not made it this far. 

To my parents Andrea and Anton: \emph{Danke für alles.}

~ \\ ~ \\ ~ \\
\emph{Jonathan Weinberger} \\
\emph{Darmstadt, October 2021 and February 2022}

%% file: summ.tex
\section{Summary of the results}

The \emph{synthetic perspective on $\inftyone$-categories} alluded to in the title is not new to this thesis, but has been established previously in~\cite{RS17}---with the central idea of using as a model the $\inftyone$-topos of simplicial spaces, in order to later in the type theory carve out the (complete) Segal types independently suggested by Joyal as well. The thesis presents further developments along these lines, relying strongly on aspects of the highly expansive foundational work in~\cite{RV21}, transferring it to the type-theoretic setting.

For the most part, this dissertation concerns the synthetic development of \emph{fibered $\inftyone$-category theory}. Despite the discussion of the actual semantics making up only a small portion of the text, we still have chosen to mention the \emph{semantics} in the title, to emphasize the immediate semantical analogies. Unlike other possibly more syntactically motivated approaches to synthetic higher category theory, simplicial HoTT is clearly semantically motivated. With the right ambient  logical and ``analytic'' theory preestablished, one readily sees that the (complete) Rezk types are clearly interpreted as (complete) Segal objects. Particularly, using the $\infty$-cosmological, formal higher catgeorical perspective, more intricate notions such as adjunctions, cartesian fibrations~\etc.~also translate to their correct semantic counterparts.

In the following, we summarize the results of each thesis chapter. For the general narrative we refer to the non-technical introduction~\Cref{ch:intro}.

\subsection{Co-/cartesian families}

In~\cite{RS17}, Riehl and Shulman have developed a synthetic account to left fibrations of $\inftyone$-categories. A left fibration $E \fibarr B$ presents a (covariant) presheaf $B \to \Space$, valued in the $\inftyone$-category of (small) $\infty$-groupoids (\aka~spaces)---the $\infty$-categorical analogue of the category of sets. We introduce a notion of \emph{cocartesian families},~\ie~covariant presheaves $B \to \Cat$ valued in the $\inftyone$-category of (small) $\inftyone$-categories.\footnote{Note that this is to be read \emph{cum grano salis}: At that point, we do not have a satisfactory internal treatment of the classifiers $\Space$ and $\Cat$ yet. This means, a type-theoretic account of these categorical universes is still under way. But these objects exist externally as is well known from ``analytic'' higher-category theory,~\cf~\cite{Ras18model} specifically for the case of (complete) Segal spaces. In fact, for the ``naive'' universes defined as $\Sigma$-types such as $\sum_{X:\UU} \isRezk(X)$ or $\sum_{X:\UU} \isDisc(X)$ it is \emph{not derivable} that they be (complete) Segal. This has been observed in simplicial spaces initially by Shulman, and later verified for a few low-dimensional special cases in unpublished work by Buchholtz and the author~\cite{BW18a}. Already for discrete types in simplicial reflexive graphs, the naive universe contains ``too many arrows'' to be Segal.} This is based on previous joint work of the author together with Buchholtz~\cite{BW21}, which contains some parts that also occur in this thesis. However, the proofs presented here have been established by the author of the thesis. To provide an essentially self-contained treatise, we also have included an expository section from~\cite[Section~2]{BW21} in the thesis at hand as well as, occasionally, some other expository paragraphs from~\emph{loc.~cit.} These have originally been written by the thesis author as well.

Generalizing classical results, and paralleling the study of~\cite[Chapter~5]{RV21}, cocartesian type families are defined as type families where one can lift arrows in the base type to dependent cocartesian arrows in the family (or, equivalently in the total type). We prove that this, as expected, can be equivalently expressed as a criterion involving the existence of a certain left adjoint right inverse (abbrev.~\emph{LARI}), \aka~\emph{Chevalley criterion}, and a different condition postulating the existence of a fibered left adjoint. These results are extended to cocartesian functors between cocartesian fibrations. These characterizations have immediate important consequences, since they entail various closure properties of cocartesian families and functors, as detailed in~\cite{BW21}. In particular, these are a type-theoretic version of the closure properties of cocartesian fibrations in $\infty$-cosmos theory.

Moreover, compared to~\cite{BW21}, we present an alternative account to \emph{cocartesian arrows}, more along the lines of~\cite[Section~5.1]{RV21}. Namely, as recently discovered by Riehl and Verity, cocartesian \emph{arrows} can also be characterized in terms of a Chevalley criterion in terms of \emph{relative adjunctions} \`{a} la Ulmer (which are briefly treated in the appendix). This yields a more formal proof of the closure properties of cocartesian arrows, and ties in with the discussion of ``LARI fibrations'' as formal generalizations of cocartesian fibrations.

This treatment of cocartesian arrows has been suggested to us by Riehl.

Note that analogous results for cartesian fibrations,~\ie~contravariant $\Cat$-valued $\infty$-presheaves, follow by (manual) dualization. Since at present, the type theory at hand lacks a type former for the opposite category, contravariant maps over $B$ still have to be encoded as families over $B$, with contravariance produced by the dual Chevalley criterion. In particular, this gives rise to a notion of cartesian and simultaneously cocartesian \aka~\emph{bi-cartesian} family.

\subsection{Beck--Chevalley families and Moens' Theorem}

As an application of our theory of synthetic cocartesian families, we provide an account of Beck--Chevalley fibrations and Moens' Theorem by adapting Streicher's methods~\cite[Section~15]{streicher2020fibered} from the $1$-categorical case to our type-theoretic setting. A \emph{Beck--Chevalley fibration} is a bicartesian fibration which moreover satisfies a fibrational version of the traditional Beck--Chevalley condition, crucial~\eg~in categorical logic. If, moreover a Beck--Chevalley fibration is lex and satisfies some further conditions, generalizing the so-called \emph{extensivity property} for categories, then it is called a \emph{lextensive} or \emph{Moens fibration}. By Moens' Theorem, lextensive fibrations over a given base $B$ are classified by lex functors \emph{from} $B$ to some lex category.

Although the Beck--Chevalley condition has had some appearance in $\inftyone$-category theory~\cite{LurAmbi}, this, to our knowledge, has not been the case for Moens' Theorem so far. In the classical theory, Moens' Theorem plays an important role for the so-called \emph{fibered view of geometric morphisms}~\cite{StrFVGM,streicher2020fibered}. Our generalization to the case of Rezk types externalizes to a version of Moens' Theorem for internal $\infty$-categories in an arbitrary $\infty$-topos.\footnote{in the sense of Grothendieck--Rezk--Lurie} We hope that this yields a starting point for further investigations along these lines in higher topos theory---be they synthetic or analytic.

\subsection{Two-sided cartesian families}

We generalize our study of co-/cartesian families $E \fibarr A$ to the $2$-variable case of spans $A \stackrel{\xi}{\twoheadleftarrow} E \stackrel{\pi}{\fibarr} B$ with \emph{mixed variance}. This constitutes a theory of \emph{two-sided cartesian fibrations}, where one leg of such a span $\xi: E \fibarr A$ is cocartesian, $\pi:E \fibarr B$ is cartesian, and some further compatibility conditions between the lifts are required. Morally,\footnote{Recall, that we do not currently have an ``$\mathrm{op}$'' in our theory.} such a two-sided fibration encodes a functor $B^{\Op} \times A \to \Cat$. This also has a discrete version, called \emph{distributor} or \emph{(bi-)module},\footnote{Common synonyms also include \emph{profunctor} or \emph{relator}, since they can be seen as relations between categories, \cf~\cite{StrDist}.} where the legs are \emph{discrete} co-/cartesian fibrations, resp. Accordingly, a distributor can be presented as a functor $B^{\Op} \times A \to \Space$ or as a \emph{presheaf of copresheaves} $B^{\Op} \to A \to \Space$. The prime example of a distributor is given by the hom-bifunctor $\hom_A: A^{\Op} \times A \to \Space$.

Analogously to our study of cocartesian families, \cf~also \cite{BW21}, we develop the theory of two-sided cartesian families and their discrete versions in simplicial type theory, after Riehl--Verity's installment for $\infty$-cosmoses~\cite[Chapter~7]{RV21}. To achieve a similar analysis, we have to give some rather explicit discussions of fibered or \emph{sliced} notions of synthetic co-/cartesian fibrations.\footnote{This is particularly owed to the fact that we currently do not have the appropriate categorical universes at hand, so we cannot reason about fibrations as literally being objects of some $\infty$-category, which is in contrast to $\infty$-comos theory.} The payoff will be a systematic analysis of the notions of two-sided fibrations and functors \`{a} la \cite{RV21}, encompassing namely Chevalley-like characterization theorems and closure properties. Furthermore, we generalize the synthetic Yoneda Lemmas from \cite{RS17,BW21} from the non-/discrete, one-sided case to a synthetic version of the two-sided case, as given in \cite[Section~7.3]{RV21}.

This involves some results about fibered (LARI) fibrations between Rezk types, generalizing~\cite[Section~11]{RS17} and \cite[Appendix~B]{BW21} which are developed in the appendix.

In our setting, we allude to internal versions of the operations of this (externally present) double-category.

\subsection{Semantics: Strict stability of extension type formers}

By the results of Riehl--Shulman~\cite[Appendix~A]{RS17}, simplicial type theory can be modeled in suitable model structures equipped with an appropriate shape theory. However, their study technically does not include universe types. On the other hand, building on previous work, Shulman~\cite{Shu19} established positively the long-standing conjecture that any Grothendieck $\inftyone$-topos gives rise to a model of homotopy type theory with universes strictly \`{a} la Tarski.

Let $\mathscr E$ be an $\inftyone$-topos. Then it follows that the $\inftyone$-topos $\sE := \mathscr E^{{\Simplex}^{\Op}}$ of simplicial objects in $\mathscr E$ \emph{almost} gives rise to a model of simplicial homotopy type theory. In light of the previous work, what's missing\footnote{setting aside questions about initiality of syntax} is a construction interpreting the extension types \`{a} la Riehl--Shulman in such a way that they are strictly stable under substitution.

We provide such a construction generalizing a splitting method originally due to Voevodsky~\cite{VVTySys,VVPi}, and later well investigated for HoTT by~\cite{KL18,LW15,Awo18}.

Since in the models $\sE$, all the shapes are fibrant, we simplify our setting so as to work completely inside the model structure presenting $\sE$. The advantage is that we can work inside an ordinary comprehension category, rather than deal with extra fibered, possibly non-fibrant structure, and separate classifiers for the ``cofibrations''. This should make the construction as transparent as possible, and also ready for adaptation to technically more intricate settings. 

We conclude by discussing how our synthetic type-theoretic notions translate in the models $\sE$. In particular, any $\infty$-topos $\mathscr E$ gives rise to an $\infty$-cosmos of \emph{Rezk objects} internal to $\mathscr E$, which can be understood as the \emph{internal $\inftyone$-categories in $\mathscr E$}.\footnote{\cf~ \cite[Example~A.14]{RS17}} Thus, simplicial HoTT captures (some of) the synthetic theory of internal $\inftyone$-categories in an $\infty$-topos, similarly to how standard HoTT captures (some of) the synthetic theory of homotopy types in an $\infty$-topos.

\section{Structure of the thesis}
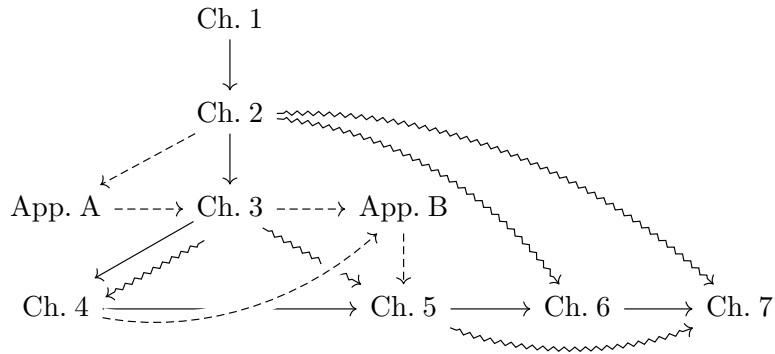
\begin{figure}
\[\begin{tikzcd}
	&& {\mathrm{Ch.~1}} \\
	&& {\mathrm{Ch.~2}} \\
	{} & {\mathrm{App.~A}} & {\mathrm{Ch.~3}} & {\mathrm{App.~B}} & {} \\
	& {\mathrm{Ch.~4}} && {\mathrm{Ch.~5}} & {\mathrm{Ch.~6}} & {\mathrm{Ch.~7}} \\
	&&&&& {} && {} \\
	&&&&&& {}
	\arrow[from=1-3, to=2-3]
	\arrow[from=2-3, to=3-3]
	\arrow[shift left=2, squiggly, from=3-3, to=4-2]
	\arrow[squiggly, from=3-3, to=4-4]
	\arrow[curve={height=-30pt}, squiggly, from=2-3, to=4-6]
	\arrow[from=4-2, to=4-4]
	\arrow[shift right=1, from=3-3, to=4-2]
	\arrow[dashed, from=2-3, to=3-2]
	\arrow[dashed, from=3-2, to=3-3]
	\arrow[dashed, from=3-3, to=3-4]
	\arrow[dashed, from=3-4, to=4-4]
	\arrow[from=4-4, to=4-5]
	\arrow[curve={height=-24pt}, squiggly, from=2-3, to=4-5]
	\arrow[curve={height=18pt}, squiggly, from=4-4, to=4-6]
	\arrow[from=4-5, to=4-6]
	\arrow[curve={height=24pt}, dashed, from=4-2, to=3-4, crossing over]
\end{tikzcd}\]
	\caption{Structure of the thesis}
	\label{fig:struct-thesis}
\end{figure}
In~\Cref{fig:struct-thesis}, we suggest a few paths to take reading this thesis. A ``complete tour'' is suggested by the regular arrows ``$\rightarrow$''. One can take  detours over the appendices as indicated via ``$\dashrightarrow$'', or simply consult them as needed. It is also possible to do straight ``jumps'' via the squiggly arrows ``$\rightsquigarrow$'', depending on individual interest. However, in any case we do recommend reading~\Cref{ch:intro} (and~\Cref{ch:prelim}) to set the stage. \Cref{ch:bicart,ch:2scart} are not dependent on each other, but they both rely on~\Cref{ch:cocart} where co-/cartesian fibrations are introduced. \Cref{ch:sem} is the part primarily concerned with semantics. As such it can be read on its own right. A conclusion and an outlook can be found in~\Cref{ch:concl-outlook}. The appendices contain mainly technical results. Namely, \Cref{app:reladj} provides a treatise of a version of relative adjunctions after Ulmer~\cite{UlmDense}. This is also interesting from a conceptual point of view because it gives a more general perspective on cocartesian arrows. Finally, in \Cref{app:fibconstr} we collect several results on fibered (LARI) adjunctions and equivalences, as well as some sliced constructions that will be used often in the main text.

\section{List of central results}

\subsection{Cocartesian families}

\begin{itemize}
	\item \Cref{thm:cocart-fams-intl-char}, p.~\pageref{thm:cocart-fams-intl-char}:~Characterization of cocartesian families (Chevalley criterion)
	\item \Cref{thm:cocartfams-via-transp}, p.~\pageref{thm:cocartfams-via-transp}:~Characterization for cocartesian families (fibered adjunction criterion)
	\item \Cref{thm:char-cocart-fun}, p.~\pageref{thm:char-cocart-fun}:~Characterization of cocartesian functors
\end{itemize}
\subsection{Bicartesian families}
\begin{itemize}
	\item \Cref{thm:moens-thm}, p.~\pageref{thm:moens-thm}:~Moens'~Theorem
\end{itemize}
\subsection{Two-sided families}
\begin{itemize}
	\item \Cref{thm:char-two-sid}, p.~\pageref{thm:char-two-sid}:~Characterization of two-sided cartesian families
	\item \Cref{thm:2scart-cosm-closure}, p.~\pageref{thm:2scart-cosm-closure}:~Closure properties of two-sided cartesian families and functors
	\item \Cref{thm:dep-yon-2s}, p.~\pageref{thm:dep-yon-2s}:~Two-sided Yoneda Lemma
	\item \Cref{prop:2s-disc}, p.~\pageref{prop:2s-disc}: Two-sided discrete cartesian families as discrete objects
\end{itemize}
\subsection{Semantics}
\begin{itemize}
	\item \Cref{thm:coh-ext}, p.~\pageref{thm:coh-ext}:~Coherence for extension types
\end{itemize}

\section{Related work}

Apart from the original paper by Riehl--Shulman~\cite{RS17} there is follow-up work on directed univalence in simplicial HoTT by Cavallo--Riehl--Sattler~\cite{CRS}. Directed Univalence in a bi-\emph{cubical} type theory has been studied by Weaver--Licata~\cite{WL19} where they prove directed univalence for the classifier of covariant discrete families, with ongoing work towards directed HITs~\cite{WeaDHIT}.

Inspired by Voevodsky's~\emph{homotopy type system (HTS)}, higher categories have been studied type-theoretically in \emph{2-level type theory (2LTT)}~\cite{2ltt,CapriottiPhD}. In this approach, one considers simultaneously a layer for HoTT and a layer for extensional type theory validating UIP, together with a cocercion function from the former to the latter. 

Directed versions of (homotopy) type theory have been studied by several people~\cite{War-DTT-IAS,LH2DTT,LicPhD,NuyMSc,NuyPhD,NorthDHoTT,ASNatFree}. Kavvos discusses ideas towards directed type theory capturing two-sided fibrations in~\cite{KavQuant}. On the semantic side, there are ongoing investigations by Cisinski and Nguyen~\cite{CisHoTTEST,NguyPhD,NguyCoContra} about directed univalence. There is also ongoing joint work by North, van den Berg, and McCloskey about directed (algebraic) weak factorization systems~\cite{NorDirectedTalk21}.

As an overarching theory, model-independent foundations of higher category theory have been established in Riehl--Verity's $\infty$-cosmos theory~\cite{RV21}.

Preceding work to this thesis has been done by Buchholtz in joint work with the thesis author~\cite{BW21}, parts of which make up~\Cref{ch:cocart} of this thesis.

Furthermore, a proof assistant for simplicial type theory is being developed by Kudasov~\cite{KudRzk}. 

\section{Declaration of authorship}

The technical chapter with the technical preliminaries~\Cref{ch:prelim} has essentially been carried over from the thesis author's joint work with Ulrik Buchholtz~\cite[Section~2]{BW21}. All the parts of this in the present text have originally been written by the author of this thesis.

Occasionally present in the text are expository and some technical parts that are also taken from~\cite{BW21} and have, again, originally been written by the thesis author.

\Cref{ch:cocart} to a large extent consists of material from~\cite{BW21}, esp.~Section~5~\emph{ibid.} The theorems stated and proven here have been stated and proven by the thesis author.

The \LaTeX\,template of this thesis uses the \emph{memoir} document class and is based on Egbert Rijke's PhD thesis~\cite{RijPhd}. Most if not all commutative diagrams have been typeset using the brillant \href{https://q.uiver.app/}{quiver} web editor developed by Nathanael Arkor. The author is most indebted for this free and open tool which makes typesetting diagrams exceedingly efficient and comfortable. The sketches of fibrations in TikZ are due to Ulrik Buchholtz. Some macros for extension types have been taken from Emily Riehl and Mike Shulman's paper~\cite{RS17}.

%% file: intro.tex
\section{Higher structures in mathematics}

Higher homotopical structures have been gaining relevance in modern-day mathematics. For example, in classical algebraic topology, the fundamental group of a pointed space is a very important invariant. But this construction loses information about the higher homotopies. Furthermore, it can be argued that its construction is not particularly natural: it relies on choosing a basepoint, parametrizations of loops, and involves taking a quotient. An improvement over this is given by the \emph{fundamental $\infty$-groupoid} or \emph{singular Kan complex} of a space, whose $n$-simplices are given by the $n$-dimensional paths in the space simply by definition.

The increasing complexity of mathematical objects studied also adds to the complexity of the structures formed by these objects. This is already visible at the level of bare sets: the collection of (small) sets naturally forms a (large) groupoid rather than a (large) set again, since the preferred notion of identity is \emph{isomorphism} rather than plain equality on the nose. This is true even more so for more complicated objects. Many objects in areas such as topology, algebraic geometry, differential geometry or even theoretical computer science naturally form \emph{weak higher-dimensional category}. In this case, contrasting ordinary $1$-dimensional categories, composition is defined in some weaker sense than as an operation on sets, and the respective laws do not hold equationally but up to coherence as witnessed by some higher-dimensional morphisms---possibly in arbitrarily high dimensions.

But the problem when reasoning about weak structures of this kind in traditional set-theoretic mathematics is that the definitions tend to become very involved. Since at the fundamental level set-theory is dealing with honest equality, these weak higher-dimensional structures have to be built out of rigid ones, often producing very involved combinatorics and the necessity to keep track of higher coherences in an explicit way. In addition, for most of these higher structures of interest there do exist various different mathematical implementations which are equivalent in some precise technical sense, but hard to compare ``directly''---not least because they do live in different categories. A remarkable amount of traditional lower-dimensional category and topos theory has been successfully generalized to weak higher-dimensional categories, producing new ramifications not present in the lower-dimensional case. But because of the different ``competing'' models\footnote{in the sense of homotopy theory, not logic/model theory} it has always been a challenge to develop the theory in a ``uniform'' manner: indeed, it is often the case that those arguments involve switching back and forth between different models---or, despite coming from traditional $1$-categorical ideas initially, one gets caught up in the challenging technicalities of the chosen model when making things fully precise.

\section{Model-independent $\infty$-category theory}

On the one hand, this situation does have the advantage that one can deliberately exploit the different practical advantages of one model over the other. On the other hand, for both technical and philosophical reasons, it seems desirable to aim for unification in the theory of higher categories and homotopical structures, at least to the extent of making the parallels to the classical theory as transparent as possible. One approach to this has been the long-standing program of Riehl and Verity, who have developed a vastly comprehensive theory of higher categories via \emph{$\infty$-cosmoses}. The idea is to work \emph{synthetically}: instead of trying to axiomatize the definition of a higher category, they do so for the supposed \emph{structures} the higher categories should live an. An \emph{$\infty$-category}, of whatever flavor,\footnote{\eg~$(\infty,n)$-categories for $0 \le n \le \infty$} is then simply an object of the respective $\infty$-cosmos. \Eg~for the very prominent case of $\inftyone$-categories this provides a unified, model-independent treatment encompassing simultaneously their different incarnations as quasi-categories, complete Segal spaces, Segal categories, and $1$-complicial sets. A crucial aspect of this theory is that a lot of the arguments (sometimes with extra care) can even be carried out in the \emph{homotopy $2$-category} of the $\infty$-cosmos---a strict $2$-categorical quotient, comparable by analogy to the homotopy ($1$-)category of a model category. For instance, remarkably the data of a $2$-adjunction in the homotopy $2$-category lifts to define a fully homotopy coherent adjunctions in the $\infty$-cosmos~\cite{RVCohAdjMnd}. Furthermore, there is a notion of cosmological functor and cosmological biequivalence between $\infty$-cosmoses which allows to transfer constructions and theorems across different models. The slogan is that equivalent $\infty$-cosmoses have equivalent $\infty$-category theories (\emph{Model Independence Theorem},~\cite[Theorem~11.1.6]{RV21}). Again, this is analogous to how model categories encode the same homotopy theories in case they are Quillen equivalent.

An $\infty$-cosmos is defined as an $\inftytwo$-category with certain $\inftytwo$-categorical limits.\footnote{In technical terms, it is presented as a (sufficiently complete) fibration category, enriched over the Joyal-model structure of quasi-categories.} They have succesfully been used to develop a synthetic $\infty$-category theory, however, using structures that are based on set-theoretic mathematics and previously known ``analytic'' results about concrete models of $\inftyone$-categories. This development has evolved quite far, but one can still ask: is there an approach to $\infty$-category theory \emph{within} a ``homotopy-invariant'' theory? An approach in this direction is \emph{homotopy type theory (HoTT)}, a logical system, which has developed in parallel to the rapidly expanding study of higher categories in the early 2000s.

\section{Homotopy type theory and univalent foundations}

Going back to Bertrand Russell, type theory was devised in the early 1900s as a logical system to avoid logical paradoxes that at shook the foundations of mathematics at the time. The basic entities in a type theory are called \emph{types} $A,B,C,\ldots$ and they can have \emph{elements} (also: \emph{inhabitants} or \emph{terms}). Basic judgments consist~\eg~of declarations of the form
\[ x:A \]
saying that $x$ is an element of type $A$. The logical content of a type theory is given by a set of rules governing standard procedures such as weakening, substitution~\emph{etc}. Out of given types, new types can be formed, again through given rules that capture the (universal) properties that the newly constructed types should have. \Eg~given two types $A,B$ one might form their product $A \times B$, coproduct $A + B$ or function type $A \to B$. This is different to the more \emph{materialist} spirit of set theory, since it captures a more \emph{structuralist} style of doing reasoning and constructions.  

Of central importance is a variant due to Martin-Löf~\cite{MLTT}, named after its creator. Martin-Löf type theory MLTT is a constructive dependent type theory (\ie~it captures \emph{type families} depending on other types) that internalizes proof-relevant equality using so-called \emph{identity types} $\Id_A$. In fact, when developing even simple fragments of arithmetic inside MLTT, one will encounter that a lot of expected equalities between elements $a,b:A$ do not hold \emph{judgmentally} (\ie~on the nose), but only \emph{propositionally}, as witnessed by a proof term $p:\Id_A(a,b)$. Compared to a strict equality judgment $a \jdeq b$, the identity type $\Id_A(a,b)$ in general contains much richer information since two elements may be equal or isomorphic to each other in more than one way. In particular, this makes MLTT an \emph{intensional} type theory.\footnote{It can be made \emph{extensional} by reflecting propositional equality into judgmental equality, but this renders type-checking undecidable.}

Any type $A$ admits a family of identity types indexed over $A \times A$. Since any $\Id_A(a,b)$ is again a type, this construction can be iterated to yield a hierarchy of types capturing identity proofs between identity proofs (between identity proofs~\emph{etc.}):
\[ A, \Id_A, \Id_{\Id_A}, \Id_{\Id_{\Id_A}}, \ldots\]
One can now ask: In MLTT, given elements $a,b:A$ and two proofs $p,q:\Id_A(a,b)$, is it derivable that there is an identity proof of identity proofs in the next dimension, \ie~a term $\alpha:\Id_{\Id_A(a,b)}(p,q)$? This was answered negatively in 1994 by Hofmann and Streicher who constructed the groupoid model~\cite{HS94} which refuted the claim. In this model, types were interpreted as groupoids, families of types as fibrations of groupoids, and identity types as the sets of isomorphisms. In fact, a motivation to come up with this model was the observation that, inside MLTT, the identity types give rise to define a groupoid-like structure on every type. The units are given by the reflexivity proofs $\refl_{A,a} : \Id_A(a,a)$, and one can define maps $\mathrm{comp}_{A,a,b,c}:\Id_A(b,c) \to \Id_A(a,b) \to \Id_A(a,c)$ and $\mathrm{inv}_{A,a,b}: \Id_A(a,b) \to \Id_A(b,a)$, \resp, for composition and inversion, \resp~However, the expected groupoid laws only hold up to higher propositional equalities.

All these considerations suggested the slogan that ``types are (weak) higher-dimensional groupoids''. This is also reminiscent to Grothendieck's \emph{Homotopy Hypothesis} stating that ``spaces are higher-dimensional groupoids''. Around 2006, independently Voevodsky and Streicher suggested to model intensional type theory in \emph{Kan complexes} since those were known from categorical homotopy theory as a notion of $\infty$-groupoids. This connected well to previous work by Awodey--Warren~\cite{AW05} who had shown that MLTT can be seen as an internal language of any model category, interpreting type families as fibrations, and interpreting identity types as factorizations of diagonals. Furthermore, van den Berg and Garner showed in 2008 that, internally in MLTT, each type can be regarded as an $\omega$-groupoid in the sense of Batanin. Also in 2008, Gambino and Garner showed that every identity type gives rise to a weak factorization system on the syntactic category, \cf~also later work by Emmenegger~\cite{EmmIdWFS}.

Further work by Voevodsky and Kapulkin--Lumsdaine up to 2012 led to establishing a model of \emph{Homotopy Type Theory (HoTT)}, an extended version of MLTT, in the Kan model structure of simplicial sets. A crucial addition due to Voevodsky is the so-called \emph{Univalence Axiom} which roughly can be subsumed under the slogan \emph{Isomorphic types are equal}.

Since the Kan model structure is also a presentation of the $\inftyone$-topos of spaces, Awodey had generalized the internal language correspondence for MLTT and model categories to the following conjecture: Any $\inftyone$-topos admits a model of homotopy type theory.

A lot of work about the semantics of HoTT has been done in the past decade, and is still ongoing. This very notably includes Shulman's positive confirmation of Awodey's conjecture in~\cite{Shu19}. Statements of these kind prove particularly challenging, since to achieve a model of type theory, the interpretation has to be given in a coherent way that is stable under substitution up to honest set-theoretic equality. This is nontrivial for model categories or similar structures, because these typically do not come with distinguished, coherent choices.

On the internal or synthetic side, a lot of progress has been going on working internally in HoTT to develop synthetic accounts to homotopy theory. Many results have successfully been formalized and verified in a computer proof assistant.

\section{Synthetic $\inftyone$-category theory in simplicial homotopy type theory}

Given the discussion up to this point, one might ask if HoTT was suited as a homotopy-invariant language for reasoning about higher categories.

However, reasoning about higher category theories in HoTT is somehow underdeveloped to this date. A related problem is defining homotopy-coherent structures directly in HoTT, which has still been proven hard in general over the years. There do exist approaches, \eg~in \emph{Two-level type theory} (going back to Voevodsky's \emph{Homotopy Type System (HTS)}). There one adds another \emph{extensional/non-fibrant layer} to the theory, allowing to work partially ``classically'' as needed, as opposed to solely up to homotopy.

Another approach has been given in~\cite{RS17} under the name of \emph{Simplicial Homotopy Type Theory (sHoTT)}. Somewhat similarly, one also adds extra layers of non-fibrant pre-types to the theory. These shall capture the \emph{shapes} known from simplicial homotopy theory, generated by the standard simplices $\Delta^n$ and their subpolytopes (such as the boundaries $\partial \Delta^n$ or the horns $\Lambda_k^n$). As a further gadget, one adds to the theory type formers which capture \emph{strict} extensions of (partial) sections in a family along a shape inclusion. For instance, this allows for defining the hom-types of a type $A$ as
\[ \hom_A(x,y) \defeq \ndexten{\Delta^1}{A}{\partial \Delta^1}{[x,y]},\]
whose terms are directed arrows $f:\Delta^1 \to A$ that coincide \emph{strictly} with $x$ and $y$ on the boundary, $f\,0 \jdeq x$ and $f\,1 \jdeq y$. Such a type, with strict computational behavior, is not definable in standard HoTT, where instead these equalities would only hold up to paths $p:\Id_A(f\,0,x)$ and $q:\Id_A(f\,1,y)$, as a consequence producing unwieldy higher coherences.\footnote{This setup also has some parallels with developments of \emph{Cubical Type Theories}~\cite{BCH,CCHM2018} due to Coquand \emph{et al.}, \cf~\cite[Remark~3.2]{RS17}.}

In this setting, the way to reason about synthetic $\infty$-categories is as follows. Using the extension types, one can form in particular function types such as $A^\Phi$ for any (non-fibrant) tope $\Phi$. Using Voevodsky's notion of \emph{type-theoretic weak equivalence} $f: X  \stackrel{\equiv}{\to} Y$, Riehl--Shulman define a (simplicial) type $A$ to be a \emph{synthetic pre-$\inftyone$-category} or \emph{Segal type} if the induced map
\[A^{\Delta^2} \to  A^{\Lambda_1^2} \]
is a weak equivalence. This map restricts any $2$-simplex (triangle with filled interior) in $A$ to the sub-diagram of shape $(\bullet \to \bullet \to \bullet)$ (\emph{$(2,1)$-horn}). Hence, this restriction being a weak equivalence means that, in $A$, any such pair of composable arrows possesses a composite, uniquely up to contractibility~(\ie~the space of all possible such composition data is homotopy-equivalent to the point).

This constitutes a synthetic version of \emph{Segal spaces}. Adding the so-called \emph{Rezk-completeness} condition\footnote{Namely, this says that ``categorical isomorphism'' (defined through the hom-types) coincides with ``homotopy equivalence'' (given through the identity type). For this reason, the condition can also be called ``local unvialence'', after Voevodsky's Univalence Axiom.} gives rise to \emph{synthetic $\inftyone$-categories} or \emph{complete Segal types}, or \emph{Rezk types}.

Analytically, Rezk spaces are known to also present $\inftyone$-categories, just as quasi-categories do, but based on \emph{bisimplicial} sets, rather than simplicial sets. This idea was also independently suggested by Joyal.

Recall that HoTT can be modeled in the Kan model structure on simplicial sets $\sSet = \mathbf{Set}^{{\Simplex}^{\Op}}$, presenting the $\inftyone$-topos of spaces $\mathscr S$. Analogously, simplicial HoTT can be modeled in the so-called \emph{Reedy model structure} on bisimplicial sets $\mathbf{Set}^{{\Simplex}^{\Op} \times {\Simplex}^{\Op}}  \cong  \sSet^{{\Simplex}^{\Op}}$, presenting the $\inftyone$-topos $\mathscr S^{{\Simplex}^{\Op}}$ of simplicial spaces. The simplicial types, interpreted by Reedy fibrant simplicial spaces, do not have intrinsic meaning for us per se. But importantly, they are presented by a model structure interpreting all of HoTT.\footnote{Interpreting directly in the model structure presenting Rezk spaces is not possible, because it does not support general $\Pi$-types, as is also the case with the Joyal model structure on simplicial sets, capturing the quasi-categories.} In our consideration, we will often restrict to (complete) Segal types, which can be done internally in the newly extended theory, since (complete) Segal-ness becomes a definable predicate.

In~\cite{RS17} Riehl--Shulman have developed a lot of basic synthetic $\inftyone$-category theory in this setting, including a study of $\infty$-(co-)presheaves and adjunctions. We present here extensions of this development, on the one hand taking place in the type-theoretic setting, on the other hand adapting and aiming to parallel parts of the theory in $\infty$-cosmoses from Riehl--Verity~\cite{RV21}.

%% file: prelims.tex
This section is essentially taken from the author's joint work with Buchholtz~\cite[Section~2]{BW21}, to make the presentation of the thesis as self-contained as possible \wrt~to the variations and peculiarities of sHoTT~\cite{RS17} that are relevant for our treatise.

The parts present here have originally been written by the author of this thesis as well.

\section{Exposition of Riehl--Shulman's synthetic
	\texorpdfstring{$\inftyone$}{(∞,1)}-category theory}\label{sec:expo}

We recall some basic features and results from Riehl--Shulman's synthetic $\inftyone$-category theory~\cite{RS17}, at a very brief and informal level. A significantly more thorough treatment is provided in the original paper. For a general introduction to homotopy type theory,~\cf~\emph{The Book}, collaboratively authored by the \emph{Univalent Foundations Project}~\cite{hottbook}, or Rijke's excellent book in progress~\cite{RijIntro}. In particular, the latter emphasizes new perspectives informed by categorical homotopy theory and higher topos theory in the synthetic setting.

\subsection{Shapes}
In terms of homotopy theory, the shape layer enable us to reason about generating anodyne cofibrations using strict equalities.

In simplicial HoTT, next to the familiar layer of (univalent) intensional Martin-L\"of type theory, there are new ``non-fibrant'' layers added that provide a logical calculus of geometric shapes. We start of from the \emph{cube layer}, \ie, a Lawvere theory generated by a single bi-pointed object $0,1:\I$, the \emph{standard $1$-cube}. A cube context
\[ \Xi \jdeq[ I_1, \ldots,  I_k] \]
is thus a finite list of cubes
\[ I_m ~\mathrm{cube}\]
for $1 \le m \le k$.

On top of the cube layer, we can form \emph{topes} through logical comprehension via (intuitionistic) conjunction $\land$, disjunction $\lor$, and equality $\jdeq$.\footnote{but \emph{no} negation/reversals!} The \emph{tope layer} hence captures sub-polytopes of $n$-cubes (with explicit embedding). A tope formula $\varphi(t_1,\ldots,t_k)$ together with a cube context $\Xi \jdeq [I_1, \ldots, I_k, \vec{J}]$ gives rise to a tope
\[ \Xi \vdash \varphi ~\mathrm{tope}.\]

The interval $\I$ under consideration shall also come equipped with an \emph{inequality} tope
\[ x,y: \I \vdash x \le y ~ \mathrm{tope}\]
making it a total order (\wrt~the \emph{strict} equality tope $\jdeq$) with $0$ and $1$ as the bottom and top element, respectively.\footnote{For a comparison with the setup of \emph{cubical type theory}~\cite{CCHM2018} cf.~\cite[Remark~3.2]{RS17}.}
In particular, we also sometimes make use of \emph{connections} on the cube terms as discussed in~\cite[Proposition~3.5]{RS17}.

A cube together with a tope is called a \emph{shape}:
\[
\frac{I\,\cube \qquad t : I \vdash \varphi\,\tope}%
{\set{t : I}{\varphi}\,\type}
\]

As an addition to the original theory by Riehl--Shulman, we will moreover coerce all shapes to be types, cf.~\Cref{ssec:fib-shapes}. This is still in accordance with the intended class of models.

\subsection{Extension types}

In addition to the strict layers, the other new feature of simplicial type theory is a new type former called the \emph{extension type}, the idea of which originally was due to Lumsdaine and Shulman. Given a shape $\Psi$ and a type family $P: \Gamma \to \Psi \to  \UU$ together with a \emph{partial} section $a: \prod_{\Gamma \times \Phi} P$, where $\Phi \subseteq \Psi$ denotes a subshape, we can form the corresponding family of extension types
\[ \Gamma \vdash \exten{t:\Psi}{P(t)}{\Phi}{a}\]
which is interpreted as a (strict) pullback, \cf~\cite[Theorem~A.16]{RS17}:
\[\begin{tikzcd}
	{\exten{t:\Psi}{P(t)}{\Phi}{a}} && {\widetilde{P}^\Psi} \\
	\Gamma && {\widetilde{P}^\Phi \times_{(\Gamma \times \Psi)^\Phi} (\Gamma \times \Psi)^\Psi} & {} \\
	&&& {}
	\arrow[two heads, from=1-1, to=2-1]
	\arrow[from=1-1, to=1-3]
	\arrow[two heads, from=1-3, to=2-3]
	\arrow["{\pair{\overline{a}}{\overline{\id_{\Gamma \times \Psi}}}}"' swap, from=2-1, to=2-3]
	\arrow["\lrcorner"{anchor=center, pos=0.125}, draw=none, from=1-1, to=2-3]
\end{tikzcd}\]
This means, the elements of $\exten{t:\Psi}{P(t)}{\Phi}{a}$ are total sections $b: \prod_{\Gamma \times \Psi} P$ such that $b|_\Phi \jdeq a$ holds judgmentally:
\[\begin{tikzcd}
	{\Phi^*\widetilde{P}} && {\widetilde{P}} \\
	{\Gamma \times \Phi} && {\Gamma \times \Psi}
	\arrow[two heads, from=1-1, to=2-1]
	\arrow[hook, from=2-1, to=2-3]
	\arrow[from=1-1, to=1-3]
	\arrow[two heads, from=1-3, to=2-3]
	\arrow["\lrcorner"{anchor=center, pos=0.125}, draw=none, from=1-1, to=2-3]
	\arrow["b"', curve={height=18pt}, dotted, from=2-3, to=1-3]
	\arrow["a", curve={height=-18pt}, dotted, from=2-1, to=1-1]
\end{tikzcd}\]
The type-theoretic rules are analogous to the familiar rules of $\Pi$-types, but with the desired judgmental equalities added, \cf~\cite[Figure~4]{RS17}.

In particular, non-dependent instances give rise to function types $A^\Phi$ where $\Phi$ is a shape rather than a type (even though, later on all of our shapes are assumed to be fibrant, cf.~Subsection~\ref{ssec:fib-shapes}). Semantically, this reflects the fact that the intended model is cotensored over simplicial sets, cf.~also the discussion in~\cite[Appendix~A]{RS17}.

From the given rules one can show that the extension types interact well with the usual $\Pi$- and $\Sigma$-types, as shown in~\cite[Subsections~4.1, 4.2]{RS17}. In particular, there is a version of the type-theoretic principle of choice\footnote{Sometimes this is referred to as the ``type-theoretic \emph{axiom} of choice'' even though it is derivable.} involving extension types that will be used a lot.

\begin{theorem}[Type-theoretic principle of choice for extension types, {\protect\cite[Theorem~4.2]{RS17}}]\label{thm:choice}
	Let $\Phi \subseteq \Psi$ be a shape inclusion. Suppose we are given families $P: \Psi \to \UU$, $Q: \prod_{t:\Psi} (P(t) \to \UU)$ and sections $a:\prod_{t:\Phi} P(t)$, $b:\prod_{t:\Phi} Q(t,a(t))$. Then there is an equivalence
	\[ \exten{t:\Psi}{\sum_{x:P(t)} Q(t,x)}{\Phi}{\lambda t.\pair{a(t)}{b(t)}} \equiv \sum_{f:\exten{t:\Psi}{P(t)}{\Phi}{a}} \exten{t:\Psi}{Q(t,f(t))}{\Phi}{b}.\]
\end{theorem}

A further important principle is \emph{relative function extensionality}, which is added as an axiom:
\begin{ax}[Relative function extensionality, {\protect\cite[Axiom 4.6]{RS17}}]\label{ax:relfunext}
	Let $\Phi \subseteq \Psi$ be a shape inclusion.
	Given a family $P: \Psi \to \UU$ such that each $P(t)$ is contractible,
	and a partial section $a:\prod_{t: \Phi} P(t)$, then the extension type $\exten{t:\Psi}{P(t)}{\Phi}{a}$ is contractible.
\end{ax}
An important consequence is the \emph{homotopy extension property (HEP)}:
\begin{prop}[Homotopy extension property (HEP), {\protect\cite[Proposition 4.10]{RS17}}]\label{prop:hep}
	Fix a shape inclusion $\Phi \subseteq \Psi$. Let $P: \Psi \to \UU$ be family, $b:\prod_{t:\Psi} P(t)$ a total section, and $a:\prod_{t:\Phi} A(t)$ a partial section. Then, given a homotopy $H:\prod_{t:\Phi} a(t) = b(t)$, there exist totalizations $a':\exten{t:\Psi}{P(t)}{\Phi}{a}$ and $H':\exten{t:\Psi}{a'(t)=b(t)}{\Phi}{H}$.
\end{prop}

\subsubsection{Semantics in simplicial spaces}

A model of simplicial type theory is given by the Reedy model structure on bisimplicial sets, which presents the $\inftyone$-topos of simplicial spaces. The main steps in proving this are discussed in~\cite[Appendix A]{RS17}, with previous work done in~\cite{ShuReedy,CisUniv}. In fact, one can replace the base by an arbitrary (Grothendieck--Rezk--Lurie) $\inftyone$-topos $\E$ so that the results developed synthetically will hold for Rezk objects (\ie, internal $\inftyone$-categories) in $\E$. This is detailed later in~\Cref{ch:sem}.

In particular, following a guiding principle of Riehl--Verity's $\infty$-cosmos theory~\cite{RV21} the definitions and constructions we are presenting fall in the frame of (synthetic) formal higher category theory: they have characterizations in terms of basic notions, such as (fibered) weak equivalences, (LARI) adjunctions, representability of distributors etc. See for instance the characterizations of cocartesian fibrations via \Cref{thm:cocart-fams-intl-char,thm:cocartfams-via-transp}.
By Riehl--Verity's results on model-independence and notably Rasekh's work on simplicial and (complete) Segal spaces one can systematically argue that, in essence, all of our internal notions externalize to their intended semantic counterparts---at least when restricting to the Rezk types, which are our objects of primary interest after all.

\subsection{Synthetic higher categories}
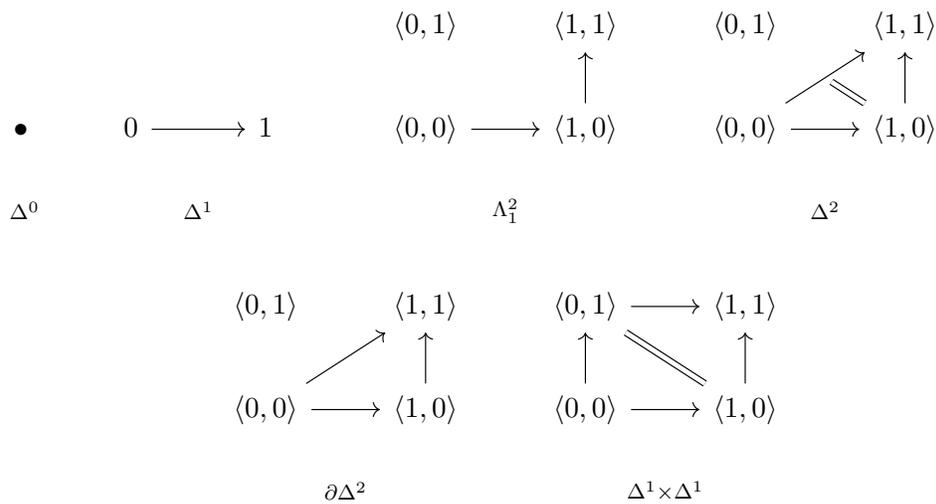
\begin{figure}
	\[\begin{tikzcd}
		&&&& {\langle 0,1 \rangle} & {\langle 1,1 \rangle} & {\langle 0,1 \rangle} & {\langle 1,1 \rangle} \\
		& \bullet & 0 & 1 & {\langle 0,0 \rangle} & {\langle 1,0 \rangle} & {\langle 0,0 \rangle} & {\langle 1,0 \rangle} \\
		{} && {} & {} & {} & {} & {} & {} \\
		&&& {\langle 0,1 \rangle} & {\langle 1,1 \rangle} & {\langle 0,1 \rangle} & {\langle 1,1 \rangle} \\
		&&& {\langle 0,0 \rangle} & {\langle 1,0 \rangle} & {\langle 0,0 \rangle} & {\langle 1,0 \rangle} \\
		&&& {} & {} & {} & {}
		\arrow[from=2-7, to=2-8]
		\arrow[from=2-8, to=1-8]
		\arrow[""{name=0, anchor=center, inner sep=0}, from=2-7, to=1-8]
		\arrow[from=2-6, to=1-6]
		\arrow[from=2-5, to=2-6]
		\arrow[from=2-3, to=2-4]
		\arrow["{\Delta^1}"{description}, draw=none, from=3-3, to=3-4]
		\arrow["{\Lambda_1^2}"{description}, draw=none, from=3-5, to=3-6]
		\arrow["{\Delta^2}"{description}, draw=none, from=3-7, to=3-8]
		\arrow["{\Delta^1 \times \Delta^1}"{description}, draw=none, from=6-6, to=6-7]
		\arrow[Rightarrow, no head, from=4-6, to=5-7]
		\arrow[from=5-7, to=4-7]
		\arrow[from=5-6, to=5-7]
		\arrow[from=5-6, to=4-6]
		\arrow[from=4-6, to=4-7]
		\arrow[from=5-4, to=5-5]
		\arrow[from=5-5, to=4-5]
		\arrow["{\partial\Delta^2}"{description}, draw=none, from=6-4, to=6-5]
		\arrow[from=5-4, to=4-5]
		\arrow["{\Delta^0}"{description}, draw=none, from=3-1, to=3-3]
		\arrow[shorten <=2pt, Rightarrow, no head, from=0, to=2-8]
	\end{tikzcd}\]
	\caption{Some important shapes}
	\label{fig:sample-shapes}
\end{figure}

Via the inequality tope of the interval $\I$ we can define simplices and subshapes familiar from simplicial homotopy theory. The first few low-dimensional simplices are given by
\begin{align*}
	\Delta^0 \defeq \{ t:\unit \; | \; \top\}, ~~ 	\Delta^1 \defeq \{ t:\I  \; | \; \top\}, ~~ \Delta^2 \defeq \{ \! \pair{t}{s} : \I \times \I \; | \;  s \le t \}.
\end{align*}
The logical connectives of the tope layer enable us to carve out subshapes, such as boundaries and horns, \eg
\begin{align*}
	& \partial \Delta^1 \defeq \{ t:\I \; | \;  t \jdeq 0 \lor t \jdeq 1\}, \quad \Lambda_1^2 \defeq \{ \pair{t}{s} : \I \times \I  \; | \;  s \jdeq 0 \lor t \jdeq 1\}, \\
	& \partial \Delta^2 \defeq \{ \pair{t}{s} : \I \times \I  \; | \;  s \jdeq t \lor s \jdeq 0 \lor t \jdeq 1\}.
\end{align*}
Cf.~\Cref{fig:sample-shapes} for an illustration and~\cite[Section 3.2]{RS17} for a detailed discussion.

We then can define, for any type $B$ and fixed elements $b,b':B$ the type of arrows from $b$ to $b'$ as
\[
\hom_B(b,b') \defeq \ndexten{\Delta^1}{B}{\partial \Delta^1}{[b,b']}.
\]
Given a type family $P:B \to \UU$ and an arrow $u:\hom_B(b,b')$ in the base, the type of arrows lying over $u$, from $e:P\,b$ to $e':P\,b'$, is given by
\[ \dhom^P_u(e,e') \defeq \exten{t:\Delta^1}{P(u(t))}{\partial \Delta^1}{[e,e']}.\]
Such an arrow is also called a \emph{dependent arrow} or \emph{dependent homomomorphism}.

We will also be considering types of $2$-cells, defined by\footnote{The boundary here is given by
	\[ [u,v,w]: \partial \Delta^2 \to B, \quad [u,v,w](t,s) \defeq
	\begin{cases}
		u & s \jdeq 0 \\
		v & t \jdeq 1 \\
		w & s \jdeq t
	\end{cases}
	\]
	and similarly for the dependent case.}
\[ \hom_B^2(u,v;w) \jdeq \ndexten{\Delta^2}{B}{\partial \Delta^2}{[u,v,w]}, \quad \hom^{2,P}_\sigma(f,g;h) \jdeq \exten{\langle t,s \rangle : \Delta^2}{P(\sigma(t,s))}{\partial \Delta^2}{[f,g,h]}.\]
We abbreviate
\[ \hom_B(b,b') \jdeq  (b \to_B b') \jdeq  (b \to b')\]
when the intent is clear from the context. For families $P : B \to \UU$, we write
\[ \hom_u^P(e,e') \jdeq (e \to^P_u e').\]
For a type $B$, we have two projections from the arrow type, given by evaluation
\[ \partial_k: B^{\Delta^1} \to B, \quad \partial_k \defeq \lambda u.u(k).\]
Similarly to the notation introduced above, for the type of natural transformations between a fixed pair of functors we abbreviate
\[ \nat{A}{B}(f,g) \jdeq  (f \Rightarrow g).\]
Sometimes, we also denote the type of $2$-simplices by
\[ \hom_B^2(u,v;w) \jdeq (u,v \Rightarrow w),\]
and likewise for the dependent version.
With these prerequisites, \cite{RS17} define a type $B$ to be a \emph{Segal type} such that the proposition
\[ \isSegal(B) \defeq \prod_{b,b',b'':B} \prod_{\substack{u:b \to b' \\ v: b' \to b''}} \isContr\Big( \sum_{w:b \to b''} \hom_B^2(u,v;w)\Big) \]
is true. This means that $B$ has \emph{weak composition of directed arrows}. After Joyal, the Segal condition can be stated as
\[
\isSegal(B) \simeq \isEquiv \Big( (\Delta^2 \to B) \to (\Lambda_1^2 \to B) \Big). \]
Segal types can be thought of as synthetic \emph{pre}-$\inftyone$-categories,\footnote{Informally, for $-2 \le k \le \infty$ and $0 \le n \le k+1$ an \emph{$(k,n)$-category} has $r$-dimensional morphisms for $r<k$, and for $r > n$, every $r$-morphism is invertible. Our study deals with $\inftyone$-categories, even though some of the ambient $\inftytwo$-categorical structure will shine through in the type theory.} which here in simplicial homotopy type theory is expressed as a \emph{property} rather than structure, echoing the familiar situation from the semantics in simplicial spaces. As discussed in~\cite[Section 5]{RS17}, Segal types can be endowed with a weak composition \emph{operation} which is weakly associative. For arrows $f:a \to b$, $g:b \to c$ in some Segal type $B$, one writes $g \circ f$ for the chosen composite arrow. The identity arrow of an element $b:B$ is given by the constant map $\id_b \defeq \lambda t.b:\Delta^1 \to B$.

Often, naturality \wrt~directed arrows comes ``for free''. In particular, any function~$f:A \to B$ between Segal types is a \emph{functor} in the sense that it preserves composition and identities up to propositional equality, as proven in~\cite[Section 6.1]{RS17}. The action of a functor on points already determines its actions on arrows as discussed in \emph{loc.~cit}.

Although the semantics is given by a structure presenting an $\inftyone$-topos---a certain kind of $\inftyone$-category---we, in fact, have access to portions of the $2$-dimensional structure present in the model as well. Since Segal types form an exponential ideal, the type $A \to B$ is Segal if $B$ is, and this allows us to study natural transformations between functors and lax diagrams of types, cf.\ Appendix~\cite[Appendix~A]{BW21} and the groundwork in~\cite[Section~6]{RS17}. This enables us to adapt several developments from Riehl--Verity's model-independent higher category theory from $\infty$-cosmoses to type theory.

Segal types come with two possible notions of isomorphism (analogous to the semantic situation for Segal spaces), the ``spatial'' one given by propositional equality, and the ``categorical'' one derived from the directed arrows. Namely, an arrow $f:a \to b$ in a Segal type $B$ is a (categorical) \emph{isomorphism} if the type
\[ \isIso(f) \defeq \sum_{g,h:b \to a} (gf = \id_a) \times (fh = \id_b) \]
is inhabited. As discussed in~\cite[Section 10]{RS17}, the type $\isIso(f)$ turns out to be a proposition, so we can define the subtypes
\[ \iso_B(a,b) \defeq \sum_{f:\hom_A(a,b)} \isIso(a,b).\]

By path induction we define the comparison map
\[ \idtoiso_B: \prod_{a,b:B} (a=_Bb) \to \iso_B(a,b), \quad \idtoiso_{B,a,a}(\refl_a) \defeq \id_a, \]
and demanding that this be an equivalence leads to the notion of a \emph{complete Segal type}, \aka~\emph{Rezk type}:
\[ \isRezk(B) \defeq \isSegal(B) \times \isEquiv(\idtoiso_B).\]
The Rezk-completeness condition can be understood as a local version of the Univalence Axiom. In the simplicial space model, Rezk types are interpreted as Rezk spaces, which justifies viewing Rezk types as \emph{synthetic $\inftyone$-categories}. Even though a lot of the development in~\cite{RS17} actually already works well on the level of (not necessarily complete) Segal types, our study of cocartesian families mostly restricts to complete Segal types, which is in line with preexisting studies of (co-)cartesian fibrations in the higher-categorical context \cite{JoyNotesQcat,LurHTT,RVyoneda,Ras17Cart,dB16segal,AFfib,BarwickShahFib}.

Among the synthetic $\inftyone$-categories, we can also consider types that are synthetic $\inftyzero$-categories, \ie, $\infty$-groupoids. These are called \emph{discrete types}, which refers to the condition that all directed arrows be invertible, namely the comparison map defined inductively by
\[ \idtoarr_B: \prod_{a,b:B} (a=_B b) \to \hom_B(a,b), \quad \idtoarr_{B,a,a}(\refl_a) \defeq \id_a,\]
be an equivalence:
\[ \isDisc(B) \defeq \prod_{a,b:B} \isEquiv(\idtoarr_{B,a,b})\]
In fact, this discreteness condition entails Rezk-ness as shown by Riehl--Shulman. Furthermore, if $B$ is a Segal type, for any $a,b:B$, the hom-type $\hom_B(a,b)$ is discrete by~\cite[Proposition~8.13]{RS17}.

Discrete Rezk spaces are also called complete Bousfield--Segal spaces and have been studied \eg~in~\cite{BergnerInvDiagII,Stenzel-CBS}.

\subsection{Covariant families}

Riehl--Shulman have introduced the notion of \emph{covariant family}, \ie, families of discrete types varying functorially \wrt~directed arrows in the base. A type family $P:B \to \UU$ over a (Segal or Rezk) type $B$ is \emph{covariant} if
\[ \isCovFam(P) \defeq \prod_{\substack{a,b:B \\ u:a \to b}} \prod_{d:P\,a} \isContr\Big( \sum_{e:P\,b} (d \to^P_u e) \Big), \]
\ie, arrows in the base can be uniquely lifted w.r.t.~a given source vertex.

Semantically, these correspond to \emph{left fibrations}, which encode $\inftyone$-copresheaves. Hence, as expected, an example is given, for any $a:B$ by the family
\[ \lambda b.\hom_B(a,b): B \to \UU.\]

The central topic of our work is to generalize this study to synthetic \emph{cocartesian fibrations}, \ie, the case where the fibers are Rezk rather than discrete.\footnote{Everything dualizes to the case of cartesian fibrations, of course, but we don't spell this out.}

\section{Fibrant shapes}\label{ssec:fib-shapes}

In our intended models, the shapes will arise as (spatially-discrete) fibrant objects. In fact, it is crucial for our treatment to have this reflected in the type theory, so we add the following rule:
\[
\frac{I\,\cube \qquad t : I \vdash \varphi\,\tope}%
{\shapety{t : I}{\varphi}\,\type}
\]
In fact, we take the interval $\Delta^1$ as a type, and the inequality relation ${\le} : \Delta^1 \to \Delta^1 \to \Prop$ as a type family, and then all shapes are types using the ordinary type formers.

Using relative function extensionality \cite[Section~4.4, Axiom~4.6]{RS17}, \ref{ax:relfunext}, one can show that every strict extension type is equivalent to its ``weak'' counterpart, where the latter (essentially) is definable in Standard HoTT, \ie~we always have an equivalence as follows.
\begin{proposition}[De-/strictification of extension types,~\cf~\protect{\cite[Section~2.4]{BW21}}]
	In cube context $I$, let $\varphi \hookrightarrow \psi$ be a shape inclusion. Let $A: \shapety I\psi \to \UU$ be a type family, and $a:\prod_{x:\shapety I\varphi}$ a partial section. Then there is an equivalence between the ensuing types of ``strict'' (judgmental) and ``weak'' (propositional) extensions:
	\[ \exten{x:\shapety I\psi}{A(x)}{\varphi}{a} \equiv \sum_{f : \prod_{x:\shapety I\psi}A(x)}\prod_{x:\shapety I\varphi}(a\,x = f\,x), \]
\end{proposition}
A proof due to Buchholtz is in~\cite[Section~2.4]{BW21}.

In particular, this gives the following formulation of (diagrammatic, weak) lifting problems in terms of (formulaic, strict) contractibility statements.

\begin{obs}\label{obs:orth-maps}
	Consider a family $P:B \to \UU$ and a shape inclusion $\Phi \subseteq \Psi$. Then, given a total diagram $\sigma:\Psi \to B$ with a partial diagram $\kappa:\prod_\Phi \sigma^*P$ lying over,
	the diagram
	\[\begin{tikzcd}
		\Phi && {\widetilde{P}} & {} \\
		\Psi && B & {}
		\arrow[hook, from=1-1, to=2-1]
		\arrow["\sigma"{description}, from=2-1, to=2-3]
		\arrow["\pi", two heads, from=1-3, to=2-3]
		\arrow["\kappa"{description}, from=1-1, to=1-3]
		\arrow[dashed, from=2-1, to=1-3]
	\end{tikzcd}\]
	possesses a diagonal filler uniquely up to homotopy if and only if the proposition
	\[ 
	\mathrm{isContr} \Big( \Big\langle\prod_{t:\Psi} P(\sigma(s))\Big|^\Phi_\kappa \Big \rangle\Big)
	\]
	is inhabited.
\end{obs}

\begin{expl}
	Recall from~\cite{RS17} that a type $A$ is Segal precisely if $A \to \unit$ is right orthogonal to $\Lambda_1^2 \hookrightarrow \Delta^2$. Another example is given by the class of covariant families, namely $P:B \to \UU$ is covariant if and only if $\totalty{P} \to B$ is right orthogonal to the initial vertex inclusion $i_0: \unit \hookrightarrow \Delta^1$.
\end{expl}

We can also type-theoretically express the \emph{Leibniz construction} familiar from categorical homotopy theory~\cite[Definitions C2.8, C2.10, C3.8]{RV21}, \cite{RieCatHtpy} as follows. We remark that Leibniz cotensor maps will be ubiquitous in our treatise since the fibrations of interest are defined by conditions on them.

\begin{defn}[Leibniz cotensor]
	Let $j:Y \to X$ be a type map or shape inclusion, and $\pi :E \to B$ a map between types. The \emph{Leibniz cotensor of $j$ and $\pi$} (\aka~\emph{Leibniz exponential of $\pi$ by $j$} or \emph{pullback hom}) is defined as the following gap map:
	\[\begin{tikzcd}
		&&& {E^X} \\
		{Y} & {E} & {E^X} && {\cdot} && {E^Y} \\
		{X} & {B} & {B^X \times_{B^Y} E^Y} && {B^X} && {B^Y}
		\arrow["{j}"{name=0, swap}, from=2-1, to=3-1]
		\arrow["{j \widehat{\pitchfork} \pi}"{name=1}, from=2-3, to=3-3]
		\arrow[from=2-5, to=3-5]
		\arrow[from=3-5, to=3-7]
		\arrow[from=2-5, to=2-7]
		\arrow[from=2-7, to=3-7]
		\arrow["{j\widehat{\pitchfork} \pi}" description, from=1-4, to=2-5, dashed]
		\arrow[from=1-4, to=2-7, curve={height=-12pt}]
		\arrow[from=1-4, to=3-5, curve={height=12pt}]
		\arrow["\lrcorner"{very near start, rotate=0}, from=2-5, to=3-7, phantom]
		\arrow["{\pi}"{name=2, swap}, from=2-2, to=3-2, swap]
		\arrow[Rightarrow, "{\defeq}"{description,pos=0.25}, from=2, to=1, shorten <=7pt, shorten >=7pt, phantom, no head]
		\arrow[Rightarrow, "{\widehat{\pitchfork}}" description, from=0, to=2, shorten <=5pt, shorten >=5pt, phantom, no head]
	\end{tikzcd}\]
\end{defn}

The map $\pi$ is right orthogonal to $j$, meaning that for any square as below there exists a filler uniquely up to homotopy
\[\begin{tikzcd}
	{Y} & {E} \\
	{X} & {B}
	\arrow["{j}"', from=1-1, to=2-1]
	\arrow[from=2-1, to=2-2]
	\arrow[from=1-1, to=1-2]
	\arrow["{\pi}", from=1-2, to=2-2]
	\arrow[from=2-1, to=1-2, dotted]
\end{tikzcd}\]
if and only if the Leibniz cotensor map is an equivalence
\[ j \widehat{\pitchfork} \pi: E^X \stackrel{\equiv}{\longrightarrow} B^X \times_{B^Y} E^Y, \]
\cf~\Cref{obs:orth-maps}.

Though sparsely explicitly present in the text, we will also mention the dual operation.

\begin{defn}[Pushout product]\label{def:po-prod}
	Let $j: Y \to X$ and $k: T \to S$ each be type maps or shape inclusions. The \emph{Leibniz tensor of $j$ and $k$} (or \emph{pushout product}) is defined as the following cogap map:
	\[\begin{tikzcd}
		{Y} & {T} & {Y \times S\bigsqcup_{Y \times T} X \times T} && {Y \times T} && {Y \times S} \\
		{X} & {S} & {X \times S} && {X \times T} && {\cdot} \\
		&&&&&&& {X \times S}
		\arrow["{j}"{name=0, swap}, from=1-1, to=2-1]
		\arrow["{k}"{name=1, swap}, from=1-2, to=2-2, swap]
		\arrow["{j\widehat{\otimes} k}"{name=2}, from=1-3, to=2-3]
		\arrow[from=1-5, to=2-5]
		\arrow[from=2-5, to=2-7]
		\arrow[from=1-5, to=1-7]
		\arrow[from=1-7, to=2-7]
		\arrow[from=2-5, to=3-8, curve={height=12pt}]
		\arrow[from=1-7, to=3-8, curve={height=-12pt}]
		\arrow["{j \widehat{\otimes}k}" description, from=2-7, to=3-8, dashed]
		\arrow["\lrcorner"{very near start, rotate=180}, from=2-7, to=1-5, phantom]
		\arrow[Rightarrow, "{\widehat{\otimes}}" description, from=0, to=1, shorten <=5pt, shorten >=5pt, phantom, no head]
		\arrow[Rightarrow, "{\defeq}" {description,pos=0.25}, from=1, to=2, shorten <=9pt, shorten >=9pt, phantom, no head]
	\end{tikzcd}\]
\end{defn}

In particular, recall from~\cite[Theorem 4.2]{RS17}, the explicit formula for the pushout product of two shape inclusions:
\[\begin{tikzcd}
	{\{t:I\,|\,\varphi\}} & {\{s:J\,|\,\chi\}} & {\{ \langle t,s\rangle : I \times J \, | \, (\varphi \land \zeta) \lor (\psi \land \chi)\}} \\
	{\{t:I\,|\,\psi\}} & {\{s:J\,|\,\zeta\}} & {\{ \langle t,s\rangle : I \times J \, | \, \psi \land \zeta\}} \\
	{}
	\arrow[""{name=0, inner sep=0}, from=1-1, to=2-1, hook]
	\arrow[""{name=1, inner sep=0}, from=1-2, to=2-2, hook]
	\arrow[""{name=2, inner sep=0}, from=1-3, to=2-3, hook]
	\arrow[Rightarrow, "{\widehat{\otimes}}" description, from=0, to=1, shorten <=7pt, shorten >=7pt, phantom, no head]
	\arrow[Rightarrow, "{\defeq}"  {description,pos=0.25}, from=1, to=2, shorten <=13pt, shorten >=13pt, phantom, no head]
\end{tikzcd}\]

\section{Families vs.~fibrations}

Recall from~\cite{hottbook} that in presence of the univalence axiom, there is an equivalence between type families and fibrations.\footnote{Assuming universes with better structural properties---such as Segalness or directed univalence---would be fruitful for further considerations, but this is part of future work.}

Consider the types
\[ \Fib(\UU) \defeq \sum_{A,B:\UU} A \to B, \qquad \Fam(\UU) \defeq \sum_{B:\UU} (B \to \UU)  \]
of functions in $\UU$ (viewed as type-theoretic fibrations\footnote{The inhabitants of $\Fib(\UU)$ are just maps between arbitrary $\UU$-small types, but viewed as ``$\UU$-small type theoretic fibrations over a $\UU$-small base''.}), and families with $\UU$-small fibers, resp. Both these types naturally are fibered over $\UU$ via the following maps:
\[\begin{tikzcd}
	{\Fib(\UU)} && \UU && {\Fam(\UU)}
	\arrow["{\partial_1 \defeq \lambda A,B,f.B}", from=1-1, to=1-3]
	\arrow["{\pr_1 \defeq \lambda B,P.B}"', from=1-5, to=1-3]
\end{tikzcd}\]
Over a type $B:\UU$, we obtain the type of \emph{maps into (or fibrations over) $B$} as the fiber:
\[\begin{tikzcd}
	{\UU/B} && {\mathrm{Fib}(\UU)} \\
	\unit && \UU
	\arrow[from=1-1, to=2-1]
	\arrow["B"', from=2-1, to=2-3]
	\arrow[from=1-1, to=1-3]
	\arrow["{\partial_1}", from=1-3, to=2-3]
	\arrow["\lrcorner"{anchor=center, pos=0.125}, draw=none, from=1-1, to=2-3]
\end{tikzcd}\]

\begin{theorem}[Typal Grothendieck construction, cf.~{\protect\cite[Theorem 4.8.3]{hottbook}}]
	There is a fiberwise quasi-equivalence
	\[\begin{tikzcd}
		{\Fib(\UU)} && {\Fam(\UU)} \\
		& \UU
		\arrow["{\partial_1}"', from=1-1, to=2-2]
		\arrow["{\pr_1}", from=1-3, to=2-2]
		\arrow["\Un", shift left=1, from=1-3, to=1-1]
		\arrow["\St", shift left=2, from=1-1, to=1-3]
	\end{tikzcd}\]
	at stage $B:\UU$ given by a pair
	\[\begin{tikzcd}
		{\UU/B} && {(B\to \UU)}
		\arrow["{\St_B}", shift left=2, from=1-1, to=1-3]
		\arrow["{\Un_B}", shift left=1, from=1-3, to=1-1]
	\end{tikzcd}\]
	with \emph{straightening}
	\[ \St_B(\pi) \defeq \lambda b.\fib_b(\pi)\]
	and \emph{unstraightening}
	\[ \Un_B(P) \defeq \pair{\totalty{P}}{\pi_P}\]
	($\pi_P: \totalty{P} \defeq \sum_{b:B} P\,b \to B$ the total space projection).
\end{theorem}

The spirit of dependent type theory somewhat favors type families over fibrations, but we will often resort to the fibrational viewpoint because it allows us to replay familiar categorical arguments. For instance, Riehl--Shulman's covariant type families are a type-theoretic version of left fibrations, and we want to be able to conveniently make use of both incarnations of the same concept which motivates the following:

\begin{defn}[Notions of families and fibrations]
	A \emph{notion of family (or notion of fibration)} is a family
	\[ \mathcal F : \Fam(\UU) \to \Prop \]
	of propositions on the type of $\UU$-small fibrations. For a notion of family $\mathcal F$, we say that a family $P:B\to \UU$ is an \emph{$\mathcal F$-family}\footnote{In practice, the name of $\mathcal F$ often will be a linguistic predicate such as ``covariant'', ``cocartesian'' etc.~in which case we drop the hyphen and treat it as part of the natural meta-language, \eg~we will simply speak of ``cocartesian'' or ``covariant fibrations''.} if and only if the proposition
	\[ \isFam_{\mathcal F}(P) \defeq \mathcal F(P)\]
	holds. A map $\pi:E \to B$ is called an \emph{$\mathcal F$-fibration} if its family of fibers $\St_B(\pi)$ is an $\mathcal F$-family.
\end{defn}
By univalence and the Grothendieck construction, this definition is well-behaved, \ie, a~($\UU$-small) map is an~$\mathcal F$-fibration if and only if it is (equivalent to) a projection associated to an $\mathcal F$-family (valued in $\UU$).

In particular, we observe the following. Considering
\[ \Fib_\mathcal F(\UU) \defeq \sum_{\substack{E,B:\UU \\ \pi:E \to B}} \isFam_{\mathcal F}(\St(\pi)), \quad \Fam_{\mathcal F}(\UU) \defeq \sum_{P:\Fam(\UU)} \isFam_{\mathcal F}(P), \]
the Grothendieck construction descends to a fiberwise equivalence, for any notion of fibration/family $\mathcal F$:

\[\begin{tikzcd}
	{\mathrm{Fib}_{\mathcal F}(\mathcal U)} && {\mathrm{Fam}_{\mathcal F}(\mathcal U)} \\
	& {\mathrm{Fib}(\mathcal U)} && {\mathrm{Fam}(\mathcal U)} \\
	& {\mathcal U}
	\arrow[hook, from=1-1, to=2-2]
	\arrow[""{name=0, anchor=center, inner sep=0}, "{\mathrm{St}}", shift left=2, from=1-1, to=1-3]
	\arrow[hook, from=1-3, to=2-4]
	\arrow[curve={height=6pt}, from=1-1, to=3-2]
	\arrow[from=1-3, to=3-2]
	\arrow[from=2-4, to=3-2]
	\arrow[from=2-2, to=3-2]
	\arrow[""{name=2, anchor=center, inner sep=0}, "{\mathrm{Un}}", shift left=2, from=1-3, to=1-1]
	\arrow[""{name=3, anchor=center, inner sep=0}, "{\mathrm{Un}}", shift left=2, from=2-4, to=2-2, crossing over]
	\arrow[""{name=1, anchor=center, inner sep=0}, "{\mathrm{St}}", shift left=2, from=2-2, to=2-4, crossing over]
	\arrow["\simeq"{description}, shorten <=1pt, shorten >=1pt, Rightarrow, from=1, to=3]
	\arrow["\simeq"{description}, shorten <=1pt, shorten >=1pt, Rightarrow, no head, from=0, to=2]
\end{tikzcd}\]

\begin{rem}
	As a convention, we will always state the definitions of the various notions of fibration in terms of \emph{families}, and the above definition schema immediately yields the respective corresponding notion in fibrational terms.
	
	We will also often denote a map which satisfies such a fibration condition (or possibly even just a usual map which is to be regarded as a type-theoretic fibration) by a double hooked arrow $\pi:E \fibarr B$, as is customary in homotopical algebra or categorical homotopy theory.\footnote{We use here the predicate \emph{typal} to indicate that the universe considered is merely a general type, rather than Rezk or Segal. A (complete) Segal universe would give rise to a \emph{categorical Grothendieck construction}, hence a synthetic version of straightening/unstraightening~\cite{LurHTT,RVComp,Ras18model,dB16segal}. But this is beyond the scope of this thesis.}
\end{rem}

Furthermore, as a consequence of univalence, any such propositionally defined notion of family/fibration is invariant under equivalence.

\begin{prop}[Homotopy invariance of notions of fibrations]
	Let $\mathcal F$ be a notion of fibration. When given a commutative square
	\[\begin{tikzcd}
		F && E \\
		A && B
		\arrow["\xi"', from=1-1, to=2-1]
		\arrow["\simeq"', from=2-1, to=2-3]
		\arrow["\simeq", from=1-1, to=1-3]
		\arrow["\pi", from=1-3, to=2-3]
	\end{tikzcd}\]
	the map $\xi$ is an $\mathcal F$-fibration if and only if $\pi$ is.
\end{prop}

\section{Comma and co-/cone types}\label{ssec:commas-cones}
\begin{defn}[Comma types]\label{def:comma-obj}
	Consider a cospan of types
	\[\begin{tikzcd}
		C && A && B
		\arrow["g", from=1-1, to=1-3]
		\arrow["f"', from=1-5, to=1-3]
	\end{tikzcd}\]
	The \emph{comma type} $f \downarrow g$ is given by the following pullback:
	\[\begin{tikzcd}
		{f \downarrow g} && {A^{\Delta^1}} \\
		{C \times B} && {A \times A}
		\arrow[from=1-1, to=2-1]
		\arrow["g \times f", from=2-1, to=2-3]
		\arrow[from=1-1, to=1-3]
		\arrow["{\langle \partial_1, \partial_0 \rangle}", from=1-3, to=2-3]
		\arrow["\lrcorner"{anchor=center, pos=0.125}, draw=none, from=1-1, to=2-3]
	\end{tikzcd}\]
	In the case that $f$ is the identity $\id_A$, we write shorthand $\comma{A}{g}$ for $\comma{f}{g}$, and dually $\comma{f}{A}$ if $g$ is the identity.
\end{defn}

\begin{defn}[Co-/cone types]
	Let $X$ be a type or a shape and $A$ a type. In a setting such as the present one, a map $u:X \to A$ is sometimes referred to as an \emph{$X$-shaped diagram in $A$}. The cospans
	\[\begin{tikzcd}
		A && {A^X} && \unit && \unit && {A^X} && A
		\arrow["{\mathrm{cst}_A}", from=1-1, to=1-3]
		\arrow["u"', from=1-5, to=1-3]
		\arrow["u", from=1-7, to=1-9]
		\arrow["{\mathrm{cst}_A}"', from=1-11, to=1-9]
	\end{tikzcd}\]
	give rise to the type $u/A \defeq \comma{u}{\cst_A}$ of \emph{cocones in $A$ under $X$}, and, dually  $A/u \defeq \comma{\cst_A}{u}$ of \emph{cones in $A$ over $X$}, resp., defined as comma objects:
	\[\begin{tikzcd}
		{u/A} && {\big(A^X)^{\Delta^1}} && {A/u} && {\big(A^X)^{\Delta^1}} \\
		{A \times \unit} && {A^X \times A^X} && {\unit \times A} && {A^X \times A^X}
		\arrow["{\cst_A \times u}", from=2-1, to=2-3]
		\arrow[from=1-1, to=1-3]
		\arrow["{\langle \partial_1, \partial_0 \rangle}", from=1-3, to=2-3]
		\arrow[from=1-1, to=2-1]
		\arrow["\lrcorner"{anchor=center, pos=0.125}, draw=none, from=1-1, to=2-3]
		\arrow[from=1-5, to=2-5]
		\arrow["{u \times \cst_A}", from=2-5, to=2-7]
		\arrow[from=1-5, to=1-7]
		\arrow["{\langle \partial_1, \partial_0 \rangle}", from=1-7, to=2-7]
		\arrow["\lrcorner"{anchor=center, pos=0.125}, draw=none, from=1-5, to=2-7]
	\end{tikzcd}\]
\end{defn}

\begin{expl}[Co-/slice types]
	For a fixed point $b:B$, the \emph{co-/slice types} are defined as $b/B$ and $B/b$, resp. Note that $b/B \equiv \comma{b}{B}$, and similarly for the slice types.
\end{expl}

\section{Orthogonal and LARI families}\label{sec:orth-lari}

An important part in our study of synthetic fibered $(\infty,1)$-categories is to provide proofs of certain closure properties, which are chosen to parallel those of $\infty$-cosmoses~\cite[Definition~1.2.1]{RV21}. Recall from~\cite{RV21}, that any $\infty$-cosmos $\mathcal K$ provides an intrinsic notion of cocartesian fibrations, which themselves form an $\infty$-cosmos ${co\mc{C}art}({\mc{K}})$. The discrete cocartesian fibrations form an embedded $\infty$-cosmos $\mc{D}iscco\mc{C}art({\mc{K}}) \hookrightarrow {co\mc{C}art}({\mc{K}})$. We prove type-theoretic analogues of the $\infty$-cosmological closure properties, formulated internally to the type theory of the ``ambient'' $\inftyone$-topos of simplicial objects.

First, we consider maps which are, more generally, defined by a unique right lifting property against an arbitrary map. Next, we discuss $j$-LARI maps which are defined by a left adjoint right inverse condition on a Leibniz cotensor map.

Specifically, let $j:Y \to X$ be some type map.\footnote{This is often a shape inclusion; recall from Subsection~\ref{ssec:fib-shapes} that we have coercion of (strict) shapes into types.} A map $\pi:E \to B$ is called \emph{$j$-orthogonal} if any square as below has a contractible space of fillers:
\[\begin{tikzcd}
	Y && E \\
	X && B
	\arrow["j"', from=1-1, to=2-1]
	\arrow[from=2-1, to=2-3]
	\arrow[from=1-1, to=1-3]
	\arrow["\pi", from=1-3, to=2-3]
	\arrow[dashed, from=2-1, to=1-3]
\end{tikzcd}\]
If a map $\pi:E \to B$ is right orthogonal to a map $j:Y \to X$, we write $j \bot \pi$. Similarly, for families $P:B \to \UU$, we write $j \bot P$ if $j \bot \Un_B(P)$.

Classes of maps defined by such lifting conditions play an important role in categorical homotopy theory and have been extensively studied in various contexts. In particular, classes defined by right orthogonal lifting conditions necessarily satisfy certain closure properties. We are giving type theoretic proofs which will later apply for the specific kinds of $j$-orthogonal maps that we are interested in, namely (iso-)inner fibrations and left fibrations \aka~discrete covariant fibrations. For instance, a map $\pi:E \to B$ (over a Segal type $B$) is a covariant fibration if and only if it is right orthogonal to the initial vertex inclusion $i_0: \unit \to \Delta^1$.

In general, $\pi:E \to B$ being $j$-orthogonal can be rephrased as the condition that the gap map in the following diagram be an equivalence:
\[\begin{tikzcd}
	{E^X} \\
	& {B^X \times_{B^Y} E^Y} && {E^Y} \\
	& {B^X} && {B^Y}
	\arrow[from=2-2, to=3-2]
	\arrow[from=3-2, to=3-4]
	\arrow[from=2-2, to=2-4]
	\arrow[from=2-4, to=3-4]
	\arrow["\lrcorner"{anchor=center, pos=0.125}, draw=none, from=2-2, to=3-4]
	\arrow[curve={height=-12pt}, from=1-1, to=2-4]
	\arrow["\equiv"{description}, dashed, from=1-1, to=2-2]
	\arrow[curve={height=12pt}, from=1-1, to=3-2]
\end{tikzcd}\]
Weakening this condition by requiring the gap map to only have a \emph{left adjoint right inverse (LARI)} leads to the notion of \emph{$j$-LARI} map, \ie, $\pi:E \to B$ is a $j$-LARI map if and only if the induced map $E^X \to B^X \times_{B^Y} E^Y$ has a LARI:
\[\begin{tikzcd}
	{E^X} \\
	&& {B^X \times_{B^Y} E^Y} && {E^Y} \\
	&& {B^X} && {B^Y}
	\arrow[from=2-3, to=3-3]
	\arrow[from=3-3, to=3-5]
	\arrow[from=2-3, to=2-5]
	\arrow[from=2-5, to=3-5]
	\arrow["\lrcorner"{anchor=center, pos=0.125}, draw=none, from=2-3, to=3-5]
	\arrow[curve={height=-24pt}, from=1-1, to=2-5]
	\arrow[curve={height=12pt}, from=1-1, to=3-3]
	\arrow[""{name=0, anchor=center, inner sep=0}, curve={height=12pt}, dotted, from=2-3, to=1-1]
	\arrow[""{name=1, anchor=center, inner sep=0}, curve={height=6pt}, from=1-1, to=2-3]
	\arrow["\dashv"{anchor=center, rotate=-110}, draw=none, from=0, to=1]
\end{tikzcd}\]
Between Rezk types, a map $\pi:E \to B$ is a cocartesian fibration if and only if it is an $i_0$-LARI map.

Both $j$-orthogonal and $j$-LARI maps are closed under dependent products, composition, and pullback. In addition, $j$-orthogonal maps are closed under sequential limits and Leibniz cotensoring. They also satisfy left canceling.

\section{(Iso-)inner families}

Since at the most general level types are not Segal, as an intermediate step to defining cocartesian families, we have to deal with families of (complete) Segal types that are not necessarily functorial. \emph{Inner families} are those type families for which the associated projection is right orthogonal to the horn inclusion $\Lambda_1^2 \hookrightarrow \Delta^2$. Hence, between Segal types, inner families correspond to fibrations in the Segal model structure. Bringing in Rezk-completeness motivates our definition of \emph{isoinner family}, which in addition to innerness requires all fibers to be Rezk-complete. Over Rezk types, this can be expressed by requiring the associated projection to be right orthogonal to the terminal projection from the free bi-invertible arrow $\walkBinv$. In the thesis, the running assumption is that the (anonymous) types considered are Rezk, so maps $E \to B$ are iso-inner automatically. For a finer analysis of more general settings~\cf~\cite[Section~4]{BW21}.

\section{Sliced constructions}
	
Sometimes, we will also make use of sliced constructions, \cf~\cite[Proposition~1.2.22]{RV21}.

\begin{defn}[Sliced cotensor, \protect{\cite[Proposition~1.2.22(vi)]{RV21}}]
	Let $\pi:E\fibarr B$ be a map, and $X$ be a type or shape. The \emph{sliced exponential (over $B$) of $\pi$ by $X$}  is given by the map $X \iexp E \to B$ defined as:
	\[\begin{tikzcd}
		{X \boxtimes E} && {E^X} \\
		B && {B^X}
		\arrow[two heads, from=1-1, to=2-1]
		\arrow["{\mathrm{cst}}"', from=2-1, to=2-3]
		\arrow[from=1-1, to=1-3]
		\arrow["{\pi^X}", two heads, from=1-3, to=2-3]
		\arrow["\lrcorner"{anchor=center, pos=0.125}, draw=none, from=1-1, to=2-3]
	\end{tikzcd}\]
This means $X \boxtimes E \simeq \sum_{b:B} X \to P\,b$.
In particular, for $X \jdeq \Delta^1$, we obtain the \emph{vertical arrow} object $\VertArr_\pi(E) \to B$.
\end{defn}

\begin{defn}[Sliced product, \protect{\cite[Proposition~1.2.22(vi)]{RV21}}]
 	Let $I$ and $B$ be types, and consider maps $\pi_i:E_i \fibarr B$ for $i:I$. The \emph{sliced product} over the $\pi_i$ is defined by pullback:
	\[\begin{tikzcd}
		{\times_{i:I}^B E_i} && {\prod_{i:I} E_i} \\
		B && {B^I}
		\arrow["{\times_{i:I}^B \pi_i}"', two heads, from=1-1, to=2-1]
		\arrow[from=2-1, to=2-3]
		\arrow[from=1-1, to=1-3]
		\arrow["{\prod_{i:I} \pi_i}", two heads, from=1-3, to=2-3]
		\arrow["\lrcorner"{anchor=center, pos=0.125}, draw=none, from=1-1, to=2-3]
	\end{tikzcd}\]
\end{defn}

\begin{defn}[Sliced comma, \protect{\cite[Proposition~1.2.22(vi)]{RV21}}]
Consider a cospan $\psi:F \to_B G \leftarrow_B E: \varphi$ of fibered functors, giving rise to the sliced comma type $\relcomma{B}{\varphi}{\psi}$:
\[\begin{tikzcd}
	{\varphi \downarrow_B \psi} && {\VertArr(G)} \\
	{F \times E} && {G \times G} \\
	& B
	\arrow[from=1-1, to=2-1]
	\arrow["{\psi \times \varphi}"{description}, from=2-1, to=2-3]
	\arrow[from=1-1, to=1-3]
	\arrow["{\langle \partial_1,\partial_0\rangle}", from=1-3, to=2-3]
	\arrow[two heads, from=2-1, to=3-2]
	\arrow[two heads, from=2-3, to=3-2]
	\arrow["\lrcorner"{anchor=center, pos=0.125}, draw=none, from=1-1, to=2-3]
\end{tikzcd}\]
\end{defn}

%% file: cocart-intro.tex
Cocartesian families $P:B \to \UU$ encode copresheaves of $\inftyone$-categories. All fibers $P\,b$ are Rezk types, and $P$ is (covariantly) \emph{functorial} in the sense that an arrow $u:a \to b$ in $B$ induces a functor $u_!: P\,a \to P\,b$, and this transport operation is natural w.r.t.~directed arrows in $B$, \ie,~it respects composition and identities. In fact, we will often reason about cocartesian families $P:B \to \UU$ in terms of their associated projection $\pi \defeq \pi_P: E \to B$. Our study is informed by~\cite[Chapter 5]{RV21} and~\cite{RVyoneda} in an essential way, where Riehl--Verity develop a model-independent theory of cocartesian fibrations intrinsic to an arbitrary $\infty$-cosmos. While this constitutes more generally a fibrational theory of $\inftyn$-categories, for $0 \le n \le \infty$, our present study restricts to $\inftyone$-categories, and at the same time extends Riehl--Shulman's treatment of synthetic $\inftyone$-categories fibered in $\infty$-groupoids.

Reminiscent to the classical (1-categorical) definition, we introduce cocartesian families in terms of the existence of enough cocartesian liftings. However, we also give alternative characterizations, such as the \emph{Chevalley criterion} which allows us to develop the theory in the style of formal category theory, as done by Riehl and Verity~\cite{RV21} for $\infty$-cosmoses. Their work in particular constitutes a vast generalization of the historic results of Gray~\cite{GrayFib} and Street~\cite{StrYon,StrBicat,StrBicatCorr} to the model-independent higher case.

Specifically, over Rezk types cocartesian families are exactly the isoinner families that are $i_0$-LARI families in the sense of Subsection~\ref{sec:orth-lari}, for $i_0: \unit \to \Delta^1$. Spelled out, this means that the gap map in the pullback
\[\begin{tikzcd}
	{E^{\Delta^1}} & {} \\
	&& {\pi \downarrow B} && E \\
	&& {B^{\Delta^1}} && B
	\arrow[two heads, from=2-3, to=3-3]
	\arrow["{\partial_0}"', from=3-3, to=3-5]
	\arrow[from=2-3, to=2-5]
	\arrow["\pi", two heads, from=2-5, to=3-5]
	\arrow["\lrcorner"{anchor=center, pos=0.125}, draw=none, from=2-3, to=3-5]
	\arrow["{\partial_0}", shift left=2, curve={height=-18pt}, from=1-1, to=2-5]
	\arrow[""{name=0, anchor=center, inner sep=0}, "\chi"', curve={height=12pt}, dashed, from=2-3, to=1-1]
	\arrow[""{name=1, anchor=center, inner sep=0}, "{i_0\widehat{\pitchfork}\pi}"', curve={height=12pt}, from=1-1, to=2-3]
	\arrow["{\pi^{\Delta^1}}"', shift right=2, curve={height=18pt}, two heads, from=1-1, to=3-3]
	\arrow["\dashv"{anchor=center, rotate=-118}, draw=none, from=0, to=1]
\end{tikzcd}\]
has a left adjoint right inverse $\chi: \comma{\pi}{B} \to E^{\Delta^1}$ which yields the up-to-homotopy uniquely determined cocartesian lifts.

We find a similar characterization for \emph{cocartesian functors} between cocartesian families (incarnated as fibrations). In~\cite[Subsection~5.2.3 and 5.3.3]{BW21}, from this we prove type-theoretic versions of the $\infty$-cosmological closure properties of cocartesian fibrations, which in our case means that the $\inftyone$-category of cocartesian fibrations is complete \wrt~to certain $\inftyone$-limits.\footnote{In more technical terms, our results can be interpreted as type-theoretic proofs of the completeness of the $\inftyone$-categorical core of the $\infty$-cosmos $\cosCocart(\cosRezk)$, itself presenting an $\inftytwo$-category (cf.~\cite[Definition 12.1.10, Remark~12.1.11]{RV21}).} Note that, ideally, these would be statements involving universe types which themselves are Rezk. These are beyond the scope of the current discussion, but nevertheless we can ``externalize'' these completeness statements to our univalent universe $\UU$ of arbitrary simplicial types, yielding \Cref{prop:cocart-cosm-closure}.

We then prove characterizations of cocartesian functors extending the ones for cocartesian functors, \eg~the Chevalley criterion for cocartesian functors says that a fibered functor
\[\begin{tikzcd}
	F && E \\
	A && B
	\arrow["\xi"', two heads, from=1-1, to=2-1]
	\arrow[from=2-1, to=2-3]
	\arrow[from=1-1, to=1-3]
	\arrow["\pi", two heads, from=1-3, to=2-3]
\end{tikzcd}\]
between cocartesian fibrations is a cocartesian functor if and only if the mate of the induced square
\[\begin{tikzcd}
	{F^{\Delta^1}} && {E^{\Delta^1}} \\
	{\xi \downarrow A} && {\pi \downarrow B}
	\arrow["i_0 \cotens \xi", swap, two heads, from=1-1, to=2-1]
	\arrow[from=2-1, to=2-3]
	\arrow[from=1-1, to=1-3]
	\arrow["i_0 \cotens \pi", two heads, from=1-3, to=2-3]
	\arrow["{=}"{description}, Rightarrow, from=2-1, to=1-3]
\end{tikzcd}\]
is invertible.\footnote{A development of the required results about adjunctions in simplicial type theory is given in \cite[Appendix~A and B]{BW21}.}

%% file: cocart-arr.tex
\subsection{Definition and properties}

The starting point are \emph{cocartesian arrows}, which are dependent arrows in a fibration characterized by a certain initial universal property. Our type-theoretic definition generalizes the classical $1$-categorical picture, and is a synthetic version of the $\inftyone$-categorical one. Thanks to the extension types, cocartesian lifts of an arrow will lie \emph{strictly} over their base arrow.

\begin{definition}[Cocartesian arrow]
	Let $B$ be a type and $P: B \to \UU$ be an inner family.
	Let $b,b': B$, $u:\hom_B(b,b')$, and $e:P\,b$, $e':P\,b'$. An arrow $f : \hom_{P\,u}(e,e')$ is a \emph{($P$-)cocartesian morphism} or \emph{($P$-)cocartesian arrow} iff
	\[
	\isCocartArr_P f :\jdeq \prod_{\sigma:\ndexten{\Delta^2}{B}{\Delta_0^1}{u}}
	\prod_{h:\prod_{t:\Delta^1} P\,\sigma(t,t)} \isContr \Big( \exten{\pair{t}{s}:\Delta^2}{P \sigma(t,s)}{\Lambda_0^2}{[f,h]}  \Big).
	\]
\end{definition}
This is illustrated in~\Cref{fig:cocart-arr}. Notice that being a cocartesian arrow is a homotopy proposition.
\begin{figure}
	\centering
	\includegraphics{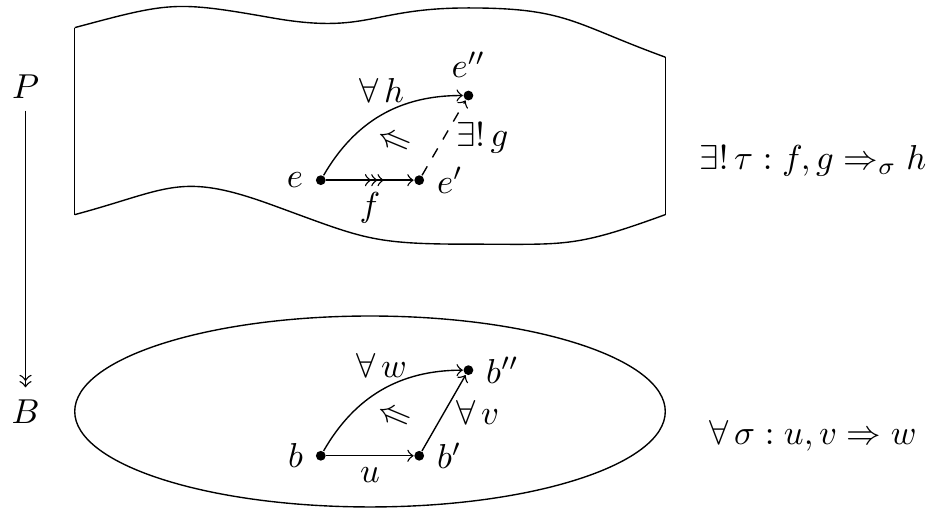}
	\caption{The universal property of cocartesian arrows.}
	\label{fig:cocart-arr}
\end{figure}
By expressing the functions on simplices in terms of
objects, morphisms and composition, we obtain an equivalent type:
\begin{equation}\label{eq:isCocartArrEquiv}
	\begin{split}
		\isCocartArr_P f \equiv
		&\prod_{b'':B} \prod_{v:\hom_B(b',b'')} \prod_{w:\hom_B(b,b'')}
		\prod_{\sigma:\hom^2_B(u,v;w)} \prod_{e'':P\,b''} \prod_{h:\dhom_{P\,w}(e,e'')} \\
		&\isContr\big( \sum_{g:\dhom_{P\,v}(e',e'')} \dhom^2_{P\,\sigma}(f,g;h) \big)
	\end{split}
\end{equation}

Diagrammatically, this is expressed as the existence of a filler, uniquely up to homotopy, as in a diagram of the following form, cf.~\Cref{ssec:fib-shapes}:
\[\begin{tikzcd}
	{\Delta^1} & {\Lambda_2^2} & {\widetilde{P}} \\
	& {\Delta^2} & {B}
	\arrow["{\{0,1\}}"', from=1-1, to=1-2]
	\arrow[from=1-2, to=2-2]
	\arrow[from=1-2, to=1-3]
	\arrow[from=2-2, to=2-3]
	\arrow["{\pi}", from=1-3, to=2-3]
	\arrow["{f}", from=1-1, to=1-3, curve={height=-18pt}]
	\arrow[from=2-2, to=1-3, dashed]
\end{tikzcd}\]

\begin{definition}[Cocartesian lift]
	Let $B$ be a type and $P: B \to \UU$ be an inner family.
	For $b,b': B$, $u:\hom_B(b,b')$, and $e:P\,b$,
	we define the type of \emph{($P$-)cocartesian lifts of $u$
		starting at $e$} to be
	\[
	\CocartLift_P(u,e) \defeq
	\sum_{e':P\,b'} \sum_{f : \hom^P_u(e,e')}
	\isCocartArr_P f.
	\]
\end{definition}
For Segal types, where composites are uniquely determined,
we can further rewrite \eqref{eq:isCocartArrEquiv}:
\footnote{Here, composition in the family is to be understood as
	\emph{dependent composition} in the sense of \cite[Remark~8.11.]{RS17}.}
\[
\isCocartArr_P f \equiv \prod_{b'':B}~\prod_{v:b'\to_B b''}~
\prod_{e'':P\,b''}~ \prod_{h:e\to^P_{v \circ u} e''}
\isContr\Bigl( \sum_{g:e'\to^P_v e''} g \circ^P f = h \Bigr).
\]

In fact, the cocartesian cells in an isoinner family are exactly the $i_0$-LARI cells in the sense of~\Cref{ssec:lari-cells}, where $i_0: \unit \hookrightarrow \Delta^1$ denotes the inclusion of the initial vertex:\footnote{We thank Emily Riehl for pointing out this perspective and its consequences---newly found in $\infty$-cosmos theory~\cite[Theorem~5.1.7]{RV21}---and for suggesting to use it in simplicial type theory, which has led to the new treatise in~\Cref{ssec:reladj,ssec:lari-stuff}. In fact, for the concrete case of cocartesian arrows, the other fibered Chevalley criterion~\cite[Theorem~5.1.7(iii)]{RV21} also is accessible in sHoTT, using a fibered version relative adjunctions. But this is omitted in this thesis.}
\[\begin{tikzcd}
	&& {E^{\Delta^1}} \\
	\unit && {B^{\Delta^1} \times_{B} E}
	\arrow[""{name=0, anchor=center, inner sep=0}, "{\langle b,v,e\rangle}"', from=2-1, to=2-3]
	\arrow["g", from=2-1, to=1-3]
	\arrow["{i_0 \cotens \pi}", two heads, from=1-3, to=2-3]
	\arrow[shorten >=7pt, Rightarrow, no head, from=1-3, to=0]
\end{tikzcd}\]
The relative adjointness says that $g$ mediates a fibered equivalence
\[ \comma{g}{E^{\Delta^1}} \equiv_E \comma{\angled{b,u,e}}{i_0 \cotens \pi}.\]

 This unfolds to the following: An arrow $g$ in $E$ over $v$ in $B$ is cocartesian if and only if the ``cubical'' version of cocartesianess property is satisfied. We will not detail on this here, but geometrically it is a routine proof using the usual equivalences between extension types. The cubical cocartesianness property is depicted as follows, where the ``short arrows'' from the previous picture have now been collapsed into points (at the initial vertex):\footnote{Note that both definitions are propositions. The cubical implies the traditional ``simplicial'' one by degenerating the right upper horizontal arrows to identities. Conversely, by composition the cubical formulation is implied by the simplicial one by composing the arrows making up right upper halves of the diagram.}
\[\begin{tikzcd}
		&& \cdot && \cdot \\
		E && \cdot && \cdot \\
		\\
		B && \cdot && \cdot \\
		&& \cdot && \cdot
		\arrow["v"', from=4-3, to=5-3]
		\arrow[from=4-3, to=4-5]
		\arrow[from=5-3, to=5-5]
		\arrow[from=4-5, to=5-5]
		\arrow["g"', from=1-3, to=2-3]
		\arrow[from=1-3, to=1-5]
		\arrow[from=1-5, to=2-5]
		\arrow[dashed, from=2-3, to=2-5]
		\arrow[two heads, from=2-1, to=4-1]
\end{tikzcd}\]

This is summarized as follows:
\begin{prop}[Chevalley Criterion for cocartesian arrows,~\cf\protect{\cite[Theorem~5.1.7]{RV21}}]
	In an isoinner family $P:B \to \UU$ over a Rezk type $B$, a dependent arrow $f$ over $u:\comma{b}{B}$ with starting vertex $e:P\,b$ is cocartesian if and only if $f$ together with the identity homotopy makes the following triangle an absolute left lifting diagram:
	\[\begin{tikzcd}
		&& {E^{\Delta^1}} \\
		\unit && {B^{\Delta^1} \times_{B} E}
		\arrow[""{name=0, anchor=center, inner sep=0}, "{\langle b,u,e\rangle}"', from=2-1, to=2-3]
		\arrow["f", from=2-1, to=1-3]
		\arrow["{i_0 \cotens \pi}", two heads, from=1-3, to=2-3]
		\arrow[shorten >=7pt, Rightarrow, no head, from=1-3, to=0]
	\end{tikzcd}\]
\end{prop}

We find several formal properties of cocartesian arrows:

\begin{prop}[Uniqueness of cocartesian lifts (in isoinner families); \cite{RV21}, Lem.~5.1.3]\label{prop:cocart-lifts-unique-in-isoinner-fams}
	Let $B$ be a Rezk type and $P: B \to \UU$ be an isoinner family. Then $P$-cocartesian lifts of arrows of $B$ are unique up to homotopy.
\end{prop}

Note that this is also implied by their characterization as relative left adjoint.

\begin{cor}\label{cor:cocart-trivfill}
	Let $P:B \to \UU$ be an isoinner family over a Rezk type.
	\begin{enumerate}
		\item\label{it:cocart-fill-comp} If $f:e \to e'$ is a $P$-cocartesian arrow, and $h:e' \to e''$ is an arbitrary arrow, then $\tyfill_f(h) \circ f = h$.
		\item\label{it:cocart-comp-fill} If $f:e \to e'$ is a $P$-cocartesian arrow, and $g:e' \to e''$ is an arbitrary arrow, then $\tyfill_f(gf) = g$.
	\end{enumerate}
\end{cor}

\begin{prop}\label{prop:cocart-arr-closure} Let $P: B \to \UU$ be a cocartesian family over a Rezk type $B$. For arrows $u:\hom_B(b,b')$, $v:\hom_B(b',b'')$, with $b,b',b'':B$, consider dependent arrows $f:\dhom_{P,u}(e,e')$, $g:\dhom_{P,v}(e',e'')$ lying over, for $e:P\,b$, $e':P\,b'$, $e'':P\,b''$.
	\begin{enumerate}
		\item\label{it:cocart-arr-comp} If both $f$ and $g$ are are cocartesian arrows, then so is their composite $g \circ f$.
		\item\label{it:cocart-arr-cancel} If $f$ and $g \circ f$ are cocartesian arrows, then so is $g$.
	\end{enumerate}
\end{prop}

\begin{lem}[\cite{RV21}, Lem.~5.1.4]\label{lem:cocart-arrows-isos}
	Let $P: B \to \UU$ be an inner family over a Segal type $B$.
	\begin{enumerate}
		\item\label{it:depisos-are-over-isos} If $f$ is a dependent isomorphism in $P$ over some morphism $u$ in $B$, then $u$ is itself an isomorphism.
		\item\label{it:depisos-are-cocart} Any dependent isomorphism in $P$ is cocartesian.
		\item\label{it:cocart-arrows-over-ids-are-isos} If $f$ is a cocartesian arrow in $P$ over an identity in $B$, then $f$ is an isomorphism.
	\end{enumerate}
\end{lem}

\subsection{Characterizations of cocartesian arrows}

\begin{prop}
	Let $B$ be a Rezk type, and $P:B \to \UU$ be an isoinner family with total type $E\defeq \totalty{P}$. Let $u:\hom_B(b,b')$, $b,b':B$, be a morphism in $B$ and $f:\dhom_{P,u}(e,e')$, $e:P\,b$, $e':P\,b'$ a dependent morphism.
	
	The morphism $f$ is cocartesian if and only if, for any $b'':B$, $e'':P\,b''$, the diagram
	\[
	\begin{tikzcd}
		\hom_E(\pair{b'}{e'}, \pair{b''}{e''}) \ar[r, "{-\circ \pair{u}{f}}"] \ar[d] & \hom_E(\pair{b}{e}, \pair{b'}{e'})  \ar[d] \\
		\hom_B(b',b'') \ar[r, "{-\circ u}"] & \hom_B(b,b'')
	\end{tikzcd}
	\]
	is a pullback.
\end{prop}

\begin{proof}
	By fibrant replacement, de-/strictification, and choice for extension types, we can replace the square in question by:
	\[
	\hspace{4cm}
	\begin{tikzcd}
		{} \\
		&& {\mathllap{\sum_{w:b \to b''} \sum_{\sigma:\ndexten{\Delta^2}{B}{\Lambda_0^2}{[u,w]}} \sum_{h:e \to^P_w e''}} \exten{\pair{t}{s}:\Delta^2}{P(\sigma(t,s))}{\Lambda_0^2}{[f,h]}}& {\sum_{w:b \to b''} (e \to_w^P e'')} && {} \\
		&& {\sum_{w:b\to b''} \ndexten{\Delta^2}{B}{\Lambda_0^2}{[u,w]}} & {\hom_B(b,b'')} \\
		&& {} &&& {}
		\arrow[from=2-3, to=3-3]
		\arrow[from=2-3, to=2-4]
		\arrow[from=2-4, to=3-4]
		\arrow[from=3-3, to=3-4]
	\end{tikzcd}\]
	Undwinding what it means for this square to be a pullback precisely recovers the condition that $f$ be cocartesian.
\end{proof}

We prove a characterization for cartesian edges, recovering the definition established by Joyal and Lurie, and transferred to complete Segal spaces by Rasekh.

Let $B$ be a
Segal type and $P: B \to \UU$ be an isoinner family. Consider its total space $E :\jdeq \sum_{b:B} P\,b$. For $b,b':B$, let $u:\hom_B(b,b')$ an arrow with a dependent arrow $f:\dhom_{P,f}(e,e')$ above it, where $e:P\,b, e':P\,b'$.

There is an induced commutative square involving comma objects, each of which can be described using extension types:
\[\begin{tikzcd}
	{} & {\mathllap{\sum_{\sigma:u \downarrow B} \exten{\pair{t}{s}:\Delta^1}{P(\sigma(t))}{\Delta_1^1}{f}} \simeq f/E} && {u/B \simeq \mathrlap{\ndexten{\Delta^2}{B}{\Delta_1^1}{u}}} & {} \\
	{} &&& {} & {} \\
	& {\mathllap{\sum_{u:b\downarrow B} \exten{t:\Delta^1}{P(u(t))}{0}{e}} \simeq e/E} && {b/B}
	\arrow[from=1-2, to=1-4]
	\arrow["{\partial_0}"', from=1-2, to=3-2]
	\arrow[from=3-2, to=3-4]
	\arrow["{\partial_0}", from=1-4, to=3-4]
\end{tikzcd}\]

\begin{prop}[Joyal's Criterion]
	Let $B$ be Segal and $P: B \to \UU$ be an inner family. Write $E\defeq \totalty{P}$ and consider the canonical projection $\pi : E \to B$. A dependent arrow $f:\dhom_{P,u}(e,e')$, $e:P\,b, e':P\,b'$, $u:\hom_B(b,b')$, $b,b':B$, is cocartesian if and only if the mediating map $f/E \to e/E \times_{b/B} u/B$ occurring in the pullback~\Cref{fig:joyals-criterion}
	is an equivalence.
\end{prop}

\begin{figure}
	\[\begin{tikzcd}
		{f/E} \\
		& { e/E \times_{b/B} u/B} && {u/B} \\
		& {e/E} && {b/B}
		\arrow[from=2-2, to=3-2]
		\arrow[from=3-2, to=3-4]
		\arrow[from=2-2, to=2-4]
		\arrow["{\partial_0}", from=2-4, to=3-4]
		\arrow[curve={height=-18pt}, from=1-1, to=2-4]
		\arrow["{\partial_0}"', curve={height=18pt}, from=1-1, to=3-2]
		\arrow["\lrcorner"{anchor=center, pos=0.125}, draw=none, from=2-2, to=3-4]
		\arrow[dashed, from=1-1, to=2-2]
	\end{tikzcd}\]
\caption{Joyal's Criterion for cocartesian edges}
\label{fig:joyals-criterion}
\end{figure}
\begin{proof}
	Note that by precondition both $B$ and $E$ are Segal types. The map $\varphi$ is an equivalence if and only if
	\[
	\prod_{b'':B} \, \prod_{v:\hom_B(b',b'')} \, \prod_{e'':P \, b''} \, \prod_{h:\dhom_{P,v \circ u}(e,e'')} \, \isContr \Big( \sum_{g:\dhom_{P.v}(e',e'')} g \circ f = h \Big)
	\]
	which is equivalent to $f$ being cocartesian.
\end{proof}

%% file: cocart-fib.tex
Cocartesian families are families of Rezk types such that every map in the base has a cocartesian lift w.r.t.~a choice of the source vertex. After showing elementary properties such as functoriality we prove the Chevalley criterion which exhibits cocartesian families as LARI fibrations (w.r.t.~to the initial vertex inclusion $i_0:\unit \to \Delta^1$). It then follows that cocartesian fibrations are closed under pullback, composition, and dependent products. We proceed by giving three kinds of examples of cocartesian families: the domain projection, the codomain projection in case the base category has all pushouts, and the cocartesian replacement of an arbitrary map between Rezk types. Relating to the latter, we show that cocartesian replacement really is a left adjoint.

Independently, definitions of cocartesian families have also been given in an unpublished section of~\cite{RS17}, and in the formalization~\cite{LicataMoreFibs}. The definition presented here is due to the author's joint work with Buchholtz~\cite{BW21}.

\subsection{Definition and basic properties}
Cocartesian type families are those (isoinner) families that have all cocartesian lifts, in the sense of~\Cref{def:enough-lari-lifts}:

\begin{defn}[Cocartesian lifting property]
	A family $P: B \to \UU$ is said to \emph{have (all) cocartesian lifts} if
	\[
	\hasCocartLifts \, P :\jdeq \prod_{b,b':B} \prod_{u:\hom_B(b,b')} \prod_{e:P\,b} \sum_{e':P\,b'} \sum_{f:\hom_{P\,u}(e,e')} \isCocartArr_P \, f.
	\]
\end{defn}

\begin{defn}[Cocartesian family]
	For any type $B$, we call a family $P : B \to \UU$ a \emph{cocartesian family} if
	\[
	\isCocartFam \, P :\jdeq \isIsoInnerFam \, P \times \hasCocartLifts \, P.
	\]
\end{defn}


If $B$ is a Rezk type and $P:B \to \UU$ is a cocartesian family, then any arrow $u:\hom_B(a,b)$ induces a map $\coliftptfammap{P}{u}: P\,a \to P\,b$ defined by
\[ \coliftptfammap{P}{u}: \jdeq \lambda d. \partial_1 \coliftarr{P}{u}{d}.\]
We will often omit the superscript if the family is clear from the context.

\begin{defn}[Vertical arrow]\label{def:vert-arr}
	Let $\pi:E \to B$ be an inner fibration over a Segal type. A dependent arrow $f: e \to e'$ is called \emph{vertical} if $\pi \, f$ is an isomorphism. We also write $f:e \rightsquigarrow e'$ to indicate that $f$ is vertical.
\end{defn}

Observe that, since being an isomorphism is a proposition in a Segal type by Proposition~1.10, \cite{RS17}, being a vertical arrow is a proposition when in an inner family over a Segal type. In a cocartesian family one recovers the classically well-known fact that any dependent arrow factors as a cocartesian arrow followed by a vertical arrow. Furthermore, in a cocartesian family, vertical arrows are stable under pullback.

Cocartesian families, generalizing $1$-categorical Grothendieck opfibrations,\footnote{Note that the liftings really are strict up to honest equality thanks to the extension types. This matches with the situation of cocartesian fibrations of quasi-categories or Rezk spaces.} implement the idea of a \emph{functorial} family of Rezk types, \ie~in addition to transport along paths---which exists for arbitrary type families---there is also a notion of transport along \emph{directed} arrows, which turns out to be compatible with the familiar path transport.

\begin{prop}[Functoriality]\label{prop:cocart-functoriality}
	Let $B$ be a Rezk type and $P:B \to \UU$ a cocartesian family. For any $a:B$ and $x:P\,a$ there is an identity
	\[ \coliftarr{P}{\id_b}{x} = \id_x, \]
	and for any $u:\hom_B(a,b)$, $v:\hom_B(b,c)$, there is an identity
	\[ \coliftarr{P}{v \circ u}{x} = \coliftarr{P}{v}{\coliftpt{u}{x}} \circ \coliftarr{P}{u}{x}. \]
\end{prop}

\begin{proof}
	The first claim follows from~\Cref{lem:cocart-arrows-isos}(\ref{it:cocart-arrows-over-ids-are-isos}), in combination with~\Cref{prop:cocart-lifts-unique-in-isoinner-fams}.
	
	The second claim follows from~\Cref{prop:cocart-arr-closure}(\ref{it:cocart-arr-comp}).
\end{proof}

\begin{prop}
	Let $B$ be a Rezk type and $P:B \to \UU$ be a cocartesian family. For any arrow $u:\hom_B(a,b)$ and terms $d:P\,a$, $e:P\,b$, we have equivalences between the types of (cocartesian) lifts of arrows from $d$ to $e$ and maps (equivalences) from $u_!d$ to $e$:
	\[\begin{tikzcd}
		{(d} & {e)} && {(u_!d} & {e)} \\
		{(d} & {e)} && {(u_!d} & {e)}
		\arrow["\equiv", from=1-2, to=1-4]
		\arrow["\equiv", from=2-2, to=2-4]
		\arrow[""{name=0, anchor=center, inner sep=0}, "u"', from=1-1, to=1-2, cocart]
		\arrow[""{name=1, anchor=center, inner sep=0}, "{P\,b}" below, from=2-4, to=2-5]
		\arrow[""{name=2, anchor=center, inner sep=0}, "u"', from=2-1, to=2-2]
		\arrow[""{name=3, anchor=center, inner sep=0}, "{P\,b}" below,equals,  from=1-4, to=1-5]
		\arrow[shorten <=9pt, shorten >=6pt, hook, from=0, to=2]
		\arrow[shorten <=6pt, shorten >=9pt, hook, from=3, to=1]
	\end{tikzcd}\]
\end{prop}

\begin{proof}
	Consider the maps
	\begin{align*}
		\Phi & :  (d \to_u e) \longrightarrow (u_!d \to^{P\,b} e), \quad \Phi(f):\jdeq \tyfill_{P_!(u,d)}(f), \\
		\Psi & : (u_!d \to^{P\,b} e) \longrightarrow (d \to_u e), \quad \Psi(p):\jdeq p \circ P_!(u,d).
	\end{align*}
	From the universal property of cocartesian fillers, we find
	\begin{align*}
		\Phi(\Psi(g)) & = \tyfill_{P_!(u,d)}(g \circ P_!(u,d)) = g, \\
		\Psi(\Phi(h)) & = \tyfill_{P_!(u,d)}(h) \circ P_!(u,d) = h.
	\end{align*}
	
	Now, if $u:d \cocartarr_u e$ is a ccoartesian arrow by uniqueness of cocartesian lifts the filler $\Phi(u)$ must be a path. Thus $\Phi$ restricts to a map $\Phi':(d \cocartarr_u e) \longrightarrow (u_!d =_{P\,b} e)$.
	
	If $p:u_!d = e$ is a path, it is in particular cocartesian, so $\Psi(p)$ is as well since cocartesian arrows are closed under composition. Hence, $\Psi$ restricts to a map
	\[ \Psi':(u_!d =_{P\,b} e) \longrightarrow (d \cocartarr_u e).\]
\end{proof}

Thus, just as in the classical case, our cocartesian families capture the notion of covariantly functorial families of categories.

Due to \Cref{prop:cocart-lifts-unique-in-isoinner-fams}, over Rezk types being a cocartesian family is a proposition, and indeed this is the setting that we are interested in. In particular, cocartesian families over Rezk types are thus ``cloven up to homotopy''. Given an arrow $u:\hom_B(b,b')$ together with $e:P\,b$, we write $P_!(u,e)$ for homotopically unique cocartesian lift of $u$. Even more, from the point of view of homotopy type theory, these cleavages are automatically ``\emph{split}''.\footnote{in the sense analogous to \cite{streicher2020fibered}, Definition~3.1} 

\subsection{Characterizations of cocartesian families}

	\begin{figure}
	\centering
	\[\begin{tikzcd}
		&& e & d && e & d \\
		e & d & b & c && b & c & e & d \\
		{u_!e} & d & {b'} & {c'} && {b'} & {c'} & {u_!e} & d
		\arrow[from=2-1, to=3-1]
		\arrow["m"', from=3-1, to=3-2]
		\arrow["k", from=2-1, to=2-2]
		\arrow[""{name=0, anchor=center, inner sep=0}, "g"', from=2-2, to=3-2]
		\arrow[""{name=1, anchor=center, inner sep=0}, "u", from=2-3, to=3-3]
		\arrow["{\pi\,m}"', from=3-3, to=3-4]
		\arrow["{\pi\,k}", from=2-3, to=2-4]
		\arrow["v"', from=2-4, to=3-4]
		\arrow["k", from=1-3, to=1-4]
		\arrow["u", from=2-6, to=3-6]
		\arrow["r"', from=3-6, to=3-7]
		\arrow["w", from=2-6, to=2-7]
		\arrow[""{name=2, anchor=center, inner sep=0}, "v"', from=2-7, to=3-7]
		\arrow["k", from=1-6, to=1-7]
		\arrow[dashed, from=3-8, to=3-9]
		\arrow["k", from=2-8, to=2-9]
		\arrow["g", from=2-9, to=3-9]
		\arrow[""{name=3, anchor=center, inner sep=0}, from=2-8, to=3-8]
		\arrow[Rightarrow, dotted, no head, from=1-3, to=2-3]
		\arrow[Rightarrow, dotted, no head, from=1-4, to=2-4]
		\arrow[Rightarrow, dotted, no head, from=1-6, to=2-6]
		\arrow[Rightarrow, dotted, no head, from=1-7, to=2-7]
		\arrow["\Phi", shorten <=6pt, shorten >=6pt, maps to, from=0, to=1]
		\arrow["\Psi", shorten <=6pt, shorten >=6pt, maps to, from=2, to=3]
	\end{tikzcd}\]
	\caption{Transposing maps of the adjunction $\chi \dashv i_0 \cotens \pi$}\label{fig:transp-chevalley}
\end{figure}

\paragraph{Cocartesian families via lifting}

We find that cocartesian families are exactly the $i_0$-LARI families, for $i_0:\unit \hookrightarrow \Delta^1$ the inclusion of the initial vertex.

\begin{theorem}[Chevalley criterion: Cocartesian families via lifting, \cite{RV21}, ~Prop.~5.1.11(ii)]\label{thm:cocart-fams-intl-char}
	Let $B$ be a Rezk type, $P: B \to \UU$ be an isoinner family, and denote by $\pi: E \to B$ the associated projection map. The family $P$ is cocartesian if and only if the Leibniz cotensor map $i_0 \cotens \pi: E^{\Delta^1} \to \comma{\pi}{B}$ has a left adjoint right inverse:
	\[\begin{tikzcd}
		{E^{\Delta^1}} & {} \\
		&& {\pi \downarrow B} && E \\
		&& {B^{\Delta^1}} && B
		\arrow[two heads, from=2-3, to=3-3]
		\arrow["{\partial_0}"', from=3-3, to=3-5]
		\arrow[from=2-3, to=2-5]
		\arrow["\pi", two heads, from=2-5, to=3-5]
		\arrow["\lrcorner"{anchor=center, pos=0.125}, draw=none, from=2-3, to=3-5]
		\arrow["{\partial_0}", shift left=2, curve={height=-18pt}, from=1-1, to=2-5]
		\arrow[""{name=0, anchor=center, inner sep=0}, "\chi"', curve={height=12pt}, dashed, from=2-3, to=1-1]
		\arrow[""{name=1, anchor=center, inner sep=0}, "{i_0\widehat{\pitchfork}\pi}"', curve={height=12pt}, from=1-1, to=2-3]
		\arrow["{\pi^{\Delta^1}}"', shift right=2, curve={height=18pt}, two heads, from=1-1, to=3-3]
		\arrow["\dashv"{anchor=center, rotate=-118}, draw=none, from=0, to=1]
	\end{tikzcd}\]
\end{theorem}

\begin{proof}
This follows formally from~\Cref{thm:lari-fams-lifting}.	
\end{proof}

\subsubsection{Cocartesian families via transport}

There is another characterization of cocartesian families in terms of an adjointness condition. Any map $\pi:E \to B$ between Rezk types is exhibited as a retract of the pullback map $\partial_0^*\pi:\comma{\pi}{B} \to B^{\Delta^1}$ in the following way:
\[\begin{tikzcd}
	E \\
	&& {\pi\downarrow B} &&& E \\
	&& {} \\
	B && {B^{\Delta^1}} &&& B
	\arrow["\pi", from=2-6, to=4-6]
	\arrow[curve={height=-24pt}, no head, from=1-1, to=2-6, equals]
	\arrow["{\mathrm{cst}}"', from=4-1, to=4-3]
	\arrow["{\partial_0^*\pi}", from=2-3, to=4-3]
	\arrow[from=2-3, to=2-6]
	\arrow["{\partial_0}"', from=4-3, to=4-6]
	\arrow["\pi"{description},from=1-1, to=4-1]
	\arrow["{\iota_\pi}"{description}, dashed, from=1-1, to=2-3]
	\arrow["\lrcorner"{anchor=center, pos=0.125}, draw=none, from=2-3, to=4-6]
	\arrow[from=1-1, to=4-3]
	\arrow[curve={height=26pt}, from=4-1, to=4-6, equals]
\end{tikzcd}\]
The mediating map
\[ \iota \defeq \iota_\pi \defeq \lambda \pair{b}{e}.\pair{\id_b}{e}: E \to \comma{\pi}{B} \]
is a fibered functor from $\pi:E \to B$ to its cocartesian replacement,\footnote{Explicitly, $\partial'_1 \defeq \lambda u,e.u(1):\comma{\pi}{B} \fibarr B$, cf.~\Cref{def:free-cocart}.} \ie~there is a commutative triangle:
\[\begin{tikzcd}
	E && {\comma{\pi}{B}} \\
	& B
	\arrow["{\iota_\pi}", from=1-1, to=1-3]
	\arrow["\pi"', two heads, from=1-1, to=2-2]
	\arrow["{\partial'_1 }", two heads, from=1-3, to=2-2]
\end{tikzcd}\]
Denote the family of fibers of $\pi$ by $P:B \to \UU$. If $P$ is cocartesian, it has ``directed transport''
\[ \tau \defeq \tau_P \defeq \transp_P\defeq \lambda u,e.u^P_!e:\comma{\pi}{B} \to E.\]
This transport map $\tau$ is easily checked to be a fibered functor from $\partial'_1$ to $\pi$. We will show that it is a fibered left adjoint of $\iota_\pi$, and conversely, the existence of a fibered left adjoint $\tau_\pi$ to $\iota_\pi$ will imply that $\pi$ is cocartesian.

\begin{theorem}[Cocartesian families via transport, \cite{RV21}, ~Prop.~5.1.11(ii)]\label{thm:cocartfams-via-transp}
	Let $B$ be a Rezk type, and $P:B \to \UU$ an isoinner family with associated total type projection $\pi:E \to B$.
	
	Then, $P$ is cocartesian if and only if the map
	\[ \iota \defeq \iota_P : E \to \comma{\pi}{B}, \quad \iota \, \pair{b}{e} :\jdeq \pair{\id_b}{e}  \]
	has a fibered left adjoint $\tau \defeq \tau_P: \comma{\pi}{B} \to E$ as indicated in the diagram:
	\[\begin{tikzcd}
		{} && E && {\pi \downarrow B} \\
		\\
		&&& B
		\arrow[""{name=0, anchor=center, inner sep=0}, "\iota"', curve={height=12pt}, from=1-3, to=1-5]
		\arrow["\pi"', two heads, from=1-3, to=3-4]
		\arrow["{\partial_1'}", two heads, from=1-5, to=3-4]
		\arrow[""{name=1, anchor=center, inner sep=0}, "\tau"', curve={height=12pt}, from=1-5, to=1-3]
		\arrow["\dashv"{anchor=center, rotate=-90}, draw=none, from=1, to=0]
	\end{tikzcd}\]
\end{theorem}

\begin{proof}
	Assume $P$ is cocartesian. For the candidate left adjoint we take the map given by cocartesian transport
	\[ \tau:\comma{\pi}{B} \to E, \quad \tau(u,e)\defeq \pair{\partial_1 \, u}{u_!e}.\]
	Then $\pi(\tau(u,e)) \jdeq \partial_1 \,u \jdeq \partial_1'(u,e)$, so $\tau$ is a fibered functor from $\partial_1'$ to $\pi$. We show that for any $\pair{u:b \to b'}{e:P\,b}$ and $\pair{c:B}{d:P\,c}$ in $E$ the maps
	\[\begin{tikzcd}
		{\hom_{E}(\langle b',u_!e \rangle, \langle c,d \rangle)} && {\hom_{\pi \downarrow B}(\langle u, e \rangle,\langle \mathrm{id}_c, d \rangle)}
		\arrow["{\Phi_{\langle u,e \rangle, \langle c,d\rangle}}", shift left=2, from=1-1, to=1-3]
		\arrow["{\Psi_{\langle u,e \rangle, \langle c,d\rangle}}", shift left=2, from=1-3, to=1-1]
	\end{tikzcd}\]
	defined by
	\[ \Phi_{\langle u,e \rangle, \langle c,d\rangle}(v,g) \defeq \angled{vu,v,g \circ P_!(u,e)}, \quad \Psi_{\langle u,e \rangle, \langle c,d\rangle}(w,v,h) \defeq \pair{v}{\tyfill_{P_!(u,e)}(h)}\]
	form a quasi-equivalence (cf.~\Cref{fig:transp-fibadj} for illustration).
	We have
	\[ \Phi(\Psi(w,v,h)) = \Phi(v, \tyfill_{P_!(u,e)}(h)) = \angled{vu,v,\tyfill_{P_!(u,e)}(h) \circ P_!(u,e)} = \angled{w,v,h}\]
	by \Cref{cor:cocart-trivfill}(\ref{it:cocart-fill-comp}), and noting that for any square $\pair{w}{v} : u \to \id_{c}$ there is an identification $w=vu$.
	Next, we find
	\[ \Psi(\Phi(v,g)) = \Psi(vu,v,g \circ P_!(u,e)) = \pair{v}{\tyfill_{P_!(u,e)}(g \circ P_!(u,e))} = \pair{v}{g}.\]
	using again the properties of the fillers defined by the cocartesian lifts, cf.~\Cref{cor:cocart-trivfill}(\ref{it:cocart-comp-fill}).
	
	So indeed $\tau$ is left adjoint to $\iota$. Moreover, it is a fibered left adjoint as can be seen as follows. The unit is defined by
	\[ \eta_{\pair{u}{e}} \defeq \Phi(\id_{b'},\id_{u_!e}) = \angled{u, \id_{b'}, P_!(u,e)}:\hom_{\comma{\pi}{B}}(\pair{u}{e},\pair{\id_{b'}}{u_!e}).\]
	Since the second component is an identity this is a vertical arrow in $\partial_1': \comma{\pi}{B} \fibarr B$ which proves the fiberedness of the adjunction.
	
	Suppose on the converse that $\tau \dashv_B \iota$ is some fibered left adjunction. Since $\tau$ is a fibered functor, for $\pair{u:b\to b'}{e:P\,b}$ we can assume
	\[ \tau(u,e) \jdeq \pair{b'}{\widehat{\tau}(u,e)}. \]
	Next, $\eta$ being a \emph{fibered} natural transformation fixes its part in $B$, \ie~since in the square the lower horizontal edge has to be an identity the upper horizontal edge must be $u$ (up to identification), so the only degree of freedom is the dependent arrow $f_{u,e}$ as indicated:
	\[\begin{tikzcd}
		& e & {\widehat{\tau}(u,e)} \\
		{\eta_{u,e}:} & b & {b'} \\
		& {b'} & {b'}
		\arrow["{f_{u,e}}", from=1-2, to=1-3]
		\arrow["u"', from=2-2, to=3-2]
		\arrow[Rightarrow, no head, from=3-2, to=3-3]
		\arrow["u", from=2-2, to=2-3]
		\arrow[Rightarrow, no head, from=2-3, to=3-3]
	\end{tikzcd}\]
	Hence, we can assume
	\[ \eta_{u,e} \jdeq \angled{u,\id_{b'}, f_{u,e}}.\]
	By assumption the transposing map induced by the unit
	\[ \Phi \defeq \Phi_\eta \defeq \lambda v,g.\iota(v,g) \circ \eta_{u,e} \jdeq \angled{vu,v,g \circ f_{u,e}}:\hom_E(\iota(u,e),\pair{c}{d}) \longrightarrow \hom_{\comma{\pi}{B}}(\pair{u}{e},\pair{\id_c}{d})\]
	is an equivalence.	Spelled out, this means for any $v:b' \to c$, $h:e \to^P_{vu} d$ there exists an arrow $g_h: \widehat{\tau}(u,e) \to_v d$, uniquely up to homotopy, s.t.~$h = g_h \circ f_{u,e}$. This says exactly that $f_{u,e}:e \to \widehat{\tau}(u,e)$ is a cocartesian lift of $u$ \wrt~$e$.
\end{proof}

\begin{figure}[ht]
	\centering
	\[\begin{tikzcd}
		&& e & d && e & d \\
		{u_!e} & d & b & c && b & c & {u_!e} & d \\
		{b'} & c & {b'} & c && {b'} & c & {b'} & c
		\arrow["v", from=3-1, to=3-2]
		\arrow["g", from=2-1, to=2-2]
		\arrow[""{name=0, anchor=center, inner sep=0}, "u", from=2-3, to=3-3]
		\arrow["v"', from=3-3, to=3-4]
		\arrow["vu", from=2-3, to=2-4]
		\arrow["{\mathrm{id}_c}", Rightarrow, no head, from=2-4, to=3-4]
		\arrow["u"', from=2-6, to=3-6]
		\arrow["v"', from=3-6, to=3-7]
		\arrow["w", from=2-6, to=2-7]
		\arrow[""{name=1, anchor=center, inner sep=0}, "{\mathrm{id}_c}"', Rightarrow, no head, from=2-7, to=3-7]
		\arrow["{g\circ P_!(u,e)}", from=1-3, to=1-4]
		\arrow["h", from=1-6, to=1-7]
		\arrow[Rightarrow, dotted, no head, from=2-1, to=3-1]
		\arrow[""{name=2, anchor=center, inner sep=0}, Rightarrow, dotted, no head, from=2-2, to=3-2]
		\arrow[Rightarrow, dotted, no head, from=1-3, to=2-3]
		\arrow[Rightarrow, dotted, no head, from=1-4, to=2-4]
		\arrow[Rightarrow, dotted, no head, from=1-6, to=2-6]
		\arrow[Rightarrow, dotted, no head, from=1-7, to=2-7]
		\arrow["{\mathrm{fill}_{P_!(u,e)}(h)}", dashed, from=2-8, to=2-9]
		\arrow["v", from=3-8, to=3-9]
		\arrow[Rightarrow, dotted, no head, from=2-9, to=3-9]
		\arrow[""{name=3, anchor=center, inner sep=0}, Rightarrow, dotted, no head, from=2-8, to=3-8]
		\arrow["\Phi", shorten <=6pt, shorten >=6pt, maps to, from=2, to=0]
		\arrow["\Psi", shorten <=6pt, shorten >=6pt, maps to, from=1, to=3]
	\end{tikzcd}\]
	\caption{Transposing maps of the fibered adjunction $\tau_P \dashv_B \iota_P$}\label{fig:transp-fibadj}
\end{figure}

\subsubsection{Examples of cocartesian families}

\begin{prop}
	Let $g:C \to A \leftarrow B:f$ be a cospan of Rezk types. Then the codomain projection from the comma object
	\[\begin{tikzcd}
		{f \downarrow g} && {A^{\Delta^1}} \\
		{C \times B} && {A \times A} \\
		C
		\arrow[from=1-1, to=2-1]
		\arrow["{g \times f}"', from=2-1, to=2-3]
		\arrow[from=1-1, to=1-3]
		\arrow["{\langle \partial_1, \partial_0 \rangle}", from=1-3, to=2-3]
		\arrow["\lrcorner"{anchor=center, pos=0.125}, draw=none, from=1-1, to=2-3]
		\arrow[from=2-1, to=3-1]
		\arrow["{\partial_1}"', curve={height=22pt}, from=1-1, to=3-1]
	\end{tikzcd}\]
	is a cocartesian fibration.
\end{prop}

\begin{proof}
	\Cf~\cite[Proposition~5.2.15]{BW21}, or later the version for fibered comma objects,~\Cref{prop:sliced-comma-is-cocart}.
\end{proof}

\begin{cor}[Codomain opfibration]\label{prop:cod-cocartfam}
	For any Rezk type $B$, the projection
	\[ \partial_1 : B^{\Delta^1} \to B, ~  \partial_1 :\jdeq \lambda f.f(1).\]
	is a cocartesian fibration, called the \emph{codomain opfibration}.
\end{cor}

Indeed the domain projection of a Rezk type is a cocartesian fibration given that the base has pushouts.
\begin{prop}[Domain opfibration]
	If $B$ is a Rezk type that has all pushouts, then the domain projection
	\[ \partial_0 : B^{\Delta^1} \to B, \quad \partial_0(u) :\jdeq u(0)\]
	is a cocartesian fibration.
\end{prop}

\begin{proof}
	\Cf~\Cite[Proposition~3.2.10]{BW21}.
\end{proof}

\subsubsection{Towards monadicity: the free cocartesian family}

As discussed in \cite{AFfib,GHT17,RVexp} cocartesian fibrations are monadic over general functors (over a fixed base). This means that for any functor $\pi:E \fibarr B$ there is a \emph{free cocartesian fibration} $L(\pi) : L(E) \fibarr B$. Due to the current lack of categorical universes in our type theory we postpone a discussion with emphasis on a global perspective similar to the cited works. However, we can still state and prove the universal property for this construction, so that later on, in the presence of the desired universes the actual monadicity statement will easily follow. Namely, we define a ``unit map'' $\iota_\pi \defeq \iota: \pi \to_B L(\pi)$, and prove that precomposition constitutes an equivalence of types\footnote{In general, $\CocartFun_B(\pi,\xi)$ is the $\Sigma$-type of fiberwise maps from $\pi$ to $\xi$ which preserve cocartesian lifts. Cf.~\Cref{sec:cocart-fun} for a more thorough treatment.}
\[ -\circ \iota_P : \CocartFun_{B}(L(\pi),\xi) \stackrel{\simeq}{\longrightarrow} \Fun_B(\pi,\xi),\]
for any \emph{cocartesian} fibration $\xi:F \to B$. 

\begin{defn}[Free cocartesian family]\label{def:free-cocart}
	Let $B$ be a Rezk type and $P: B \to \UU$ be an isoinner family. Then the family
	\[ L(\pi) \defeq \lambda b. \sum_{u:\comma{B}{b}} P(\partial_0u)  : B \to \UU \]
	is the \emph{free cocartesian family associated to $P$}, or the \emph{cocartesian replacement of $P$}.
\end{defn}

In more categorical terms, the free cocartesian family---in its incarnation as a fibration---is constructed by first pulling back the map $\pi:E \fibarr B$ along the domain projection, and then postcomposing with the \emph{codomain} projection:
\[\begin{tikzcd}
	{L(\pi)} && E \\
	{B^{\Delta^1}} && B \\
	B
	\arrow[from=1-1, to=1-3]
	\arrow[two heads, from=1-1, to=2-1]
	\arrow["{\partial_0}"', from=2-1, to=2-3]
	\arrow["\pi", two heads, from=1-3, to=2-3]
	\arrow["{\partial_1}", two heads, from=2-1, to=3-1]
	\arrow["\lrcorner"{anchor=center, pos=0.125}, draw=none, from=1-1, to=2-3]
	\arrow["{\partial_1'}"', curve={height=24pt}, two heads, from=1-1, to=3-1]
\end{tikzcd}\]
Morphisms in the cocartesian replacement can be depicted as follows:
\[\begin{tikzcd}
	e && {e'} \\
	a && {a'} \\
	b && {b'} \\
	b && {b'}
	\arrow["w", from=2-1, to=2-3]
	\arrow["u", from=3-1, to=3-3]
	\arrow["{v'}", from=2-3, to=3-3]
	\arrow["v"', from=2-1, to=3-1]
	\arrow["h", from=1-1, to=1-3]
	\arrow["u", from=4-1, to=4-3]
	\arrow[Rightarrow, dotted, no head, from=1-1, to=2-1]
	\arrow[Rightarrow, dotted, no head, from=1-3, to=2-3]
\end{tikzcd}\]

We will see that, indeed the free cocartesian family is a cocartesian family.
\begin{theorem}
	If $B$ is a Rezk type, and $P:B \to \UU$ is an isoinner family then the family $L(P) :B \to \UU$  is cocartesian.
\end{theorem}
\begin{proof}
	By the closure properties of isoinner families, since $P$ is an isoinner family, so is $L(P)$.
	
	Let $u:\hom_B(b,b')$ be an arrow in $B$, and $\pair{v}{e}:LP(b)$ a point over $b$, where $v:\hom_B(a,b)$ and $e:P\,a$.
	
	We define the candidate lift to be $\pair{\id_a, u}{\id_e}:\dhom_{LP,u}(\pair{v}{e},\pair{vu}{e})$, \ie:
	\[
	\begin{tikzcd}
		e \ar[r, equal] & e \\
		a \ar[r, equal]
		\ar[d, "v" swap] & a \ar[d, "vu"] \\
		b \ar[r, "u" swap] & b'
	\end{tikzcd}
	\]
	Cocartesianness is readily verified.\footnote{Compare cf.~\Cref{prop:cod-cocartfam}.} Namely, for $u':\hom_B(b',b'')$, let $t:\hom_B(a',b'')$, $e':P\,a'$, together with $w:\hom_B(a,a')$ and $f:\dhom_{P,u'u}(e,e')$ s.t.~$t \circ w = (u'u) \circ v$. We find the ensuing filler over $v$ as indicated:
	\[
	\begin{tikzcd}
		L(\pi) \ar[dddddd, two heads] &
		e \ar[r, equal]
		\ar[rr, bend left = 40, "f"] & e \ar[r, dashed, "f" swap] & e' \\
		& & & \\
		& a \ar[r, equal]
		\ar[rr, bend left = 40, "w"] \ar[d, "v" swap] & a \ar[d, "vu"] \ar[r, dashed, "w"] & a' \ar[d, "t"]\\
		& b \ar[r, "u" swap]  \ar[rr, bend right = 40, "u'u" swap]  & b' \ar[r, dashed, "u'" swap] & b''
		& & & \\
		& & & \\
		& & & \\
		B & b \ar[r, "u" swap] \ar[rr, bend left = 40, "u'u"] & b' \ar[r, "u'" swap] & b'' \\
	\end{tikzcd}
	\]
	By construction the dashed arrows are unique up to homotopy.
\end{proof}

We define the ``unit map''
\[\begin{tikzcd}
	E && {\pi \downarrow B} \\
	& B
	\arrow["\pi"', two heads, from=1-1, to=2-2]
	\arrow["\iota", from=1-1, to=1-3]
	\arrow["{\partial_1'}", two heads, from=1-3, to=2-2]
\end{tikzcd}\]
as the ``inclusion''
\[ \iota\defeq \lambda b,e.\pair{\id_b}{e}:\prod_{b:B} P\,b \to (LP)\,b.\]

\begin{prop}[Universal property of cocartesian replacement]\label{prop:univ-prop-cocart-repl}
	For a Rezk type $B$, consider an isoinner fibration $\pi:E \to B$, and a cocartesian fibration $\xi:F \to B$. Then the map
	\[ \CocartFun_{B}(L(\pi),\xi) \stackrel{\Psi \defeq -\circ \iota_P}{\longrightarrow} \Fun_B(\pi,\xi)\]
	is an equivalence of types.
\end{prop}

\begin{proof}
	We aim to give a quasi-inverse of the precomposition map. Let
	\[ \Psi \defeq \lambda \varphi.\varphi': \Fun_B(\pi, \xi) \to \CocartFun_{B}(L(\pi),\xi)\]
	where
	\[ \varphi'_b(v,e) \defeq v_!^Q(\varphi_a \,e),\]
	for $v:a \to b$, $e:P\,a$.
	First, we are to show that this operation is really valued in cocartesian functors. 
	For this, we have to show that, for any $v:a \to b$, $e:P\,a$, the arrow
	\[ \lambda t.(u(t) \circ v)_!^Q(\varphi_b \,e) : v_!Q(\varphi_b\,e) \longrightarrow^P_u (uv)_!^Q(\varphi_b\,a)\]
	is $Q$-cocartesian. To that end, we observe the following. Let $a:B$, $d:Q\,a$ be fixed. Consider the maps $\cst(d), \tau_Q(-,d): \comma{a}{B} \to E$ defined by
	\[ \cst(d)(v:a \to b) \defeq  d, \quad \tau_Q(-,d)(v:a \to b) \defeq v^Q_!(d):d . \]
	We define the natural transformation
	\[ Q_!(-,d):\nat{\comma{\pi}{B}}{E}(\cst(d),\tau_Q(-,d)), \quad Q_!(-,d)(v:a \to b) \defeq Q_!(v,d) : d \cocartarr_v v^Q_!(d).\]
	Morphisms in $\comma{a}{B}$ are given by commutative triangles $u:v \to w$, so for fixed $v$ the type of morphisms in $\comma{a}{B}$ starting at $v$ is equivalent to the type $\comma{(\partial_1 \, v)}{B}$. Hence, any morphism in $\comma{a}{B}$ can be taken to be of the form $u:v \to uv$, for $v:a \to b$, $u:b \to b'$. The naturality squares of $Q_!(-,d)$ thus are of the following form:
	\[\begin{tikzcd}
		d && d \\
		{v_!^Q(d)} && {(uv)_!^Q(d)} \\
		a && a \\
		b && {b'}
		\arrow[Rightarrow, no head, from=3-1, to=3-3]
		\arrow["u"', from=3-1, to=4-1]
		\arrow["v"', from=4-1, to=4-3]
		\arrow["vu", from=3-3, to=4-3]
		\arrow[from=1-1, to=2-1, cocart]
		\arrow["{(u:v \to uv)^Q_!(d)}"', from=2-1, to=2-3]
		\arrow[Rightarrow, no head, from=1-1, to=1-3]
		\arrow[from=1-3, to=2-3, cocart]
	\end{tikzcd}\]
	Note that the lower vertical arrow is given by
	\[ \lambda t.(u(t) \circ v)^Q_!(d) = (u:v \to uv)^Q_!(d).\]
	By right cancelation, $(u:v \to uv)^Q_!(d)$ is cocartesian, and hence we have an identity of arrows:
	\[\begin{tikzcd}
		&& {(uv)_!^Q(d)} \\
		{v^Q_!(d)} && {u^Q_!(v_!^Q(d))}
		\arrow["{Q_!(u,v_!^Q(d))}", from=2-1, to=2-3, cocart, swap]
		\arrow["{(u:v \to uv)^Q_!(d)}", from=2-1, to=1-3, cocart]
		\arrow[Rightarrow, no head, from=1-3, to=2-3]
	\end{tikzcd}\]
	In the cocartesian replacement $\partial_1': \comma{\pi}{B} \fibarr B$, the cocartesian lift of $u:b \to b'$ w.r.t.~$\angled{v:a \to b,e:P\,a}$ is given by $\angled{\id_a,u,\id_e}$. Now, by the previous discussion we have
	\begin{align*}
		\varphi'_u(\id_a,u,\id_e)  & = \lambda t.(u(t) \circ v)_!^Q(\varphi_a\,e) \\
		& =  (u:v \to uv)^Q_!(\varphi_a\,e) \\
		& = Q_!(u,v^Q_!(\varphi_a\,e)) \\
		& = Q_!(u, \varphi'_a(v,e))
	\end{align*}
	which shows that $\varphi': L(\pi) \to_B \xi$ is a cocartesian functor, as desired.
	
	We now turn to showing that precomposing with $\iota$ gives an equivalence
	\[ \Fun_B(\pi, \xi) \simeq \CocartFun_{B}(L(\pi),\xi).\]
	We define
	\[ \Phi \defeq \lambda \psi. \iota^*\psi \defeq \lambda \psi.\psi \circ \iota : \Fun_B(\pi, \xi) \to \CocartFun_{B}(L(\pi),\xi),\]
	and recall that in the converse direction
	\[ \Psi \defeq \lambda \varphi. \varphi' : \CocartFun_{B}(L(\pi),\xi) \to  \Fun_B(\pi, \xi)\]
	with $\varphi'_b(v,e) \defeq v_!^Q(\varphi_a\,e)$.
	Let $\psi: L\,E \to_B F$ be a cocartesian functor.
	We compute
	\[ (\iota^*\psi)_b'(v,e) = v_!^Q(\iota^*\psi_a(e)) = v_!^Q(\psi_a(\id_a,e)).\]
	Since $\psi$ is cocartesian, we have $v_!^Q(\psi_a(\id_a,e)) = \psi_b(v_!^{LP}(\id_a,e))$. Now, the $LP$-cocartesian lift of $v:a \to b$ w.r.t~$\pair{\id_a}{e}$ is given by $\angled{\id_a, v,\id_e}: \pair{\id_a}{e} \to \pair{v}{e}$:
	\[\begin{tikzcd}
		e && e \\
		a && a \\
		a && b \\
		a && b
		\arrow[Rightarrow, no head, from=1-1, to=1-3]
		\arrow[Rightarrow, no head, from=2-1, to=3-1]
		\arrow["v"', from=3-1, to=3-3]
		\arrow[Rightarrow, no head, from=2-1, to=2-3]
		\arrow["v", from=2-3, to=3-3]
		\arrow["v", from=4-1, to=4-3]
	\end{tikzcd}\]
	As a dependent arrow in $LP$, the codomain of this morphism is the pair $\pair{v}{e}$. In sum, this means
	\[ v_!^Q(\psi_a(\id_a,e)) = \psi_b(v_!^{LP}(\id_a,e)) = \psi_b(v,e) = (\iota^*\psi)_b'(v,e),\]
	\ie~$\Psi(\Phi(\psi)) = \psi$.
	On the other hand, for an arbitrary fiberwise map $\varphi:E \to_B F$, we find that
	\[ (\iota^* \varphi)'(b,e)= \varphi_b'(\id_b,e) = (\id_b)_!^Q(\varphi_b\,e) = \varphi_b(e)\]
	since cocartesian lifts of identities are themselves identities. This gives an identification $\Phi(\Psi(\varphi))= \varphi$.
\end{proof}

%% file: cocart-fun.tex
We turn to the study of the right notion of morphism between cocartesian fibrations: The cocartesian functors. These are fibered functors preserving the cocartesian arrows. Again, these can be characterized in terms of Chevalley criteria, and by~\cite{BW21} several closure properties hold, capturing internally the closure properties of the $\infty$-cosmos of cocartesian fibrations of $\inftyone$-categories. 

\subsection{Definition and properties}

Let us first review fibered maps between type families.
\begin{defn}[Fiberwise maps]\label{def:fib-maps}
	Let $P: B \to \UU$ and $Q: C \to \UU$ be families. A \emph{fiberwise map $\Phi$ from $P$ to $Q$} is a pair of functions
	\begin{itemize}
		\item $j: B \to C$,
		\item $\varphi : \prod_{b:B} (P\,b \to Q\,j(b))$.
	\end{itemize}
	We call $F$ an \emph{equivalence} if $j$ and $\varphi$ are equivalences.
	
	We write $\FibMap_{B,C}(P,Q)$ for the ensuing type of fiberwise maps.
\end{defn}
Note that (by fibrant replacement) the type of commutative squares is equivalent to the type of maps between families.

Observe that given a map between families $P$ and $Q$ as above we get a strictly commutative square:
\[
\begin{xy}
	\xymatrix{
		\totalty{P} \ar[r]^{\pair{j}{\varphi}} \ar[d]_{\pi_P}\ar[r] & \totalty{Q}  \ar[d]^{\pi_Q}\\
		B \ar[r]_{j} & C
	}
\end{xy}
\]

In the above setting, for any $u:\hom_B(a,b)$, $a,b$, the fiberwise map $\Phi$ acts on arrows over $u$ in the following way. For $f:\dhom_{P,u}(d,e)$, $d:Pa$, $e:Pb$, we define
\[ \varphi_u(f):\jdeq \lambda t.\varphi_{u(t)}(f(t)) : \dhom_{Q,ju}(\varphi_a(d), \varphi_b(e)).\]
One can think of the following picture:
\[
\begin{tikzcd}
	(d \stackrel{f}{\longrightarrow} e) \ar[d, mapsto] \ar[rr, mapsto] && (\varphi_a d \stackrel{\varphi_u f}{\longrightarrow} \varphi_b e)  \ar[d, mapsto] \\
	(a \stackrel{u}{\longrightarrow} b) \ar[rr, mapsto] && (ja \stackrel{ju}{\longrightarrow} jb)
\end{tikzcd}
\]

We now turn to the desired cocartesian functors

\begin{defn}[Cocartesian functors]\label{def:cocart-fun}
	If $P$ and $Q$ are cocartesian families, and furthermore the map
	\[ \Phi: \totalty{P} \to \totalty{Q}, \quad \Phi \, b \, e :\jdeq \pair{j(b)}{\varphi_j(e) } \]
	preserves cocartesian arrows, then we call $\Phi$ a \emph{cocartesian map}:\footnote{This is a proposition because being a cocartesian arrow is a proposition.}
	\[ \isCocartFun_{P,Q}(\Phi) :\jdeq \prod_{f:\Delta^1\to \widetilde{P}} \isCocartArr_P(f) \to \isCocartArr_Q(\varphi f).\]
	In particular, if $B$ and $C$ in addition are Segal (or Rezk) types we speak of a \emph{cocartesian functor}.
	
	We define
	\[ \CocartFun_{B,C}(P,Q) :\jdeq \sum_{F:\FibMap_{B,C}(P,Q)} \isCocartFun_{P,Q}(F).\]
	
	Given families $P: B \to \UU$ to $Q: B \to \UU$, a \emph{fibered functor} from $P \to Q$ is a section $\varphi : \prod_{x:B} Px \to Qx$. It is cocartesian if
	\[ \isCocartFun_{P,Q}(\pair{\id_B}{\varphi}) \equiv \prod_{\substack{u:\Delta^1 \to B \\ f: \Delta^1 \to u^*P}} \isCocartArr_P(f) \to \isCocartArr_Q(\varphi f).\]
	
	We define
	\[ \CocartFun_{B}(P,Q) :\jdeq \sum_{\varphi:\prod_B P \to Q} \isCocartFun_{P,Q}(\pair{\id_B}{\varphi}).\]
\end{defn}

This is, again, an instance of~\Cref{def:lari-fun}, for the map $i_0:\unit \hookrightarrow \Delta^1$.

\begin{prop}[Naturality of cocartesian liftings]\label{prop:nat-cocartlift-arr}
	Let $B$ be a Rezk type, $P:B \to \UU$, $Q:C \to \UU$ cocartesian families. Then a fibered functor $\Phi \jdeq \pair{\varphi}{j}$ from $P$ to $Q$ is a cocartesian functor if and only if $\Phi$ commutes with cocartesian lifts, \ie~for any $u:\hom_B(a,b)$ there is an identification of arrows
	\[ \varphi\big(P_!(u,d)\big) =_{\Delta^1 \to (ju)^*Q} Q_!(ju,\varphi_ad) \]
	and hence of endpoints
	\[ \varphi_b(u_!^Pd) =_{Q(jb)} (ju)_!^Q(\varphi_ad). \]
	In particular there is a homotopy commutative square:
	\[
	\begin{tikzcd}
		Pa \ar[rr, "\varphi_a"] \ar[d, "\coliftptfammap{P}{u}" swap] & & Qa \ar[d, "\coliftptfammap{Q}{(ju)}"] \\
		Pb \ar[rr, "\varphi_b" swap] && Qb
	\end{tikzcd}
	\]
\end{prop}

\begin{proof}
	The first claim follows formally from~\cref{prop:nat-larifun}.
	Specifically, for the naturality square, we find the following. For $u:\hom_B(a,b)$ and $d:P\,a$, consider the $P$-cocartesian lift $\coliftarr{P}{u}{d}:\dhom_{P,u}(d,\coliftptfam{P}{u}{d})$. Since $\varphi$ is a cocartesian functor the arrow $\varphi_u(\coliftarr{P}{u}{d}):\dhom_{Q,ju}(\varphi_ad,\varphi_b(\coliftptfam{P}{u}{d}))$ is $Q$-cocartesian. On the other hand, $\coliftarr{Q}{ju}{\varphi_ad}:\dhom_{Q,ju}(\varphi_ad, \coliftptfam{Q}{ju}{\varphi_ad})$ is as well a $Q$-cocartesian lift of $ju$ with domain $\varphi_ad$, thus coincides with $\varphi_u(\coliftarr{P}{u}{d})$ up to a path, in particular this gives an identification $\varphi_b(\coliftptfam{P}{u}{d}) = \coliftptfam{Q}{ju}{\varphi_ad}$.
\end{proof}

\begin{cor}[Naturality over a common base (discrete case: \cite{RS17}, Prop.~8.17)]\label{prop:nat-cocartlift-pt}
	Consider a Rezk type $B$, cocartesian families $P,Q: B \to \UU$, and a cocartesian functor $\varphi:\CocartFun_{B}(P,Q)$. Then $\varphi$ commutes with the actions of arrows, \ie~for any $a,b:B$, $u:\hom_B(a,b)$, $d:P(a)$, we get an
	identification
	\[ \varphi_b(\coliftptfam{P}{u}{d}) =_{Qb} \coliftptfam{Q}{u}{\varphi_a(d)},\]
	thus a homotopy commutative square:
	\[
	\begin{tikzcd}
		Pa \ar[rr, "\varphi_a"] \ar[d, "\coliftptfammap{P}{u}" swap] & & Qa \ar[d, "\coliftptfammap{Q}{u}"] \\
		Pb \ar[rr, "\varphi_b" swap] && Qb
	\end{tikzcd}
	\]
\end{cor}

\subsection{Characterization of cocartesian functors}

\begin{theorem}[{\protect\cite[Theorem~5.3.4]{RV21}}]\label{thm:char-cocart-fun}
	Let $A$ and $B$ be Rezk types, and consider cocartesian families $P:B \to \UU$ and $Q:A \to \UU$ with total types $E\defeq \totalty{P}$ and $F\defeq \totalty{F}$, resp.
	
	For a fibered functor $\Phi\defeq \pair{j}{\varphi}$ giving rise to a square
	\[
	\begin{tikzcd}
		F \ar[r, "\varphi"] \ar[d, "\xi" swap] & E \ar[d, "\pi"] \\
		A \ar[r, "j" swap] & B
	\end{tikzcd}
	\]
	the following are equivalent:
	\begin{enumerate}
		\item The fiberwise map $\Phi$ is a cocartesian functor.
		\item The mate of the induced canonical fibered natural isomorphism is invertible, too:
		\[\begin{tikzcd}
			{F} && {E} & {} & {F} && {E} \\
			{\xi \downarrow A} && {\pi \downarrow B} & {} & {\xi \downarrow A} && {\pi \downarrow B}
			\arrow["{i}"', from=1-1, to=2-1]
			\arrow["{\varphi}", from=1-1, to=1-3]
			\arrow["{i'}", from=1-3, to=2-3]
			\arrow[Rightarrow, "{=}", from=2-1, to=1-3, shorten <=7pt, shorten >=7pt]
			\arrow["{\rightsquigarrow}" description, from=1-4, to=2-4, phantom, no head]
			\arrow["{\kappa}", from=2-5, to=1-5]
			\arrow["{\varphi}", from=1-5, to=1-7]
			\arrow["{\varphi \downarrow j}"', from=2-5, to=2-7]
			\arrow["{\kappa'}"', from=2-7, to=1-7]
			\arrow[Rightarrow, "{=}"', from=2-7, to=1-5, shorten <=7pt, shorten >=7pt]
			\arrow["{\varphi \downarrow j}"', from=2-1, to=2-3]
		\end{tikzcd}\]
		
		\item The mate of the induced canonical natural isomorphism is invertible, too:
		\[\begin{tikzcd}
			{F^{\Delta^1}} && {E^{\Delta^1}} & {} & {F^{\Delta^1}} && {E^{\Delta^1}} \\
			{\xi \downarrow A} && {\pi \downarrow B} & {} & {\xi \downarrow A} && {\pi \downarrow B}
			\arrow["{r}"', from=1-1, to=2-1]
			\arrow["{\varphi \downarrow j}"', from=2-1, to=2-3]
			\arrow["{\varphi^{\Delta^1}}", from=1-1, to=1-3]
			\arrow["{r'}", from=1-3, to=2-3]
			\arrow["{\rightsquigarrow}" description, from=1-4, to=2-4, phantom, no head]
			\arrow[Rightarrow, "{=}", from=2-1, to=1-3, shorten <=7pt, shorten >=7pt]
			\arrow["{\ell}", from=2-5, to=1-5]
			\arrow["{\varphi \downarrow j}"', from=2-5, to=2-7]
			\arrow["{\varphi^{\Delta^1}}", from=1-5, to=1-7]
			\arrow["{\ell'}"', from=2-7, to=1-7]
			\arrow[Rightarrow, "{=}"', from=2-7, to=1-5, shorten <=7pt, shorten >=7pt]
		\end{tikzcd}\]
	\end{enumerate}
	
\end{theorem}

\begin{proof}
	Consider the situation of the first (fibered) adjunction, where the mate of the canonical isomorphism cell is constructed through the following pasting diagram:
	\[\begin{tikzcd}
		{\xi \downarrow A} & {F} && {E} \\
		& {\xi \downarrow A} && {\pi \downarrow B} & {E}
		\arrow["{i}", from=1-2, to=2-2]
		\arrow["{\kappa}", from=1-1, to=1-2]
		\arrow["{\varphi}", from=1-2, to=1-4]
		\arrow["{\varphi\downarrow j}"', from=2-2, to=2-4]
		\arrow["{i'}"', from=1-4, to=2-4]
		\arrow[""{name=0, inner sep=0}, from=1-4, to=2-5, no head, equals]
		\arrow["{\kappa'}"', from=2-4, to=2-5]
		\arrow[""{name=1, inner sep=0}, from=1-1, to=2-2, no head, equals]
		\arrow[Rightarrow, "{=}", from=2-2, to=1-4, shorten <=7pt, shorten >=7pt]
		\arrow[Rightarrow, "{\eta}", from=1, to=1-2, shorten <=2pt, shorten >=2pt]
		\arrow[Rightarrow, "{=}", from=2-4, to=0, shorten <=1pt, shorten >=1pt]
	\end{tikzcd}\]
	The unit $\eta: \hom_{\xi \downarrow A}(\id_F, \kappa i)$ at $\pair{u:a \to a'}{d:P_a}$ is given as follows:
	\[\begin{tikzcd}
		{d} && {u_!d} \\
		{a} && {a'} \\
		{a'} && {a'}
		\arrow["{P_!(u,d)}", from=1-1, to=1-3]
		\arrow["{u}"', from=2-1, to=3-1]
		\arrow[from=3-1, to=3-3, no head, equals]
		\arrow["{u}", from=2-1, to=2-3]
		\arrow[from=2-3, to=3-3, no head, equals]
	\end{tikzcd}\]
	The pasting $2$-cell can be identified with the natural transformation
	\[ \alpha: \hom_{F \to \pi \downarrow B}(\kappa' \circ \varphi \downarrow j, \varphi \circ \kappa) \]
	whose components at $\pair{u:a\to a'}{d:Pa}$ are given by the fillers
	\[ \alpha_{\pair{u}{d}}:\jdeq (\kappa' \circ \varphi \downarrow j)\eta_{\pair{u}{d}}: \pair{jd'}{(ju)_!(\varphi d)} \to \pair{jd'}{\varphi(u_!d)} \]
	\[\begin{tikzcd}
		&&&&&& {\varphi(u_!d)} \\
		{d} && {u_!d} & {} & {\varphi(d)} && {(ju)_!(\varphi d)} \\
		{a} && {a'} & {} & {j(a)} && {j(a')}
		\arrow["{Q_!(u,d)}", from=2-1, to=2-3]
		\arrow["{u}", from=3-1, to=3-3]
		\arrow["{\rightsquigarrow}", from=2-4, to=3-4, phantom, no head]
		\arrow["{j(u)}", from=3-5, to=3-7]
		\arrow["{P_!(ju,\varphi d)}"', from=2-5, to=2-7,  cocart]
		\arrow["{\alpha_{\langle u,d \rangle}}"', from=2-7, to=1-7, dashed]
		\arrow["{\varphi\big(P_!(u,d)\big)}", from=2-5, to=1-7]
	\end{tikzcd}\]
	If $\Phi \jdeq \pair{j}{\varphi}$ is a cocartesian functor there is an identification $\varphi(Q_!(u,d)) = P_!(ju, \varphi d)$ in $E^{\Delta^1}$, hence $\alpha_{\pair{u}{d}}$ is an identity.
	
	On the other hand, if the induced filler $\alpha_{\pair{u}{d}}$ happens to be an isomorphism, and thus an identity, we obtain an identification $\varphi(Q_!(u,d)) = P_!(ju, \varphi d)$ rendering $\Phi$ a cocartesian functor.
	
	In case of the second adjunction, this is an instance of~\Cref{thm:char-lari-fun}.
\end{proof}

\subsection{Closure properties of cocartesian functors}

In \cite{BW21}, we have shown that one obtains the following list of closure properties, capturing the structure of the $\infty$-cosmos internally:
\begin{prop}[Cosmological closure properties of cocartesian families, \protect{\cite[Proposition~5.3.17]{BW21}}]\label{prop:cocart-cosm-closure}
	Over Rezk bases, it holds that:
	
	Cocartesian families are closed under composition, dependent products, pullback along arbitrary maps, and cotensoring with maps/shape inclusions. Families corresponding to equivalences or terminal projections are always cocartesian.
	
	Between cocartesian families over Rezk bases, it holds that:
	Cocartesian functors are closed under (both horizontal and vertical) composition, dependent products, pullback, sequential limits,\footnote{all three objectwise limit notions satisfying the expected universal properties \wrt~to cocartesian functors} and Leibniz cotensors.
	
	Fibered equivalences and fibered functors into the identity of $\unit$ are always cocartesian.
\end{prop}

The following proposition is not contained in~\cite{BW21}, and is recorded here for later use in the chapter on two-sided cartesian fibrations.
\begin{proposition}[Pullback of fibered cocartesian sections]\label{prop:cocart-sect-pb}
	For a Rezk type $B$, consider cocartesian families $P:B \to \UU$, $Q:\totalty{P} \to \UU$, and a fiberwise map $\varphi : \prod_{b:B} P\,b \to Q\,b$. We write the unstraightenings as $\pi: E \defeq \totalty{P} \fibarr B$, $\xi: F \defeq \totalty{Q} \fibarr B$. Consider the following diagram, induced by a section $\ell$ of (the totalization of) $\varphi$ over $B$, and $A$ and a map $k:A \to B$ between Rezk types:
	\[\begin{tikzcd}
		{F'} &&& F \\
		& {E'} &&& E \\
		A &&& B
		\arrow["\varphi"{description}, from=1-4, to=2-5]
		\arrow[from=1-1, to=1-4]
		\arrow["{\varphi'}"{description}, from=1-1, to=2-2]
		\arrow["\xi"{description, pos=0.3}, two heads, from=1-4, to=3-4]
		\arrow["\pi"{description}, two heads, from=2-5, to=3-4]
		\arrow["{\xi'}"{description, pos=0.3}, two heads, from=1-1, to=3-1]
		\arrow["{\pi'}"{description}, two heads, from=2-2, to=3-1]
		\arrow["k", from=3-1, to=3-4]
		\arrow["\ell"{description}, curve={height=12pt}, dotted, from=2-5, to=1-4]
		\arrow[from=2-2, to=2-5, crossing over]
		\arrow["\lrcorner"{anchor=center, pos=0.125}, draw=none, from=2-2, to=3-4]
		\arrow["\lrcorner"{anchor=center, pos=0.125}, shift right=5, draw=none, from=1-1, to=3-4]
		\arrow["{\ell'}"{description, pos=0.3}, curve={height=18pt}, dotted, from=2-2, to=1-1, crossing over]
	\end{tikzcd}\]
	If $\ell$ is a cocartesian functor, then the induced section $\ell'$ is, too.
\end{proposition}

\begin{proof}
	First, fibrant replacement yields:
	\[E \defeq \sum_{b:B} P\,b, \quad F \defeq \sum_{\substack{b:B \\ e:P\,b}} Q\,b\,e,
	E' \defeq \sum_{a:A} P\,ka, \quad F' \defeq \sum_{\substack{a:B \\ d:P\,ka}} Q\,ka\,d.  \]
	The section $\ell$ is then taken to be
	\[ \ell(b,e) \defeq \angled{b,e,\widehat{\ell}(b,e)}\]
	for $b:B$, $e:P\,b$.
	Cocartesianness means that there is a path
	\begin{align*}
		\ell(u,P_!(u,e)) & \jdeq \angled{u,P_!(u,e),\widehat{\ell}(u,P_!(u,e))} \\
		& \jdeq \angled{u,P_!(u,e),Q_!(u, \pair{u^P_!e}{\widehat{\ell}(b,e)})}
	\end{align*}
	for $u:b\to_B b'$, $e:P\,b$.
	The induced section $\ell'$ arises as $\ell'(a,d) \jdeq \ell(ka,d) \jdeq \angled{ka,d \widehat{\ell}(ka,d)}$ for $a:A$, $d:P'\,a \simeq P\,ka$. Applying this to the $P'$-cocartesian lift of $v:a \to a'$ \wrt~$d:Q\,a$ yields
	\begin{align*}
		\ell'(v,P'_!(v,d)) & \jdeq \ell(kv,P'_!(kv,d)) = \ell(kv,P_!(v,d)) \defeq \angled{kv,P_!(v,d),\widehat{\ell}(kv,P_!(v,d))} \\
		& = \angled{kv,P_!(kv,d),Q_!(kv, \pair{(kv)^P_!(d)}{\widehat{\ell}(kv,P_!(kv,d))})} \\
		& = \angled{kv,P'_!(v,d),Q'_!(v, \pair{v^{P'}_!(d)}{\widehat{\ell'}(v,P'_!(v,d))})} \\
	\end{align*}
	confirming the claim.
\end{proof}

%% file: cart-fib.tex
Completely dually, one can formulate a theory of \emph{cartesian families} which are \emph{contravariantly functorial} \wrt~to directed paths. \emph{I.e.}, for $P:B \to \UU$ a cartesian family, for any arrow $u:b \to a$ and $d:P\,a$ there exists a \emph{cartesian} lift $f \defeq P^*(u,d): u^*d \cartarr d$, satisfying the dual universal property: For any $v:c \to a$, and any $h:e \to^P_{uv} d$ there exists a filler $g:e \to^P_v u^*\,d$, uniquely up to homotopy, \st~$h=P^*(u,d) \circ g$  In particular, this induces a map
\[ u^*:P\,a \to P\,b.\]
Likewise, we have a notion of \emph{cartesian functor}. The Chevalley condition(s) turn out to be \emph{right adjoint} right inverse (RARI) conditions instead. Furthermore, the cartesian arrows are pullback stable, and any dependent arrow factors as $(\cdot \rightsquigarrow \cdot \cartarr)$.

Sometimes, we will distinguish in the notation between cartesian and cocartesian filling by writing $\cartFill_{\ldots}(\ldots)$ or $\cocartFill_{\ldots}(\ldots)$, resp. Also, especially in~\Cref{ch:bicart}, we will denote vertical arrows (resp.~their types) by a squigglyarrow $\cdot \rightsquigarrow \cdot$.

\subsection{Lex families}

\begin{figure}
	\[\begin{tikzcd}
		E \\
		& e \\
		& {\zeta_b} && r & {\zeta_b} && {\zeta_a} && r \\
		B & b && z & b && a && z
		\arrow["{!_b}"', from=4-2, to=4-4]
		\arrow[two heads, from=1-1, to=4-1]
		\arrow[""{name=0, anchor=center, inner sep=0}, from=3-2, to=3-4, cart]
		\arrow["{!_e}", from=2-2, to=3-4]
		\arrow["{!^b_{e}}"', dashed, from=2-2, to=3-2]
		\arrow[""{name=1, anchor=center, inner sep=0}, "{!_{\zeta_b}}"', curve={height=12pt}, from=3-2, to=3-4]
		\arrow["u", from=4-5, to=4-7]
		\arrow["{!_a}", from=4-7, to=4-9]
		\arrow[from=3-7, to=3-9, cart]
		\arrow[curve={height=-24pt}, from=3-5, to=3-9, cart]
		\arrow["h"', dashed, from=3-5, to=3-7]
		\arrow["{!_b}"', curve={height=18pt}, from=4-5, to=4-9]
		\arrow[shorten <=2pt, shorten >=2pt, Rightarrow, no head, from=0, to=1]
	\end{tikzcd}\]
	\caption{Construction of fiberwise terminal objects}
	\label{fig:fib-term-obj}	
\end{figure}
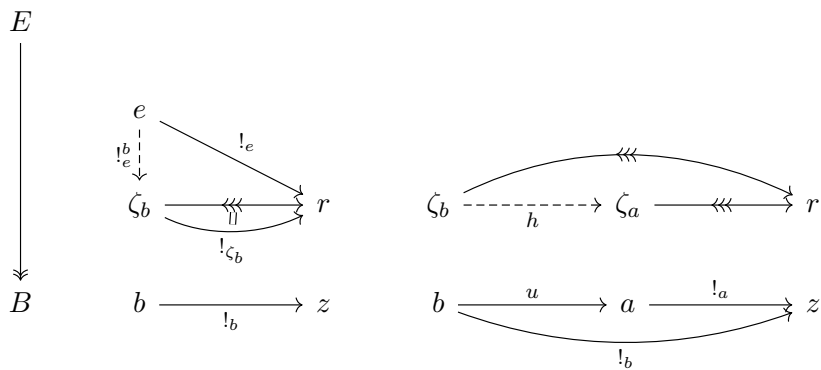

A specific class that becomes important in the next chapter are the so-called \emph{lex families}, which have fibered terminal elements and pullbacks.

Note that, \eg~for a functor $f:A \to B$ preserving a terminal object $z:A$ is a propositional condition. If $z'$ denotes the terminal object in $B$ there is a path $f(z) = z'$ if and only if $f(z):B$ is terminal,~\ie~$\prod_{b:B}\isContr(b \to_B f(z))$ if and only if the homotopically unique arrow $!_{f(z)}: f(z) \to z'$ is an isomorphism.

We will not discuss this further here, but similar considerations hold for limits in general, by their defining universal property as terminal objects of the respective Rezk types of cones.

In principle, we also think in the synthetic setting there could be a more uniform and abstract treating of ``$X$-shaped limit fibrations'', for a given shape or type $X$, after~\cite[Definition~8.5.1]{BorHandb2}, but we do not develop this here.

Instead we follow the account of~\cite[Section~8]{streicher2020fibered}, adapting it to the synthetic setting.

The aim is to recover the standard characterization of lex cartesian fibrations: Fix a base with the desired limits. Then the total type has those limits and they are preserved by the fibration if and only if the fibers each have the respective limits, and the reindexing functors preserve them.

First, we consider the case of terminal elements.

\begin{proposition}\label{prop:term-obj-fib}
	Let $P:B \to \UU$ be a cartesian family and $B$ be a Rezk type with terminal object $z:B$. Denote by $\pi:E \fibarr B$ the unstraightening of $P$. Then the following are equivalent:
	\begin{enumerate}
		\item\label{it:term-obj-fib-i} The total Rezk type $E$ has a terminal object $r$, and $\pi$ preserves it, \ie~$\pi(r):B$ is terminal.
		\item\label{it:term-obj-fib-ii} For all $b:B$, the fiber $P\,b$ has a terminal object, and for all arrows $u:b \to_B a$ the functor $u^*: P\,a \to P\,b$ preserves the terminal object.
	\end{enumerate}
\end{proposition}

\begin{proof}
	\begin{description}
		\item[$\ref{it:term-obj-fib-i}\implies\ref{it:term-obj-fib-ii}$:] For the visualization of both parts, \cf~\Cref{fig:fib-term-obj}. Denote by $\pair{z}{r}:E$ the terminal object of $E$, with $z:B$ terminal. For $b:B$, consider the point $\zeta_b \defeq (!_b)^*r:P\,b$. We claim that this is the terminal object of the fiber $P\,b$. Indeed, consider the canonical arrow $!_r:e \to_{!_b} r$. Then there is a unique arrow $!^b_e:e \to_{P\,b} \zeta_b$ s.t.~$P^*(!_b,r) \circ !^b_e = !_r$. But by terminality of $r$, the cartesian lift $P^*(!_b,r)$ also is propositionally equal to the terminal projection $!_{\zeta_b}: \zeta_b \to r$. Now, for any given map $g:e \to_{P\,b} \zeta_b$ we have that $!_{\zeta_b} \circ g = !_e$, but by the universal property of $!_{\zeta_b}$ there is only a unique such arrow $g$ up to homotopy. Hence, $\zeta_b$ is terminal in $P\,b$.
		
		Let $u:b \to_B a$. We will show that there is a path~$u^*\,\zeta_a = \zeta_b$. As we have just seen, we have $\zeta_a = (!_a)^*(r)$, and similarly for $\zeta_b$. Consider their terminal projections to $r$, which are necessarily cartesian arrows. From this and the identification $!_b = !_a \circ u$ in $B$, we get a unique filler $h:\zeta_b \to_u \zeta_a$. Moreover, $h$ is cartesian by left cancelation, so $h = P^*(u,\zeta_a): \zeta_b \cartarr_u^P \zeta_a$. This establishes the desired path.
		\item[$\ref{it:term-obj-fib-i}\implies\ref{it:term-obj-fib-ii}$:] Conversely, consider the section $\zeta:\prod_{b:B} P\,b$ choosing the terminal element in each fiber. Let $b:B$. By assumption, the cartesian lift of the terminal map $!_b:b \to z$ has $\zeta_a$ as its source vertex, up to a path. Let $e:P\,b$ be some point. Since $\zeta_b$ is terminal in $P\,b$, there exists a unique morphism $!_e^b: e \to_{P\,b} \zeta_b$, and post-composition with the cartesian lift $P^*(!_b,\zeta_z): \zeta_b \cartarr \zeta_z$ gives a morphism $t_e: e \to \zeta_z$:
		\[\begin{tikzcd}
			E & e \\
			& {\zeta_b} && {\zeta_z} \\
			B & b && z
			\arrow["{!_b}", from=3-2, to=3-4]
			\arrow[from=2-2, to=2-4, cart]
			\arrow[two heads, from=1-1, to=3-1]
			\arrow["{t_e}", from=1-2, to=2-4]
			\arrow["{!_e^b}"', dashed, from=1-2, to=2-2]
		\end{tikzcd}\]
		Finally, any morphism $f:e \to \zeta_z$, up to homotopy, lies over $!_b: b \to z$, and necessarily has the same factorization again, hence is identified with $t_e$. Therefore, $\pair{z}{\zeta_z}$ defines the (``global'') terminal element of $E$, and we have $\pi(z,\zeta_z) \defeq z$.
	\end{description}
\end{proof}

We are now turning to the analogous statement for pullbacks, which requires more preparation.

First, we give two conditions on dependent squares being pullbacks.
\begin{lemma}[\protect{\cite[Lemma~8.1(1)]{streicher2020fibered}}]\label{lem:cart-arr-pb}
	Let $P:B \to \UU$ be a cartesian family. Writing $E$ for the total type, then any a square in $E$ all of whose sides are cartesian arrows, is a pullback.
\end{lemma}

\begin{proof}
	Consider a square in $E$ together with a cone, and the fillers induced by cartesianness of $f'$ and $g'$, resp.:
	\[\begin{tikzcd}
		x \\
		& {d'} && d \\
		& {e'} && e
		\arrow["{f'}"', from=2-2, to=3-2, cart]
		\arrow["g"', from=3-2, to=3-4, cart]
		\arrow["{g'}", from=2-2, to=2-4, cart]
		\arrow["f", from=2-4, to=3-4, cart]
		\arrow[""{name=0, anchor=center, inner sep=0}, "h", curve={height=-18pt}, from=1-1, to=2-4]
		\arrow[""{name=1, anchor=center, inner sep=0}, "{h'}"', curve={height=24pt}, from=1-1, to=3-2]
		\arrow["{m'}"', shift right=2, dashed, from=1-1, to=2-2]
		\arrow["m", shift left=2, dashed, from=1-1, to=2-2]
		\arrow[shorten >=6pt, Rightarrow, no head, from=2-2, to=1]
		\arrow[shorten >=4pt, Rightarrow, no head, from=2-2, to=0]
	\end{tikzcd}\]
	We have $fh = (fg')m$ and $gh'=g(f'm') = (fg')m'$. But also $fg=gh'$, so $(fg')m = (fg')m'$. But since $fg'$ is cartesian as the composition of two cartesian arrows, $m=m'$ as desired.
\end{proof}

\begin{lemma}[\protect{\cite[Lemma~8.1(2)]{streicher2020fibered}}]\label{lem:op-sides-pb}
	Let $P:B \to \UU$ be a fibration. Then any dependent square in $P$ of the form
	\[\begin{tikzcd}
		{e'''} && {e''} \\
		{e'} && e
		\arrow["g"', squiggly, from=1-1, to=2-1]
		\arrow["f"', from=2-1, to=2-3, cart]
		\arrow["{f'}", from=1-1, to=1-3, cart]
		\arrow["{g'}", squiggly, from=1-3, to=2-3]
	\end{tikzcd}\]
	is a pullback.
\end{lemma}

\begin{proof}
	Consider a point $d$ and maps $h:d \to 'e$, $h':d \to e''$ s.t.~ $g'h' = fh$. By cartesianness of $f'$, there uniquely exists $k:d \to e'''$ s.t.~$f'k = h'$:
	\[\begin{tikzcd}
		d \\
		& {e'''} && {e''} \\
		& {e'} && e
		\arrow["g"', squiggly, from=2-2, to=3-2]
		\arrow["f"', from=3-2, to=3-4, cart]
		\arrow["{f'}", from=2-2, to=2-4, cart]
		\arrow["{g'}", squiggly, from=2-4, to=3-4]
		\arrow["k"{description}, dashed, from=1-1, to=2-2]
		\arrow[""{name=0, anchor=center, inner sep=0}, "{h'}", curve={height=-12pt}, from=1-1, to=2-4]
		\arrow[""{name=1, anchor=center, inner sep=0}, "h"', curve={height=12pt}, from=1-1, to=3-2]
		\arrow[shorten <=10pt, shorten >=10pt, Rightarrow, no head, from=3-2, to=2-4]
		\arrow[shorten >=4pt, Rightarrow, no head, from=2-2, to=0]
		\arrow["{(?)}"{description}, Rightarrow, draw=none, from=1, to=2-2]
	\end{tikzcd}\]
	To show that also $h= gk$, it suffices to show that $f(gk) = g'h'$, since also $fh = g'h'$, which taken together implies $h=gk$ by cartesianness of $f$. Indeed, by the above we have a chain of paths
	\[ f(gk) = (fg)k = (g'f')k = g'(f'k') = g'h',\]
	which implies the claim that $h=gk$. Furthermore, $k$ is already unique with the property $f'k= h'$, so we are done.
\end{proof}

The next lemma presents a sufficient condition for the ``local'' pullbacks being ``global'' pullbacks.
\begin{lemma}[\cf~\protect{\cite[Lemma~8.2]{streicher2020fibered}}]\label{lem:loc-pb-is-pb}
	Let $P:B \to \UU$ be a cartesian family and $B$ be a Rezk type where all pullbacks exist. Assuming that all fibers have pullbacks, and these are preserved by the reindexing functors, we have: A pullback in a fiber $P\,b$ is also a pullback in $\totalty{P}$.
\end{lemma}

\begin{proof}
	Let $\pi:E \fibarr B$ be the unstraightening of $P$.
	
	For some $b:B$, consider a pullback square in $P\,b$, together with a cone in $E$, as follows:
	\begin{equation}
	\label{cd:local-pb}
	\begin{tikzcd}
		d \\
		& {e'} && {e_2} \\
		& {e_1} && e
		\arrow["{f_1}"', squiggly, from=3-2, to=3-4]
		\arrow["{f_2}", squiggly, from=2-4, to=3-4]
		\arrow["{g_2}", squiggly, from=2-2, to=2-4]
		\arrow["{g_1}", squiggly, from=2-2, to=3-2]
		\arrow["{h_2}", curve={height=-12pt}, from=1-1, to=2-4]
		\arrow["{h_1}"', curve={height=12pt}, from=1-1, to=3-2]
	\end{tikzcd}
	\end{equation}
	Projecting down we find that $\pi(h_1) = u = \pi(h_2)$ for some $u:b \to_B a$. Consider the factorizations
	\[\begin{tikzcd}
		d && {u^*\,e_n} && {e_n}
		\arrow["{m_n}"', squiggly, from=1-1, to=1-3]
		\arrow["{k_n}"', from=1-3, to=1-5, cart]
		\arrow["{h_n}", curve={height=-24pt}, from=1-1, to=1-5]
	\end{tikzcd}\]
	for $n=1,2$.
	We claim that $(u^*\,f_1) \circ m_1 = (u^*\,f_2) \circ m_2$. To see this, consider the following induced diagram:
	\[\begin{tikzcd}
		d && {u^*\,e_2} && {e_2} \\
		{u^*\,e_1} && {u^*\,e} && e \\
		{e_1}
		\arrow["{f_2}", squiggly, from=1-5, to=2-5]
		\arrow[from=1-3, to=1-5, cart, "k_2"]
		\arrow[from=2-3, to=2-5, cart, "k"]
		\arrow[squiggly, from=1-1, to=1-3, "m_2"]
		\arrow[from=2-1, to=2-3, squiggly, "{u^*\,f_1}"]
		\arrow[from=2-1, to=3-1, cart, "k_1"]
		\arrow["{h_1}"', curve={height=24pt}, from=1-1, to=3-1]
		\arrow["{f_1}"', squiggly, from=3-1, to=2-5]
		\arrow[squiggly, from=1-3, to=2-3, "{u^*\,f_2}"]
		\arrow["{h_2}", curve={height=-24pt}, from=1-1, to=1-5]
		\arrow["m_1"' swap, squiggly, from=1-1, to=2-1]
		\arrow[shorten <=10pt, shorten >=10pt, Rightarrow, no head, from=3-1, to=2-3]
		\arrow[shorten <=10pt, shorten >=10pt, Rightarrow, no head, from=2-3, to=1-5]
	\end{tikzcd}\]
	Now, in fact the remaining sub-square commutes as well because both sides are equalized by the cartesian arrow $k$: By assumption we have $(f_1k_1)m_1 = (f_2k_2)m_2$, \ie~$(k (u^*f_1))m_1 = (k (u^*f_2))m_2$, hence $(u^*f_1)m_1 = (u^*f_2)m_2$ as claimed.
	
	Now, by assumption~the square of vertical arrows in~(\ref{cd:local-pb}), is a pullback in $P\,b$, and gets preserved by $u^*:P\,a\to P\,b$. Then the gap map $\ell:d \rightsquigarrow u^*\,e'$ as indicated below is vertical:
	\[\begin{tikzcd}
		d \\
		& {u^*\,e'} && {u^*\,e_2} \\
		& {u^*\,e_1} && {u^*\,e}
		\arrow[squiggly, from=2-2, to=3-2]
		\arrow[squiggly, from=2-2, to=2-4]
		\arrow[squiggly, from=2-4, to=3-4]
		\arrow["{m_2}"', curve={height=-12pt}, squiggly, from=1-1, to=2-4]
		\arrow["{m_1}"', curve={height=12pt}, squiggly, from=1-1, to=3-2]
		\arrow["\ell", squiggly, from=1-1, to=2-2, dashed]
		\arrow[squiggly, from=3-2, to=3-4]
	\end{tikzcd}\]
	Then, for $k'\defeq P^*(u,e'): u^*\,e' \cartarr_u e'$ we claim that the mediating arrow for the original diagram~\ref{cd:local-pb} is given by
	\[ \ell' \defeq k' \circ \ell: d \to e'.\]
	Indeed, we find
	\[ g_n \ell' = g_n(k'\ell) = (k_n \circ u^*\,g_n)\ell = k_n m_n = h_n\]
	for $n=1,2$.
	Furthermore, $\ell'$ is unique with this property because its vertical component is determined uniquely up to homotopy as a gap map of a pullback in $P\,b$.
\end{proof}

Finally, we can state the desired chracaterization.

\begin{proposition}[\cf~\protect{\cite[Theorem~8.3]{streicher2020fibered}}]\label{prop:pb-fib}
	Let $P:B \to \UU$ be a cartesian family and $B$ be a Rezk type where all pullbacks exist. Denote by $\pi:E \fibarr B$ the unstraightening of $P$. Then the following are equivalent:
	\begin{enumerate}
		\item\label{it:pb-fib-i} The total Rezk type $E$ has all pullbacks, and $\pi$ preserves them.
		\item\label{it:pb-fib-ii} For all $b:B$, the fiber $P\,b$ has all pullbacks, and for all arrows $u:b \to_B a$ the functor $u^*: P\,a \to P\,b$ preserves them.
	\end{enumerate}
\end{proposition}

\begin{proof}
	\item[$\ref{it:pb-fib-i} \implies \ref{it:pb-fib-ii}$:] Since $\pi$ preserves pullbacks, every pullback of vertical arrows in a fiber $P\,b$ is a pullback in $P\,b$, \ie~given a cone of vertical arrows, the mediating arrow is necessarily vertical as well. What is left to show is that the reindexing functors preserve the pullbacks. Let $u:b \to_B a$. Consider a pullback square in $P\,a$, together with the cartesian liftings of $u$ \wrt~to each point. Then, by~\Cref{lem:op-sides-pb} the ensuing squares are pullbacks, as indicated in:
	\[\begin{tikzcd}
		{} && E & {u^*\,e'''} &&& {e'''} \\
		&&&& {u^*\,e''} &&& {e''} \\
		&&& {u^*\,e'} &&& {e'} \\
		&&&& {u^*\,e} &&& e \\
		&& B && b &&& a
		\arrow[from=1-4, to=1-7, cart]
		\arrow[squiggly, from=1-7, to=3-7]
		\arrow[squiggly, from=1-7, to=2-8]
		\arrow[squiggly, from=2-8, to=4-8]
		\arrow[squiggly, from=3-7, to=4-8]
		\arrow[from=3-4, to=3-7, cart]
		\arrow[dashed, from=1-4, to=2-5, squiggly]
		\arrow[dashed, from=3-4, to=4-5, squiggly]
		\arrow[dashed, from=1-4, to=3-4, squiggly]
		\arrow[two heads, from=1-3, to=5-3]
		\arrow["u", from=5-5, to=5-8]
		\arrow["\lrcorner"{anchor=center, pos=0.125, rotate=-45}, draw=none, from=1-7, to=4-8]
		\arrow["\lrcorner"{anchor=center, pos=0.125}, draw=none, from=2-5, to=4-8]
		\arrow["\lrcorner"{anchor=center, pos=0.125}, shift left=2, draw=none, from=1-4, to=3-7]
		\arrow[from=2-5, to=2-8, cart, crossing over]
		\arrow[from=4-5, to=4-8, cart, crossing over]
		\arrow[dashed, from=2-5, to=4-5, crossing over, squiggly]
	\end{tikzcd}\]
	By~\cite[Remark~26.1.5(ii)]{RijIntro}, we obtain that the left hand square is a pullback, as desired.
	\item[$\ref{it:pb-fib-ii} \implies \ref{it:pb-fib-i}$:] Conversely, consider a cospan an $E$, comprised of dependent arrows $(f:d \to e \leftarrow d':f')$. First, we consider their vertical/cartesian-factorizations, $f=gk$, $f'=g'm$. This gives rise to the following situation, which we will readily explain:
	\[\begin{tikzcd}
		{d'''} && {d''} && {d'} \\
		{e''''} && {e'''} && {e''} \\
		d && {e'} && e
		\arrow[""{name=0, anchor=center, inner sep=0}, "k"{description}, squiggly, from=3-1, to=3-3]
		\arrow["g", from=3-3, to=3-5, cart]
		\arrow["{g'}"', from=2-5, to=3-5, cart]
		\arrow["m"', squiggly, from=1-5, to=2-5]
		\arrow["{m'}"', squiggly, from=1-3, to=2-3]
		\arrow["{h'}"', from=2-3, to=3-3, cart]
		\arrow[""{name=1, anchor=center, inner sep=0}, "h", from=2-3, to=2-5, cart]
		\arrow[""{name=2, anchor=center, inner sep=0}, "{h'''}", from=1-3, to=1-5, cart]
		\arrow["f"{description}, curve={height=18pt}, from=3-1, to=3-5]
		\arrow["{f'}"{description}, curve={height=-18pt}, from=1-5, to=3-5]
		\arrow["{h''}"', from=2-1, to=3-1, cart]
		\arrow[""{name=3, anchor=center, inner sep=0}, "{k'}"{description}, squiggly, from=2-1, to=2-3]
		\arrow["\lrcorner"{anchor=center, pos=0.125}, draw=none, from=1-3, to=2-5]
		\arrow["\lrcorner"{anchor=center, pos=0.125}, draw=none, from=2-3, to=3-5]
		\arrow["\lrcorner"{anchor=center, pos=0.125}, draw=none, from=2-1, to=3-3]
		\arrow["\tau"{description}, draw=none, from=2-3, to=3-5]
		\arrow["{\ell'}"', squiggly, from=1-1, to=2-1]
		\arrow["\ell", squiggly, from=1-1, to=1-3]
		\arrow["{\tau'}"{description}, "\lrcorner"{anchor=center, pos=0.125}, draw=none, from=1-1, to=2-3]
		\arrow["\sigma"{description}, Rightarrow, draw=none, from=3, to=0]
		\arrow["{\sigma'}"{description}, Rightarrow, draw=none, from=2, to=1]
	\end{tikzcd}\]
	First of all, the diagram $\tau$ is a pulback by~\Cref{lem:cart-arr-pb}. The (vertical) fillers $k'$ and $m'$, resp.~ are induced by $h'$ and $h$ being cartesian, resp. Then by~\Cref{lem:op-sides-pb}, the squares $\sigma$, $\sigma'$ are pullbacks, too. Since the fibers have pullbacks, the square $\tau'$ exists. As the reindexings preserve the local pullbacks, we can apply~\Cref{lem:loc-pb-is-pb}, so $\tau'$ is a pullback in $E$. Altogether, this yields the pullback square of $f'$ along $f$.
\end{proof}

\begin{definition}[Lex families]
	Let $B$ be a Rezk type with terminal object and all pullbacks. A cartesian fibration over $B$ is \emph{lex} if it satisfies the conditions from~\Cref{prop:term-obj-fib,prop:pb-fib},~\ie~all the fibers of $P$ have terminal objects and pullbacks, and both notions are preserved by the reindexing functors.
	
	A Rezk type is \emph{lex} if it it has all pullbacks and a terminal object. A functor between Rezk types is a \emph{lex functor} if it preserves these notions.
\end{definition}

By the above, this is equivalent to $\totalty{P}$ having terminal objects and pullbacks, and the projection to $B$ preserving them.

%% file: bicart-fib.tex
In this section, we consider families that are both cartesian and concartesian, corresponding to \emph{bicartesian fibrations}. Specifically, we are interested in such fibrations satisfying a so-called Beck--Chevalley condition (BCC). This form of the BCC has its origins in the work of B\'{e}nabou--Roubaud leading to their famous chracterization of descent data of a fibration~\cite{BRMonDesc}. In the $\inftyone$-categorical context such fibrations play a role in homotopical ambidexterity~\cite{LurAmbi}.

In our context, we are eventually interested in a specific subclass of BCC fibrations, which go by the name of \emph{(l)extensive} or \emph{Moens fibrations}. These are a fibrational generalization of (l)extensive categories~\cite{CLWExt}.

Ultimately, this leads to Moens' Theorem which says that Moens fibrations over a fixed base type can be identified with lex functors \emph{from} this type into some other lex type. This is crucial to develop the \emph{fibered view of geometric morphisms}, \cf~\cite[Section~15 \emph{et seq}]{streicher2020fibered}, \cite{StrFVGM,LietzDip}. Applications in realizability have been given by Frey in his doctoral thesis~\cite{FreyPhD}.

We successively generalize Streicher's exposition and proofs~\cite[Section~15]{streicher2020fibered} to the synthetic $\inftyone$-categorical setting, also making explicit some arguments not detailed in~\emph{op.~cit.}

\section{Bicartesian families}

\subsection{Bicartesian families}

\begin{definition}
	Let $B$ be a Rezk type. A \emph{bicartesian family} is a type family $P:B \to \UU$ which is both cartesian and cocartesian.
\end{definition}

Bicartesian families are hence equipped with both co- and contravariant transport operations for directed arrows. In fact, these induce adjunctions on the fibers.

\begin{proposition}
	Let $P:B \to \UU$ be a bicartesian family. For any $a,b:B$, $u:a \to_B b$, there is an adjunction:
	\[\begin{tikzcd}
		{P\,a} && {P\,b}
		\arrow[""{name=0, anchor=center, inner sep=0}, "{u_!}"{description}, curve={height=-12pt}, from=1-1, to=1-3]
		\arrow[""{name=1, anchor=center, inner sep=0}, "{u^*}"{description}, curve={height=-18pt}, from=1-3, to=1-1]
		\arrow["\dashv"{anchor=center, rotate=-90}, draw=none, from=0, to=1]
	\end{tikzcd}\]
\end{proposition}

\begin{proof}
	For fixed $a,b:B$, we define a pair of maps
	\[ \Phi:(u_!\,d \to e) \rightleftarrows (d \to u^*e) : \Psi,\]
	intended to be quasi-inverse to each other, through
	\[\Phi(g) \defeq \cartFill_{P^*(u,e)}^P(g \circ P_!(u,d)), \quad \Psi(k) \defeq \cocartFill_{P_!(u,d)}^P(P^*(u,e)\circ k).\]
	We write $f \defeq P_!(u,d): d \cocartarr u_!\,d$ and $r \defeq P^*(u,e): u^*e \cartarr e$. For a dependent arrow $g:u_!\,d \to^P e$ we have $g' \defeq \Phi(g) \circ r = g \circ f$ by construction. Next, we find that $\Psi(g') \circ f = r \circ \Phi(g)$. Combining these identities it follows that $\Psi(g') = \Psi(\Phi(g)) = g$ by cocartesianness of $f$. The other roundtrip is analogous.
\end{proof}

We now explain a few important constructions producing bicartesian fibrations.

\subsection{The family fibration}

Recall the construction of the free cocartesian fibration associated to an arbitrary map $\pi:E \fibarr B$ between Rezk types.

\begin{definition}
If the base $B$ has all pullbacks and $\pi:E \fibarr B$ is assumed to be a cartesian fibration, then the free cocartesian fibration
\[\begin{tikzcd}
	{L(\pi)} && E \\
	{B^{\Delta^1}} && B \\
	B
	\arrow["\pi", two heads, from=1-3, to=2-3]
	\arrow["{\partial_0}"', from=2-1, to=2-3]
	\arrow[from=1-1, to=1-3]
	\arrow["{\partial_1}", two heads, from=2-1, to=3-1]
	\arrow["{\partial_1'}"', curve={height=24pt}, two heads, from=1-1, to=3-1]
	\arrow["\lrcorner"{anchor=center, pos=0.125}, draw=none, from=1-1, to=2-3]
	\arrow[two heads, from=1-1, to=2-1]
\end{tikzcd}\]
is itself a also cartesian fibration. This bifibration is called the \emph{family fibration associated to $P$}.
\end{definition}

\begin{proposition}[Cartesian lifts in the family fibration]\label{prop:cartlift-famfib}
	Let $\pi:E \fibarr B$ be a cartesian fibration over a Rezk type $B$ with all pullbacks. The cartesian lift of an arrow $u:a \to b$ in $B$ w.r.t.~$\pair{v:b' \to b}{e:P\,b'}$ in $L(\pi)$ is given by $\pair{\mathrm{pb}(u,v)}{P^*(v^*u,e)}$.
\end{proposition}

\begin{proof}
	This follows by computing the lifts fiberwisely,~\cf~\cite[Subsection~5.2.3]{BW21}.
\end{proof}

\begin{figure}
	\centering
	\[\begin{tikzcd}
		{L(\pi)} && {(u')^*e} && e \\
		&& {a \times_b b'} && b' \\
		&& {a} && b \\
		B && {a} && b
		\arrow["u", from=4-3, to=4-5]
		\arrow["{v'}"', from=2-3, to=3-3]
		\arrow["u"', from=3-3, to=3-5, swap]
		\arrow["{u'}"', from=2-3, to=2-5]
		\arrow["v", from=2-5, to=3-5]
		\arrow["{P^*(u',e)}", from=1-3, to=1-5, cart]
		\arrow[Rightarrow, dotted, no head, from=1-3, to=2-3]
		\arrow[Rightarrow, dotted, no head, from=1-5, to=2-5]
		\arrow["\lrcorner"{anchor=center, pos=0.125}, draw=none, from=2-3, to=3-5]
		\arrow["{\partial_1'}"', two heads, from=1-1, to=4-1]
	\end{tikzcd}\]
	\caption{Cartesian lifts in the family fibration}\label{fig:cart-lifts-famfib}
\end{figure}

\subsection{The Artin gluing fibration}

\begin{definition}[Artin gluing]
	Let $B,C$ be Rezk types and $F:B \to C$ a functor. Then the map $\gl(F): \comma{C}{F} \fibarr B$ constructed by pullback
	\[\begin{tikzcd}
		{C \downarrow F} && {C^{\Delta^1}} \\
		B && C
		\arrow["F"', from=2-1, to=2-3]
		\arrow["{\partial_1}", two heads, from=1-3, to=2-3]
		\arrow["{\mathrm{gl}(F)}"', two heads, from=1-1, to=2-1]
		\arrow[from=1-1, to=1-3]
		\arrow["\lrcorner"{anchor=center, pos=0.125}, draw=none, from=1-1, to=2-3]
	\end{tikzcd}\]
	is called the \emph{Artin gluing} (or simply \emph{gluing}) of $F$.
\end{definition}

Since the codomain projection $\partial_1 : C^{\Delta^1} \fibarr C$ is always a cocartesian fibration the gluing $\gl(F)$ is a cocartesian fibration as well. We will be concerned with the case that $C$ has all pullbacks. In this case $\partial_1: C^{\Delta^1} \fibarr C$ also is a cartesian fibration, hence a bifibration, and consequently the same is true for $\gl(F): \comma{C}{F} \fibarr B$.

Hence, from the description of the co-/cocartesian lifts in pullback fibrations, the respective lifts in $\gl(F)$ can be computed as illustrated in~\Cref{fig:lifts-gluing}. Given an arrow $u:b \to b'$ in $B$, a cocartesian lift w.r.t.~an arrow $v:c \to F\,b$ in $C$ is the square with boundary $[v,\id_b,F\,u,F\,u \circ v]$. A cartesian lift of $u:b\to b'$ w.r.t.~an arrow $w:b \to F\,c'$ is given by the pullback square $\mathrm{pb}(F\,u,w)$.

A vertical arrow in the gluing fibration is exactly given by a square of the form:
\[\begin{tikzcd}
	{C\downarrow F} && c && {c'} \\
	&& {F\,b} && {F\,b} \\
	B && b && b
	\arrow[from=1-3, to=2-3]
	\arrow[Rightarrow, no head, from=2-3, to=2-5]
	\arrow[from=1-3, to=1-5]
	\arrow[from=1-5, to=2-5]
	\arrow[Rightarrow, no head, from=3-3, to=3-5]
	\arrow[two heads, from=1-1, to=3-1]
\end{tikzcd}\]
\begin{figure}
	\[\begin{tikzcd}
		{C \downarrow F} && c && c && {F\,b \times_{F\,b'} c} && c \\
		&& {F\,b} && {F\,b'} && {F\,b} && {F\,b'} \\
		B && b && b && b && b
		\arrow["v"', from=1-3, to=2-3]
		\arrow["{F\,u}"', from=2-3, to=2-5]
		\arrow[Rightarrow, no head, from=1-3, to=1-5]
		\arrow["{Fu \circ v}"{pos=0.6}, from=1-5, to=2-5]
		\arrow[from=1-7, to=2-7]
		\arrow["{F\,u}"', from=2-7, to=2-9]
		\arrow[from=1-7, to=1-9]
		\arrow["w", from=1-9, to=2-9]
		\arrow["\lrcorner"{anchor=center, pos=0.125}, draw=none, from=1-7, to=2-9]
		\arrow["u", from=3-3, to=3-5]
		\arrow["u", from=3-7, to=3-9]
		\arrow[two heads, from=1-1, to=3-1]
	\end{tikzcd}\]
\caption{Cocartesian and cartesian lifts in the gluing fibration}
\label{fig:lifts-gluing}
\end{figure}

%% file: fl-fib.tex
\section{Beck--Chevalley families}

In categorical logic, Beck--Chevalley conditions say that substitution (\ie~pullback) commutes with existential quantification (\ie~dependent sums). This generalizes to the fibrational setting by considering cartesian arrows in place of substitutions (acting contravariantly), and cocartesian arrows in place of dependent sums (acting covariantly). We refer to~\cite[Section~6]{streicher2020fibered} for more explanation.

\subsection{Beck--Chevalley condition}

\begin{definition}[Beck--Chevalley condition, \protect{\cite[Definition~6.1]{streicher2020fibered}}]\label{def:bcc}
Let $P:B \to \UU$ be an isoinner family over a Rezk type. Then $P$ is said to satisfy the \emph{Beck--Chevalley condition (BCC)} if for any dependent square of the form
\[\begin{tikzcd}
	& {d'} && {e'} \\
	{\totalty{P}} & d && e \\
	& {a'} && {b'} \\
	B & a && b
	\arrow["{f'}", from=1-2, to=1-4]
	\arrow["{g'}"', from=1-2, to=2-2, cart]
	\arrow["g", from=1-4, to=2-4, cart]
	\arrow["f"', from=2-2, to=2-4, cocart]
	\arrow["{u'}", from=3-2, to=3-4]
	\arrow["u"', from=4-2, to=4-4]
	\arrow["v", from=3-4, to=4-4]
	\arrow["{v'}"', from=3-2, to=4-2]
	\arrow[two heads, from=2-1, to=4-1]
	\arrow["\lrcorner"{anchor=center, pos=0.125}, draw=none, from=3-2, to=4-4]
\end{tikzcd}\]
it holds that: if $f$ is cocartesian, and $g,g'$ are cartesian, then $f'$ is cocartesian.
\end{definition}

\begin{proposition}[Dual of the Beck--Chevalley conditions]
Let $P:B \to \UU$ be an isoinner family over a Rezk type. Then $P$ satisfies the BCC~\Cref{def:bcc} if and only if it satisfies the \emph{dual BCC}, which says: Given any dependent square
\[\begin{tikzcd}
	{d'} && {e'} \\
	d && e
	\arrow["{f'}", from=1-1, to=1-3]
	\arrow["{g'}"', from=1-1, to=2-1, cart]
	\arrow["g", from=1-3, to=2-3]
	\arrow["f"', from=2-1, to=2-3, cocart]
\end{tikzcd}\]
in $P$ over a pullback, then: If $f$ and $f'$ are cocartesian, and $g'$ is cartesian, then $g$ is cartesian as well.
\end{proposition}

\begin{proof}
	Assume, the BCC from~\Cref{def:bcc} is satisfied. We consider a square, factoring the arrow $g$ as a vertical arrow followed by a cartesian arrow, we obtain:
	\[\begin{tikzcd}
		{d'} && {e'} \\
		&&& {v^*e} \\
		d && e
		\arrow["{g'}"', from=1-1, to=3-1, cart]
		\arrow["f", from=3-1, to=3-3, cocart]
		\arrow["{f'}", from=1-1, to=1-3, cocart]
		\arrow["g", from=1-3, to=3-3]
		\arrow["m", squiggly, from=1-3, to=2-4]
		\arrow["k", from=2-4, to=3-3, cart]
		\arrow["{m'}"{description}, from=1-1, to=2-4, crossing over]
	\end{tikzcd}\]
	Then, applying the BCC to the ``smaller'' square (still lying over the same pullback since $m$ is vertical), $m' \defeq m \circ f'$ must be cocartesian. But then $m$ is, too, by right cancelation, as $f'$ is cocartesian. But since it is also vertical, it is an isomorphism, so $g$ is cartesian, as claimed.
	
	The converse direction is analogous.
\end{proof}

\subsection{Beck--Chevalley families}

Preparing the treatment of Moens fibrations, we state a few first results about BCC fibrations,~\aka~\emph{fibrations with internal sums}.

\begin{definition}
A map between $\pi:E \fibarr B$ is a \emph{Beck--Chevalley fibration} or a \emph{cartesian fibration with internal sums} if:
\begin{enumerate}
	\item The map $\pi$ is a bicartesian fibration, \ie~a cartesian and cocartesian fibration.
	\item The map $\pi$ satisfies the Beck--Chevalley condition.
\end{enumerate}
\end{definition}

Recall the criterion characterizing cocartesian fibrations via the existence of a fibered left adjoint which acts as the ``cocartesian transport'' functor. The Beck--Chevalley condition is equivalent to this functor being \emph{cartesian}.
\begin{theorem}[Beck--Chevalley fibrations via cartesianness of the coocartesian transport functor, \cf~\protect{\cite[Theorem~6.1]{streicher2020fibered}}]
Let $P: B \to \UU$ be a cartesian family over a Rezk type $B$ which has all pullbacks, with unstraightening $\pi:E \fibarr B$ a cartesian fibration. Then $\pi$ is a Beck--Chevalley fibration if and only if the mediating fibered functor 
\[ \iota_P: E \to_B \comma{\pi}{B}, ~ \iota_P(b,e)\defeq \angled{b, \id_b, e} \]
has a fibered left adjoint which is also a cartesian functor:
\[\begin{tikzcd}
	E && {\pi \downarrow B} \\
	& B
	\arrow["\pi"', two heads, from=1-1, to=2-2]
	\arrow["{\partial_1'}", two heads, from=1-3, to=2-2]
	\arrow[""{name=0, anchor=center, inner sep=0}, "{\tau_\pi}"{description}, curve={height=12pt}, dashed, from=1-3, to=1-1]
	\arrow[""{name=1, anchor=center, inner sep=0}, "{\iota_\pi}"{description}, curve={height=6pt}, from=1-1, to=1-3]
	\arrow["\dashv"{anchor=center, rotate=-91}, draw=none, from=0, to=1]
\end{tikzcd}\]
\end{theorem}

\begin{proof}
Recall from \Cref{thm:cocartfams-via-transp} that the existence of the fibered left adjoint is equivalent to $\pi:E \fibarr B$ being a \emph{cocartesian} fibration (already without requiring $B$ to have pullbacks and $\pi$ to be a cartesian fibration). For $b:B$ the action of the fiberwise map $\tau \defeq \tau_\pi$ at $b:B$ is given by
\[ \tau_b(v:a \to b, e:P(a)) \defeq \partial_1 \, P_!(v,e).\]
Let $u:a \to b$ be a morphism in $B$. Over $u$, the action on arrows of $\tau$ maps a pair $\pair{\sigma}{f}$ consisting of a commutative square $\sigma$ in $B$ with boundary $[v',u',u,v]$ and a dependent arrow $f:d \to^P_{u'} e$ to the dependent arrow
\[ \tau_u(\sigma,f) = \tyfill_{P_!(v,e)\circ f}(P_!(v',e)),\]
cf.~\Cref{fig:cocart-transp}.
\begin{figure}
\[\begin{tikzcd}
	& d && e && d && e \\
	{} & {a'} && {b'} \\
	& a && b && {v'_!(d)} && {v_!(e)}
	\arrow["{v'}"', from=2-2, to=3-2]
	\arrow[""{name=0, anchor=center, inner sep=0}, "u"', from=3-2, to=3-4]
	\arrow[""{name=1, anchor=center, inner sep=0}, "{u'}", from=2-2, to=2-4]
	\arrow["v", from=2-4, to=3-4]
	\arrow["f", from=1-2, to=1-4]
	\arrow[Rightarrow, dotted, no head, from=1-2, to=2-2]
	\arrow[Rightarrow, dotted, no head, from=1-4, to=2-4]
	\arrow["f", from=1-6, to=1-8]
	\arrow[""{name=2, anchor=center, inner sep=0}, "{P_!(v',d)}"{description, pos=0.3}, from=1-6, to=3-6, cocart]
	\arrow["{\tau_u(\sigma,f)}"', dashed, from=3-6, to=3-8]
	\arrow["{P_!(v,e)}"{description, pos=0.3}, from=1-8, to=3-8, cocart]
	\arrow["\rightsquigarrow"{description}, draw=none, from=2-4, to=2]
	\arrow["\sigma"{description}, Rightarrow, draw=none, from=1, to=0]
\end{tikzcd}\]
\caption{Action on arrows of $\tau$ (over $u:a \to b$ in $B$)}
\label{fig:cocart-transp}
\end{figure} 
Recall the description of cartesian lifts in the family fibration, \Cref{prop:cartlift-famfib}. Then, $\tau$ mapping these $L(P)$-cartesian lifts to $P$-cartesian arrows is equivalent to the Beck--Chevalley condition.
\end{proof}

Next is a useful result stating that any functor (between Rezk types with pullbacks) preserves pullback if and only if its Artin gluing satisfies the Beck--Chevalley condition.

\begin{proposition}[Internal sums for gluing, \protect{\cite[Lemma~13.2]{streicher2020fibered}}]\label{prop:bcc-for-pb-pres}
	Let $A$ and $B$ be Rezk types with pullbacks and $F:A \to B$ an arbitrary functor (hence an isoinner map). Then the following are equivalent:
	\begin{enumerate}
		\item\label{it:gl-intsums-i} The functor $F$ preserves pullbacks.
		\item\label{it:gl-intsums-ii} The gluing fibration $\gl(F) \jdeq \partial_1: \comma{B}{F} \fibarr A$ is a Beck--Chevalley fibration.
	\end{enumerate}
\end{proposition}

\begin{proof}
	We note first that since $B$ has all pullbacks, $\gl(F): \comma{B}{F} \fibarr A$ is a fibration since it is a pullback of the fundamental fibration $\partial_1 : B^{\Delta^1} \fibarr B$. Hence, $\gl(F)$ is a bifibration in this case.
	
	\begin{description}
		\item[$\ref{it:gl-intsums-i} \implies \ref{it:gl-intsums-ii}$:] For a pullback square in $A$ we consider a square lying over in the gluing fibration which is a cube as in~\Cref{fig:gluing-depsqare}
		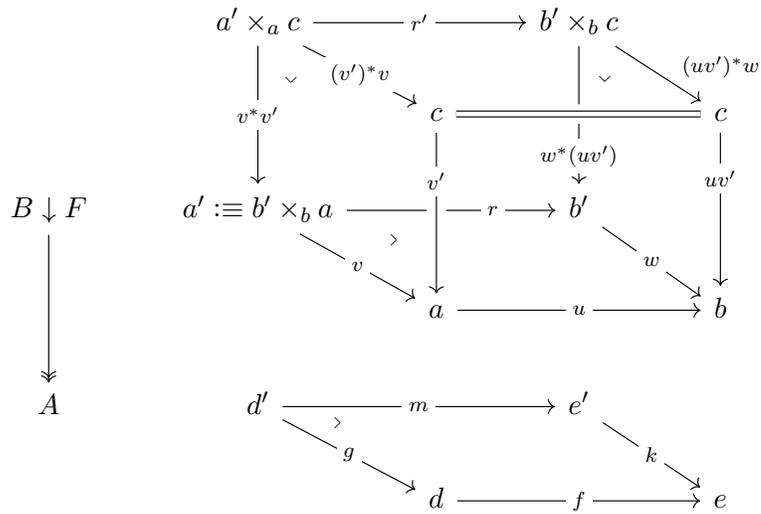
\begin{figure}
			\[\begin{tikzcd}
				& {a'\times_ac} && {b'\times_bc} \\
				&& c && c \\
				{B \downarrow F} & {a'\defeq b'\times_ba} && {b'} \\
				&& a && b \\
				A & {d'} && {e'} \\
				&& d && e
				\arrow["m"{description}, from=5-2, to=5-4]
				\arrow["g"{description}, from=5-2, to=6-3]
				\arrow["f"{description}, from=6-3, to=6-5]
				\arrow["k"{description}, from=5-4, to=6-5]
				\arrow["\lrcorner"{anchor=center, pos=0.125, rotate=45}, draw=none, from=5-2, to=6-5]
				\arrow["v"{description}, from=3-2, to=4-3]
				\arrow["u"{description}, from=4-3, to=4-5]
				\arrow["w"{description}, from=3-4, to=4-5]
				\arrow["r"{description, pos=0.7}, from=3-2, to=3-4]
				\arrow["{v^*v'}"{description}, from=1-2, to=3-2]
				\arrow["{(v')^*v}"{description}, from=1-2, to=2-3]
				\arrow["{uv'}"{description, pos=0.3}, from=2-5, to=4-5]
				\arrow["{r'}"{description}, from=1-2, to=1-4]
				\arrow["{(uv')^*w}"{pos=0.7}, from=1-4, to=2-5]
				\arrow["{w^*(uv')}"{description, pos=0.8}, from=1-4, to=3-4]
				\arrow["\lrcorner"{anchor=center, pos=0.125, rotate=-45}, draw=none, from=1-4, to=4-5]
				\arrow["\lrcorner"{anchor=center, pos=0.125, rotate=-45}, draw=none, from=1-2, to=4-3]
				\arrow["\lrcorner"{anchor=center, pos=0.125, rotate=45}, draw=none, from=3-2, to=4-5]
				\arrow[two heads, from=3-1, to=5-1]
				\arrow[Rightarrow, no head, from=2-3, to=2-5, crossing over]
				\arrow["{v'}"{description, pos=0.3}, from=2-3, to=4-3, crossing over]
			\end{tikzcd}\]
			\caption{Verifying the Beck--Chevalley condition}
			\label{fig:gluing-depsqare}
		\end{figure}
		where the pullback square at the bottom of the cube is the image of the given pullback square in $C$. By composition and right cancelation of pullbacks, the square on the left is also a pullback:
		\[\begin{tikzcd}
			{a' \times_a c} && {b'\times_b c} && {b'} \\
			c && c && b
			\arrow["{(v')^*v}"', from=1-1, to=2-1]
			\arrow[Rightarrow, no head, from=2-1, to=2-3]
			\arrow["{r'}", from=1-1, to=1-3]
			\arrow["{(uv')^*w}"', from=1-3, to=2-3]
			\arrow["{w^*(uv')}", from=1-3, to=1-5]
			\arrow["{uv'}"', from=2-3, to=2-5]
			\arrow["w", from=1-5, to=2-5]
			\arrow["\lrcorner"{anchor=center, pos=0.125}, draw=none, from=1-3, to=2-5]
			\arrow["\lrcorner"{anchor=center, pos=0.125}, draw=none, from=1-1, to=2-3]
		\end{tikzcd}\]
		Then, $r'$ turns out to be an isomorphism. By Rezk-completeness, it can be taken to be the identity $\id_c: c \to c$, exhibiting the back square of the cube as a cocartesian arrow in the gluing fibration, as desired.
		
		\item[$\ref{it:gl-intsums-ii} \implies \ref{it:gl-intsums-i}$:] From the Beck--Chevalley condition we obtain, for any pullback square in $A$ a commutative cube above as follows:
		\[\begin{tikzcd}
			& {a'} && {a'} \\
			&& c && c \\
			{B \downarrow F} & {a'} && {b'} \\
			&& a && b \\
			A & {d'} && {e'} \\
			&& d && e
			\arrow["m"{description}, from=5-2, to=5-4]
			\arrow["g"{description}, from=5-2, to=6-3]
			\arrow["f"{description}, from=6-3, to=6-5]
			\arrow["k"{description}, from=5-4, to=6-5]
			\arrow["\lrcorner"{anchor=center, pos=0.125, rotate=45}, draw=none, from=5-2, to=6-5]
			\arrow["v"{description}, from=3-2, to=4-3]
			\arrow["u"{description}, from=4-3, to=4-5]
			\arrow["w"{description}, from=3-4, to=4-5]
			\arrow["r"{description, pos=0.7}, from=3-2, to=3-4]
			\arrow[Rightarrow, no head, from=1-2, to=3-2]
			\arrow[from=1-2, to=2-3]
			\arrow["{uv'}"{description, pos=0.3}, from=2-5, to=4-5]
			\arrow[Rightarrow, no head, from=1-2, to=1-4]
			\arrow["{(uv')^*w}"{pos=0.7}, from=1-4, to=2-5]
			\arrow["{w^*(uv')}"{description, pos=0.75}, from=1-4, to=3-4]
			\arrow["\lrcorner"{anchor=center, pos=0.125, rotate=-45}, draw=none, from=1-4, to=4-5]
			\arrow["\lrcorner"{anchor=center, pos=0.125, rotate=-45}, draw=none, from=1-2, to=4-3]
			\arrow[two heads, from=3-1, to=5-1]
			\arrow[Rightarrow, no head, from=2-3, to=2-5, crossing over]
			\arrow[Rightarrow, no head, from=2-3, to=4-3, crossing over]
		\end{tikzcd}\]
		Since the right outer square is a pullback, by composition also the bottom square is. This shows that $\gl(F)$ preserves pullbacks.
	\end{description}
\end{proof}

\section{Moens families}

\subsection{(Pre-)Moens families and internal sums}

Recall from classical $1$-category theory that a category $\mathbb C$ with pullbacks and coproducts is \emph{extensive} (or \emph{lextensive} depending on convention) if and only if, for all small families $(A_i)_{i\in I}$ the induced functor $\prod_{i \in I} \mathbb C/A_i \to \mathbb C/\coprod_{i\in I} A_i$ is an equivalence. This is equivalent to the condition that injections of finite sums are stable under pullback, and for any family of squares
\[\begin{tikzcd}
	{B_k} && B \\
	{A_k} && {\coprod_{i\in I} A_i}
	\arrow["{g_k}", from=1-1, to=1-3]
	\arrow["{f_k}"', from=1-1, to=2-1]
	\arrow[from=2-1, to=2-3]
	\arrow[from=1-3, to=2-3]
\end{tikzcd}\]
all of these are pullbacks if and only if all $g_k: B_k \to B$ are coproduct cones. This generalizes fibrationally as follows.

\begin{definition}[Stable and disjoint internal sums]
	Let $P:B \to \UU$ be a lex fibration with internal sums over a Rezk type $B$. Then $P$ has \emph{stable} internal sums if cocartesian arrows are stable under arbitrary pullbacks. The internal sums of $P$ are \emph{disjoint}\footnote{In a category, a coproduct is \emph{disjoint} if the inclusion maps are monomorphisms, and the intersection of the summands is an initial object.} if for every cocartesian arrow $f:d \cocartarr^P e$ the fibered diagonal is cocartesian, too:
	\[\begin{tikzcd}
		d \\
		& {d \times_ed} && d \\
		& d && e
		\arrow["f"', from=3-2, to=3-4, cocart]
		\arrow[from=2-2, to=2-4]
		\arrow["f", from=2-4, to=3-4, cocart]
		\arrow[curve={height=-24pt}, Rightarrow, no head, from=1-1, to=2-4]
		\arrow[curve={height=18pt}, Rightarrow, no head, from=1-1, to=3-2]
		\arrow[from=2-2, to=3-2]
		\arrow["{\delta_f}", from=1-1, to=2-2, cocart]
		\arrow["\lrcorner"{anchor=center, pos=0.125}, draw=none, from=2-2, to=3-4]
	\end{tikzcd}\]
\end{definition}

\begin{definition}[(Pre-)Moens families]
Let $B$ be a lex Rezk type. A lex Beck--Chevalley family $P:B \to \UU$ is a \emph{pre-Moens family} if it has stable internal sums. We call a pre-Moens family $P:B \to \UU$ \emph{Moens family} (or \emph{extensive family}  or \emph{pre-geometric family}) if, moreover, all its (stable) internal sums are also disjoint.
\end{definition}

An immediate result is the following.
\begin{lemma}[\cite{streicher2020fibered}, Lem.~15.1]\label{lem:cocart-leg-pb}
	Let $B$ be a lex Rezk type and $\pi:E \fibarr B$ be a Moens fibration. Then, for $d,e,e':E$, and cocartesian morphisms $g:e \cocartarr e'$, in any pullback of the following form, the gap map is cocartesian, too:
\[\begin{tikzcd}
	d \\
	& {d'} && e \\
	& d && {e'}
	\arrow[from=2-2, to=3-2]
	\arrow["h"', from=3-2, to=3-4]
	\arrow[from=2-2, to=2-4]
	\arrow["g", from=2-4, to=3-4,cocart]
	\arrow["f", curve={height=-12pt}, from=1-1, to=2-4]
	\arrow[curve={height=12pt}, Rightarrow, no head, from=1-1, to=3-2]
	\arrow["\lrcorner"{anchor=center, pos=0.125}, draw=none, from=2-2, to=3-4]
	\arrow["k", dashed, from=1-1, to=2-2, cocart]
\end{tikzcd}\]
\end{lemma}

\begin{proof}
Consider the following diagram, arising from canonical factorizations over the pullbacks $d' \jdeq d \times_{e'} e$ and $e \times_e e'$, resp:
	\[\begin{tikzcd}
		d && e \\
		{d'} && {e \times_{e'} e} && e \\
		d && e && {e'}
		\arrow["k"', from=1-1, to=2-1, cocart]
		\arrow[from=2-1, to=2-3]
		\arrow["f", from=1-1, to=1-3]
		\arrow["{\delta_g}"', from=1-3, to=2-3, cocart]
		\arrow[from=2-1, to=3-1]
		\arrow["f"', from=3-1, to=3-3]
		\arrow[from=2-3, to=3-3]
		\arrow[from=2-3, to=2-5]
		\arrow["g"', from=3-3, to=3-5, cocart]
		\arrow["g", from=2-5, to=3-5, cocart]
		\arrow[Rightarrow, no head, from=1-3, to=2-5]
		\arrow["\lrcorner"{anchor=center, pos=0.125}, draw=none, from=1-1, to=2-3]
		\arrow["\lrcorner"{anchor=center, pos=0.125}, draw=none, from=2-1, to=3-3]
		\arrow["\lrcorner"{anchor=center, pos=0.125}, draw=none, from=2-3, to=3-5]
		\arrow[curve={height=24pt}, Rightarrow, no head, from=1-1, to=3-1, crossing over]
		\arrow[curve={height=40pt}, Rightarrow, no head, from=1-3, to=3-3, crossing over]
	\end{tikzcd}\]
By disjointness of sums, $\delta_g: e \cocartarr e \times_{e'} e$ is cocartesian, and by general pullback stability, so is $k:d \cocartarr d'$.
\end{proof}

The preceding lemma can be used to characterize disjointness given that stable internal sums exist. 

\begin{proposition}[Characterizations of disjointness of stable internal sums, \cite{streicher2020fibered}, Lem.~15.2]\label{prop:char-disj-stable}
Let $B$ be a lex Rezk type and $P:B \to \UU$ be a pre-Moens family. Then the following are equivalent:\footnote{Streicher \cite{streicher2020fibered} points out that the first three points only require stability of cocartesian arrows along \emph{vertical} maps.}
\begin{enumerate}
	\item\label{it:disj-pre-moens-i} The family $P$ is a Moens family, \ie~internal sums are disjoint (and stable).
	\item\label{it:disj-pre-moens-ii} Cocartesian arrows in $P$ satisfy left canceling, \ie~if $g$, $g \circ f$ are cocartesian then so is $f$.
	\item\label{it:disj-pre-moens-iii} Cocartesian transport is conservative, \ie~if $k$ is vertical, and both $f$ and $f \circ k$ are cocartesian, then $k$ is an isomorphism.
	\item\label{it:disj-pre-moens-iv} Any dependent square in $P$ of the form
	\[\begin{tikzcd}
		d && e \\
		{d'} && {e'}
		\arrow["f"', squiggly, from=1-1, to=2-1]
		\arrow["g"', from=2-1, to=2-3, cocart]
		\arrow["h", from=1-1, to=1-3, cocart]
		\arrow["k", squiggly, from=1-3, to=2-3]
	\end{tikzcd}\]
where $f$, $k$ are vertical and $g$, $h$ are cocartesian is a pullback.
\end{enumerate}	
\end{proposition}

\begin{proof}
	\begin{description}
		\item[$\ref{it:disj-pre-moens-i}  \implies \ref{it:disj-pre-moens-ii}$:] Let $f:e \to e'$ and $g:e' \cocartarr e''$ s.t.~$g \circ f: e \cocartarr e''$ is cocartesian. Since $P$ is a Moens family we can apply~\Cref{lem:cocart-leg-pb} to obtain that the gap map $\varphi: e \cocartarr e' \times_{e''} e$ as in
		\[\begin{tikzcd}
			e \\
			& {e'''} && e \\
			& {e'} && {e''}
			\arrow["{g'}", from=2-2, to=3-2]
			\arrow["g"', from=3-2, to=3-4, cocart]
			\arrow[from=2-2, to=2-4]
			\arrow["gf", from=2-4, to=3-4, cocart]
			\arrow[curve={height=-18pt}, Rightarrow, no head, from=1-1, to=2-4]
			\arrow["f"', curve={height=12pt}, from=1-1, to=3-2]
			\arrow["\varphi", dashed, from=1-1, to=2-2, cocart]
			\arrow["\lrcorner"{anchor=center, pos=0.125}, draw=none, from=2-2, to=3-4]
		\end{tikzcd}\]
		is cocartesian. By stability, $g'$ is cocartesian, too, hence so is $f= g' \circ \varphi$.
		\item[$\ref{it:disj-pre-moens-ii}  \implies \ref{it:disj-pre-moens-i}$:] By stability of sums, for a cocartesian arrow $f:d \cocartarr e$, the map $f^*f: d \times_e d \to d$ is cocartesian. Then, the gap map $\delta_d: d \to d \times_e d$ is cocartesian, since $\id_d = f^*f \circ \delta_d$.
		\item[$\ref{it:disj-pre-moens-ii}  \implies \ref{it:disj-pre-moens-iii}$:] This follows since an arrow that is vertical and cocartesian necessarily is an isomorphism.
		\item[$\ref{it:disj-pre-moens-iii}  \implies \ref{it:disj-pre-moens-ii}$:] Let $g$, $f$ be given s.t.~$gf$ exists, and both $gf$ as well as $g$ are cocartesian. Consider the factorization $f = mh$ where $m$ is vertical and $h$ is cocartesian. Then, by right cancelation of cocartesian arrows, since $gf = (gm)h$ and $h$ both are cocartesian, so must be $gm$. 
		\item[$\ref{it:disj-pre-moens-iii}  \implies \ref{it:disj-pre-moens-iv}$]: Consider the induced pullback square:
		\[\begin{tikzcd}
			d \\
			& {e''} && e \\
			& {d'} && {e'}
			\arrow["{k'}"', squiggly, from=2-2, to=3-2]
			\arrow["{g'}", from=2-2, to=2-4, cocart]
			\arrow["k", squiggly, from=2-4, to=3-4]
			\arrow["h", curve={height=-18pt}, from=1-1, to=2-4, cocart]
			\arrow["f"', curve={height=18pt}, squiggly, from=1-1, to=3-2]
			\arrow["{f'}", squiggly, from=1-1, to=2-2, dashed]
			\arrow["\lrcorner"{anchor=center, pos=0.125}, draw=none, from=2-2, to=3-4]
			\arrow["g"', from=3-2, to=3-4, cocart]
		\end{tikzcd}\]
		Since $P$ is a bifibration, the vertical arrows are stable under pullback along any arrow, and satisfy left cancelation.\footnote{This can be shown by hand. But it should also be possible to exhibit them as the left and a right class of the two ensuing orthogonal factorization systems, pending an apprioprate synthetic formulation.} Hence they  Hence, since $k$ is vertical, so is $k'$, and consequently $f'$ as well (since $f$ is). By the assumption in $(\ref{it:disj-pre-moens-iii})$ since $f'$ is vertical and both $g'$ and $h=g' \circ f'$ are cocartesian $f'$ is an isomorphism.
		\item[$\ref{it:disj-pre-moens-iv}  \implies \ref{it:disj-pre-moens-iii}$:] Any square of the form
		\[\begin{tikzcd}
			d && e \\
			{d'} && e
			\arrow["k"', squiggly, from=1-1, to=2-1]
			\arrow["f"', from=2-1, to=2-3,  cocart]
			\arrow["fk", from=1-1, to=1-3, cocart]
			\arrow[Rightarrow, no head, from=1-3, to=2-3]
			\arrow["\lrcorner"{anchor=center, pos=0.125}, draw=none, from=1-1, to=2-3]
		\end{tikzcd}\]
	is a pullback by precondition, hence $k$ is an identity.

	\end{description}
\end{proof}

\subsubsection{Extensive internal sums}

We can now provide a characterization of Moens families among the BCC families. 

In particular, we obtain a fibered version of \emph{Lawvere-extensivity} as an alternative characterization for (internal) extensivity. Classically, a category $\mathbb C$ is \emph{Lawvere-extensive} if for any small set $I$, the categories $\mathbb C^I$ and $\mathbb C/\coprod_{i\in I} \unit$ are canonically isomorphic.

To prepare, consider first the following construction.

Let $B$ be a lex Rezk type and $P:B \to \UU$ be a cocartesian family.

\begin{definition}[Terminal transport functor]
	For a terminal element $z:B$, we define the functor\footnote{In \cite{streicher2020fibered}, the functor $\omega$ is called $\mathbf{\Delta}$.}\footnote{Note that we can suppress the dependency on a specified terminal element $z:B$.}
	\[ \omega_{P,z} \defeq \omega : \widetilde{P} \to P\,z, \quad \omega \defeq \lambda b,e.(!_b)_!(e). \]
\end{definition}
\begin{figure}
	\[\begin{tikzcd}
		e && {\omega(e)} \\
		{e'} && {\omega(e')} \\
		b && z \\
		{b'}
		\arrow["f"', from=1-1, to=2-1]
		\arrow["{P_!(!_{b'},e')}"', from=2-1, to=2-3,cocart]
		\arrow["{P_!(!_b,e)}", from=1-1, to=1-3,cocart]
		\arrow["{\omega(f)}", dashed, from=1-3, to=2-3]
		\arrow["u"', from=3-1, to=4-1]
		\arrow["{!_b}", from=3-1, to=3-3]
		\arrow["{!_{b'}}"', from=4-1, to=3-3]
	\end{tikzcd}\]
	\caption{Action on morphisms of the transport functor $\omega_{P,z}$}
	\label{fig:transp-term}
\end{figure}
The action on arrows of this functor is illustrated in \ref{fig:transp-term}. The arrow $\omega_f$ is vertical over the terminal element $z:B$, for any $f: \Delta^1 \to B$. 

\begin{definition}[Choice of terminal elements]
	Let $B$ be a Rezk type and $P:B \to \UU$ be an isoinner family such that every fiber has a terminal element. Then we denote, by the Principle of Choice, the section choosing fiberwise terminal elements by
	\[ \zeta_P \defeq \zeta: \prod_{b:B} P\,b,\]
	\ie~for any $b:B$ the element $\zeta_b : P\,b$ is terminal.
\end{definition}

We define
\[ \omega' \jdeq \omega'_P: B \to P\,z, \quad \omega'(b)\defeq \omega(\zeta_b) \jdeq (!_b)_!(\zeta_b).\]

We are now ready for the promised characterization.

\begin{proposition}[Stable disjoint sums in terms of extensive sums, \cite{streicher2020fibered}, Lem.~15.3]\label{prop:ext-sums}
Let $B$ be a lex Rezk type and $P:B \to \UU$ be a Beck--Chevalley family. Then, the following are equivalent:
\begin{enumerate}
	\item\label{it:ext-sums-stable} The family $P$ is a Moens family, \ie~$P$ has stable disjoint sums.
	\item\label{it:ext-sums-ext} The bicartesian family $P$ has \emph{internally extensive} sums, \ie~for vertical arrows $f:d \to d'$, $k:e \to e'$, cocartesian arrows $g: e \cocartarr e'$, in a square

	\[\begin{tikzcd}
		d && e \\
		{d'} && {e'}
		\arrow["f"', squiggly, from=1-1, to=2-1]
		\arrow["g"', from=2-1, to=2-3, cocart]
		\arrow["h", from=1-1, to=1-3]
		\arrow["k", squiggly, from=1-3, to=2-3]
	\end{tikzcd}\]
	the arrow $h:d \to e$ is a cocartesian arrow if and only if the square is a pullback.
	\item\label{it:ext-sums-lawvere} The internal sums in $P$ are \emph{Lawvere-extensive}, \ie~in any square of the form
	\[\begin{tikzcd}
		d && e \\
		{\zeta_a} && {\omega'(a)}
		\arrow["!_d^a"', from=1-1, to=2-1, squiggly]
		\arrow["{P_!(!_a, \zeta_a)}"', from=2-1, to=2-3, cocart]
		\arrow["h", from=1-1, to=1-3]
		\arrow["k", from=1-3, to=2-3, squiggly]
	\end{tikzcd}\]
	where $k:e \vertarr \omega'(a)$ is vertical the arrow $h:d \to e$ is cocartesian if and only if the given square is a pullback.
	\item\label{it:ext-sums-transp} Let $z:B$ be a terminal element in $B$. For any $a:B$, the transport functor $(!_a)_!: P\,a \to P\,z$ reflects isomorphisms and $k^*P_!(!_a,\zeta_a)$ is cocartesian in case $k$ is vertical.
\end{enumerate}
\end{proposition}

Again, as remarked by Streicher, the equivalences between all but the first statement hold already in the case that cocartesian arrows are only stable under pullback along vertical arrows.
\begin{proof}
	\begin{description}
		\item[$(\ref{it:ext-sums-stable}) \implies (\ref{it:ext-sums-ext})$:] Consider a square as given in (\ref{it:ext-sums-ext}). If it is a pullback we have an identification $h= k^*g$, and by stability $h$ is cocartesian, too. Conversely, given such a square where $h$ is cocartesian, consider the factorization:
		\[\begin{tikzcd}
			d \\
			& {d''} && e \\
			& {d'} && {e'}
			\arrow["k'", from=2-2, to=3-2, squiggly]
			\arrow["g"', from=3-2, to=3-4, cocart]
			\arrow["{g'}"', from=2-2, to=2-4, cocart]
			\arrow["k", from=2-4, to=3-4, squiggly]
			\arrow["h", curve={height=-12pt}, from=1-1, to=2-4, cocart]
			\arrow["f"', curve={height=12pt}, squiggly, from=1-1, to=3-2]
			\arrow["{f'}"', dashed, from=1-1, to=2-2]
			\arrow["\lrcorner"{anchor=center, pos=0.125}, draw=none, from=2-2, to=3-4]
		\end{tikzcd}\]
	The arrow $g' = k^*g$ is cocartesian by stability of sums. The arrow $k' = g^*k$ is vertical since $k$ is. By the same reason, they are left cancelable, hence $f'$ is vertical. But since $P$ is a Moens family, by~\Cref{prop:char-disj-stable}, cocartesian arrows also satisfy left cancelation, hence $f'$ is cocartesian, too, and thus an equivalence.
	\item[$(\ref{it:ext-sums-ext}) \implies (\ref{it:ext-sums-lawvere})$:] The latter is an instance of the former.
	\item[$(\ref{it:ext-sums-lawvere}) \implies (\ref{it:ext-sums-transp})$:] By assumption, for any vertical arrow $f: d \vertarr d'$ in $P\,a$, the induced arrow $(!_a)_!(f):(!_a)_!(d) \to (!_a)_!(d')$ is a path:
	\[\begin{tikzcd}
		d && {(!_a)_!(d)} \\
		{d'} && {(!_a)_!(d')}
		\arrow["f"', squiggly, from=1-1, to=2-1]
		\arrow[from=1-1, to=1-3,  cocart]
		\arrow[from=2-1, to=2-3,  cocart]
		\arrow["{(!_a)_!(f)}", Rightarrow, no head, from=1-3, to=2-3]
	\end{tikzcd}\]
	Consider the cube induced by cocartesian filling \wrt~$P_!(!_a,\zeta_a) \circ !_d$ and $P_!(!_a,\zeta_a) \circ !_d'$, resp.:
	\[\begin{tikzcd}
		d && {(!_a)_!(d)} \\
		& {\zeta_a} && {\omega'(a)} \\
		{d'} && {(!_a)_!(d')} \\
		& {\zeta_a} && {\omega'(a)}
		\arrow["{!_{d'}}", squiggly, from=3-1, to=4-2, swap]
		\arrow[from=4-2, to=4-4,  cocart, near end]
		\arrow[from=3-1, to=3-3,  cocart, near end]
		\arrow[squiggly, from=3-3, to=4-4, "\tyfill"]
		\arrow[from=1-1, to=3-1, "f", swap, squiggly]
		\arrow[from=1-1, to=1-3,  cocart, near end]
		\arrow[Rightarrow, no head, from=1-3, to=3-3]
		\arrow[Rightarrow, no head, from=2-2, to=4-2, curve={height=15pt}]
		\arrow["{!_d}", squiggly, from=1-1, to=2-2, swap]
		\arrow[squiggly, from=1-3, to=2-4, "\tyfill"]
		\arrow["\lrcorner"{anchor=center, pos=0.125, rotate=45}, draw=none, from=3-1, to=4-4]
		\arrow["\lrcorner"{anchor=center, pos=0.125, rotate=45}, draw=none, from=1-1, to=2-4]
		\arrow[Rightarrow, no head, from=2-4, to=4-4, crossing over]
		\arrow[from=2-2, to=2-4,  cocart, near end, crossing over]
\end{tikzcd}\]
The bottom and top squares are pullbacks by Lawvere extensivity. Then, by~\Cref{prop:htopy-inv-of-pb}, the map $f$ is an identity as well.

For $a:B$ and a vertical arrow $k: e \vertarr_z \omega'(a)$ consider the pullback square:
\[\begin{tikzcd}
	{e'} && e \\
	{\zeta_a} && {\omega'(a)}
	\arrow["{k'}"', from=1-1, to=2-1]
	\arrow["{P_!(!_a,\zeta_a)}"', from=2-1, to=2-3, cocart]
	\arrow["f", from=1-1, to=1-3]
	\arrow["k", squiggly, from=1-3, to=2-3]
	\arrow["\lrcorner"{anchor=center, pos=0.125}, draw=none, from=1-1, to=2-3]
\end{tikzcd}\]
Since $P$ is a bifibration, $k'$ is vertical. By Lawvere extensivity, $f$ is cocartesian.
\item[$(\ref{it:ext-sums-transp}) \implies (\ref{it:ext-sums-lawvere})$:] Given a square as in $(\ref{it:ext-sums-lawvere})$, we see that the arrow $h$ is cocartesian if and only if it is a pullback, by the second condition in $(\ref{it:ext-sums-lawvere})$.
\item[$(\ref{it:ext-sums-lawvere}) \implies (\ref{it:ext-sums-ext})$:] This is mostly analogous to~\Cite[Lemma~15.3]{streicher2020fibered}.
	\end{description}
\end{proof}

\begin{corollary}[\protect{\cite[Corollary~15.4]{streicher2020fibered}}]
	Let $B$ be a Rezk type and $P:B \to \UU$ a Moens family.
	
	Then, for all arrows $u:a \to b$ in $B$ and points $d:P\,a$, the functor
	\[ \comma{u_!}{d}: \comma{P\,a}{d} \to \comma{P\,b}{u_!d}\]
	is an equivalence.
	In particular, for $P\,a \equiv \comma{P\,a}{\zeta_a}$ we have equivalences
	\[ \comma{u_!}{\zeta_a} : \comma{P\,a}{\zeta_a} \equiv \comma{P\,b}{u_!\zeta_a}, \quad \comma{(!_a)_!}{\zeta_a}:\comma{P\,a}{\zeta_a} \simeq \comma{P\,z}{\omega'(a)}. \]
\end{corollary}

\begin{corollary}[Left exactness of covariant transport in Moens families, \protect{\cite[Corollary~15.5]{streicher2020fibered}}]\label{cor:lex-transp-moens}
	Let $B$ be a Rezk type and $P:B \to \UU$ a Moens family.
	
	Then for all $u:a \to b$ in $B$, the covariant transport functor
	\[ u_!:P\,a \to P\,b\]
	is left exact.
\end{corollary}

As one crucial ingredient for Moens' Theorem, the gluing of pullback-preserving functors always is a Moens fibration.
\begin{prop}\label{prop:gluing-moens}
Let $B$ and $C$ be a lex Rezk types and $F:B \to C$ be a pullback-preserving functor. Then $\gl(F): \comma{C}{F} \fibarr B$ is a Moens fibration.	
\end{prop}

\begin{proof}
 Since $F$ preserves pullback, $\gl(F)$ is a Beck--Chevalley fibration by~\Cref{prop:bcc-for-pb-pres}. Then, by~\Cref{prop:ext-sums}, it suffices to prove that the internal sums in $\gl(F)$ are extensive. But this follows from considering a dependent cube as given in~\Cref{fig:gluing-moens} and the fact that $k:c \to c'$ is an isomorphism if and only if the top square is a pullback.\footnote{The identities in the cube are part of the prerequisites to prove extensivity.}
\end{proof}

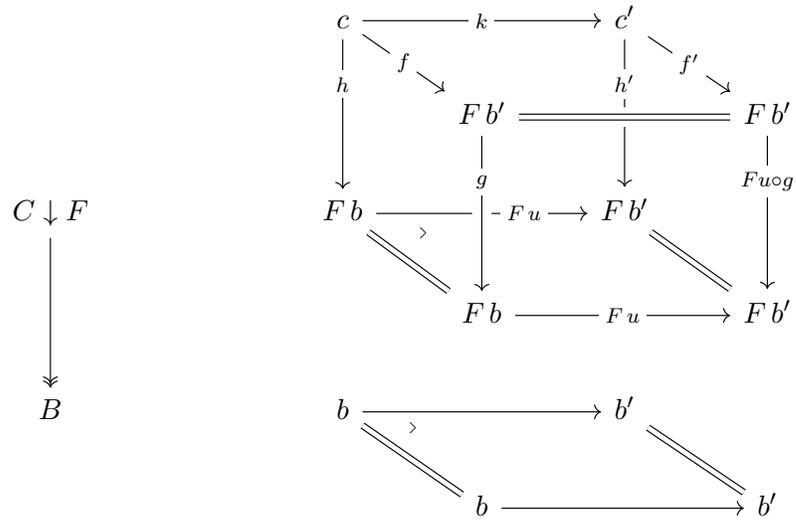
\begin{figure}
\[\begin{tikzcd}
	&&& c && {c'} \\
	&&&& {F\,b'} && {F\,b'} \\
	{C \downarrow F} &&& {F\,b} && {F\,b'} \\
	&&&& {F\,b} && {F\,b'} \\
	B &&& b && {b'} \\
	&&&& b && {b'}
	\arrow[Rightarrow, no head, from=5-4, to=6-5]
	\arrow[from=5-4, to=5-6]
	\arrow[Rightarrow, no head, from=5-6, to=6-7]
	\arrow[from=6-5, to=6-7]
	\arrow["h"{description, pos=0.3}, from=1-4, to=3-4]
	\arrow[Rightarrow, no head, from=3-4, to=4-5]
	\arrow["{F\,u}"{description}, from=4-5, to=4-7]
	\arrow[Rightarrow, no head, from=3-6, to=4-7]
	\arrow["{F\,u}"{description, pos=0.7}, from=3-4, to=3-6]
	\arrow["f"{description}, from=1-4, to=2-5]
	\arrow["k"{description}, from=1-4, to=1-6]
	\arrow["{f'}"{description}, from=1-6, to=2-7]
	\arrow["{Fu\circ g}"{description, pos=0.3}, from=2-7, to=4-7]
	\arrow["\lrcorner"{anchor=center, pos=0.125, rotate=45}, draw=none, from=5-4, to=6-7]
	\arrow["\lrcorner"{anchor=center, pos=0.125, rotate=45}, draw=none, from=3-4, to=4-7]
	\arrow["{h'}"{description, pos=0.3}, from=1-6, to=3-6]
	\arrow[two heads, from=3-1, to=5-1]
	\arrow["g"{description, pos=0.3}, from=2-5, to=4-5, crossing over]
	\arrow[Rightarrow, no head, from=2-5, to=2-7, crossing over]
\end{tikzcd}\]
\caption{Extensivity of sums in $\gl(F)$}
\label{fig:gluing-moens}
\end{figure}

\begin{lemma}[\cf~\protect{\cite[Lemma~15.6]{streicher2020fibered}}]\label{lem:gap-to-vert-cocart}
	Let $B$ be a lex Rezk type and $P:B \to \UU$ be a Moens fibration. Then the gap arrow in any diagram of the form
	\[\begin{tikzcd}
		d &&&& {d'} \\
		&& {e'''} && {e'} \\
		{d''} && e && {e''}
		\arrow["{f'}"', from=1-1, to=3-1,cocart]
		\arrow["f", from=1-1, to=1-5,cocart]
		\arrow["\lrcorner"{anchor=center, pos=0.125}, draw=none, from=2-3, to=3-5]
		\arrow["{h'}"', from=3-3, to=3-5, squiggly]
		\arrow["{g'}"', from=3-1, to=3-3,cocart]
		\arrow["g", from=1-5, to=2-5, cocart]
		\arrow["h", squiggly, from=2-5, to=3-5]
		\arrow["k", from=2-3, to=2-5, squiggly]
		\arrow["m"', squiggly, from=2-3, to=3-3]
		\arrow["\lrcorner"{anchor=center, pos=0.125}, draw=none, from=1-1, to=2-3]
		\arrow["{f''}"', curve={height=12pt}, dashed, from=1-1, to=2-3, cocart]
	\end{tikzcd}\]
	is cocartesian as well.
\end{lemma}

\begin{proof}
	Since $P$ is a Moens fibration, cocartesian arrows and (vertical arrows, too) are stable under pullback along arbitrary arrows. By the Pullback Lemma, this gives rise to the following diagram:
	\[\begin{tikzcd}
		d \\
		& {y'} && x && {d'} \\
		& y && {e'''} && {e'} \\
		& {d''} && e && {e''}
		\arrow["{g'}"', from=4-2, to=4-4, cocart]
		\arrow["{h'}"', squiggly, from=4-4, to=4-6]
		\arrow["h", squiggly, from=3-6, to=4-6]
		\arrow["g", from=2-6, to=3-6, cocart]
		\arrow["m"', squiggly, from=3-4, to=4-4]
		\arrow["{m''}"', from=2-4, to=3-4, cocart, swap]
		\arrow["{k'}"', squiggly, from=2-4, to=2-6]
		\arrow["k"', squiggly, from=3-4, to=3-6]
		\arrow["\lrcorner"{anchor=center, pos=0.125}, draw=none, from=3-4, to=4-6]
		\arrow["\lrcorner"{anchor=center, pos=0.125}, draw=none, from=2-4, to=3-6]
		\arrow["{m'}"', squiggly, from=3-2, to=4-2]
		\arrow["{g''}"', from=3-2, to=3-4, cocart]
		\arrow["{m'''}"', from=2-2, to=3-2, cocart]
		\arrow["{g'''}"', from=2-2, to=2-4, cocart]
		\arrow["\lrcorner"{anchor=center, pos=0.125}, draw=none, from=2-2, to=3-4]
		\arrow["\lrcorner"{anchor=center, pos=0.125}, draw=none, from=3-2, to=4-4]
		\arrow["f", curve={height=-18pt}, from=1-1, to=2-6, cocart]
		\arrow["{f'}"', curve={height=18pt}, from=1-1, to=4-2, cocart]
		\arrow["\cong"{description}, dashed, from=1-1, to=2-2]
		\arrow[curve={height=6pt}, from=2-2, to=3-4]
		\arrow["{f''}", from=1-1, to=3-4, crossing over,curve={height=-27pt}]
	\end{tikzcd}\]
	Hence, $f''$ is (up to identification) the composition of two cocartesian arrows, hence cocartesian itself.
\end{proof}

\begin{proposition}[Left exactness of terminal transport, \protect{\cite[Lemma~15.16]{streicher2020fibered}}]\label{prop:ttrans-lex}
Let $B$ be a lex Rezk type and $P:B \to \UU$ be a Moens family. Then the terminal transport functor $\omega: \widetilde{P} \to P\,z$ is lex.
\end{proposition}

\begin{proof}
	Preservation of the terminal object follows from 
	\[ \omega(z,\zeta_z) \jdeq (!_z)_!(\zeta_z) = \id_{\zeta_z}.\]
	Consider a pullback
	\[\begin{tikzcd}
		{e'''} && {e''} \\
		{e'} && e
		\arrow["{f'}", from=1-1, to=1-3]
		\arrow["{g'}"', from=1-1, to=2-1]
		\arrow["f"', from=2-1, to=2-3]
		\arrow["g", from=1-3, to=2-3]
		\arrow["\lrcorner"{anchor=center, pos=0.125}, draw=none, from=1-1, to=2-3]
	\end{tikzcd}\]
	in $E$, where $f$ lies over an arrow $u$, and $g$ over an arrow $v$ in $B$. Considering the induced diagram
	\[\begin{tikzcd}
		{e'''} &&&& {e''} \\
		&& d && {v_!\,e''} \\
		{e'} && {u_!\,e'} && e
		\arrow["m", squiggly, from=3-3, to=3-5]
		\arrow["k", from=3-1, to=3-3, cocart]
		\arrow["{r'}"', squiggly, from=2-3, to=3-3]
		\arrow["{m'}", squiggly, from=2-3, to=2-5]
		\arrow["\ell"', from=1-5, to=2-5, cocart]
		\arrow["{f'}", from=1-1, to=1-5]
		\arrow["{g'}"', from=1-1, to=3-1]
		\arrow["\lrcorner"{anchor=center, pos=0.125}, draw=none, from=2-3, to=3-5]
		\arrow["r"', squiggly, from=2-5, to=3-5]
		\arrow["h", from=1-1, to=2-3]
		\arrow["\lrcorner"{anchor=center, pos=0.125}, shift right=4, draw=none, from=1-1, to=2-3]
		\arrow["f"{description}, curve={height=14pt}, from=3-1, to=3-5]
		\arrow["g"{description}, curve={height=-18pt}, from=1-5, to=3-5]
		\arrow["\sigma"{description}, draw=none, from=2-3, to=3-5]
	\end{tikzcd}\]
	by~\Cref{lem:gap-to-vert-cocart}, we find that the gap map $h$ is cocartesian. By naturality, the image of the outer diagram under $\omega$ is identifiable with the image of $\sigma$ under cocartesian transport $(!_b)_!$, which preserves pullbacks by~\Cref{cor:lex-transp-moens}.
\end{proof}

\subsection{Moens'~Theorem}

We are now ready to prove a version of Moens' Theorem,\footnote{Note that in the absence of categorical universes, we have to consider the ``naive'' $\Sigma$-types instead of the respective Rezk types of fibrations or lex functors, resp. Thus, our version of the theorem should be considered as a statement ``on the level of objects'' but not arrows.} chracterizing the type of Moens fibrations over a fixed lex base as the type of lex functors from this type into some other lex type.

\begin{theorem}[Moens'~Theorem, \protect{\cite[Theorem~5.18]{streicher2020fibered}}]\label{thm:moens-thm}
For a small lex Rezk type $B:\UU$ the type
\[ \MoensFam(B) \defeq \sum_{P:B \to \UU} \isMoensFam \,P\]
of $\UU$-small Moens families is equivalent to the type
\[ B \downarrow^\lex \mathrm{LexRezk} \defeq\sum_{C:\LexRezk} (B \to^\lex C)\]
of lex functors from $B$ into the type $\LexRezk$ of $\UU$-small lex Rezk types.
\end{theorem}

\begin{proof}
	We define a pair of quasi-inverses
	\[\begin{tikzcd}
		{\mathrm{MoensFam}(B)} &&& {B \downarrow^\lex \mathrm{LexRezk}}
		\arrow[""{name=0, anchor=center, inner sep=0}, "\Phi", curve={height=-18pt}, from=1-1, to=1-4]
		\arrow[""{name=1, anchor=center, inner sep=0}, "\Psi", curve={height=-18pt}, from=1-4, to=1-1]
		\arrow["\simeq"{description}, Rightarrow, draw=none, from=0, to=1]
	\end{tikzcd}\]
by setting
\begin{align*} 
	\Phi(P:B \to \UU) & \defeq \omega_P' \jdeq \lambda b.(!_b)_!(\zeta_b): B \to P\,z, \\
	\Psi(F:B \to C)&  \defeq \big(\gl(F):\comma{C}{F} \to B\big).
\end{align*} 
The values of $\Phi$ and $\Psi$, resp., are indeed terms of the respective types due to~\Cref{prop:ttrans-lex,prop:gluing-moens}.

For the first roundtrip, let $P:B \to \UU$ be a Moens family. We have $\Phi(P) \jdeq \omega_P':B \to P\,z$ and $\Psi(\Phi(P)) \jdeq \gl(\omega'): \comma{P\,z}{\omega'} \fibarr B$. We want to give an identification $P =  \gl(\omega'_P)$ in the type of Moens families over $B$ which amounts to a fiberwise equivalence $\widetilde{P} \simeq_B \comma{P\,z}{\omega_P'}$. We are to define a pair of fibered quasi-inverses:
\[\begin{tikzcd}
	{\widetilde{P}} && {P\,z \downarrow \omega'} \\
	& B
	\arrow[two heads, from=1-1, to=2-2]
	\arrow[two heads, from=1-3, to=2-2]
	\arrow[""{name=0, anchor=center, inner sep=0}, "\varphi"{description}, curve={height=-12pt}, from=1-1, to=1-3]
	\arrow[""{name=1, anchor=center, inner sep=0}, "\psi"{description}, curve={height=-6pt}, from=1-3, to=1-1]
	\arrow["\simeq"{description}, Rightarrow, draw=none, from=0, to=1]
\end{tikzcd}\]

We introduce the following notation. For $b:B$, $e:P\,b$, consider the following canonical square:
\[\begin{tikzcd}
	{\widetilde{P}} && e && {(!_b)_!e} \\
	&& {\zeta_b} && {\omega'(b)} \\
	B && b && z
	\arrow["{!_b}", dashed, from=3-3, to=3-5]
	\arrow["{\nu_b}"', from=2-3, to=2-5, cocart]
	\arrow[from=1-1, to=3-1, two heads]
	\arrow["{!_e^b}"', dashed, from=1-3, to=2-3]
	\arrow["{\kappa_e}", from=1-3, to=1-5, cocart]
	\arrow["{\mu_e}", dashed, from=1-5, to=2-5]
\end{tikzcd}\]
We denote by $!_e^b:e \to \zeta_b$ the terminal map of $e$ in $P\,b$, and by $\nu_b: \zeta_b \cocartarr_{!_b} \omega'(b)$ the cocartesian lift of $!_b$ \wrt~$\zeta_b:P\,b$.

By $\kappa_e$, we denote the cocartesian lift of $!_b$ \wrt~$e:P\,b$. We abbreviate by $\mu_e:(!_b)_!e \to \omega'(b)$ the filler
\[ \mu_e \defeq \tyfill_{\kappa_e}(\nu_e \circ !_e^b).\]
Then, we define 
\[ \varphi:\prod_{b:B} P\,b \to \comma{P\,z}{\omega'\,b}, \quad \varphi_b(e)\defeq \mu_{e}: (!_b)_!e \vertarr_z \omega'(b)  \]
and
\[ \psi:\prod_{b:B}\comma{P\,z}{\omega'\,b} \to P\,b, \quad \psi_b(f:e\vertarr_z \omega'(b)) \defeq \zeta_b \times_{\omega'(b)} e : P\,b. \]
For the first part of the round trip, we take $e:P\,b$ which gets mapped to $\varphi_b(e) = \mu_e: (!_b)_!e \vertarr_z \omega'(b)$. Computing the pullback $\psi_b(\mu_e)$ recovers $e$ by~\Cref{prop:ext-sums}(\ref{it:ext-sums-lawvere}) (or (\ref{it:ext-sums-ext})).
The reverse direction is established as follows. Starting with a (vertical) arrow $f:e \vertarr_z \omega'(b)$, we consider the dependent pullback, with a pasted identity of arrows:
\[\begin{tikzcd}
	&& {(!_b)_!e} \\
	{e'} && e \\
	{\zeta_b} && {\omega'(b)}
	\arrow[squiggly, from=2-1, to=3-1, "!_e^b", swap]
	\arrow["{\nu_b}"', from=3-1, to=3-3, cocart]
	\arrow[from=2-1, to=2-3, cocart]
	\arrow[squiggly, from=2-3, to=3-3, "f"]
	\arrow["\lrcorner"{anchor=center, pos=0.125}, draw=none, from=2-1, to=3-3]
	\arrow[from=2-1, to=1-3, cocart]
	\arrow[Rightarrow, no head, from=1-3, to=2-3]
\end{tikzcd}\]
Then, the arrow $\psi_b(\varphi_b(e))$ is given by the composite
\[ (!_b)_!e  \xlongequal{~~} e \xlongrightarrow{f} \omega'(b) \]
which can be identified with $f$. In sum, we have proven $\Psi \circ \Phi = \id$. We are left with the other direction.

Let $C$ be a lex Rezk type and $F:B \to C$ a lex functor. Then $\Psi(f) = \gl(F): \comma{C}{F} \fibarr B$. The section chosing the terminal elements is given by
\[ \zeta \defeq \lambda b.\angled{b,F\,b, \id_{F\,b}}:\prod_{b:B} \sum_{c:C} c \to F\,b.\]
Since $F$ is lex we have an identification $F(z) = y$ where $y:C$ is terminal. Then the terminal transport functor of the fibration $\gl(F)$ yields
\[ \omega' \defeq \pair{F\,b}{F(!_z):F\,b \to y}: B \to \sum_{c:C} c \to y. \]
We have $\Phi(\Psi(F)) = \omega'_{\St_B(\gl(F))}$, and since $y$ is terminal, we can identify this map with $F$. 
\end{proof}

%% file: two-sid-intro.tex
Two-sided cartesian families are type families $P:A \to B \to \UU$ which fibrationally are presented by \emph{spans}
\[\begin{tikzcd}
	& E \\
	A && B
	\arrow["\xi"', two heads, from=1-2, to=2-1]
	\arrow["\pi", two heads, from=1-2, to=2-3]
\end{tikzcd}\]
where $\xi$ is cocartesian, $\pi$ is cartesian, and some compatibility conditions between the two respective liftings are satisfied. An instructive example is given by the ``hom span'' $\partial_1: A \leftarrow A^{\Delta^1} \to A:\partial_0$ of a Rezk type $A$, and from ensuing properties one also obtains comma spans $g:C \leftarrow A \to B:f$.\footnote{In fact, these are even \emph{discrete} two-sided fibrations,~\cf~\cite[Section~8.6]{RS17}, \cite[Section~7.2]{RV21}, \cite[Theorem~2.3.3]{LorRieCatFib}.} Semantically, two-sided families correspond to \emph{categorical $\inftyone$-distributors},~\ie~bifunctors $A^{\Op} \times B \to \Cat$ into the $\inftyone$-category of small $\inftyone$-categories.\footnote{Even though this cannot be expressed in our theory yet,~\cf~\Cref{sec:outlook}.} The significance for $\infty$-cosmos theory is that the discrete variant,~\ie~the $\inftyone$-distributors or \emph{modules}, form a \emph{virtual equipment}~\cite{CruShuMulticat}, a rich double-categorical structure that presents the formal $\infty$-category theory of an $\infty$-cosmos. The Model Independence theorem stats that a biequivalence between $\infty$-cosmoses lifts to a biequivalence of the associated virtual equipments.\footnote{Note the parallel to axiomatic homotopy theory where a Quillen equivalence between ``homotopy theories'' presented through model categories lifts to an equivalence of their associated homotopy categories.} In this thesis, however we will deal with the categorical two-sided case. Namely, we will provide a structured analysis, leading up to characterizations and closure properties generalizing the one-sided case. This follows the thread of of~\cite[Section~7.1]{RV21}, but with a more explicit accounts of various (auxiliary) notions of fibered (or \emph{sliced}) fibrations, owed to the lack of categorical universes in the present theory. Our treatise nevertheless often times make use of techniques from ``formal'' category theory, by reasoning about the various conditions in terms of statements about (fibered) adjunctions, and their closure properties. We view this as a fruitful pratical effect of the $\infty$-cosmological philosophy on the synthetic theory formulated in simplicial type theory.

Our treatise ends with a two-sided Yoneda Lemma, and a (very brief) note on discrete two-sided families.

%% file: relcocart-fib.tex
\subsection{Sliced cocartesian families}

\begin{definition}[Sliced cocartesian families]
Let $B$ be a Rezk type.  A \emph{sliced cocartesian family over $B$} is given by the following data:
	\begin{itemize}
		\item an isoinner family $P: B \to \UU$,
		\item an isoinner family $K: \totalty{P} \to \UU$,
		\item and, writing $Q \defeq \Sigma_P K$, a witness for the proposition\footnote{\cf~\Cref{prop:cocart-lifts-unique-in-isoinner-fams}}
		\[ \prod_{b:B} \prod_{\substack{f:\Delta^1 \to P\,b \\ x:Q(b,f0)}} \sum_{\substack{x':Q(b,f1) \\ k:x \to^Q_{b,f} x'}} \isCocartArr^Q_f(k).\]
	\end{itemize}
	We call $K$ a \emph{cocartesian family sliced over $B$ with base $P$}, and denote the ensuing cocartesian lifts as
	\[ K_!^b(f,x) \defeq K_!(b,f,x):x \cocartarr^K_{\pair{b}{f}} f_!^{K,b}\,x.\]
\end{definition}

Perhaps more familiarly, in fibrational terms, a \emph{cocartesian fibration sliced over $B$ with base $\pi$} is given by a fibered functor $\varphi: \xi \to_B \pi$, where $\xi: F \fibarr B$ and $\pi:E \fibarr B$ are isoinner fibrations over $B$, visualized through
\[\begin{tikzcd}
	F && E \\
	& B
	\arrow["\varphi", from=1-1, to=1-3]
	\arrow["\xi"', two heads, from=1-1, to=2-2]
	\arrow["\pi", two heads, from=1-3, to=2-2]
\end{tikzcd}\]
moreover satisfying the analogous lifting property: any $\pi$-vertical arrow has a $\varphi$-cocartesian lift.

As previously with ordinary cocartesian families, we will often bring in the fibrational viewpoint and reason diagrammatically.

This is a type-theoretic formulation of what, more generally in $\infty$-cosmos theory, defines for any $\infty$-cosmos $\mathcal K$ and an object $B \in \mathcal K$ a cocartesian family in the slice-$\infty$-cosmos $\mathcal K/B$. This is captured internally by the following theorem, which shows that the above condition precisely amounts to the sliced version of the familiar LARI condition for cocartesian families.


\begin{theorem}[Characterization of sliced cocartesian families]\label{thm:sl-cocart-fam-char}
	Given a Rezk type $B$, let $P:B \to \UU$ and $K: \totalty{P} \to \UU$ be isoinner families. We write $Q \jdeq \Sigma_{P}K:B \to \UU$, and denote
	\[ \pi \defeq \Un_B(P) : E \defeq \totalty{P} \fibarr B, \quad \xi \defeq \Un_B(Q) : F \defeq \totalty{Q} \fibarr B, \quad \varphi \defeq \Un_E(K): F \to E: \]
\[\begin{tikzcd}
	F && E && {\widetilde{Q} \simeq \widetilde{K}} && {\widetilde{P}} \\
	& B & {} && {} & B
	\arrow["\varphi", from=1-1, to=1-3]
	\arrow["\xi"', two heads, from=1-1, to=2-2]
	\arrow["\pi", two heads, from=1-3, to=2-2]
	\arrow["{\pi_K}", from=1-5, to=1-7]
	\arrow["{\pi_Q \jdeq \pi_{\Sigma_P\,K}}"{description}, two heads, from=1-5, to=2-6]
	\arrow["{\pi_P}", two heads, from=1-7, to=2-6]
	\arrow[squiggly, tail reversed, from=1-3, to=1-5]
\end{tikzcd}\]
	Then the following are equivalent propositions:
	\begin{enumerate}
		\item\label{it:sl-cocart-fam-char-i} The family $K: E \to \UU$ is a cocartesian family sliced over $B$.
		\item\label{it:sl-cocart-fam-char-ii} The sliced Leibniz cotensor~$i_0 \cotens_B \varphi: \VertArr_{\xi}(F) \to_B \relcomma{B}{\varphi}{E}$ has a fibered LARI:
		\[\begin{tikzcd}
			{\VertArr_{\xi}(F) } &&  {\varphi \downarrow_B E} \\
			\\
			& B
			\arrow[""{name=0, anchor=center, inner sep=0}, "{i_0 \widehat{\pitchfork}_B \varphi}"', curve={height=6pt}, from=1-1, to=1-3]
			\arrow[two heads, from=1-1, to=3-2]
			\arrow[two heads, from=1-3, to=3-2]
			\arrow[""{name=1, anchor=center, inner sep=0}, "{\chi_B}"', curve={height=6pt}, dashed, from=1-3, to=1-1]
			\arrow["\dashv"{anchor=center, rotate=-90}, draw=none, from=1, to=0]
		\end{tikzcd}\]
		\item\label{it:sl-cocart-fam-char-iii} The fibered inclusion map~$\iota_\varphi: F \to_E \relcomma{B}{\varphi}{E}$ has a fibered left adjoint:
	\[\begin{tikzcd}
		F &&&& {\varphi \downarrow_B E} \\
		&& E \\
		\\
		&& B
		\arrow["\varphi"{description}, from=1-1, to=2-3]
		\arrow["{\partial_1}"{description}, from=1-5, to=2-3]
		\arrow["\pi"{description}, two heads, from=2-3, to=4-3]
		\arrow["\xi"{description}, two heads, from=1-1, to=4-3]
		\arrow["{\partial_1'}"{description}, two heads, from=1-5, to=4-3]
		\arrow[""{name=0, anchor=center, inner sep=0}, "{\iota_\varphi}"{description}, from=1-1, to=1-5]
		\arrow[""{name=1, anchor=center, inner sep=0}, "{\tau_\varphi}"{description}, curve={height=18pt}, dashed, from=1-5, to=1-1]
		\arrow["\dashv"{anchor=center, rotate=-91}, draw=none, from=1, to=0]
	\end{tikzcd}\]
	\end{enumerate}

\end{theorem}

\begin{proof}
\begin{description}
	\item[$\ref{it:sl-cocart-fam-char-ii} \implies \ref{it:sl-cocart-fam-char-i}$:]
	We abbreviate $r\defeq i_0 \widehat{\pitchfork}_B \varphi$. After the usual fibrant replacement, we can identify it as the fiberwise map with components
	\[ r_b \big( f:e \to_{P\,b} e', k:x \to^Q_f x' \big) \defeq \langle e, f, x \rangle \jdeq \langle \partial_0 \,f, f, \partial_0 k \rangle \]
	for $b:B$.
	
	Assume, the stated fibered LARI condition is satisfied.
	First we note that the invertible unit, for every $b:B$, exhibits $\chi_{B,b}$ as a (strict) section of $r_b$, \ie~given $e:P\,b$, $f:\comma{e}{P\,b}$ and $x:Q(b,e)$, we can assume $\chi_{B,b}(e,f,x):x \to^Q_f x'$ for some $x':Q(b,\partial_1\,f)$.
	We have a fibered equivalence
	\[  \prod_{\substack{b,b':B \\ u:b \to_B b'}} \prod_{\substack{e,e':P\,b \\ f:e \to_{P\,b} e' \\ x:Q(b,e)}} \prod_{\substack{d,d':P(b') \\ g:d \to_{P\,b'} d' \\ y:Q(b',d), y':Q(b',d') \\ m:y \to^Q_g y'}} ( \chi_B(e,f,x) \to_u \langle g,m \rangle) \stackrel{\simeq}{\longrightarrow} ( \langle e,f,x \rangle \to_u \langle d, g,y \rangle).  \]
	Just as in the second part of the proof of \Cref{thm:cocart-fams-intl-char}, by specializing to the case that $m= \id_{x''}$ and $g=\id_{e''}$ for some $e'':P(b')$ and $x'':Q(b',e'')$, we find that the lift $\chi_{B,b}(e,f,x)$ is a $K$-cocartesian lift of the \emph{$P$-vertical} arrow $f:e \to_{P\,b} e'$.
	
	\item[$\ref{it:sl-cocart-fam-char-i} \implies \ref{it:sl-cocart-fam-char-ii}$:]
	On the other hand, suppose that $K$-cocartesian lifts of all $P$-cocartesian maps exist, \wrt~to a given initial vertex. Accordingly, we define $\chi_{B,b}(e,f,x)\defeq K^b_!(f,x) : x \cocartarr_f^K x'$. Again, analogously to the first part of the proof, we define a pair of maps
	\[\begin{tikzcd}
		{\left( \chi_{B,b}(e,f,x) \to_u \langle g,m\rangle\right)} && {\left( \langle e,f,x\rangle \to_u \langle d,g,y\rangle\right)}
		\arrow[""{name=0, anchor=center, inner sep=0}, "\Phi", shift left=2, from=1-1, to=1-3]
		\arrow[""{name=1, anchor=center, inner sep=0}, "\Psi", shift left=2, from=1-3, to=1-1]
		\arrow["\simeq"{description}, shorten <=1pt, shorten >=1pt, Rightarrow, from=0, to=1]
	\end{tikzcd}\]
	by
	\begin{align*}
	 & \Phi(h:e \to_u^P d, h': e' \to^P_u d', k: x \to^K_h y, k': x' \to^K_{h'} y') \defeq \langle h, h', k \rangle  \\
	& \Psi(h:e \to_u^P d, h': e' \to^P_u d', k:x \to_h^K y) \defeq \langle h, h', k, \tyfill^K_{\chi_{B,b}(e,f,x)}(m \circ k) \rangle.
	\end{align*}
	Due to cocartesianness of $\chi_{B,b}(e,f,x)$ these are quasi-inverse to each other. In particular, the components of $\Phi$ are defined by applying the right adjoint $i_0 \cotens_B \varphi$. For the unit of the adjunction, we take reflexivity, and taken together this defines a fibered LARI adjunction.
	
	\item[$\ref{it:sl-cocart-fam-char-i} \implies \ref{it:sl-cocart-fam-char-iii}$:] The fiberwise map $\iota_\varphi: F \to_E \relcomma{B}{\varphi}{E}$ is given by
	\[ \iota_\varphi(b,e,x)\defeq \angled{b,e,e,\id_e,x}. \]
	Because of the preconditions we can define the candidate fibered left adjoint $\tau_\varphi: \relcomma{B}{\varphi}{E} \to_E F$ by
	\[ \tau_\varphi(b,e',e,f:e \to_{P\,b} e', x:Q(b,e)) \defeq \angled{b,e',f^Q_!(x):Q(b,e')}, \]
	as we would expect analogously to~\Cref{thm:cocartfams-via-transp}.
	To obtain a fibered adjunction as desired, recalling~\Cref{thm:char-fib-adj},~\Cref{it:fib-ladj-sliced}, we want to define a family of equivalences
	\[ \Phi: \prod_{\substack{b:B \\ e':P\,b}} \prod_{\substack{e:P\,b, x:K\,b\,e\\ f:e \to_{P\,b} e'}} \prod_{x':K\,b\,e'} \Big( \underbrace{\tau_\varphi(e,f,x)}_{\jdeq f_!x} \longrightarrow_{K(b,e')} x' \Big) \stackrel{\simeq}{\to} \Big( x \longrightarrow^{Q}_{b,f} x'\Big) : \Psi, \]
	generalizing~\Cref{thm:cocartfams-via-transp}, by\footnote{Note that for the codomain of the equivalence we have identified the type of morphisms $\Big(\angled{e,f,x} \longrightarrow \angled{e',\id_{e'}, x'} \Big)$ (in the fiber $\sum_{g:\comma{P\,b}{e'}} Q(b,\partial_0\,g)$) with simply $(x \to^{b^*Q}_f x')$.}
	\[ \Phi_{b,e'}\big( m:f_!\,x \to_{Q\,b\,e'} x'\big) \jdeq m \circ Q^b_!(f,e), \quad \Psi_{b,e'}\big( k:x \to_f^{b^*Q} x\big) \defeq \tyfill_{Q^b_!(f,x)}^\varphi(k),\]
	see~\Cref{fig:sl-cocart-transp} for an illustration. By \emph{$Q$-cocartesianness} of the lifts of $P$-vertical arrows in $K$ one can show---analogously to the proof of~\Cref{thm:cocartfams-via-transp}---that the maps are quasi-inverse to one another.
	\item[$\ref{it:sl-cocart-fam-char-iii} \implies \ref{it:sl-cocart-fam-char-i}$:] By assumption, there exists a fibered functor $\tau_\varphi: \relcomma{B}{\varphi}{E} \to_E F$ and a fibered natural transformation
	\[ \eta:\big( \id_{\relcomma{B}{\varphi}{E}}\Rightarrow^K_E \iota_\varphi \circ \tau_\varphi \big) \simeq \prod_{\substack{b:B \\e:P\,b}} \prod_{\substack{d:P\,b \\ f:d \to_{P\,b}e \\ x:Q(b,d)}} \angled{d,f,x} \longrightarrow^{\relcomma{B}{K}{E}}_{\pair{b}{e}} \angled{e,\id_e,f_!x},  \]
	where we write $f_!x$ for the respective component, for the sake of foreshadowing.
	Here,
	\[ \relcomma{B}{K}{E} \defeq \lambda b,e.\sum_{\substack{d:P\,b \\ f:d \to_{P\,b}e}} Q(b,d):E \to \UU \]
	is the straightening of $\partial_1: \relcomma{B}{\varphi}{E} \fibarr E$. Note that there is an equivalence
	\[ \Big( \angled{d,f,x} \to_{(\relcomma{B}{K}{E})(b,e)} \angled{e,\id_e,f_!\,x}  \Big) \simeq \Big( x \to^{b^*Q}_{\pair{b}{f}} f_!\,x\Big) ,\]
	as illustrated by:
	\[\begin{tikzcd}
		F & x && {f_!\,x} \\
		& d && e \\
		E & e && e \\
		B && b
		\arrow["f"', from=2-2, to=3-2]
		\arrow[Rightarrow, no head, from=3-2, to=3-4]
		\arrow["f", from=2-2, to=2-4]
		\arrow[Rightarrow, no head, from=2-4, to=3-4]
		\arrow["{\eta_x}", from=1-2, to=1-4, cocart]
		\arrow[Rightarrow, dotted, no head, from=1-2, to=2-2]
		\arrow[Rightarrow, dotted, no head, from=1-4, to=2-4]
		\arrow[Rightarrow, dotted, no head, from=3-2, to=4-3]
		\arrow[Rightarrow, dotted, no head, from=3-4, to=4-3]
		\arrow[two heads, from=3-1, to=4-1]
		\arrow[two heads, from=1-1, to=3-1]
	\end{tikzcd}\]
 	Furthermore, by the assumption, the induced transposing map\footnote{again, identifying $\angled{d,f,x} \longrightarrow_{\pair{v}{g}} \angled{e',\id_{e'}, x'}$ with $x \to^K_{\pair{v}{gf}} x'$} is a family of equivalences:
	\begin{align*}
		& \Phi : \prod_{\substack{b,b':B \\e:P\,b \\ e':P\,b'}} \prod_{\substack{v:b \to_B b' \\ g:e \to^P_v e'}} \prod_{\substack{d:P\,b \\ f:d \to_{P\,b} e \\ x:Q(b,d)}} \prod_{x':Q(b',e')} \Big( f_!\,x \longrightarrow_{\pair{v}{g}}^K x' \Big) \stackrel{\simeq}{\to} \Big( x \longrightarrow^K_{\pair{v}{gf}} x' \Big), \\
		 & \Phi\big(m: f_!\,x \to^K_{\pair{v}{g}} x' \big) \defeq \big( m \circ \eta_x: x \to_{Q\,b\,e}  f_!\,x \to^K_{\pair{v}{g}} x' \big)
	\end{align*}
	Now, $\Phi$ being a fiberwise equivalence means the proposition
	\[ \prod_{\substack{b,b':B \\e:P\,b \\ e':P\,b'}} \prod_{\substack{v:b \to_B b' \\ g:e \to^P_v e'}} \prod_{\substack{d:P\,b \\ f:d \to_{P\,b} e \\ x:Q(b,d)}} \prod_{x':Q(b',e')} \prod_{k: x \to^K_{\angled{v,gf}} x'} \isContr \Big( \sum_{m:f_!\,x \to^K_{\pair{v}{g}} x'} m \circ \eta_x = k \Big) \]
	is satisfied, cf.~\Cref{fig:sl-vert-fill}. This exhibits $\eta_x: x \to^K_{\pair{b}{f}} \tau_\varphi(f,x)$ as $K$-cocartesian lift of the $P$-vertical arrow $f:d \to_{P\,b} e$ (starting at $x:Q(b,d)$), as claimed.
\end{description}
\end{proof}

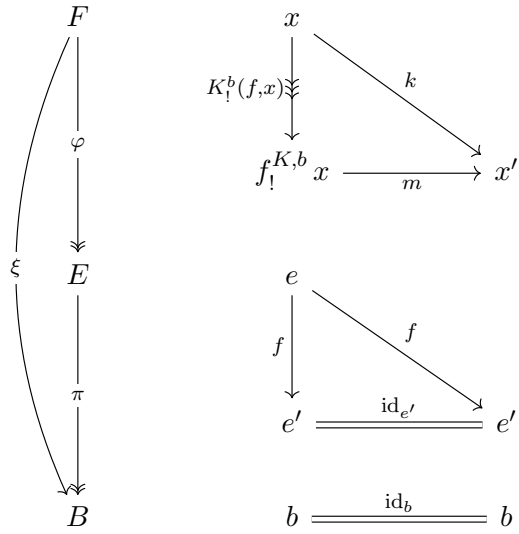
\begin{figure}
	\centering
\[\begin{tikzcd}
	F && x \\
	\\
	&& {f_!^{K,b}\,x} && {x'} \\
	E && e \\
	\\
	&& {e'} && {e'} \\
	B && b && b
	\arrow["k", from=1-3, to=3-5]
	\arrow["f"', from=4-3, to=6-3]
	\arrow["{\id_{e'}}", Rightarrow, no head, from=6-3, to=6-5]
	\arrow["f", from=4-3, to=6-5]
	\arrow["\pi"{description}, two heads, from=4-1, to=7-1]
	\arrow["\varphi"{description}, two heads, from=1-1, to=4-1]
	\arrow["\xi"{description}, curve={height=30pt}, from=1-1, to=7-1]
	\arrow["{K^b_!(f,x)}"', from=1-3, to=3-3, cocart]
	\arrow["m"', from=3-3, to=3-5]
	\arrow["{\id_b}", Rightarrow, no head, from=7-3, to=7-5]
\end{tikzcd}\]
	\caption{Cocartesian transport for sliced cocartesian families}
	\label{fig:sl-cocart-transp}
\end{figure}

\begin{figure}
	\[\begin{tikzcd}
		&& x \\
		F && {f_!\,x} && {x'} \\
		&& d \\
		E && e && {e'} \\
		B && b && {b'}
		\arrow["{\eta_x}"', from=1-3, to=2-3, cocart]
		\arrow["{\exists!\, m}"', dashed, from=2-3, to=2-5]
		\arrow["{\forall\,k}", from=1-3, to=2-5]
		\arrow["f"', from=3-3, to=4-3]
		\arrow["g"', from=4-3, to=4-5]
		\arrow["gf", from=3-3, to=4-5]
		\arrow["v", from=5-3, to=5-5]
		\arrow[from=4-1, to=5-1, two heads]
		\arrow[from=2-1, to=4-1, two heads]
	\end{tikzcd}\]
	\caption{Cocartesian filling in sliced cocartesian families}
	\label{fig:sl-vert-fill}
\end{figure}
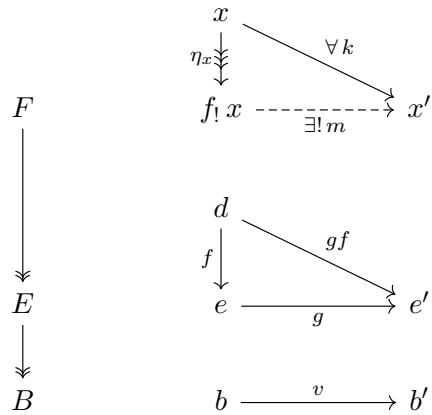

\begin{figure}
		\[\begin{tikzcd}
			x & {x'} & x & {x'} \\
			{g'_!\,x} & {f_!\,x'} &&& {g'_!\,x} & {f_!\,x'} \\
			e & {e'} & e & {e'} & e & {e'} \\
			{e''} & {e'''} & {e''} & {e'''} & {e''} & {e'''} \\
			b & {b'} & b & {b'} & b & {b'} \\
			{\VertArr_\xi(F)} && {\relcomma{B}{\varphi}{E}} && {} & F
			\arrow["v", from=5-3, to=5-4]
			\arrow["{f'}"', from=4-3, to=4-4]
			\arrow["{g'}"', from=3-3, to=4-3]
			\arrow["f", from=3-4, to=4-4]
			\arrow["v", from=5-5, to=5-6]
			\arrow["{g'}"', from=3-5, to=4-5]
			\arrow["{f'}"', from=4-5, to=4-6]
			\arrow["g", from=3-5, to=3-6]
			\arrow["f", from=3-6, to=4-6]
			\arrow["k", from=1-3, to=1-4]
			\arrow["{k'}", dashed, from=2-5, to=2-6]
			\arrow[Rightarrow, dotted, no head, from=1-3, to=3-3]
			\arrow[Rightarrow, dotted, no head, from=1-4, to=3-4]
			\arrow["g", from=3-3, to=3-4]
			\arrow["{\tau_B}", curve={height=-12pt}, maps to, from=3-4, to=3-5]
			\arrow["g", from=3-1, to=3-2]
			\arrow["{g'}"', from=3-1, to=4-1]
			\arrow["f", from=3-2, to=4-2]
			\arrow["{f'}"', from=4-1, to=4-2]
			\arrow["v", from=5-1, to=5-2]
			\arrow["{\chi_B}"', curve={height=12pt}, maps to, from=3-3, to=3-2]
			\arrow["{k'}", dashed, from=2-1, to=2-2]
			\arrow["B", two heads, from=6-3, to=6-1]
			\arrow["{m'}"', from=1-1, to=2-1, cocart]
			\arrow["k", from=1-1, to=1-2]
			\arrow["m", from=1-2, to=2-2, cocart]
			\arrow[curve={height=-18pt}, Rightarrow, dotted, no head, from=2-6, to=4-6]
			\arrow[curve={height=18pt}, Rightarrow, dotted, no head, from=2-5, to=4-5]
			\arrow["E"', two heads, from=6-3, to=6-6]
		\end{tikzcd}\]
	\caption{Action on arrows of lifting and transport of sliced cocartesian families}
	\label{fig:actn-lift-transp-sl-cocart}
\end{figure}
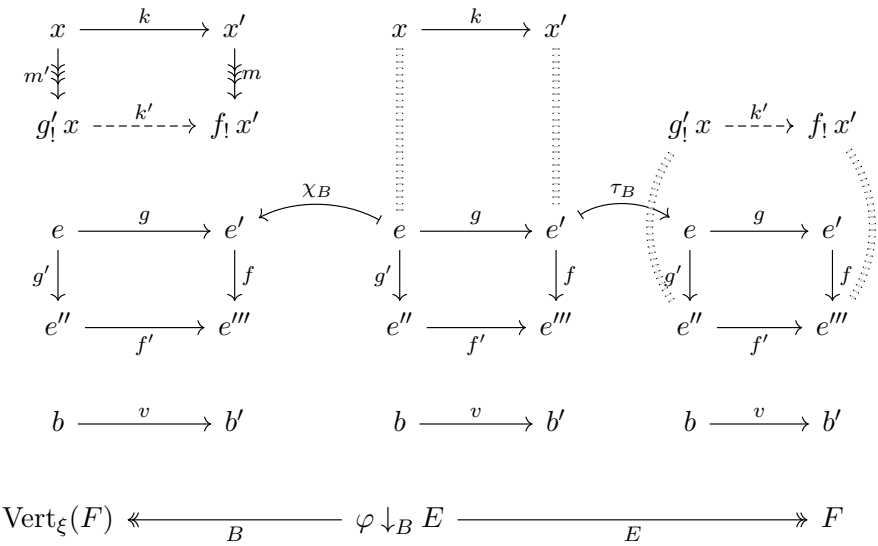

\begin{remark}[Sliced cocartesian families: actions on arrows of the induced functors]\label{rem:sl-cocart-fam-actn-on-arrs}
	For future reference, we record here the actions on arrows of the induced (fibered) lifting and transport functors as established in the proof of~\Cref{thm:sl-cocart-fam-char}. Informally, these can be described as follows. An arrow in $\relcomma{B}{\varphi}{E}$ is given by a dependent square $\sigma$ in $P$ whose vertical edges $f \defeq \lambda t.\sigma(1,t)$, $g' \defeq \lambda t.\sigma(0,t)$ are $P$-vertical, and, moreover, an arrow $k$ in $Q$ over the edge $\pair{t}{0}$. The lifting functor $\chi_B$ maps this to the ensuing dependent square in $Q$, produced by adding the $Q$-cocartesian lifts of the $P$-vertical arrows and the induced filling edge. The functor $\tau_B$ has yields only the filling edge of this square.
	
	Formally, this reads as\footnote{suppressing the repeated data in the lower layers}
	\begin{align}
		\chi_{B,v}(\angled{g,f,g',f'}, k:x \to_g x') & \defeq \angled{k,m:x' \cartarr_{f}^Q f_!\,x', m': x \cartarr_{f}^Q g'_!\,x, k': g_!'x \to_{f'}^Q f_!\,x},\label{eq:lift-sl-cocart}\\
		\tau_{B,f'}(\angled{g,f,g',f'}, k:x \to_g x') & \defeq (k':g_!\,x \to_{f'}^Q f_!\,x'). \label{eq:trans-sl-cocart}
	\end{align}
\end{remark}

The following proposition reflects the known fact that, given a cocartesian functor between cocartesian fibrations, it is a sliced cocartesian fibration if and only if it is a cocartesian fibration in the usual sense.
\begin{proposition}
	Let $B$ be a Rezk type. Assume $\xi:F \fibarr B$ and $\pi:E \fibarr B$ are cocartesian fibrations, and $\varphi:F \to_B E$ is a cocartesian functor:
	\[\begin{tikzcd}
		F && E \\
		& B
		\arrow["\varphi", from=1-1, to=1-3]
		\arrow["\xi"', two heads, from=1-1, to=2-2]
		\arrow["\pi", two heads, from=1-3, to=2-2]
	\end{tikzcd}\]
	Then $\varphi:F \to_B E$ is a cocartesian fibration sliced over $B$ if and only if it is a cocartesian fibration $F \fibarr E$ in the usual sense.
\end{proposition}

\begin{proof}
	In case $\varphi: F \to E$ is a cocartesian fibration it is also a cocartesian fibration sliced over $B$ since it automatically satisfies the weaker existence condition for lifts.
	
	For the converse, we fibrantly replace the given diagram, considering the straightenings $P \defeq \St_B(\pi): B \to \UU$, $Q \defeq \St_B(\pi): B \to \UU$, $K \defeq  \St_E(\varphi): E \to \UU$.
	
	We assume $K$ to be a cocartesian family sliced over $B$, and want to show that it is also a cocartesian family in the usual sense. For an illustration of what follows, \cf~\Cref{fig:abs-from-sliced-cocart-fibs}. Consider an arrow $\pair{u:b \to_B b'}{f:e \to_u^\pi e'}$ in $E$ together with a point $x:K(b,e)\jdeq Q(b,e)$. First, consider the $P$-cocartesian lift of $u:b\to_B b'$ \wrt~$e:P\,b$, given by~$g \defeq P_!(u,e) : e \cocartarr^P_{u} u^P_!\,x$. This induces a filler $h \defeq \tyfill^P_g(f): u^P_!\,x \to_{P\,b'} e'$ that is in particular vertical. Since $K:E \to \UU$ is a sliced cocartesian family we have a lift \wrt~to the \emph{$Q$-cocartesian} transport of the point $x:Q(b,e)$, namely an arrow $m \defeq K_!(h,u_!^Q\,x): u_!^Q\,x \to^K_{\pair{u}{f}} x'$ to some point $x':Q(b,e')$. But by assumption, $m$ (together with its $\varphi$-image $h$) is also a $Q$-cocartesian arrow, hence so is the composite
	\[ k \defeq m \circ Q_!(u,x) : x \to^K_{g \circ \varphi(Q_!(u,x))} x'.\]
	The functor $\varphi$ being cocartesian means $\varphi(Q_!(u,x))$ is identified with $P_!(u,e) \jdeq g$, so up to homotopy, the dependent arrow $k:x \to^Q_u x'$ lies over the composite $h \circ g = f$---hence we can assume it does so strictly. Now, $k$ (together with its projection $f$) being a $Q$-cocartesian arrow means it is in particular $K$-cocartesian (cf.~\Cref{fig:abs-from-sliced-cocart-fibs} for illustration).
\end{proof}

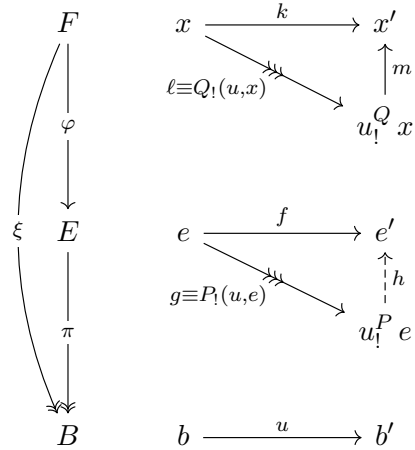
\begin{figure}
	\centering
	\[\begin{tikzcd}
		F & x && {x'} \\
		&&& {u_!^Q\,x} \\
		E & e && {e'} \\
		&&& {u^P_!\,e} \\
		B & b && {b'}
		\arrow["k", from=1-2, to=1-4]
		\arrow["m"', from=2-4, to=1-4]
		\arrow["f", from=3-2, to=3-4]
		\arrow["{g \jdeq P_!(u,e)}"', from=3-2, to=4-4, cocart]
		\arrow["h"', dashed, from=4-4, to=3-4]
		\arrow["u", from=5-2, to=5-4]
		\arrow["\varphi"{description}, from=1-1, to=3-1]
		\arrow["\pi"{description}, two heads, from=3-1, to=5-1]
		\arrow["{\ell \jdeq Q_!(u,x)}"', from=1-2, to=2-4, cocart]
		\arrow["\xi"{description}, curve={height=24pt}, two heads, from=1-1, to=5-1]
	\end{tikzcd}\]
\caption{Absolute from sliced cocartesian fibrations}\label{fig:abs-from-sliced-cocart-fibs}
\end{figure}

\begin{proposition}\label{prop:sliced-comma-is-cocart}
Consider a cospan $\psi:F \to_B G \leftarrow_B E: \varphi$ of fibered functors between isoinner fibrations over a Rezk type $B$, giving rise to the sliced comma type:
\[\begin{tikzcd}
	{\varphi \downarrow_B \psi} &&&& E \\
	\\
	F &&&& G & {} \\
	\\
	&& B
	\arrow[from=1-1, to=3-1]
	\arrow[two heads, from=3-1, to=5-3]
	\arrow[two heads, from=3-5, to=5-3]
	\arrow[from=1-1, to=1-5]
	\arrow["\varphi", from=1-5, to=3-5]
	\arrow[two heads, from=1-1, to=5-3]
	\arrow[two heads, from=1-5, to=5-3]
	\arrow[shorten <=35pt, shorten >=35pt, Rightarrow, from=1-5, to=3-1]
	\arrow["\psi"', from=3-1, to=3-5, crossing over]
\end{tikzcd}\]
The codomain projection
\[ \partial_1: \relcomma{B}{\varphi}{\psi} \fibarr F \]
is a cocartesian fibration.
\end{proposition}

\begin{proof}
Consider the families $R:B \to \UU$, $P:G \to \UU$, and $Q:G \to \UU$ associated to $G \fibarr B$, $\varphi: E \fibarr G$, and $\psi: F \fibarr G$, resp. By fibrant replacement we have
\[ \partial_1: \relcomma{B}{\varphi}{\psi} \simeq \sum_{b:B} \sum_{\stackrel{x,x':R\,b}{k:x \to^R_u x'}} P\,b\,x \times Q\, b\,x' \fibarr F \simeq \sum_{b:B} \sum_{x':R\,b} Q\,b\,x'. \]
For an arrow in $F$, given by data $u:b \to_B b'$, $m:x' \to^R_u x''$, $f:d \to^Q_m d'$, we posit the cocartesian lift w.r.t.~the starting vertex~$\angled{b,k:x \to_{P\,b} x', d,e}$ to be the ``tautological extension''
\[ \angled{\pair{\id_x}{m}: k \rightrightarrows^R_{u} mk, \id_e: e =_{P\,x} e,f:d \to^Q_m d'} \]
as illustrated:
\[\begin{tikzcd}
	{\varphi \downarrow_B \psi} && e && e \\
	&& d && {d'} &&&& d && {d'} \\
	&& x && x & {} && {} \\
	&& {x'} && {x''} &&&& {x'} && {x''} \\
	B && b && {b'} &&& F & b && {b'}
	\arrow[Rightarrow, no head, from=1-3, to=1-5]
	\arrow["f", from=2-3, to=2-5]
	\arrow["k"', from=3-3, to=4-3]
	\arrow["m", from=4-3, to=4-5]
	\arrow[Rightarrow, no head, from=3-3, to=3-5]
	\arrow["mk", from=3-5, to=4-5, swap]
	\arrow["u", from=5-3, to=5-5]
	\arrow[two heads, from=1-1, to=5-1]
	\arrow["{\partial_1}", two heads, from=3-6, to=3-8]
	\arrow["u", from=5-9, to=5-11]
	\arrow["m", from=4-9, to=4-11]
	\arrow["f", from=2-9, to=2-11]
	\arrow[curve={height=18pt}, Rightarrow, dotted, no head, from=2-3, to=4-3]
	\arrow[curve={height=-18pt}, Rightarrow, dotted, no head, from=2-5, to=4-5]
	\arrow[curve={height=18pt}, Rightarrow, dashed, no head, from=1-3, to=3-3]
	\arrow[curve={height=-18pt}, Rightarrow, dashed, no head, from=1-5, to=3-5]
	\arrow[Rightarrow, dotted, no head, from=2-9, to=4-9]
	\arrow[Rightarrow, dotted, no head, from=2-11, to=4-11]
\end{tikzcd}\]
In the picture the right hand side indicates the action of the projection $\partial_1 : \relcomma{B}{\varphi}{\psi} \fibarr F$
That this arrow in the sliced comma $\relcomma{B}{\varphi}{\psi}$ is in fact cocartesian is seen as follows. A postcomposing arrow in $F$ consists of data $v:b' \to b''$, $\ell: x'' \to_v^R x'''$, $g:d' \to^Q_\ell d''$. A dependent arrow in $\relcomma{B}{\varphi}{\psi}$ over the composite arrow in $F$ is given by
\[ \angled{\pair{\ell'}{\ell m}: k \rightrightarrows^R_{vu} k', r:e \to^P_{\ell'} e', g \circ f: d \to^Q_{\ell \circ m} d''}\]
where $\ell':x \to_v^R y$, $k':y \to_{R\,b''} x'''$:
\[\begin{tikzcd}
	{\varphi \downarrow_B \psi} && e && e && {e'} \\
	&& d && {d'} && {d''} \\
	&& x && x && y & {} \\
	&& {x'} && {x''} && {x'''} \\
	B && b && {b'} && {b''}
	\arrow[Rightarrow, no head, from=1-3, to=1-5]
	\arrow["f", from=2-3, to=2-5]
	\arrow["k"', from=3-3, to=4-3]
	\arrow["m", from=4-3, to=4-5]
	\arrow[Rightarrow, no head, from=3-3, to=3-5]
	\arrow["mk", from=3-5, to=4-5]
	\arrow["u", from=5-3, to=5-5]
	\arrow[two heads, from=1-1, to=5-1]
	\arrow["{\ell'}", from=3-5, to=3-7]
	\arrow["\ell", from=4-5, to=4-7]
	\arrow["r", from=1-5, to=1-7]
	\arrow["g", from=2-5, to=2-7]
	\arrow["{k'}", from=3-7, to=4-7]
	\arrow["v", from=5-5, to=5-7]
\end{tikzcd}\]
The filler is constructed just by repeating the missing data $\ell':x \to y$, $r:e \to e'$ which also shows uniqueness up to contractibility~\wrt~the given data.
\end{proof}

\begin{corollary}
	For a cospan of maps between Rezk types $g: C \to A \leftarrow B:f$, the codomain projection $\partial_1: \comma{f}{g} \fibarr C$ is a cocartesian fibration.
\end{corollary}

\subsection{Cocartesian families in cartesian families}

The following proposition will be of relevance for the characterization of two-sided cartesian fibrations in~\Cref{thm:char-two-sid}.

\begin{proposition}
	\label{prop:char-cocart-fib-in-cart-fib}
Let $\varphi:F \to_B E$ be a cartesian functor between cartesian fibrations, as well as a \emph{cocartesian fibration} sliced over $B$:
\[\begin{tikzcd}
	F && E \\
	& B
	\arrow["\varphi", from=1-1, to=1-3]
	\arrow["\xi"', two heads, from=1-1, to=2-2]
	\arrow["\pi", two heads, from=1-3, to=2-2]
\end{tikzcd}\]
Then the following are equivalent:
\begin{enumerate}
	\item\label{it:gen-lift-comm-i}
	The sliced cocartesian lifting map, \ie~the fibered LARI
	\[\begin{tikzcd}
		{\mathrm{Vert}_\xi(F)} && {\varphi \downarrow_B E } \\
		& B
		\arrow["{\overline{\xi}}"', two heads, from=1-1, to=2-2]
		\arrow["{\overline{\varphi}}", two heads, from=1-3, to=2-2]
		\arrow[""{name=0, anchor=center, inner sep=0}, "{\chi_B}"', curve={height=12pt}, dotted, from=1-3, to=1-1]
		\arrow[""{name=1, anchor=center, inner sep=0}, "{i_0 \cotens \pi}"', curve={height=6pt}, from=1-1, to=1-3]
		\arrow["\dashv"{anchor=center, rotate=-88}, draw=none, from=0, to=1]
	\end{tikzcd}\]
	is a \emph{cartesian functor} between cartesian fibrations over $B$
	\item\label{it:gen-lift-comm-ii} 
	The sliced cocartesian transport map, \ie~the fibered left adjoint
	\[\begin{tikzcd}
		F && {\varphi \downarrow_B E} \\
		& E \\
		& B
		\arrow["\varphi"{description}, from=1-1, to=2-2]
		\arrow[""{name=0, anchor=center, inner sep=0}, "{\iota_B}"{description}, curve={height=6pt}, from=1-1, to=1-3]
		\arrow["{\partial_1}"{description}, from=1-3, to=2-2]
		\arrow["\xi"{description}, two heads, from=1-1, to=3-2]
		\arrow["{\overline{\varphi}}"{description}, two heads, from=1-3, to=3-2]
		\arrow["\pi"{description}, two heads, from=2-2, to=3-2]
		\arrow[""{name=1, anchor=center, inner sep=0}, "{\tau_B}"{description}, curve={height=12pt}, dashed, from=1-3, to=1-1]
		\arrow["\dashv"{anchor=center, rotate=-93}, draw=none, from=1, to=0]
	\end{tikzcd}\]
	is a cartesian functor (from $\overline{\varphi}$ to $\xi$).
	\item\label{it:gen-lift-comm-iii} For all elements $b,b':B$, arrows $v:b' \to_B v$, vertical arrows $f:e' \to_{P\,b}e$ and $x:Q(b,e')$, let us make the following abbreviations:\footnote{Note that all the cocartesian lifts exist because they are over $P$-vertical arrows.}
	\begin{align}\label{eq:abbrv-cocart-in-cart-gen}
		g 	& \defeq P^*(v,e'): v^*(b'e') \cartarr^P_{v} e', & f' & \defeq P^*(v,e'): v^*(b,e) \cartarr^P_{v} e, \\
		g ' & \defeq \cartFill_{f'}^P(fg), & k & \defeq Q^*(g,x): x \cartarr^Q_g g^*x, \\
		k' & \defeq \cocartFill^Q_{m'}(mk): g'_!g^*x \to^Q_{f'} f_!\,x, & k'' & \defeq Q^*(f',f_!\,x): (f')^*f_!\,x \cartarr^Q_{f'} f_!x, \\
		m 	& \defeq Q_!(f,x): x \cocartarr^Q_f f_!\,x. & m' & \defeq Q_!(g',g^*x): g^*x \cocartarr^Q_{g'} g_!'g^*x, \\
		m'' & \defeq \cartFill^Q_{k''}(mk): g^*x \to (f')^*f_!\,x. && 
	\end{align}
	Then there is a homotopy $r$ such that:
	\[\begin{tikzcd}
		{(f')^*f_!\,x} && {f_!\,x} \\
		& {g'_!g^*\,x}
		\arrow["{k'}", from=1-1, to=1-3]
		\arrow["r"', Rightarrow, no head, from=1-1, to=2-2]
		\arrow["{k''}"', from=2-2, to=1-3, cart]
	\end{tikzcd}\]
	\item\label{it:gen-lift-comm-iv} With the notation from~\Cref{it:gen-lift-comm-iii} there is a homotopy $r$ such that:
	\[\begin{tikzcd}
		{g^*x\,} && {g_!'g^*\,x} \\
		& {(f')^*f_!\,x}
		\arrow["{m''}", from=1-1, to=1-3, cocart]
		\arrow["{m'}"', from=1-1, to=2-2]
		\arrow["r"', Rightarrow, no head, from=2-2, to=1-3]
	\end{tikzcd}\]
\end{enumerate}

\end{proposition}

\begin{proof}
We prove the equivalence of these four conditions by first explicating~\Cref{it:gen-lift-comm-i}. We will readily see that it is equivalent to either of the three remaining condition.

Recall the action of the fibered lifting $\chi_B$ and transport functor $\tau_B$, resp, \Cref{thm:sl-cocart-fam-char,rem:sl-cocart-fam-actn-on-arrs}. In the first case, assume we have an identity
\begin{align}\label{eq:cart-sl-transp} 
	\chi_B(\overline{\varphi}^*(v,\pair{f}{x})) = \overline{\xi}^*(v,\chi_B(f,x))
\end{align}
for all $v:b'\to_B b$, $f:e' \to_{P\,b} e$, $x:Q(b,e')$. Consider the abbreivations from~(\ref{eq:abbrv-cocart-in-cart-gen}).
Specifically, the case for $\chi_B$ will involve
\begin{align}\label{eq:abbrv-cocart-in-cart-chi}
	m' & \jdeq Q_!(g',g^*x): g^*x \cocartarr^Q_{g'} g_!'g^*x, & & k' \jdeq \cocartFill^Q_{m'}(mk): g'_!g^*x \to^Q_{f'} f_!\,x,
\end{align}
whereas for $\tau_B$ we will need:
\begin{align}\label{eq:abbrv-cocart-in-tau}
	k'' & \jdeq Q^*(f',f_!\,x): (f')^*f_!\,x \cartarr^Q_{f'} f_!x, & & m'' \jdeq \cartFill^Q_{k''}(mk): g^*x \to (f')^*f_!\,x.	
\end{align}
As detailed in~\cite[Subsection~5.2.3]{BW21}, lifts of co-/cartesian families are fiberwise. Hence, we find for the left hand side in~(\ref{eq:cart-sl-transp}):
\begin{align}
		& \chi_B(\overline{\varphi}^*(v,\pair{f}{x})) = \chi_B(v,\angled{g,f,g',f'},k:g^*x \cartarr_{g}^Q x) \\
	= & \angled{v,\angled{g,f,g',f'},\angled{k,m,m',k'}} \label{eq:abbrv-cocart-in-cart-lhs}
\end{align}
where, and for the right hand side
\begin{align}
		& \overline{\xi}^*(v,\chi_B(f,x)) = \overline{\xi}^*(v,\angled{g,f,g',f'},\angled{k,m,m'',k''}) = \\
	=	& \angled{v,\angled{g,f,g',f'},\angled{k,m,m'',k''}}. \label{eq:abbrv-cocart-in-cart-rhs}
\end{align}
Recall that a functor being cartesian is a proposition.
A path between (\ref{eq:abbrv-cocart-in-cart-lhs}) and (\ref{eq:abbrv-cocart-in-cart-rhs}) amounts to an isomorphism $r:(f')^*f_!\,x =_{v^*(b,e)} g'_!g^*x$ such that the entire following diagram commutes:
\[\begin{tikzcd}
	{g^*\,x} && x \\
	{(f')^*f_!\,x} && {f_!\,x} \\
	& {g_!'g^*x}
	\arrow["{m'}"', from=1-1, to=2-1, cocart]
	\arrow["k"{pos=0.6}, from=1-1, to=1-3]
	\arrow["m", from=1-3, to=2-3, cocart]
	\arrow["r"', Rightarrow, no head, from=2-1, to=3-2]
	\arrow["{k''}"', from=3-2, to=2-3, cart]
	\arrow["{k'}"'{pos=0.6}, from=2-1, to=2-3]
	\arrow["{m''}"{pos=0.3}, from=1-1, to=3-2, crossing over]
\end{tikzcd}\]
More generally, it can be shown that there exists a filler $r:(f')^*f_!\,x \to_{v^*(b,e)} g_!'g^*x$ s.t.~$m'' = rm'$ and~$k' = k''r$. Hence, this propositional condition is equivalent to this induced arrow being invertible. But moreover, we can see by universality that this is equivalent to the existence of either identification $m'=m''$ or $k'=k''$. In particular, the action by the transport functor $\tau_B$ yields just the latter. Hence, all the four conditions claimed are equivalent.
\end{proof}

\begin{definition}[Cocartesian fibrations in cartesian fibrations]\label{def:cocart-fibs-in-cart-fibs}
	Given $\varphi: \xi \to_B \pi$ as in~\Cref{prop:char-cocart-fib-in-cart-fib}, if in addition $\varphi$ is also a cartesian functor, we call $\varphi$ a \emph{cocartesian fibration in cartesian fibrations}.\footnote{For the official naming we prefer the fibrational variant since it is closer to its semantic counterpart, but of course by the typal Grothendieck construction there exists an indexed variant as well.}
\end{definition}

\begin{proposition}[Closure of sliced cocartesian fibrations under product]\label{prop:clos-sl-cocart-fib-prod}
	For a small indexing type $I:\UU$, let $B:I \to \UU$ be a family of Rezk types. Let $P:\prod_{i:I}(B_i \to \UU)$ be a family and $K: \prod_{i:I} (\totalty{P}_i \to \UU)$ be another family. We define $Q\defeq \lambda i.\Sigma_{P_i}K_i.\prod_{i:I} (B_i \to \UU)$. For every $i:I$, we denote
	\[ \pi_i \defeq \Un_{B_i}(P_i) : E_i \defeq \totalty{P_i} \fibarr B_i, \xi_i \defeq \Un_{B_i}(Q_i) : F_i \defeq \totalty{Q_i} \fibarr B_i, \varphi \defeq \Un_{E_i}(K_i): F_i \to E_i \]
	giving rise to diagrams:
	\[\begin{tikzcd}
		{F_i} && {E_i} \\
		& {B_i} & {} && {}
		\arrow["{\xi_i}"', two heads, from=1-1, to=2-2]
		\arrow["{\pi_i}", two heads, from=1-3, to=2-2]
		\arrow["{\varphi_i}", from=1-1, to=1-3]
	\end{tikzcd}\]
	If each $\varphi_i$ is a sliced cocartesian fibration, then so is the product:
	\[\begin{tikzcd}
		{\prod_{i:I} F_i} && {\prod_{i:I} E_i} \\
		& {\prod_{i:I} B_i} & {} && {}
		\arrow["{\prod_{i:I} \xi_i}"', two heads, from=1-1, to=2-2]
		\arrow["{\prod_{i:I} \pi_i}", two heads, from=1-3, to=2-2]
		\arrow["{\prod_{i:I} \varphi_i}", from=1-1, to=1-3]
	\end{tikzcd}\]
	Moreover, if $\varphi_i$ is a cocartesian fibration in cartesian fibrations in the sense of~\ref{def:cocart-fibs-in-cart-fibs}, then so is $\prod_{i:I} \varphi_i$.
\end{proposition}

\begin{proof}
	Since dependent products commute with sliced commas by~\Cref{prop:dep-prod-comm-sl-commas} we find
	\begin{align}
		\prod_{i:I} \VertArr_{\xi_i}(F_i) & \equiv \prod_{i:I} \relcomma{B_i}{F_i}{F_i} \equiv   \relcomma{\prod_{i:I} B_i}{\big(\prod_{i:I}F_i\big)}{\big(\prod_{i:I}F_i\big)} \\
		\prod_{i:I} \relcomma{B_i}{\varphi_i}{E_i} & \equiv \relcomma{\prod_{i:I} B_i}{\big(\prod_{i:I} \varphi_i\big)}{\big(\prod_{i:I} E_i\big)}.
	\end{align}
	Since sliced LARIs are preserved by the dependent product~\Cref{prop:fib-lari-pres-by-sl-prod} we obtain an induced fibered LARI between these commas, establishing that $\prod_{i:I} \varphi$ is sliced cocartesian by~\Cref{thm:sl-cocart-fam-char}.
	
	Moreover, since cartesian fibrations and co-/cartesian functors are preserved under the dependent product the analogous closure statement for the stronger notion of cocartesian fibration in cartesian fibrations follows readily.
\end{proof}

\begin{proposition}[Closure of sliced cocartesian fibrations under composition]\label{prop:clos-sl-cocart-fib-comp}
	Let $P,Q,R:B \to \UU$ be isoinner families over a Rezk type $B$ with unstraightenings
	\[ \xi \defeq \Un_B(Q) : F \fibarr B, \pi \defeq \Un_B(P) : E \fibarr B, \rho \defeq \Un_B(R) : G \fibarr B. \]
	Furthermore, assume we have fibered functors $\varphi:F \to_B E$, $\psi: E \to_B G$ that are sliced cocartesian over $B$. Then, so is their composite $\kappa:F \to_B G$:
	\[\begin{tikzcd}
		F && E && G \\
		\\
		&& B
		\arrow["\varphi", from=1-1, to=1-3]
		\arrow["\pi"{description}, two heads, from=1-3, to=3-3]
		\arrow["\psi", from=1-3, to=1-5]
		\arrow["\xi"{description}, two heads, from=1-1, to=3-3]
		\arrow["\rho"{description}, two heads, from=1-5, to=3-3]
		\arrow["\kappa"{description}, curve={height=-24pt}, from=1-1, to=1-5]
	\end{tikzcd}\]
	Moreover, if $\varphi$ and $\psi$ are cocartesian fibrations in cartesian fibrations, then so is $\psi \circ \varphi$.
\end{proposition}

\begin{proof}
	This proof works analogously to the one for the absolute situation in~\cite[Proposition~2.3.7]{BW21}. First of all, we fibrantly replace the objects at play (with some abbreviation for the term declariations):
	\begin{align*}
		G & \equiv \sum_{b:B} Q\,b, & E & \equiv \sum_{\substack{b:B \\ x:R\,b}} P\,b\,x, \\
		F & \equiv \sum_{\substack{b:B \\ x:R\,b \\ e:P\,b\,x}} Q\,b\,\,x\,e, & \relcomma{B}{\psi}{G} & \equiv \sum_{b,x,e} \sum_{x':R\,b} (x \to_{R\,b} x'), \\
		\relcomma{B}{\varphi \psi}{G} & \equiv \sum_{b,x,x',e} \sum_{d:Q\,b\,x\,e} (x \to_{R\,b} x'), & \relcomma{B}{\varphi}{E} & \equiv \sum_{b,x,x',e,d} \sum_{u:x \to_{R\,b} x'} (e \to_u^P e'), \\
		\VertArr_\pi(E) & \equiv \sum_{b,x,x',u,e,e'} (e \to_u^P e'). & 	& & 
	\end{align*}
	We are to construct from the given fibered LARI adjunctions
	\[\begin{tikzcd}
		{\VertArr(F)} && {\relcomma{B}{\varphi}{E}} && {\VertArr(E)} && {\relcomma{B}{\psi}{G}} \\
		& B &&&& B
		\arrow[""{name=0, anchor=center, inner sep=0}, "r"{description}, curve={height=12pt}, from=1-1, to=1-3]
		\arrow[""{name=1, anchor=center, inner sep=0}, "{r'}"{description}, curve={height=12pt}, from=1-5, to=1-7]
		\arrow[""{name=2, anchor=center, inner sep=0}, "\ell"{description}, curve={height=18pt}, dotted, from=1-3, to=1-1]
		\arrow[""{name=3, anchor=center, inner sep=0}, "{\ell'}"{description}, curve={height=18pt}, dotted, from=1-7, to=1-5]
		\arrow[two heads, from=1-1, to=2-2]
		\arrow[two heads, from=1-3, to=2-2]
		\arrow[two heads, from=1-5, to=2-6]
		\arrow[two heads, from=1-7, to=2-6]
		\arrow["\dashv"{anchor=center, rotate=-90}, draw=none, from=2, to=0]
		\arrow["\dashv"{anchor=center, rotate=-90}, draw=none, from=3, to=1]
	\end{tikzcd}\]
	a fibered LARI adjunction:
	\[\begin{tikzcd}
		{\VertArr(E)} && {\relcomma{B}{\varphi \psi}{G}} \\
		& B
		\arrow[""{name=0, anchor=center, inner sep=0}, "{r''}"{description}, curve={height=12pt}, from=1-1, to=1-3]
		\arrow[""{name=1, anchor=center, inner sep=0}, "{\ell''}"{description}, curve={height=18pt}, dotted, from=1-3, to=1-1]
		\arrow[two heads, from=1-1, to=2-2]
		\arrow[two heads, from=1-3, to=2-2]
		\arrow["\dashv"{anchor=center, rotate=-90}, draw=none, from=1, to=0]
	\end{tikzcd}\]
	Using the fibrant replacements, indeed we find the diagram analogous to the proof in~\cite[Proposition~2.3.7]{BW21}, all fibered over $B$:
	\[\begin{tikzcd}
		{\VertArr_\xi(F)} \\
		&& {\relcomma{B}{\varphi}{E}} && {\relcomma{B}{\varphi \psi}{G}} && F \\
		&& {\VertArr_\pi(E)} && {\relcomma{B}{\psi}{G}} && E \\
		&&&& B
		\arrow[from=2-7, to=3-7]
		\arrow[from=3-5, to=3-7]
		\arrow[from=2-5, to=3-5]
		\arrow[from=2-5, to=2-7]
		\arrow[""{name=0, anchor=center, inner sep=0}, curve={height=6pt}, from=2-3, to=2-5]
		\arrow[from=2-3, to=3-3]
		\arrow[""{name=1, anchor=center, inner sep=0}, "r"', curve={height=6pt}, from=3-3, to=3-5]
		\arrow[from=3-3, to=4-5]
		\arrow[from=3-5, to=4-5]
		\arrow[from=3-7, to=4-5]
		\arrow["\lrcorner"{anchor=center, pos=0.125}, draw=none, from=2-5, to=3-7]
		\arrow["\lrcorner"{anchor=center, pos=0.125}, draw=none, from=2-3, to=3-5]
		\arrow[curve={height=-18pt}, from=1-1, to=2-7]
		\arrow[curve={height=12pt}, from=1-1, to=3-3]
		\arrow[""{name=2, anchor=center, inner sep=0}, "\ell"{description}, curve={height=12pt}, dotted, from=3-5, to=3-3]
		\arrow[""{name=3, anchor=center, inner sep=0}, curve={height=12pt}, dotted, from=2-5, to=2-3]
		\arrow[""{name=4, anchor=center, inner sep=0}, "{r'}"{description}, curve={height=6pt}, from=1-1, to=2-3]
		\arrow[""{name=5, anchor=center, inner sep=0}, "{\ell'}"{description}, curve={height=12pt}, dotted, from=2-3, to=1-1]
		\arrow["\dashv"{anchor=center, rotate=-90}, draw=none, from=2, to=1]
		\arrow["\dashv"{anchor=center, rotate=-90}, draw=none, from=3, to=0]
		\arrow["\dashv"{anchor=center, rotate=-144}, draw=none, from=5, to=4]
	\end{tikzcd}\]
	As before, the proclaimed fibered LARI arises by pulling back and then composing.
	
	The closure property descends to cocartesian fibrations in cartesian fibrations by closedness under composition of cartesian functors, and pullback stability of fibered cartesian sections by the dual of~\Cref{prop:cocart-sect-pb}.
\end{proof}

\begin{proposition}[Closure of sliced cocartesian fibrations under pullback]\label{prop:clos-sl-cocart-fib-pb}
	Let $\varphi:F \to_B E$ be a sliced cocartesian fibration over a Rezk type $B$. For any map $k:A \to B$ consider the pullback:
	\[\begin{tikzcd}
		{k^*F} && F \\
		& {k^*E} && E \\
		A && B
		\arrow["\xi"{description, pos=0.25}, two heads, from=1-3, to=3-3]
		\arrow["\varphi", curve={height=-6pt}, from=1-3, to=2-4]
		\arrow["\pi"{description}, two heads, from=2-4, to=3-3]
		\arrow[from=1-1, to=1-3]
		\arrow["{k^*\varphi}"{pos=0.8}, curve={height=-6pt}, dashed, from=1-1, to=2-2]
		\arrow["{k^*\xi}"{description, pos=0.3}, two heads, from=1-1, to=3-1]
		\arrow["{k^*\pi}"{description}, two heads, from=2-2, to=3-1]
		\arrow["k"', from=3-1, to=3-3]
		\arrow["\lrcorner"{anchor=center, pos=0.125}, draw=none, from=2-2, to=3-3]
		\arrow["\lrcorner"{anchor=center, pos=0.125}, shift right=2, draw=none, from=1-1, to=3-3]
		\arrow[crossing over, from=2-2, to=2-4]
	\end{tikzcd}\]
	Then the induced fibered functor $k^*\varphi: k^*F \to_A k^*E$ is a sliced cocartesian fibration over $A$.
	
	In particular, the analogous statement is true if $\varphi$ is assumed to be a cocartesian fibration in cartesian fibrations.
\end{proposition}

\begin{proof}
	Since pullback commutes with sliced commas\footnote{as can \eg~be checked by fibrant replacement} we get the following square:
	\[\begin{tikzcd}
		{k^*\VertArr_{k^*\xi}(F)} && {\VertArr_\xi(F)} \\
		& {\relcomma{A}{k^*\varphi}{k^*E}} && {\relcomma{B}{\varphi}{E}} \\
		A && B
		\arrow[two heads, from=1-3, to=3-3]
		\arrow[""{name=0, anchor=center, inner sep=0}, "{\varphi'}", curve={height=-6pt}, from=1-3, to=2-4]
		\arrow[two heads, from=2-4, to=3-3]
		\arrow[from=1-1, to=1-3]
		\arrow["{k^*\varphi'}"{pos=0.8}, curve={height=-6pt}, dashed, from=1-1, to=2-2]
		\arrow[two heads, from=1-1, to=3-1]
		\arrow[two heads, from=2-2, to=3-1]
		\arrow["k"', from=3-1, to=3-3]
		\arrow["\lrcorner"{anchor=center, pos=0.125}, draw=none, from=2-2, to=3-3]
		\arrow["\lrcorner"{anchor=center, pos=0.125}, shift right=2, draw=none, from=1-1, to=3-3]
		\arrow[""{name=1, anchor=center, inner sep=0}, curve={height=-6pt}, dotted, from=2-4, to=1-3]
		\arrow["\dashv"{anchor=center, rotate=12}, draw=none, from=1, to=0]
		\arrow[from=2-2, to=2-4, crossing over]
	\end{tikzcd}\]
	The fibered LARI on the right gets pulled back to define a fibered LARI on the left, as desired, hence the pulled back functor $k^*\varphi$ is sliced cocartesian as well.
	
	If $\varphi:F \to_B E$ is a cartesian functor between cartesian functors, then $\varphi'$ pulls back to define a cartesian functor between cartesian fibrations by (the dual of)~\cite[Proposition~5.3.21]{BW21}, \cf~also \cite[Lemma~5.3.5]{RV21}. In case the fibered LARI is cartesian, the induced fibered LARI is as well, as one can see by the dual of~\Cref{prop:cocart-sect-pb}.

\end{proof}

%% file: two-sided-fibs.tex
\subsection{Two-variable families and bifibers}\label{ssec:two-var}

	For types $A,B: \UU$ consider a family
	\[ P:A \to B \to \UU.\]

	For $a:A$ and $b:B$, the type $P(a,b)$ is called the \emph{bifiber} of $P$ at $a$ and $b$. Fixing one of the elements gives rise to the definitions
	\[ P_b \defeq \lambda a.P(a,b): A \to \UU, \quad P^a \defeq \lambda b.P(a,b): B \to \UU.\]
	The two ``legs'' of the family $P$ are given by the families
	\[ P_B \defeq \lambda a.\sum_{b:B} P_b(a):A \to \UU, \quad  P^A \defeq \lambda b.\sum_{a:B} P^a(b): B \to \UU.\]

We obtain the following version of the typal Grothendieck construction. By transposition we have a chain of fiberwise equivalences:
\[\begin{tikzcd}
	{\sum_{A,B:\UU} A \to B \to \UU} && {\sum_{A,B:\UU} (A \times B) \to \UU} && {\sum_{A,B:\UU} \Fib_\UU(A \times B)} \\
	&& {\UU \times \UU}
	\arrow[from=1-1, to=2-3]
	\arrow[from=1-3, to=2-3]
	\arrow["\simeq", from=1-3, to=1-5]
	\arrow[from=1-5, to=2-3]
	\arrow["\simeq", from=1-1, to=1-3]
\end{tikzcd}\]
Hence two-sided families $P:A \to B \to \UU$ correspond to maps over $\pair{\xi}{\pi}:\widetilde{P} \to A \times B$.

Diagrammatically, this manifests as follows:
\[\begin{tikzcd}
	& {\sum_{\substack{a:A \\ b:B}} P\,a\,b } \\
	& {A \times B} \\
	A && B
	\arrow["\varphi"', two heads, from=1-2, to=2-2]
	\arrow[two heads, from=2-2, to=3-1]
	\arrow[two heads, from=2-2, to=3-3]
	\arrow["{\xi \defeq \Un_A(P_B)}"{description}, curve={height=18pt}, two heads, from=1-2, to=3-1]
	\arrow["{ \Un_B(P^A) \defeq \pi}"{description}, curve={height=-18pt}, two heads, from=1-2, to=3-3]
\end{tikzcd}\]

The fibers at $a:A$ or $b:B$, resp., are obtained as follows
\[\begin{tikzcd}
	{E_b} && E && {E^a} && E \\
	A && {A \times B} && B && {A \times B}
	\arrow["{\xi_b}"', from=1-1, to=2-1]
	\arrow["{\langle \id_A,b\rangle}"', from=2-1, to=2-3]
	\arrow[from=1-1, to=1-3]
	\arrow["{\langle\xi, \pi\rangle}", from=1-3, to=2-3]
	\arrow["{\pi^a}"', from=1-5, to=2-5]
	\arrow["{\langle a,\id_B\rangle}"', from=2-5, to=2-7]
	\arrow[from=1-5, to=1-7]
	\arrow["{\langle\xi, \pi\rangle}", from=1-7, to=2-7]
	\arrow["\lrcorner"{anchor=center, pos=0.125}, draw=none, from=1-5, to=2-7]
	\arrow["\lrcorner"{anchor=center, pos=0.125}, draw=none, from=1-1, to=2-3]
\end{tikzcd}\]
where the projections arise as unstraightenings
\[ \xi_b \defeq \Un_A(P_b): E_b \to A, \quad \pi^a \defeq \Un_B(P^a): E^a \to B. \]
The notation comes from the convention of denoting the components of the projection $E \to A \times B$ conceived as a fibered functor in two different ways
\[\begin{tikzcd}
	E && {A \times B} && E && {A \times B} \\
	& B &&&& A
	\arrow["{\langle \xi,\pi\rangle}", from=1-1, to=1-3]
	\arrow["\pi"', from=1-1, to=2-2]
	\arrow["q", from=1-3, to=2-2]
	\arrow["{\langle \xi,\pi\rangle}", from=1-5, to=1-7]
	\arrow["\xi"', from=1-5, to=2-6]
	\arrow["p", from=1-7, to=2-6]
\end{tikzcd}\]
as
\[ \xi_b \defeq \pair{\xi}{\pi}_b :E_b \to (A\times B)_b \simeq A, \quad  \pi^a \defeq  \pair{\xi}{\pi}^a : E^a \to (A\times B)^a \simeq B. \]

\subsection{Cocartesianness on the left}

\begin{definition}[Cocartesian on the left]
	A two-sided family $P:A \to B \to \UU$ is \emph{cocartesian on the left} if the family $P_B:A \to \UU$ is cocartesian, and every $P_B$-cocartesian arrow in $\totalty{P}$ is $P^A$-vertical.
\end{definition}

\begin{example}[Cocartesian families]\label{ex:cocart-fams-cocart-left}
	A family $P:A \to \UU$ is cocartesian if and only if it is cocartesian on the left, seen as a family $P: (A \to \UU) \simeq (A \to 1 \to \UU)$.
\end{example}

\begin{proposition}[Characterizations of cocartesianness on the left, {\protect\cite[Lemma~7.1.1]{RV21}}]\label{prop:char-fib-cocart-left}
	For a two-sided family $P:A \to B \to \UU$, corresponding to $\pair{\xi}{\pi}:E \to A \times B$, the following are equivalent:
	\begin{enumerate}
		\item\label{it:coc-left-i} The fibered functor
		\[\begin{tikzcd}
			E && {A \times B} \\
			& B
			\arrow["{\langle \xi, \pi \rangle}", two heads, from=1-1, to=1-3]
			\arrow["\pi"', from=1-1, to=2-2]
			\arrow["q", from=1-3, to=2-2]
		\end{tikzcd}\]
		is a cocartesian fibration sliced over $B$.
		\item\label{it:coc-left-ii} The fibered functor
		\[\begin{tikzcd}
			E && {A \times B} \\
			& A
			\arrow["{\langle \xi, \pi \rangle}", from=1-1, to=1-3]
			\arrow["\xi"', two heads, from=1-1, to=2-2]
			\arrow["p", two heads, from=1-3, to=2-2]
		\end{tikzcd}\]
		is a cocartesian functor between cocartesian fibrations.
		\item\label{it:coc-left-iii} The fibered functor given by
		\[ \iota_\xi: E \to_{A \times B} \comma{\xi}{A}, \quad \iota_\xi(a,b,e) \defeq \angled{a,a,\id_a,b,e} \]
		has a fibered left adjoint $\tau_\xi$:
		\[\begin{tikzcd}
			E &&&& {\xi \downarrow A} \\
			&& {A \times B}
			\arrow[""{name=0, anchor=center, inner sep=0}, "{\iota_\xi}"{description}, from=1-1, to=1-5]
			\arrow["{\langle \xi,\pi \rangle}"', from=1-1, to=2-3]
			\arrow["{\langle \partial_1,\pi \circ \partial_0\rangle}", from=1-5, to=2-3]
			\arrow[""{name=1, anchor=center, inner sep=0}, "{\tau_\xi}"{description}, curve={height=18pt}, dashed, from=1-5, to=1-1]
			\arrow["\dashv"{anchor=center, rotate=-90}, draw=none, from=1, to=0]
		\end{tikzcd}\]
		\item\label{it:coc-left-iv} The two-sided family $P:A \to B \to \UU$ is cocartesian on the left.
	\end{enumerate}
\end{proposition}

\begin{proof} We abbreviate $\varphi \defeq \pair{\xi}{\pi}: E \to_B A \times B$.
	\begin{description}
		\item[$\ref{it:coc-left-i} \iff \ref{it:coc-left-iv}$:] The fibered functor $\varphi$ being a cocartesian fibration sliced over $B$ is equivalent to the condition that for all $q$-vertical maps exist a $\varphi$-cocartesian lift (\wrt~to a given initial vertex). By Rezk-completeness, this is equivalent to any arrow $\pair{u:a \to a'}{\id_b}$ having a $\varphi$-cocartesian lift $f$ with prescribed initial vertex $e:P(a,b)$. Note that $f$ is $\pi$-vertical. Projecting away the $B$-component, we obtain that $f$ is $\xi$-cocartesian. 
		\item[$\ref{it:coc-left-ii} \iff \ref{it:coc-left-iv}$:] Since $p:A \times B \fibarr A$ is a cocartesian fibration in any case, the assumption is equivalent to $\xi:E \fibarr A$ being a cocartesian fibration and every $\xi$-cocartesian arrow being mapped to $p$-cocartesian arrows under $\varphi$. But $\varphi$ is just the projection pairing $\pair{\xi}{\pi}$, and $p$-cocartesian arrows are exactly given by arrows whose $B$-component is an identity. Projecting down to $B$ this means exactly that the $\xi$-cocartesian arrows are $\pi$-vertical.
		\item[$\ref{it:coc-left-iii} \implies\ref{it:coc-left-iv}$:] We denote
		\[ \tau_\xi(a',b,u:a\to a',e) \defeq  \angled{a',b,\widehat{\tau}_{\xi,a',b}(u,e)}. \]
		Again, similarly to the considerations in the proof of~\Cref{thm:cocartfams-via-transp}, the unit is a family of arrows $\eta:\prod_{\substack{a':A \\ b:A}} \prod_{\substack{a:A \\ u:a \to_A a' \\ e:P(a,b)}} \angled{a',b,a,u,e} \to \angled{a',b,a',\id_{a'}, \widehat{\tau}_{\xi,a',b}(u,e)}$, illustrated as follows:
		\[\begin{tikzcd}
			e && {\widehat{\tau}_{\xi,a',b}(u,e)} \\
			a && {a'} \\
			{a'} && {a'} \\
			{a'} && {a'} \\
			b && b
			\arrow["{\eta_{a',b,a,u,e}}", from=1-1, to=1-3, cocart]
			\arrow["u"', from=2-1, to=3-1]
			\arrow["u", from=2-1, to=2-3]
			\arrow["{\id_{a'}}", Rightarrow, no head, from=3-1, to=3-3]
			\arrow["{\id_{a'}}", Rightarrow, no head, from=2-3, to=3-3]
			\arrow["{\id_{a'}}", Rightarrow, no head, from=4-1, to=4-3]
			\arrow["{\id_{b}}", Rightarrow, no head, from=5-1, to=5-3]
			\arrow[curve={height=12pt}, Rightarrow, dotted, no head, from=1-1, to=2-1]
			\arrow[curve={height=-12pt}, Rightarrow, dotted, no head, from=1-3, to=2-3]
			\arrow[curve={height=12pt}, Rightarrow, dotted, no head, from=3-1, to=4-1]
			\arrow[curve={height=-12pt}, Rightarrow, dotted, no head, from=3-3, to=4-3]
		\end{tikzcd}\]
		By assumption, the transposing map induced by $\eta$ is an equivalence:
		\begin{align*}
			\Phi_\eta & : \prod_{\substack{a',a'':A \\ b,b':B}} \prod_{\substack{u':a' \to_A a'' \\ v:b \to_B b'}} \prod_{\substack{a:A \\ u:a \to_A a' \\ e:P(a,b)}} \prod_{\substack{a'':A \\ b':B \\ d:P(a'',b')}} \left( \angled{a',b,\widehat{\tau}_{\xi,a',b}(u,e):P(a',b)} \to_{\pair{u'}{v}} \angled{a'',b',d} \right) \\
					\stackrel{\simeq}{\longrightarrow} & \left( \angled{a',b,a,u:a \to_A a', e:P(a,b)} \to_{\pair{u'}{v}}  \angled{a'',b',a'', \id_{a''}, d:P(a'',b')} \right), \\
			\Phi_\eta & \defeq \lambda u',v,e,d,g.\iota_\xi(g) \circ \eta_{u,e} : \big(e \to_{\pair{v}{g}}^P d\big)
		\end{align*}
		After contracting away redundant data, this is equivalent to the proposition
		\[ \prod_{h:e \to^P_{\pair{u'u}{v}} d} \isContr\Big( \sum_{g:\widehat{\tau}_{\xi,a',b}(u,e) \to^P_{\pair{u'}{v}} d} g \circ^P_{\eta_{u,e}} = h \Big), \]
		\cf~\Cref{fig:lift-cocart-left-fib-adj} for illustration. But this precisely means that $\xi:E \fibarr A$ is a cocartesian fibration whose cocartesian lifts all are $\pi$-vertical, namely the components of the fibered unit $\eta$.
		\item[$\ref{it:coc-left-iv} \implies\ref{it:coc-left-iii}$:] We can strictify the diagram as follows, including the fibered left adjoint to-be-defined:
		\[\begin{tikzcd}
			{E \simeq \sum_{\substack{a':A \\ b:B}} P(a',b)} &&&& {\sum_{\substack{a,a':A \\ b:B}} (a \to_A a') \times P(a,b) \simeq \xi \downarrow A} \\
			\\
			&& {A \times B}
			\arrow["{\langle \partial_1, \pi \circ \partial_0\rangle}", two heads, from=1-5, to=3-3]
			\arrow["{\langle \xi,\pi \rangle}"', two heads, from=1-1, to=3-3]
			\arrow[""{name=0, anchor=center, inner sep=0}, "{\iota_\xi}"{description}, curve={height=12pt}, from=1-1, to=1-5]
			\arrow[""{name=1, anchor=center, inner sep=0}, "{\tau_\xi}"', curve={height=6pt}, from=1-5, to=1-1]
			\arrow["\dashv"{anchor=center,rotate=-90}, draw=none, from=1, to=0]
		\end{tikzcd}\]
		The fibered ``inclusion'' map is defined as
		\[ \iota_\xi(a',b,e) \defeq \angled{a',b,a',\id_{a'},e}.\]
		By the preconditions from~\Cref{it:coc-left-iii}, the map $\xi:E \fibarr A$ is a cocartesian fibration, moreover whose cocartesian lifts are all $\pi$-vertical. We let
		\[ \tau_\xi(a',b,a,u:a \to a',e) \defeq \angled{a',b,u_!\,e}.\]
		Now, similarly, as in the proof of~\Cref{thm:cocartfams-via-transp}, we exhibit the fibered adjunction as given by a fiberwise equivalence\footnote{Here, we write $\xi_b \defeq (\id_A \times b)^*\pair{\xi}{\phi}: E_b \simeq \sum_{a:A} P(a,b) \fibarr A$, giving rise to the comma object $\sum_{a:A} (a \to_A a') \times P(a,b)$} 
		\[ 
		\prod_{\substack{a':A \\ b:B}} \prod_{\substack{a:A \\ u:a \to_A a' \\ e:P(a,b)}} \prod_{e':P(a',b)} \big( \tau_{\xi,a',b}(a,u,u_!\,e) \to_{P(a',b)} e'\big) \stackrel{\simeq}{\longrightarrow} \big( \angled{a,u,e} \longrightarrow_{\comma{\xi_b}{A}}  \angled{a',\id_{a'}, e'} \big)
		\]
		as follows: Over a point $\pair{a'}{b}:A \times B$ in the base, fix $\angled{a:A, u:a \to_A a', e:P(a,b)}$, $e':P(a',b)$, and define maps between the transposing hom types, in opposite directions,
		\begin{align*}
		& \Phi(g:u_!e \to_{P(a',b)} e') \defeq g \circ \xi_!(u,e), \\
		& \Psi(u:a \to_A a', u:u \to_{\comma{A}{a}} \id_{a'}, h:e \to_u^P e') \defeq \tyfill_{\xi_!(u,e)}^\xi(h). 
		\end{align*}
	Again, by the universal property of cocartesian lifts it can be checked that the two maps are quasi-inverses. Note that by assumption, the $\xi$-cocartesian lifts are $\pi$-vertical, so everything stays in ``the fibers over $b$'', or, more precisely, in the pullback type $(\id_A \times b)^*E$.

	\end{description}

\begin{figure}
	\centering
	\[\begin{tikzcd}
		e && d \\
		{\widehat{\eta}(u,e)} && d \\
		a && {a''} \\
		{a'} && {a''} \\
		{a'} && {a''} \\
		b && {b'}
		\arrow["{\eta_{u,e}}"', swap, description, from=1-1, to=2-1, cocart]
		\arrow["u"', from=3-1, to=4-1]
		\arrow["{u'u}", from=3-1, to=3-3]
		\arrow["{u'}", from=4-1, to=4-3]
		\arrow["{\id_{a''}}"', Rightarrow, no head, from=3-3, to=4-3]
		\arrow["{u'}", from=5-1, to=5-3]
		\arrow["v", from=6-1, to=6-3]
		\arrow["{\forall \,h}", from=1-1, to=1-3]
		\arrow[Rightarrow, no head, from=1-3, to=2-3]
		\arrow[curve={height=30pt}, Rightarrow, dotted, no head, from=1-1, to=3-1]
		\arrow[curve={height=-30pt}, Rightarrow, dotted, no head, from=1-3, to=3-3]
		\arrow["{\exists! \, g}", from=2-1, to=2-3, dashed]
		\arrow[curve={height=30pt}, Rightarrow, dotted, no head, from=2-1, to=4-1, crossing over]
		\arrow[curve={height=-30pt}, Rightarrow, dotted, no head, from=2-3, to=4-3, crossing over]
	\end{tikzcd}\]
	\caption{Fibered adjunction criterion for cocartesian-on-the-left families}
	\label{fig:lift-cocart-left-fib-adj}
\end{figure}
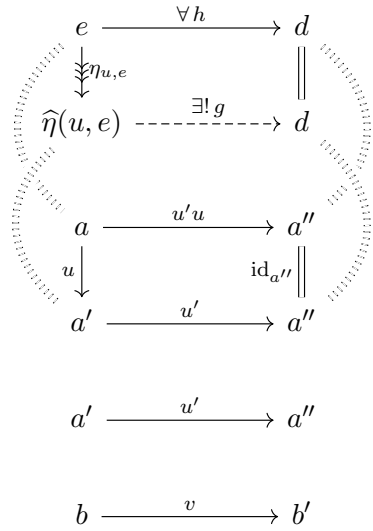

		\begin{figure}
			\centering
			\[\begin{tikzcd}
				&& e && {e'} \\
				E && {u_!'e} \\
				&& {a''} && {a'} \\
				&& a && {a'} \\
				{A \times B} && b && {b'}
				\arrow["{\xi_!(u',e)}"', from=1-3, to=2-3, cocart]
				\arrow["h", from=1-3, to=1-5]
				\arrow["g"', dashed, from=2-3, to=1-5]
				\arrow["{u'}"', from=3-3, to=4-3]
				\arrow["u", from=4-3, to=4-5]
				\arrow["{uu'}", dashed, from=3-3, to=3-5]
				\arrow[Rightarrow, no head, from=3-5, to=4-5]
				\arrow["v", from=5-3, to=5-5]
				\arrow[two heads, from=2-1, to=5-1]
			\end{tikzcd}\]
			\caption{Lifts in a family that is cocartesian on the left}
			\label{fig:lift-cocart-left}
		\end{figure}
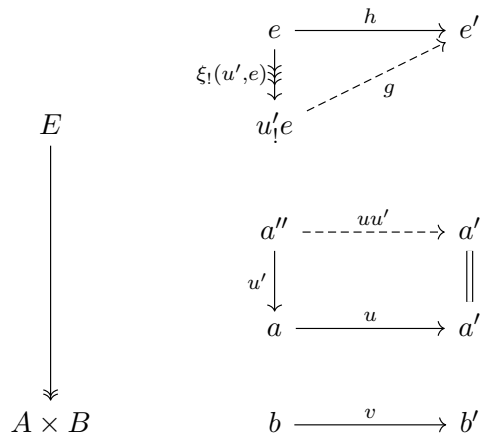
	
\end{proof}

By dualization, one obtains the notion of a two-sided family \emph{cartesian on the right}. As a corollary, we obtain a characterization of the conjunction of both properties.

\begin{corollary}[{\protect\cite[Corollary~7.1.3]{RV21}}]\label{cor:cocart-left-and-cart-right}
	A two-variable family $P:A \to B \to \UU$ is cocartesian on the left and cartesian on the right if and only if the following equivalent propositions are true.
	\begin{enumerate}
		\item The fibered functor
		\[\begin{tikzcd}
			E && {A \times B} \\
			& B
			\arrow["{\langle \xi,\pi \rangle}", two heads, from=1-1, to=1-3]
			\arrow["\pi"', two heads, from=1-1, to=2-2]
			\arrow["q", two heads, from=1-3, to=2-2]
		\end{tikzcd}\]
		is both a cartesian functor between cartesian fibrations and a relative cocartesian fibration over $B$.
		\item The fibered functor
		\[\begin{tikzcd}
			E && {A \times B} \\
			& A
			\arrow["{\langle \xi,\pi \rangle}", two heads, from=1-1, to=1-3]
			\arrow["\xi"', two heads, from=1-1, to=2-2]
			\arrow["p", two heads, from=1-3, to=2-2]
		\end{tikzcd}\]
		is both a cocartesian functor between cocartesian fibrations and a relative cartesian fibration over $B$.
	\end{enumerate}
\end{corollary}

In the case that $P:A \to B \to \UU$ is cocartesian on the left and cartesian on the right, we denote the ensuing lifting operations as follows. Given $a:A$, $b:B$, $e:P\,a\,b$, for arrows $u:a \downarrow A$, $v:B \downarrow b$, there are lifts
\[ P_!(u,e): e \cocartarr_{\pair{u}{b}} u_!\,e, \quad P^*(v,e): v^*\,e \cartarr_{\pair{a}{v}} e,\]
where in the notation we identify elements and identity maps.

The notion of \emph{two-sided cartesian fibration} adds on top a certain compatibility condition between the lifts of different variance. Before defining two-sided cartesian fibrations in~\Cref{ssec:2s-cart-fams}, we first investigate the compatibility condition in the following lemma.

\begin{lemma}[Comparing cartesian and cocartesian transport]\label{lem:comp-cart-cocart-transp}
	For Rezk types $A$ and $B$, let $P: A \to B \to \UU$ be a two-variable family which is cocartesian on the left and cartesian on the right. Denoting its unstraightening as~$\varphi\jdeq \pair{\xi}{\pi}:E \fibarr A \times B$, for arrows $u:a \to_A a'$, $v:b' \to_B ba$ and a point $e:P\,a\,b$ we abbreviate:
	\begin{align*}
		k & \defeq \xi_!(u,e):e \cocartarr^\xi_{\pair{u}{b}} u_!\,e  &  k' & \defeq \pi^*(v,e):v^*\,e \cartarr^\pi_{\pair{a}{v}} e \\
		f & \defeq \xi_!(u,v^*\,e) : v^*\,e \cocartarr^\xi_{\pair{u}{b'}} u_!\,v^*\,e & f' & \defeq \pi^*(v,u_! \,e) : v^*\,u_!\,e \cartarr^\pi_{\pair{a'}{v}} u_!\,e \\
		g & \defeq \tyfill_f^\xi(k'k): u_!\,v^*\,e \to_{\pair{a'}{v}}^\xi u_!\,e & g' & \defeq \tyfill_{f'}^\pi(k'k): v^*\,e \to^\pi_{\pair{u}{b'}} v^*\, u_!\,e
	\end{align*}
	We claim that there is an identification between the following two induced morphisms $h,h':u_!\,v^*\,e \to_{P(a',b')} v^*\,u_!\,e$~ (\cf~\ref{fig:comp-cart-cocart-transport}):
	\[\begin{tikzcd}
		{v^*\,e} && e && {u_!\,e} && {v^*\,e} && e && {u_!\,e} \\
		{u_!\,v^*\,e} &&&& {v^*\,u_!\,e} && {u_!\,v^*\,e} &&&& {v^*\,u_!\,e}
		\arrow["f"', from=1-1, to=2-1, cocart]
		\arrow["{k'}", from=1-1, to=1-3, cart]
		\arrow["k", from=1-3, to=1-5, cocart]
		\arrow["{f'}", from=2-5, to=1-5, cart]
		\arrow["{h \defeq \tyfill^\pi_{f'}(g)}"', dashed, from=2-1, to=2-5]
		\arrow["g"{description}, dashed, from=2-1, to=1-5]
		\arrow["f"', from=1-7, to=2-7, cocart]
		\arrow[from=1-7, to=1-9, cart, "k'"]
		\arrow[from=1-9, to=1-11, cocart, "k"]
		\arrow["{h' \defeq \tyfill^\xi_{f}(g')}"', dashed, from=2-7, to=2-11]
		\arrow["{f'}", from=2-11, to=1-11, cart]
		\arrow["{g'}"{description}, dashed, from=1-7, to=2-11]
	\end{tikzcd}\]
\end{lemma}

\begin{proof}
	It is sufficient to provide an identification $g=f'h'$. For this, it is sufficient---and necessary---to provide a witness for $f'(h'f) = k'k$. But this follows from $f'h=g$, since $gf=k'k$.
\end{proof}

\begin{figure}
	\centering
	\[\begin{tikzcd}
		E && {v^*e} && e && {u_!e} \\
		&&& {u_!v^*e} &&&& {v^*u_!e} \\
		{A \times B} && {\langle a,b'\rangle} && {\langle a,b \rangle} && {\langle a',b\rangle} \\
		{} &&& {\langle a',b' \rangle} &&&& {\langle a',b'\rangle}
		\arrow["{\xi_!(u,v^*e)}"'{pos=0.2}, from=1-3, to=2-4, cocart]
		\arrow["{\pi^*(v,e)}", from=1-3, to=1-5, cart]
		\arrow["{\xi_!(u,e)}", from=1-5, to=1-7, cocart]
		\arrow["{\pi^*(v,u_!e)}"'{pos=0.1}, from=2-8, to=1-7, cart]
		\arrow[dashed, from=2-4, to=2-8]
		\arrow["{\langle u,b' \rangle}"', from=3-3, to=4-4]
		\arrow["{\langle a,v \rangle}", from=3-3, to=3-5]
		\arrow["{\langle u,b \rangle}", from=3-5, to=3-7]
		\arrow[two heads, from=1-1, to=3-1]
		\arrow["{\langle a',v \rangle}"', from=4-8, to=3-7]
		\arrow[Rightarrow, no head, from=4-4, to=4-8]
	\end{tikzcd}\]
	\caption{Comparing cartesian and cocartesian transport}\label{fig:comp-cart-cocart-transport}
\end{figure}

\subsection{Two-sided cartesian families}\label{ssec:2s-cart-fams}

\begin{definition}[Two-sided cartesian families]\label{def:2s-cart}
	Let $P:A \to B \to \UU$ be a (two-sided) family, where $A$ and $B$ are Rezk types. We call $P$ a \emph{two-sided family} if
	\begin{enumerate}
		\item $P$ is cocartesian on the left and cartesian on the right,
		\item\label{it:lifts-commute} and $P$ satisfies the condition that \emph{cocartesian and cartesian lifts commute} : for any $a:A$, $b:B$, $e:P\,a\,b$ and arrows $u:a \to_A a'$, $v:b' \to_B b$, the filler $\kappa:u_!(v^*e) \to v^*(u_!e)$ from \Cref{lem:comp-cart-cocart-transp} is an isomorphism, hence there is an identity $u_!v^*e =_{P(a',b')} v^*u_!e$.
	\end{enumerate}
\end{definition}

\begin{proposition}[Commutation of cocartesian and cartesian lifts]\label{prop:comm-lifts}
	Let $P:A \to B \to \UU$ be a family with both $A$ and $B$ Rezk which is cocartesian on the left and cartesian on the right. We denote by $\xi:F \fibarr A$ and $\pi:E \fibarr B$, resp., the unstraightenings.
	Then cocartesian and cartesian lifts commute if and only if the following property is satisfied:
	Given $u:a \to_A a'$, $v:b' \to_B b$, $e:P\,a\,b$, and a diagram
\[\begin{tikzcd}
	{v^*\,e} && e \\
	d && {u_!\,e}
	\arrow["f"', from=1-1, to=2-1]
	\arrow["{k'\defeq \pi^*(v,e)}", from=1-1, to=1-3, cart]
	\arrow["g"', from=2-1, to=2-3]
	\arrow["{k \defeq \xi_!(u,e)}", from=1-3, to=2-3, cocart]
\end{tikzcd}\]
	where $g$ (and necessarily $k'$) is $\xi$-vertical and $f$ (and necessarily $k$) is $\pi$-vertical. Then $f$ is $\xi$-cocartesian if and only if $g$ is $\pi$-cartesian.
\end{proposition}

\begin{proof}
	In light of~\Cref{lem:comp-cart-cocart-transp}, the commutation condition is equivalent to the type of paths $h:u_!\,v^*\,e =_{P(a',b')} v^*\,u_!\,e$ together with witnesses, necessarily propositional, that the following diagram commutes:
	\[\begin{tikzcd}
		&& {v^*\,e} &&&& e \\
		\\
		&& {u_!\,v^*\,e} &&&& {u_!\,e} \\
		\\
		{v^*\,u_!\,e}
		\arrow["{f \defeq \xi_!(u,v^*\,e)}", from=1-3, to=3-3, cocart]
		\arrow["{g \defeq \tyfill_f^\xi(kk')}", dashed, from=3-3, to=3-7]
		\arrow["{k' \defeq \pi^*(v,e)}", from=1-3, to=1-7, cart]
		\arrow["{k \defeq \xi_!(u,e)}", from=1-7, to=3-7, cocart]
		\arrow["{g' \defeq \tyfill_{f'}^\pi(kk')}"', dashed, from=1-3, to=5-1]
		\arrow["{f' \defeq \pi^*(v,u_!\,e)}"', from=5-1, to=3-7, cart]
		\arrow["h"{description}, Rightarrow, no head, from=5-1, to=3-3]
	\end{tikzcd}\]
	But since these diagrams commute in any case by the assumptions (recall~\Cref{lem:comp-cart-cocart-transp}) said type is equivalent to the proposition that the filler $h$ is an isomorphism:
	\[\begin{tikzcd}
		&& {u_!\,v^*\,e} \\
		{v^*\,e} && {v^*\,u_!\,e} && {u_!\,e}
		\arrow["h"', dashed, from=1-3, to=2-3]
		\arrow["{f'}"', from=2-3, to=2-5, cart]
		\arrow["g", from=1-3, to=2-5]
		\arrow["{g'}"', from=2-1, to=2-3]
		\arrow["f", from=2-1, to=1-3, cocart]
	\end{tikzcd}\]
	Finally due to the commutation of both of the ``completed squares'' above, this is equivalent to the new alternative criterion: up to identification, $f$ is $\xi$-cocartesian if and only if $g$ is $\pi$-cartesian.
\end{proof}

\begin{figure}
	\centering
\[\begin{tikzcd}
	{\mathrm{Vert}_\pi(E)} & {v^*\,e} && e & {E \times_{A \times B} \mathrlap{\mathrm{Vert}_q(A \times B)}} \\
	{} & {v^*\,e'} && {e'} && {v^*e} & e & {} & {} \\
	& a && a && a & a && {} \\
	& {a'} && {a'} && {a'} & {a'} && {} \\
	B & b && {b'} & B & b & {b'}
	\arrow["{\pi^*(v,e)}", from=1-2, to=1-4, cocart]
	\arrow["{\mathrm{fill}}"', dashed, from=1-2, to=2-2]
	\arrow["{\pi^*(v,e')}"', from=2-2, to=2-4, cocart]
	\arrow["f", from=1-4, to=2-4]
	\arrow["u"', from=3-2, to=4-2]
	\arrow[Rightarrow, no head, from=3-2, to=3-4]
	\arrow["u", from=3-4, to=4-4]
	\arrow[Rightarrow, no head, from=4-2, to=4-4]
	\arrow["v", from=5-2, to=5-4]
	\arrow["u"', from=3-6, to=4-6]
	\arrow["\psi"{description}, from=1-5, to=5-5]
	\arrow["{\pi'}"{description}, two heads, from=1-1, to=5-1]
	\arrow["{\pi^*(v,e)}", from=2-6, to=2-7, cocart]
	\arrow[Rightarrow, no head, from=3-6, to=3-7]
	\arrow["u", from=3-7, to=4-7]
	\arrow[Rightarrow, no head, from=4-6, to=4-7]
	\arrow["v", from=5-6, to=5-7]
\end{tikzcd}\]
	\label{fig:lifts-two-sided-fib-lari}\caption{Cartesian lifts in the induced fibrations~$\pi'$~and~$\psi$}
\end{figure}
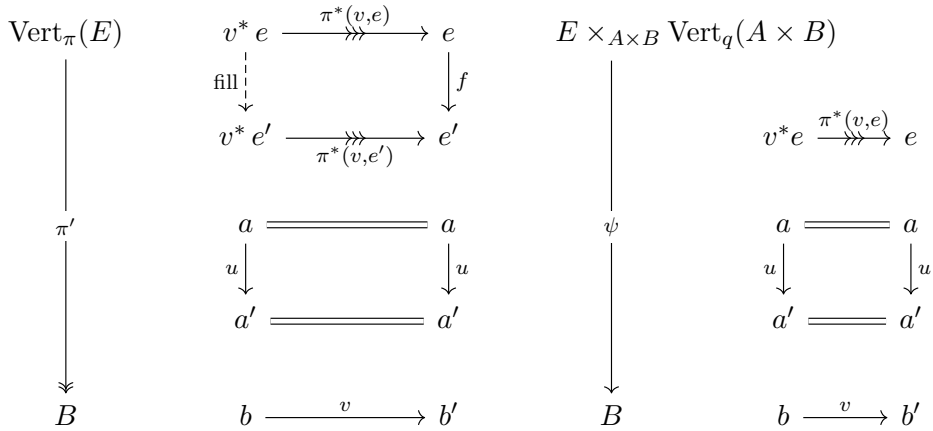

The following theorem finally contains several characterizations of a two-variable family being two-sided cartesian.\footnote{I am indebted to Emily Riehl for helpful explanations and discussions about~\cite[Thm.~7.1.4]{RV21}.} This consists in several sliced Chevalley/fibered adjoint criteria and a criterion formulated directly on the level of two-variable families.

\begin{theorem}[Characterizations of two-sided families, {\protect\cite[Thm.~7.1.4]{RV21}}]\label{thm:char-two-sid}
For a family $P:A \to B \to \UU$, corresponding to $\varphi \defeq \pair{\xi}{\pi}: E \defeq \sum_{a:A,b:B} P(a,b) \to A \times B$, the following are equivalent:
	\begin{enumerate}
		\item\label{it:char-two-sid-i} The two-variable family $P$ is two-sided.
		\item\label{it:char-two-sid-ii} Considering
		\[\begin{tikzcd}
			E && {A \times B} \\
			& B
			\arrow["\varphi", from=1-1, to=1-3]
			\arrow["\pi"', two heads, from=1-1, to=2-2]
			\arrow["q", two heads, from=1-3, to=2-2]
		\end{tikzcd}\]
		the map $\pi$ 
		is a cartesian fibration, the functor $\varphi$ is cartesian and a \emph{cocartesian} fibration sliced over $B$. Furthermore, the fibered LARI
		\[ \chi_B: \VertArr_q(A \times B) \times_{A\times B} E \to_B \VertArr_\pi(E)\]
		is a \emph{cartesian} functor.
		\item\label{it:char-two-sid-iii'} Considering
		\[\begin{tikzcd}
			E && {A \times B} \\
			& B
			\arrow["\varphi", from=1-1, to=1-3]
			\arrow["\pi"', two heads, from=1-1, to=2-2]
			\arrow["q", two heads, from=1-3, to=2-2]
		\end{tikzcd}\]
		the map $\pi$ 
		is a cartesian fibration, the functor $\varphi$ is cartesian and a \emph{cocartesian} fibration sliced over $B$. Furthermore, the fibered LARI
		\[ \tau_B: \VertArr_q(A \times B) \times_{A\times B} E \to_B E\]
		is a \emph{cartesian} functor.
		\item\label{it:char-two-sid-iii} Considering
		\[\begin{tikzcd}
			E && {A \times B} \\
			& A
			\arrow["\varphi", from=1-1, to=1-3]
			\arrow["\xi"', two heads, from=1-1, to=2-2]
			\arrow["p", two heads, from=1-3, to=2-2]
		\end{tikzcd}\]
	the map $\xi$
	is a cocartesian fibration, the functor $\varphi$ is cocartesian and a \emph{cartesian} fibration sliced over $A$. Furthermore, the fibered RARI
	\[ \chi^A: \VertArr_p(A \times B) \times_{A\times B} E \to_A \VertArr_\xi(E)\]
	is a \emph{cocartesian} functor.
	\item\label{it:char-two-sid-iv}
	The fibered adjoints in the following diagram exist:
\[\begin{tikzcd}
	E &&& {B \downarrow \pi} \\
	\\
	&& {\xi \downarrow A} &&& {\xi \downarrow A \times_E B \downarrow \pi} \\
	\\
	&&& {A\times B}
	\arrow[""{name=0, anchor=center, inner sep=0}, "{\iota^\pi}"{description}, from=1-1, to=1-4]
	\arrow[""{name=1, anchor=center, inner sep=0}, "{\iota_\xi}"{description}, from=1-1, to=3-3]
	\arrow[""{name=3, anchor=center, inner sep=0}, "{\langle \iota_\xi \circ \partial_1',\id \rangle}"{description, pos=0.25}, from=1-4, to=3-6]
	\arrow["{\langle \partial_1, \pi \, \partial_0 \rangle}"{description}, two heads, from=3-3, to=5-4]
	\arrow["{\langle \partial_1, \partial_0\rangle}"{description}, two heads, curve={height=-15pt}, from=3-6, to=5-4]
	\arrow["{\langle \xi\,\partial_1, \partial_0 \rangle}"{description, pos=0.3}, two heads, from=1-4, to=5-4]
	\arrow["\varphi"{description}, shift right=2, curve={height=30pt}, two heads, from=1-1, to=5-4]
	\arrow[""{name=4, anchor=center, inner sep=0}, "{\tau_\xi}"{description}, curve={height=-18pt}, dashed, from=3-3, to=1-1]
	\arrow[""{name=5, anchor=center, inner sep=0}, "\ell"{description}, curve={height=18pt}, dashed, from=3-6, to=1-4]
	\arrow[""{name=7, anchor=center, inner sep=0}, "{\tau^\pi}"{description}, curve={height=18pt}, dashed, from=1-4, to=1-1]
	\arrow["\dashv"{anchor=center, rotate=-91}, draw=none, from=7, to=0]
	\arrow["\dashv"{anchor=center, rotate=-133}, draw=none, from=5, to=3]
	\arrow["\dashv"{anchor=center, rotate=44}, draw=none, from=4, to=1]
	\arrow[""{name=6, anchor=center, inner sep=0}, "r"{description}, curve={height=-24pt}, dashed, from=3-6, to=3-3, crossing over]
	\arrow[""{name=2, anchor=center, inner sep=0}, "{\langle \id, \iota^\pi \circ \partial_0' \rangle}"{description, pos=0.6}, from=3-3, to=3-6, crossing over]
	\arrow["\dashv"{anchor=center, rotate=108}, draw=none, from=6, to=2]
\end{tikzcd}\]
where the pullback type is given by:
\[\begin{tikzcd}
	{\comma{\xi}{A} \times_E \comma{B}{\pi}} && {\comma{\xi}{A}} \\
	{\comma{B}{\pi}} && E
	\arrow[from=1-1, to=2-1]
	\arrow["{\partial_1}"', from=2-1, to=2-3]
	\arrow[from=1-1, to=1-3]
	\arrow["{\partial_0'}", from=1-3, to=2-3]
	\arrow["\lrcorner"{anchor=center, pos=0.125}, draw=none, from=1-1, to=2-3]
\end{tikzcd}\]
Moreover, the mate of the identity $2$-cell defines a fibered isomorphism
\[ \prod_{a:A, b:B} (\tau_\xi \circ r)_{a,b} =_{Q(a,b) \to P(a,b)} (\tau^\pi \circ \ell)_{a,b},\]
where
\[ Q: A \times B \to \UU, \quad Q(a,b) \simeq \comma{a}{A} \times \comma{B}{b} \times P(a,b)\]
is the straightening of the map
\[ F \defeq \comma{\xi}{A} \times_E \comma{B}{\pi} \fibarr A \times B.\]
	\end{enumerate}
\end{theorem}

\begin{proof}

	\begin{description}
		\item[$\ref{it:char-two-sid-i} \iff \ref{it:char-two-sid-ii}$:] By \Cref{cor:cocart-left-and-cart-right}, the map $P: A \to B \to \UU$ is cocartesian on the left and cartesian on the right if and only if $\varphi: E \to_B A \times B$ is both a cartesian functor between cartesian fibrations and a cocartesian fibration sliced over $B$.
		
		In the following, we assume this is satisfied for $P$.
		
		Thus, we are left to show that, under this assumption---$P$ being cocartesian on the left and cartesian on the right---the following holds:
		\begin{align*} 
			& \text{``The fibered LARI $\chi_B: \VertArr_q(A \times B) \times_{A\times B} E \to_B \VertArr_\pi(E)$ is a cocartesian functor.''} \\
		& \iff  \text{``In $P$, cocartesian and cartesian lifts commute.''}
		\end{align*}
		We write $F \defeq \mathrm{Vert}_q(A \times B) \times_{A \times B} E \fibarr B$, so by fibrant replacement, we consider the projection
		\[ \psi : \sum_{b:B} \sum_{a:A} (\comma{a}{A} \times P(a,b)) \simeq F \fibarr B. \]		
		The induced sliced Leibniz cotensor is given by
		\[ \kappa : \VertArr_\pi(E) \to_B F,\quad \kappa_b\big(u:\Delta^1 \to A,f:\prod_{t:\Delta^1} P(u(t),b) \big) \defeq \angled{\partial_0\,u,u, \partial_0\,f}. \]
		It  has a fibered LARI $\mu: F \to_B \VertArr_\pi(E)$ as indicated in: 
		\[\begin{tikzcd}
			{\mathrm{Vert}_\pi(E)} && {\mathrm{Vert}_q(A \times B) \times_{A \times B} E} \\
			\\
			& B
			\arrow[""{name=0, anchor=center, inner sep=0}, "\kappa"{description}, curve={height=12pt}, from=1-1, to=1-3]
			\arrow[""{name=1, anchor=center, inner sep=0}, "\mu"{description}, curve={height=18pt}, dashed, from=1-3, to=1-1]
			\arrow["{\pi'}"{description}, two heads, from=1-1, to=3-2]
			\arrow["\psi"{description}, two heads, from=1-3, to=3-2]
			\arrow["\dashv"{anchor=center, rotate=-94}, draw=none, from=1, to=0]
		\end{tikzcd}\]
	
		\begin{figure}
		\centering
		\[\begin{tikzcd}
			{e''} & e & {e''} & e && {e'} & e & {e'} & e \\
			{e'''} & {e'} &&&&&& {u_!\,e'} & {u_!\,e} \\
			{a''} & a & {a''} & a && {a''} & a & {a''} & a \\
			{a'''} & {a'} & {a'''} & {a'} && {a'''} & {a'} & {a'''} & {a'} \\
			{b'} & b & {b'} & b && {b'} & b & {b'} & b
			\arrow["g", from=1-1, to=1-2]
			\arrow["{f'}"', from=1-1, to=2-1]
			\arrow["f", from=1-2, to=2-2]
			\arrow["{g'}"', from=2-1, to=2-2]
			\arrow["m", from=3-1, to=3-2]
			\arrow["{u'}", from=3-1, to=4-1]
			\arrow[""{name=0, anchor=center, inner sep=0}, "u"', from=3-2, to=4-2]
			\arrow["v", from=5-1, to=5-2]
			\arrow["g", from=1-3, to=1-4]
			\arrow[""{name=1, anchor=center, inner sep=0}, "{u'}", from=3-3, to=4-3]
			\arrow[Rightarrow, dotted, no head, from=1-3, to=3-3]
			\arrow[Rightarrow, dotted, no head, from=1-4, to=3-4]
			\arrow["u"', from=3-4, to=4-4]
			\arrow["m", from=3-3, to=3-4]
			\arrow["{m'}"', from=4-3, to=4-4]
			\arrow["v", from=5-3, to=5-4]
			\arrow["{m'}"', from=4-1, to=4-2]
			\arrow["v", from=5-8, to=5-9]
			\arrow["u", from=3-6, to=4-6]
			\arrow[from=3-6, to=3-7]
			\arrow[from=4-6, to=4-7]
			\arrow[""{name=2, anchor=center, inner sep=0}, "u"', from=3-7, to=4-7]
			\arrow["f", from=1-6, to=1-7]
			\arrow[Rightarrow, dotted, no head, from=1-6, to=3-6]
			\arrow[Rightarrow, dotted, no head, from=1-7, to=3-7]
			\arrow["{\xi_!(u,e')}"', from=1-8, to=2-8, cocart]
			\arrow["f", from=1-8, to=1-9]
			\arrow["{\xi_!(u,e)}", from=1-9, to=2-9, cocart]
			\arrow["{\mathrm{fill}}"', dashed, from=2-8, to=2-9]
			\arrow[""{name=3, anchor=center, inner sep=0}, "u", from=3-8, to=4-8]
			\arrow[from=3-8, to=3-9]
			\arrow["u"', from=3-9, to=4-9]
			\arrow[from=4-8, to=4-9]
			\arrow["v", from=5-6, to=5-7]
			\arrow["\kappa", shorten <=6pt, shorten >=6pt, maps to, from=0, to=1]
			\arrow["\mu", shorten <=6pt, shorten >=6pt, maps to, from=2, to=3]
		\end{tikzcd}\]
		\caption{Action on arrows of the fibered functors $\kappa: F\to_B \VertArr_\pi(E)  : \mu$}\label{fig:kappa-act-arrows}
	\end{figure}

		By our discussion of cocartesian-on-the-left fibrations, \cf~\Cref{prop:char-fib-cocart-left}, the fibered LARI $\mu:F \to \VertArr_\pi(E)$ at~$b:B$ is given by
		\[ \mu_b(a:A,u:\comma{a}{A},e:P(a,b)) \defeq \pair{u:\comma{a}{A}} {\xi_!(u,e):e \to_{\pair{u}{b}}^P u_!\,e}. \]
		By the closure properties of cartesian fibrations~\cite[Corollary~5.2.10, Proposition~5.2.14]{BW21} we obtain that the pulled back maps $\pi' \defeq \cst^*\pi^{\Delta^1}: \VertArr_\pi(E) \fibarr B$ and $q' \defeq \cst^*q:\VertArr_q(A \times B) \fibarr B$ are cartesian fibrations. Moreover, by \emph{op.~cit.}, Proposition~5.3.10, so is $\psi \defeq q' \times_q \pi : F \fibarr B$. By the computations of lifts, as elobarated in \emph{op.~cit.}, Propositions~5.2.9, 5.3.9,~and 5.3.10, the cartesian lifts in $\pi': \VertArr_\pi(E) \fibarr B$ and, resp.~$\psi: F \fibarr B$ are given by as follows (\cf~\Cref{fig:cart-ness-of-mu} for illustration):
		\begin{align*}
			& (\pi')^*\left( v:b \to_B b', \pair{u:a \to_A a'}{f:e \to^P_{\pair{u}{a}} e'}\right) \\
			& = \left \langle v, \pair{\id_a}{\id_a'} : u \rightrightarrows_A u, \pair{\pi^*(v,e)}{\pi^*(v,e')}: \tyfill \rightrightarrows^P f \right \rangle \\
			& \\
			& \psi^*\left( v:b\to_B b', \pair{u:a \to_A a'}{e:P(a,b)} \right) \\
			& = \left \langle v, \langle \pair{\id_a}{\id_a'}: u \rightrightarrows_A u, \pi^*(v,e): v^* \cartarr^\pi_v e \right \rangle
		\end{align*}
		Note that, instead of using the formulas for the lifts, one can also directly verify that the given maps are indeed cartesian. Since for $\psi$ this is straightforward to see we only discuss the case of $\pi': \VertArr_\pi(E) \fibarr B$. Consider probing maps as indicated in:
		\[\begin{tikzcd}
			d && {v^*\,e} && e \\
			{d'} && {v^*\,e'} && {e'} \\
			{a''} && a && a \\
			{a'''} && {a'} && {a'} \\
			b && b && b
			\arrow["v"{description}, from=5-3, to=5-5]
			\arrow["{v'}"{description}, from=5-1, to=5-3]
			\arrow["{v'v}"{description}, curve={height=18pt}, from=5-1, to=5-5]
			\arrow[Rightarrow, no head, from=4-3, to=4-5, shorten <=9pt, shorten >=9pt]
			\arrow["{m'}"{description}, dashed, from=4-1, to=4-3]
			\arrow["m"{description, pos=0.7}, curve={height=18pt}, from=4-1, to=4-5]
			\arrow["{u'}"', from=3-1, to=4-1]
			\arrow["u"', from=3-3, to=4-3]
			\arrow[Rightarrow, no head, from=3-3, to=3-5, shorten <=9pt, shorten >=9pt]
			\arrow["u", from=3-5, to=4-5]
			\arrow["{f'}"', from=1-1, to=2-1]
			\arrow["{g'}"{description}, from=2-1, to=2-3]
			\arrow[""{name=0, anchor=center, inner sep=0}, "g"{description}, from=1-1, to=1-3]
			\arrow["{k :\jdeq \mathrm{fill}}"', dashed, from=1-3, to=2-3]
			\arrow["\ell", from=1-3, to=1-5]
			\arrow[""{name=1, anchor=center, inner sep=0}, "{\ell'}", from=2-3, to=2-5]
			\arrow["f", from=1-5, to=2-5]
			\arrow[""{name=2, anchor=center, inner sep=0}, "{h'}"{description}, curve={height=18pt}, from=2-1, to=2-5]
			\arrow[""{name=3, anchor=center, inner sep=0}, "h"{description}, curve={height=-18pt}, from=1-1, to=1-5]
			\arrow[Rightarrow, no head, from=2-3, to=1-5, shorten <=9pt, shorten >=9pt]
			\arrow["{(?)}"{description}, draw=none, from=2-1, to=1-3]
			\arrow[shorten <=9pt, shorten >=9pt, Rightarrow, no head, from=2, to=1]
			\arrow[shorten <=8pt, shorten >=8pt, Rightarrow, no head, from=0, to=3]
			\arrow["m"{description}, dashed, from=3-1, to=3-3]
			\arrow["m"{description, pos=0.7}, curve={height=-16pt}, from=3-1, to=3-5, crossing over]
		\end{tikzcd}\]
	By the property of the $\pi$-cartesian lifts the two triangles and the right hand square commute as indicated. For the square in question on the left hand side we employ a line of reasoning familiar from fibered $1$-category theory: to give a homotopy $kg = g'f'$ it suffices to show that the $\pi$-cartesian arrow $\ell'$ equalizes both composite arrows. Note that from a path $\ell'k = f\ell$ we obtain a chain of homotopies
	\[ (\ell' k)g = (f\ell)g = fh = h'f' = (\ell' g')f'\]
	as desired. Hence the whole diagram $(?)$ commutes.
	
	Now, the fibered transport functor $\mu:F \to \VertArr_\pi(E)$ is cartesian if and only if it maps $\psi$-cartesian arrows to $\pi'$-cartesian arrows. Its action on $\psi$-cartesian arrows is given by
	\begin{align*}
			& \mu_v(\psi^*(v,\pair{u}{e}))  = \mu_v\big(\pair{\id_a}{\id_{a'}}:u \rightrightarrows_A u, \pi^*(v,e)\big) \\
		= & \left \langle \pair{\id_a}{\id_{a'}}:u \rightrightarrows_A u, \big \langle \pi^*(v,e), \tyfill^\xi_{\xi_!(u,v^*e)}(\xi_!(u,e) \circ \pi^*(v,e))\big \rangle :\xi_!(u,v^*e) \rightrightarrows^P \xi_!(u,e) \right \rangle,
	\end{align*}
	for $u:a \to_A a'$, $v:b \to_B  b'$, $e:P(a,b)$.
	Conversely, $\pi'$-cartesian lifts of $\mu$-images are of the form
	\begin{align*}
			& (\pi')^*(v,\mu_b(u,e))  = (\pi')^*(v, \pair{u}{\xi_!(u,e)}) \\
		= &	\left \langle \pair{\id_a}{\id_{a'}} : u \rightrightarrows_A u, \pair{\pi^*(v,e)}{\pi^*(v,u_!\,e)}: \tyfill^\pi_{\pi^*(v,u_!\,e)}(\xi_!(u,e) \circ \pi^*(v,e)) \rightrightarrows^P \xi_!(u,e)\right \rangle.
	\end{align*}
	But having an identification between those squares is exactly equivalent to the commutation condition, by~\Cref{prop:comm-lifts}.

	\begin{figure}
		\[\begin{tikzcd}
			{v^*\,e} & e & {v^*\,e} & e &&& e & {v^*\,e} & e \\
			&& {u_!\,v^*\,e} & {u_!\,e} &&& {u_!\,e} & {v^*\,u_!\,e} & {u_!\,e} \\
			a & a & a & a &&& a & a & a \\
			{a'} & {a'} & {a'} & {a'} &&& {a'} & {a'} & {a'} \\
			b & {b'} & b & {b'} && b & {b'} & b & {b'}
			\arrow["u", from=3-1, to=4-1]
			\arrow[Rightarrow, no head, from=4-1, to=4-2]
			\arrow[Rightarrow, no head, from=3-1, to=3-2]
			\arrow[""{name=0, anchor=center, inner sep=0}, "u"', from=3-2, to=4-2]
			\arrow["v", from=5-1, to=5-2]
			\arrow["{\pi^*(v,e)}", from=1-1, to=1-2, cart]
			\arrow["{\pi^*(v,e)}", from=1-3, to=1-4, cart]
			\arrow["{\pi_!(u,v^*\,e)}"' swap, from=1-3, to=2-3, cocart]
			\arrow["{\xi_!(u,e)}", from=1-4, to=2-4, cocart]
			\arrow["{\mathrm{fill}}", dashed, from=2-3, to=2-4]
			\arrow["v", from=5-3, to=5-4]
			\arrow[""{name=1, anchor=center, inner sep=0}, "u", from=3-3, to=4-3]
			\arrow[Rightarrow, no head, from=4-3, to=4-4]
			\arrow[Rightarrow, no head, from=3-3, to=3-4]
			\arrow["u"', from=3-4, to=4-4]
			\arrow[Rightarrow, dotted, no head, from=1-1, to=3-1]
			\arrow[Rightarrow, dotted, no head, from=1-2, to=3-2]
			\arrow["v", from=5-6, to=5-7]
			\arrow[""{name=2, anchor=center, inner sep=0}, "u", from=3-8, to=4-8]
			\arrow[Rightarrow, no head, from=4-8, to=4-9]
			\arrow[Rightarrow, no head, from=3-8, to=3-9]
			\arrow["u"', from=3-9, to=4-9]
			\arrow["{\mathrm{fill}}"', dashed, from=1-8, to=2-8]
			\arrow["{\pi^*(v,u_!\,e)}", from=2-8, to=2-9, cart]
			\arrow["{\xi_!(u,e)}" swap, from=1-9, to=2-9, cocart]
			\arrow["{\pi^*(v,e)}", from=1-8, to=1-9, cart]
			\arrow["v", from=5-8, to=5-9]
			\arrow[""{name=3, anchor=center, inner sep=0}, "u"', from=3-7, to=4-7]
			\arrow["{\xi_!(u,e)}"', from=1-7, to=2-7, cocart]
			\arrow["{\mu_v}", shorten <=7pt, shorten >=7pt, maps to, from=0, to=1]
			\arrow["{(\pi')^*(v,-)}", shorten <=7pt, shorten >=7pt, maps to, from=3, to=2]
		\end{tikzcd}\]
		\caption{Cartesianness of fibered lifting functor}\label{fig:cart-ness-of-mu}
	\end{figure}
	\item[$\ref{it:char-two-sid-ii} \iff \ref{it:char-two-sid-iii'}$:] This follows from the characterization~\Cref{prop:char-cocart-fib-in-cart-fib} of cocartesian families in cartesian families, namely the equivalence of the conditions from~\Cref{it:gen-lift-comm-i} and~\Cref{it:gen-lift-comm-ii}.
	\item[$\ref{it:char-two-sid-i} \iff \ref{it:char-two-sid-iii}$:] This is dual to the previous case.
	\item[$\ref{it:char-two-sid-ii} \iff \ref{it:char-two-sid-iv}$:] We adapt the proof of~\cite[Theorem~7.1.4]{RV21}.
	First, observe that we have
	\[ \comma{\xi}{A} \times_E \comma{B}{\pi} \simeq \sum_{\substack{a:A \\ b:B}} \comma{a}{A} \times \comma{B}{b} \times P(a,b).\]
	Now, the assumption of $\varphi$ being cocartesian on the left and cartesian on the right is equivalent to the existence of the fibered adjoints $\tau_\xi$ and $\tau^\pi$. We will only write down the steps starting with $\tau_\xi$, since the case of $\tau^\pi$ is dual. By pulling back the fibered LARI adjunction $\tau_\xi \dashv_{A \times B} \iota_\xi$, we obtain
	\[\begin{tikzcd}
		{B \downarrow \pi} & {} &&& E \\
		& {\xi \downarrow A \times_E B \downarrow \pi} &&&& {\xi \downarrow A} \\
		{A \times B^{\Delta^1}} &&&& {A \times B}
		\arrow[from=1-1, to=1-5]
		\arrow["{\lambda u,v,e.\langle u1,v\rangle}"{description}, from=2-2, to=3-1]
		\arrow[""{name=0, anchor=center, inner sep=0}, "{\iota_\xi}"{description}, from=1-5, to=2-6]
		\arrow["\varphi"{description, pos=0.3}, from=1-5, to=3-5]
		\arrow["{\langle \partial_1, \pi \partial_0\rangle}"{description}, from=2-6, to=3-5]
		\arrow["{\id_A \times \partial_1}"{description}, from=3-1, to=3-5]
		\arrow["{\lambda a',v,e.\langle a',v,e \rangle}"{description}, from=1-1, to=3-1]
		\arrow[""{name=1, anchor=center, inner sep=0}, "{\tau_\xi}"{description}, curve={height=18pt}, dotted, from=2-6, to=1-5]
		\arrow["\lrcorner"{anchor=center, pos=0.125}, draw=none, from=2-2, to=3-5]
		\arrow[""{name=2, anchor=center, inner sep=0}, "{\langle \tau_\xi \circ \partial_1, \id \rangle}"{description}, from=1-1, to=2-2]
		\arrow[""{name=3, anchor=center, inner sep=0}, "\ell"{description}, curve={height=24pt}, dotted, from=2-2, to=1-1]
		\arrow["\lrcorner"{anchor=center, pos=0.125}, shift right=3, draw=none, from=1-2, to=3-5]
		\arrow["\dashv"{anchor=center, rotate=-136}, draw=none, from=1, to=0]
		\arrow["\dashv"{anchor=center, rotate=-127}, draw=none, from=3, to=2]
		\arrow[from=2-2, to=2-6, crossing over]
	\end{tikzcd}\]
	with fibrant replacements
	\begin{align*}
		\xi \downarrow A & \simeq \sum_{\substack{a':A \\ b:B}} \sum_{u:A \downarrow a'} P(u0,b), &  E & \simeq \sum_{\substack{a':A \\ b:B}} P(a',b), \\
		A \times B^{\Delta^1} & \simeq \sum_{\substack{a':A \\ b:B}} B\downarrow b, &  B \downarrow \pi & \simeq \sum_{\substack{a':A \\ b:B}} B \downarrow b \times P(a',b). 
	\end{align*}
	Postcomposition with $\id_A : A \times B^{\Delta^1} \to A \times B$ preserves the fibered adjunction, yielding as desired:
	\[\begin{tikzcd}
		{B \downarrow \pi} &&&& {\xi \downarrow A \times_E B \downarrow \pi} \\
		\\
		&& {A \times B^{\Delta^1}} \\
		\\
		&& {A \times B}
		\arrow[""{name=0, anchor=center, inner sep=0}, "{\langle \tau^\xi \circ \partial_1 ,\id \rangle}"{description}, curve={height=12pt}, from=1-1, to=1-5]
		\arrow[from=1-1, to=3-3]
		\arrow[from=1-5, to=3-3]
		\arrow["{\mathrm{id}_A \times \partial_0}"{description}, two heads, from=3-3, to=5-3]
		\arrow["{\langle \xi \partial_1, \partial_0\rangle}"{description}, two heads, from=1-1, to=5-3]
		\arrow["{\langle \partial_1, \partial_0\rangle}"{description}, two heads, from=1-5, to=5-3]
		\arrow[""{name=1, anchor=center, inner sep=0}, "\ell"{description}, curve={height=12pt}, dashed, from=1-5, to=1-1]
		\arrow["\dashv"{anchor=center, rotate=-90}, draw=none, from=1, to=0]
	\end{tikzcd}\]
	This, together with the dual case, yields the claimed adjoints in the fibered square of~\Cref{it:char-two-sid-iv}.
	
	In sum, this is equivalent to $\pi:E \fibarr B$ being a cartesian fibration, the fibered functor
	\[\begin{tikzcd}
		{\comma{\xi}{A}} && E \\
		& B
		\arrow["{\tau_\xi}", from=1-1, to=1-3]
		\arrow["{\pi\partial_0}"', two heads, from=1-1, to=2-2]
		\arrow["\pi", two heads, from=1-3, to=2-2]
	\end{tikzcd}\]
	being a cartesian functor and a cocartesian fibration sliced over $B$. The invertibility of the mate is then equivalent to the functor
	\[\begin{tikzcd}
		E && {A \times B} \\
		& B
		\arrow["\varphi", from=1-1, to=1-3]
		\arrow["\pi"', two heads, from=1-1, to=2-2]
		\arrow["q", two heads, from=1-3, to=2-2]
	\end{tikzcd}\]
	being cartesian.
	
	By~\Cref{thm:char-cocart-fun}, this is equivalent to the mate of the identity of
	\[\begin{tikzcd}
		{\comma{\xi}{A}} && E \\
		{\comma{\xi}{A} \times_E \comma{B}{\pi}} && {\comma{B}{\pi}}
		\arrow[from=1-1, to=1-3]
		\arrow[from=1-1, to=2-1]
		\arrow[from=2-1, to=2-3]
		\arrow[from=1-3, to=2-3]
	\end{tikzcd}\]
	being invertible, as
	\[ \comma{B}{\pi\partial_0} \simeq \comma{\xi}{A} \times_E \comma{B}{\pi}.\]
	\end{description}

\end{proof}

%% file: two-sided-cartesian-functors.tex
\begin{definition}[Two-sided cartesian functors]
Let $P,Q:A \to B \to \UU$ be two-sided cartesian families. A fibered map $h: P \to_{A \times B} Q$ is called \emph{two-sided cartesian functor} (or simply \emph{cartesian}) if it constitutes a cocartesian functor $h: P_B \to_A Q_B$ and a cartesian functor $h:P_A \to_B Q_A$.
\end{definition}

An immediate reformulation is that, for all $u:a \to_A a'$ and $v:b' \to_B v$, $e,d:P(a,b)$, we have identities\footnote{In particular, again the types of each of these identities is a proposition, hence so is their product.}
\[ h_{u,b}(P_!(u,b,e)) = Q_!(u,b,h_{a,b}(e)), \quad h_{a,v}(P^*(a,v,e)) = Q^*(a,v,h_{a,b}(e)). \]

We will state versions of the closure properties \wrt~to different bases as well as the sliced or relative versions where the base stays fixed throughout.

\subsection{Composition and whiskering}

\begin{proposition}[Composition stability of two-sided cartesian functors]\label{prop:2s-fun-comp}
	Let $A,B,C,D,S$, and $T$ be Rezk types. Assume given two-sided cartesian families $P: A \to B \to \UU$, $Q: C \to D \to \UU$, and $R: S \to T \to \UU$, as well as two-sided cartesian functors $h:P \to_{A \times B,C \times D} Q$, $k:Q \to_{C \times D,S \times T} R$. Then the composite fibered functor $k \circ h: P \to_{A \times B,S \times T} R$ is a two-sided cartesian functor as well.
\end{proposition}

\begin{proof}
	This follows since cartesian and cocartesian functors are both closed under composition, \cf~\cite[Proposition~5.3.6, Item 1]{BW21}.
\end{proof}

\begin{corollary}[Composition stability of two-sided cartesian functors in a slice]\label{prop:2s-fun-comp-sl}
	Let $P,Q,R:A \to B \to \UU$ be two-sided cartesian families, and $h:P \to_{A \times B} Q$, $k:Q \to_{A \times B} R$ cartesian functors. Then the composite fibered functor $k \circ h: P \to_{A \times B} R$ is a two-sided cartesian functor as well.
\end{corollary}

\begin{proposition}[Whiskering with co-/cartesian fibrations~\protect{\Cite[Lemma~7.2.5]{RV21}}]\label{prop:2s-cart-comp-prod-fib}
	Let $A,B,C,D$ be Rezk types. Assume $\varphi \defeq \pair{\xi}{\pi}:E \fibarr A \times B$ is a two-sided cartesian fibration. If $k:A \fibarr C$ is a cocartesian fibration and $m:B \fibarr D$ is a cartesian fibration, then the composite
	\[\begin{tikzcd}
		E && {A \times B} && {C \times D}
		\arrow["{\langle \xi,\pi \rangle}", two heads, from=1-1, to=1-3]
		\arrow["{k \times m}", two heads, from=1-3, to=1-5]
	\end{tikzcd}\]	
	is a two-sided cartesian fibration as well.
\end{proposition}

\begin{proof}
	We argue as in~\cite[Lemma~7.2.5]{RV21}. By the characterization of two-sided cartesian fibrations via cocartesian fibrations in cartesian fibrations~ \Cref{thm:char-two-sid}, \Cref{it:char-two-sid-ii}, we reason as follows. Since $k:A \fibarr C$ and $\id_B: B \fibarr B$ both are two-sided cartesian fibrations as well also their cartesian product is, by~\Cref{prop:2s-cart-closed-pi}. Hence, the following fibered maps are cocartesian fibrations in cartesian fibrations
	\[\begin{tikzcd}
		E && {A \times B} & {A \times B} && {C \times B} \\
		& B &&& B
		\arrow["\varphi", from=1-1, to=1-3]
		\arrow["\pi"', two heads, from=1-1, to=2-2]
		\arrow["q", two heads, from=1-3, to=2-2]
		\arrow["{k \times \id_B}", from=1-4, to=1-6]
		\arrow["q"', two heads, from=1-4, to=2-5]
		\arrow[two heads, from=1-6, to=2-5]
	\end{tikzcd}\]
	and so is their horizontal composite $(k \times \id_B) \circ \varphi: E \fibarr_B C \times B$.
	
	One can argue similarly for the case $\id_C: C \fibarr C$ and $m:B \fibarr D$, which establishes the claim.
\end{proof}

\begin{figure}
	\[\begin{tikzcd}
		P & e & {u_*\,e} & {e'} && d & d & {d'} & Q \\
		&& B & b & b & {b'} \\
		A & a & {a'} & {a''} && c & c & {c'} & C
		\arrow["u", from=3-2, to=3-3]
		\arrow["{\forall \,u'}", from=3-3, to=3-4]
		\arrow["g"', dashed, from=1-3, to=1-4]
		\arrow["{\id_d}"', Rightarrow, no head, from=1-6, to=1-7]
		\arrow["r"', dashed, from=1-7, to=1-8]
		\arrow["{\id_b}", Rightarrow, no head, from=2-4, to=2-5]
		\arrow["v", dashed, from=2-5, to=2-6]
		\arrow["{\id_c}", Rightarrow, no head, from=3-6, to=3-7]
		\arrow["{\forall \,w}", from=3-7, to=3-8]
		\arrow["{u'u}"{description}, curve={height=18pt}, from=3-2, to=3-4]
		\arrow["{\forall\,v}"{description}, curve={height=18pt}, from=2-4, to=2-6]
		\arrow["{\forall \,r}"{description}, curve={height=-18pt}, from=1-6, to=1-8]
		\arrow["{\forall \,h}"{description}, curve={height=-18pt}, from=1-2, to=1-4]
		\arrow["f"', from=1-2, to=1-3, cocart]
		\arrow["w"{description}, curve={height=18pt}, from=3-6, to=3-8]
	\end{tikzcd}\]
	\caption{Cocartesian lift and universal property in the span composite}
	\label{fig:cocart-lift-spcomp}
\end{figure}

\begin{proposition}[Span composition of two-sided cartesian fibrations, \cf~\protect{\cite[Proposition~7.2.6]{RV21}}]
	Let $P:A\to B \to \UU$, $Q:B \to C \to \UU$ be two-sided cartesian families over Rezk types $A$, $B$, $C$. Then the family defined by \emph{span composition}
	\[ Q \spancomp P \defeq \lambda a,c.\sum_{b:B} P\,a\,b \times Q\,b\,c: A \to C \to \UU \]
	is also two-sided cartesian.
	
	In particular, the cocartesian and cartesian lifts, resp., are given as follows: For $u:a \to_A a'$, $w:c'\to_C c$, $b:B$, $e:P(a,b)$, and $d:Q(b,c)$ we have
	\begin{align*}
		(Q \spancomp P)_!(u:a \to_A a', \id_c,\angled{b,e,d}) & \defeq \angled{u,\id_c,\id_b, P_!(u,b,e), \id_d}, \\
		(Q \spancomp P)^*(\id_a, w,\angled{b,e,d}) & \defeq \angled{\id_a,w,\id_b, \id_e, Q^*(b,w,d)}.
	\end{align*}
\end{proposition}

\begin{proof}
	We can argue on the level of fibrations just as in \cite[Proposition~7.2.6]{RV21}. Let $\varphi \defeq \pair{\xi}{\pi} \defeq \Un_{A,B}(P): E \fibarr A \times B$, $\psi \defeq \pair{\kappa}{\mu} \Un_{B,C}(Q): E \fibarr B \times C$. The unstraightening of $Q \spancomp P$ corresponds to the composite
	\[\begin{tikzcd}
		{E \times _B F} && {E \times F} && {A \times C}
		\arrow["{\langle q,p\rangle}", from=1-1, to=1-3]
		\arrow["{\langle \xi,\mu\rangle}", from=1-3, to=1-5]
		\arrow["{\psi \boxdot \varphi}"{description}, curve={height=-24pt}, from=1-1, to=1-5]
	\end{tikzcd}\]
	where $q:E \times_B F \to E$ and $p:E \times_B F \to F$ are the projections from the pullback object. Now, the map $\psi \spancomp \varphi:A \to C \to \UU$ is constructed by first taking the pullback
	\[\begin{tikzcd}
		{E \times_B F} && E \\
		{A \times F} && {A \times B}
		\arrow["{\id_A \times \kappa}"', from=2-1, to=2-3]
		\arrow[from=1-1, to=1-3]
		\arrow["\varphi", two heads, from=1-3, to=2-3]
		\arrow["{\langle \xi \circ q,p \rangle}"', two heads, from=1-1, to=2-1]
		\arrow["\lrcorner"{anchor=center, pos=0.125}, draw=none, from=1-1, to=2-3]
	\end{tikzcd}\]
	and then postcomposing the map on $\pair{\xi q}{p}:{E \times_B F} \fibarr A \times F$ with $\id_A \times \mu: A \times F \to A \times C$. Pullback along products of maps preserves two-sided cartesian fibrations by~\Cref{prop:2s-cart-closed-pb}, and so does postcomposition with the cartesian product of a cocartesian and a cartesian fibration by~\Cref{prop:2s-cart-comp-prod-fib}. Hence, the resulting map $\psi \spancomp \varphi: E \times_B F \fibarr A \times C$ is two-sided cartesian as well.
	
	The proclaimed description of the co-/cartesian lifts comes out of this construction, using the descriptions of the lifts from the constructions in~\cite[Subsections~3.2.4 and~5.3.3]{BW21}. Alternatively, one can verify the universal property directly, \cf~\Cref{fig:cocart-lift-spcomp}. \Eg, for the cocartesian case, given any $u':a'\to_A a''$, $w:c \to_C c'$, an arrow lying over the (component-wise) composite with domain $\angled{b,e,d}$ consists of some arrow $v:b \to b'$ and dependent arrows $f:e \to^P_{\pair{u'u}{v}} e'$, $r:d \to^Q_{\pair{v}{w}} d'$. By initiality, $v:\id_b \to v$ is the unique filler in the comma object $\comma{b}{B}$, and so is $r:\id_d \to r$, lying over $v:\id_b \to v$ and $w:\id_c \to c'$. By cocartesianness, we also find~\wrt~the data given the unique filler $g \defeq \cocartFill_{P_*(u,b,e)}(f)$ with $g \circ f = h$, for $f \defeq P_*(u,b,e):e \cocartarr^P_{\pair{u}{b}}u_*\,e$ in $P$ as desired. 
\end{proof}

\subsection{Pullback and reindexing}

\begin{proposition}[Pullback stability of two-sided cartesian families, \protect{\cf~\cite[Proposition~7.2.4]{RV21}}]\label{prop:2s-cart-closed-pb}
Let $P:A \to B \to \UU$ be a two-sided cartesian family over Rezk types $A$ and $B$. Then for any pair of maps $k:C \to A$, $m:D \to B$, the pullback family
\[ (k \times m)^*P: C \to D \to \UU \]
is two-sided as well.
Diagrammatically, if the two-sided fibration $\varphi: E \fibarr A \times B$ denotes the unstraightening of $P$, this means that the map $\psi$ in the following diagram is a two-sided fibration:
\[\begin{tikzcd}
	{(k \times m)^*E} && E \\
	{C \times D} && {A \times B}
	\arrow["{(k \times m)^*\varphi}"', two heads, from=1-1, to=2-1]
	\arrow["{k \times m}"', from=2-1, to=2-3]
	\arrow[from=1-1, to=1-3]
	\arrow["\varphi", two heads, from=1-3, to=2-3]
	\arrow["\lrcorner"{anchor=center, pos=0.125}, draw=none, from=1-1, to=2-3]
\end{tikzcd}\]
Furthermore, we claim that this square is a two-sided cartesian functor.
\end{proposition}

Recalling the notation from~\Cref{ssec:two-var}, we write
\begin{align*}
	& P_m:A \to \UU, & P_m(a) & \defeq \sum_{b':B'} P_{m\,b'}(a) \jdeq \sum_{b':B'} P(a,m\,b'), \\
	& P^k:B \to \UU, & P^k(b) & \defeq \sum_{a':A'} P^{k\,a'}(b) \jdeq \sum_{a':A'} P(k\,a',b), \\
\end{align*}
for the families induced from reindexing on just one side, and then projection down,~\ie: $P_m$ and $P^k$ arise fibrationally as follows:
\[\begin{tikzcd}
	{m^*\widetilde{P}} && {\widetilde{P}} && {k^*\widetilde{P}} && {\widetilde{P}} \\
	{A \times B'} && {A \times B} && {A' \times B} && {A \times B} \\
	A &&&& B
	\arrow[two heads, from=1-5, to=2-5]
	\arrow["{k \times \id_B}"', from=2-5, to=2-7]
	\arrow[two heads, from=1-7, to=2-7]
	\arrow[two heads, from=2-5, to=3-5]
	\arrow[from=1-5, to=1-7]
	\arrow["\lrcorner"{anchor=center, pos=0.125}, draw=none, from=1-5, to=2-7]
	\arrow["{\Un_B(P^k)}"'{pos=0.2}, curve={height=30pt}, from=1-5, to=3-5]
	\arrow[two heads, from=1-3, to=2-3]
	\arrow[from=1-1, to=1-3]
	\arrow[two heads, from=1-1, to=2-1]
	\arrow[two heads, from=2-1, to=3-1]
	\arrow["{\Un_A(P_m)}"'{pos=0.2}, curve={height=30pt}, two heads, from=1-1, to=3-1]
	\arrow["{\id_A \times m}"', from=2-1, to=2-3]
	\arrow["\lrcorner"{anchor=center, pos=0.125}, draw=none, from=1-1, to=2-3]
\end{tikzcd}\]
In particular, for the case of $k= \id_A: A \to A$ and $m=\id_B : B \to B$ we have $P^A = P^{\id_A}$ and $P_B = P_{\id_B}$, \cf~\Cref{ssec:two-var}.

\begin{proof}
This follows by employing the characterization~\Cref{thm:char-two-sid}, \Cref{,it:char-two-sid-iii}, and then the closure property~\Cref{prop:clos-sl-cocart-fib-pb}. In particular, letting either of the maps $k,m$ be an identity, we can conclude that $P^k$ is cocartesian, and $P_m$ is cartesian.

That the square is a two-sided cartesian functor follows by separately projecting to the factors in the base, and then using~\Cite[Proposition~5.3.9]{BW21} or its dual. Namely, \eg~since $P_m:A \to \UU$ is cocartesian, so is its pullback along $k:A' \to A$ which arises as
\[\begin{tikzcd}
	{(k \times m)^*\widetilde{P} \simeq k^*\widetilde{P_m}} && {\widetilde{P_m}} \\
	{A'} && A
	\arrow[two heads, from=1-3, to=2-3]
	\arrow["k"', from=2-1, to=2-3]
	\arrow[two heads, from=1-1, to=2-1]
	\arrow[from=1-1, to=1-3]
	\arrow["\lrcorner"{anchor=center, pos=0.125}, draw=none, from=1-1, to=2-3]
\end{tikzcd}\]
and the pullback square is known to be a cocartesian functor.
\end{proof}

\begin{proposition}[Pullback stability of two-sided cartesian functors]\label{prop:2s-cart-fun-pb}
	In the following, let all types be Rezk. Consider two-sided cartesian families $P,Q:A \to B\to \UU$ with unstraightenings $E \fibarr \to A \times B$ of $P$, and $F \fibarr A \times B$ of $Q$, resp. Let $\kappa:P\to_{A \times B} Q$ be a two-sided cartesian functor. Given maps $k:A' \to A$, $m:B' \to B$, then the functor $\kappa': P' \to_{A' \times B'} Q'$ induced by pullback along $k \times m$ is two-sided cartesian as well:
\[\begin{tikzcd}
	{E'} &&& E \\
	& {Q'} &&& Q \\
	{A' \times B'} &&& {A \times B}
	\arrow["\kappa'", dashed, from=1-1, to=2-2]
	\arrow[from=1-1, to=1-4]
	\arrow["\kappa", from=1-4, to=2-5]
	\arrow[two heads, from=1-4, to=3-4]
	\arrow[two heads, from=2-5, to=3-4]
	\arrow[two heads, from=1-1, to=3-1]
	\arrow[two heads, from=2-2, to=3-1]
	\arrow["{k \times m}"{description}, from=3-1, to=3-4]
	\arrow["\lrcorner"{anchor=center, pos=0.125}, draw=none, from=2-2, to=3-4]
	\arrow["\lrcorner"{anchor=center, pos=0.125}, shift right=4, draw=none, from=1-1, to=3-4]
	\arrow[from=2-2, to=2-5, crossing over]
\end{tikzcd}\]
\end{proposition}

\begin{proof}
	Let $u':a_0' \to_{A'} a_1'$, $b':B'$, $e:P'(a_0',b') \simeq P(k\,a_0',m\,b')$. Straightforward calculation gives
	\begin{align*}
		& \kappa'_{u',b'}(P_!'(u',b',e)) = \kappa_{k\,u',m\,'}(P_!(k\,u',m\,b',e)) \\
		& = Q_!(k\,u',m\,b',\kappa_{k\,u',m\,b'}(e)) = Q_!'(u',b',\kappa'_{u',b'}(e)), 
	\end{align*}
	where the second identity is given by $\kappa$ being cocartesian. The dual case for cartesian lifts works similarly (\cf~also~\cite[Proposition~5.3.18]{BW21}).
\end{proof}

\subsection{Dependent and sliced product}

\begin{proposition}[Product stability of two-sided cartesian families]\label{prop:2s-cart-closed-pi}
Let $A,B:I \to \UU$ be families of Rezk types for a small type $I$. Consider a two-sided family $P:\prod_{i:I} A_i \to B_i \to \UU$. Then the induced product family
\[ \prod_{i:I} P_i : \prod_{i:I} A_i \to \prod_{i:I} B_i \to \UU\]
is two-sided cartesian as well.

Moreover, denoting the unstraightenings of the $P_i$ by $\varphi_i: E_i \fibarr A_i \times B_i$, the squares
\[\begin{tikzcd}
	{\prod_{i:I} E_i} && {E_k} \\
	{\prod_{i:I} A_i \times \prod_{i:I} B_i} && {A_k \times B_k}
	\arrow[two heads, from=1-1, to=2-1]
	\arrow[from=2-1, to=2-3]
	\arrow[from=1-1, to=1-3]
	\arrow[two heads, from=1-3, to=2-3]
\end{tikzcd}\]
are two-sided cartesian functors. Furthermore, these product cones are terminal \wrt~two-sided cartesian functors.\footnote{Here, and in the following we will not formally spell out the universal properties, but they are analogous to the respective propositions in~\Cite[Subsection~5.3.3]{BW21}. The addition/generalization is that the base types are binary products, and the fibrations and functors are \emph{two-sided} cartesian.}
\end{proposition}

Fibrationally, the proposition says that given a family of two-sided fibrations $\varphi_i: E_i \fibarr A_i \times B_i$ for $i:I$, the product fibration $\prod_{i:I} \varphi_i : \prod_{i:I} E_i \fibarr \prod_{i:I} A_i \times \prod_{i:I} B_i$ is also two-sided cartesian.

\begin{proof}
	This is a consequence of the characterization~\Cref{thm:char-two-sid}, \Cref{it:char-two-sid-iii}, in combination with the closure property~\Cref{prop:clos-sl-cocart-fib-prod}. Two-sided cartesianness of the projection squares follows upon postcomposition with the respective projection, and then employing either~\Cite[Proposition~5.3.7]{BW21} or its dual. Similarly, one argues for the universal property for two-sided cartesian functors, using~\Cite[Proposition~5.3.8]{BW21} or its dual, resp.
\end{proof}

\begin{corollary}[Sliced product stability of two-sided cartesian families]\label{prop:2s-cart-closed-pi-loc}
	Let $A,B$ be small Rezk types. Consider a two-sided family $P:\prod_{i:I} A \to B \to \UU$. Then the induced fiberweise product family
	\[ \times_{i:I}^{A \times B} P_i : A \to B \to \UU\]
	is two-sided cartesian as well.
	
	Moreover, for every $k:I$ there is an induced canonical commutative triangle
	\[\begin{tikzcd}
		{\times_{i:I}^{A \times B} E_i} && {E_k} \\
		& {A \times B}
		\arrow[from=1-1, to=1-3]
		\arrow["{\times_{i:I}^{A \times B} \varphi_i}"', two heads, from=1-1, to=2-2]
		\arrow["{\varphi_k}", two heads, from=1-3, to=2-2]
	\end{tikzcd}\]
	which is a two-sided functor. The two-sided fibration $\prod_{i:I} E_i \fibarr A \times B$ is the terminal cone over the $\varphi_k: E_k \fibarr A \times B$~\wrt~(triangle-shaped) cones into the $\varphi_k$ whose horizontal map is two-sided cartesian.
\end{corollary}

\begin{proof}
	Recall that we have equivalences
	\begin{align*}
		\prod_i E_i & \simeq \sum_{\substack{\alpha:I \to A \\ \beta:I \to B}} \prod_{i:I} P_i(\alpha_i,\beta_i) \fibarr A^I \times B^I, \\
		\times_{i:I}^{A \times B} E_i & \simeq \sum_{\substack{a:A \\ b:B}} \prod_{i:I} P_i(a,b) \fibarr A \times B.
	\end{align*}

	Denote by $E_i \fibarr A \times B$ the unstraightening of the family $P_i$. By~\Cref{prop:2s-cart-closed-pb}, the induced map $\prod_{i:I} E_i \fibarr (A \times B)^I$ is two-sided cartesian:
	\[\begin{tikzcd}
		{\times_{i:I}^{A \times B} E_i} && {\prod_{i:I} E_i} \\
		{A \times B} && {(A \times B)^I}
		\arrow[two heads, from=1-1, to=2-1]
		\arrow["{\mathrm{cst}}"', from=2-1, to=2-3]
		\arrow[from=1-1, to=1-3]
		\arrow[two heads, from=1-3, to=2-3]
		\arrow["\lrcorner"{anchor=center, pos=0.125}, draw=none, from=1-1, to=2-3]
	\end{tikzcd}\]
	Invoking pullback-stability, and then considering the straightening of this map to recover a type family establishes the claim.
	
	Now, by the above description via fibrant replacement, we have evaluation maps yielding the desired cones $\ev_k: {\times_{i:I}^{A \times B} E_i} \to_{A \times B} E_k$ for $k:I$. But by the universal property of the standard dependent product~\Cref{prop:2s-cart-closed-pi} (\cf~\Cite[Proposition~5.3.8]{BW21}), these factor as follows
	\[\begin{tikzcd}
		{\times_{i:I}^{A \times B} E_i} && {\prod_{i:I} E_i} && {E_k} \\
		{A \times B} && {A^I \times B^I} && {A \times B}
		\arrow["{\mathrm{ev}_k}"{description}, from=1-3, to=1-5]
		\arrow[from=1-3, to=2-3]
		\arrow[from=2-3, to=2-5]
		\arrow[from=1-5, to=2-5]
		\arrow[from=1-1, to=2-1]
		\arrow["{\mathrm{cst}}"', dashed, from=2-1, to=2-3]
		\arrow[dashed, from=1-1, to=1-3]
		\arrow["{\mathrm{ev}_k}"{description}, curve={height=-18pt}, from=1-1, to=1-5]
		\arrow["\lrcorner"{anchor=center, pos=0.125}, draw=none, from=1-1, to=2-3]
	\end{tikzcd}\]
	where the upper horizontal induced functor is two-sided cartesian, as are the evaluations from the standard dependent product. Hence, so is their composite, as desired, by~\Cref{prop:2s-fun-comp}.
\end{proof}

\begin{proposition}[Pullback cones are two-sided cartesian functors]\label{prop:2s-cart-fun-pb-cones}
	Consider two-sided families over Rezk types
	\[ P:A \to B \to \UU, \quad  P':A' \to B' \to \UU, \quad  P:A'' \to B'' \to \UU. \]
	Furthermore, assume there are maps
	\[ \alpha:A' \to A, \quad \alpha'':A'' \to A, \quad \beta':B' \to B, \quad \beta'':B'' \to B \]
	and two-sided cartesian functors
	\[ \kappa':P' \to_{A'\times B', A \times B} P, \quad  \kappa'':P'' \to_{A''\times B'', A \times B} P.\]
	Denote by
	\[ \varphi:E \fibarr A \times B, \quad \varphi': E' \fibarr A' \times B', \quad \varphi'': E'' \fibarr A'' \times B'' \]
	the unstraightenings of $P$, $P'$, and $P''$, resp. Consider the induced pullback:
	\[\begin{tikzcd}
		{E' \times_E E''} &&& {E''} \\
		& {E'} &&& E \\
		{(A' \times B') \times_{A \times B} (A'' \times B'')} &&& {A'' \times B''} \\
		& {A' \times B'} &&& {A \times B}
		\arrow["{\varphi'\times_\varphi \varphi''}"{description}, dashed, two heads, from=1-1, to=3-1]
		\arrow[from=3-1, to=3-4]
		\arrow[from=1-1, to=1-4]
		\arrow["{\varphi''}"{description, pos=0.2}, two heads, from=1-4, to=3-4]
		\arrow[from=1-1, to=2-2]
		\arrow["\varphi"{description, pos=0.3}, two heads, from=2-5, to=4-5]
		\arrow[from=3-1, to=4-2]
		\arrow["{\alpha '' \times \beta''}"{description}, from=3-4, to=4-5]
		\arrow["{\kappa''}"{description}, from=1-4, to=2-5]
		\arrow["{\alpha' \times \beta'}"{description}, from=4-2, to=4-5]
		\arrow["\lrcorner"{anchor=center, pos=0.125, rotate=45}, draw=none, from=3-1, to=4-5]
		\arrow["\lrcorner"{anchor=center, pos=0.125, rotate=45}, draw=none, from=1-1, to=2-5]
		\arrow["{\varphi'}"{description, pos=0.3}, two heads, from=2-2, to=4-2, crossing over]
		\arrow["{\kappa'}"{description}, from=2-2, to=2-5, crossing over]
	\end{tikzcd}\]
	Then the mediating map
	\[ \varphi''' \defeq \varphi'\times_\varphi \varphi'': E'''\defeq {E' \times_E E''} \fibarr A''' \times B'''\]
	where
	\[ A''' \defeq A' \times_A A'', \quad B''' \defeq B' \times_B B'' \]
	is a two-sided cartesian fibration.\footnote{Note in particular that we have an equivalence $A''' \times B''' \equiv (A' \times B') \times_{A \times B} (A'' \times B'')$}
	
	Moreover, each of the projection squares from $\varphi'''$ is a two-sided cartesian functor, and $\varphi''':E''' \to A''' \times B'''$ satisfies the expected terminal universal property for cones which are two-sided cartesian functors (analogous to~\Cite[Propositions~5.3.10,11]{BW21}).
\end{proposition}

\begin{proof}
	We use fibrant replacement so that we can take the fibers of $\varphi'''$ to be\footnote{Where $a:A$, $a':A'$ lies strictly over $a$ via $k'$~\etc.}
	\[ P'''(a,a',a'',b,b',b'') \jdeq P'(a',b') \times_{P(a,b)} P''(a'',b'').\]
	We claim that the cocartesian lifts in $P'''$ are then given by
	\[ P_!'''(u,u',u'',b,b',b'',\angled{e,e',e''}) \jdeq \angled{P_!(u,b,e), P_!'(u',b',e'), P_!''(u'',b'',e'')},\]
	which can be checked to be cocartesian since the conditions are validated fiberwise. In particular, the cocartesian lifts in $P'$ and $P''$ indeed lie over the ones in $P$ by two-sided cartesian-ness of $\kappa'$ and $\kappa''$. So far, this is analogous to~\Cite[Proposition~5.3.10]{BW21}, but we have the additional triple of points $\angled{b,b',b''}$ as data.
	
	The argument for the cartesian lifts works dually.
	Now, the compatibility condition from~\Cref{prop:comm-lifts} have to be checked. But by the fibrant replacement above, the ensuing proposition is just witnessing that the condition is satisfied component-wise for triples $\angled{\sigma,\sigma',\sigma''}$ where $\sigma$ is a square of the form
	\[\begin{tikzcd}
		E && \bullet & \bullet \\
		&& \bullet & \bullet \\
		& \bullet & \bullet & \bullet & \bullet \\
		{A \times B} & \bullet & \bullet & \bullet & \bullet
		\arrow[cocart, from=1-4, to=2-4]
		\arrow[from=2-3, to=2-4]
		\arrow[from=1-3, to=2-3]
		\arrow[cart, from=1-3, to=1-4]
		\arrow[Rightarrow, no head, from=3-2, to=3-3]
		\arrow[from=3-3, to=4-3]
		\arrow[from=3-2, to=4-2]
		\arrow[Rightarrow, no head, from=4-2, to=4-3]
		\arrow[Rightarrow, no head, from=3-4, to=4-4]
		\arrow[from=4-4, to=4-5]
		\arrow[from=3-4, to=3-5]
		\arrow[Rightarrow, no head, from=3-5, to=4-5]
		\arrow[from=1-1, to=4-1]
	\end{tikzcd}\]
	and $\sigma',\sigma''$ are of the same shape, lying above. Since the compatbility condition is satisfied for each of those, we are done. This shows that $\varphi'''$ is a two-sided cartesian fibration, as desired.
	
	From the discussion of the lifts, it is also clear that both the projection squares are two-sided cartesian functors, since we just project to the respective coordinates. Furthermore, the universal property is established, again, by postcomposing separately with the projections to either $A'''$ or $B'''$, then applying either~\Cite[Proposition~5.3.11]{BW21} for the one-sided cocartesian case, or its dual for the cartesian case.
\end{proof}

\begin{corollary}[Pullback cones in a slice are two-sided cartesian functors]\label{prop:2s-cart-fun-pb-cones-sliced}
	Consider two-sided families over Rezk types $P,P',P'':A \to B \to \UU$ with unstraightenings $\varphi:E \fibarr A \times B$, $\varphi':E' \fibarr A' \times B'$, and $\varphi'':E'' \fibarr A'' \times B''$. Given two-sided cartesian functors $\kappa':P' \to_{A \times B} P$ and $\kappa'':P'' \to P$,
	consider the induced pullback over $A \times B$:
	\[\begin{tikzcd}
		{E' \times_E E''} &&& {E''} \\
		& {E'} &&& E \\
		&& {A \times B}
		\arrow[from=1-1, to=1-4]
		\arrow[from=1-1, to=2-2]
		\arrow["{\kappa''}"{description}, from=1-4, to=2-5]
		\arrow["\lrcorner"{anchor=center, pos=0.125, rotate=45}, draw=none, from=1-1, to=2-5]
		\arrow[two heads, from=2-5, to=3-3]
		\arrow[two heads, from=1-4, to=3-3]
		\arrow[two heads, from=2-2, to=3-3]
		\arrow[curve={height=30pt}, from=1-1, to=3-3]
		\arrow["{\kappa'}"{description, pos=0.3}, from=2-2, to=2-5, crossing over]
	\end{tikzcd}\]
	Then the mediating map
	\[ \varphi''' \defeq \varphi'\times_\varphi \varphi'': E'''\defeq {E' \times_E E''} \fibarr A \times B\]
	is a two-sided cartesian fibration.
	
	Moreover, each of the projection squares from $\varphi'''$ is a two-sided cartesian functor, and $\varphi''':E''' \to A \times B$ satisfies the expected terminal universal property for cones which are two-sided cartesian functors over $A \times B$.
\end{corollary}

\subsection{Sequential limit}

\begin{prop}[Sequential limit cones are cocartesian functors]\label{prop:cocart-fun-seqlim}
	Consider an inverse diagram of two-sided cartesian fibrations as below where all of the connecting squares are two-sided cartesian functors:
	\[\begin{tikzcd}
		& \ldots &&&& {E_\infty} \\
		\cdots && {E_2} && {E_1} && {E_0} \\
		& \ldots &&&& {A_\infty \times B_\infty} \\
		\cdots && {A_2 \times B_2} && {A_1 \times B_1} && {A_0 \times B_0}
		\arrow["{\pi_\infty}"{description, pos=0.3}, dashed, from=1-6, to=3-6]
		\arrow["{f_0}"{description, pos=0.7}, from=4-5, to=4-7]
		\arrow["{\kappa_0}"{description}, dashed, from=1-6, to=2-7]
		\arrow["{\langle \alpha_0,\beta_0 \rangle}"{description}, dashed, from=3-6, to=4-7]
		\arrow["{f_1}"{description, pos=0.7}, from=4-3, to=4-5]
		\arrow["{g_1}"{description, pos=0.7}, from=2-3, to=2-5]
		\arrow["{g_2}"{description, pos=0.7}, from=2-1, to=2-3]
		\arrow["{f_2}"{description, pos=0.7}, from=4-1, to=4-3]
		\arrow["{\kappa_1}"{description}, dashed, from=1-6, to=2-5]
		\arrow["{\kappa_2}"{description}, dashed, from=1-6, to=2-3]
		\arrow["{\langle \alpha_2,\beta_2 \rangle}"{description, pos=0.7}, curve={height=6pt}, dashed, from=3-6, to=4-3]
		\arrow["{\langle \alpha_1,\beta_1 \rangle}"{description}, dashed, from=3-6, to=4-5]
		\arrow["{\kappa_3}"{description}, curve={height=12pt}, dashed, from=1-6, to=2-1]
		\arrow["{\langle \alpha_3,\beta_3 \rangle}"{description}, curve={height=12pt}, dashed, from=3-6, to=4-1]
		\arrow["{g_0}"{description, pos=0.7}, from=2-5, to=2-7, crossing over]
		\arrow["{\varphi_2}"{description, pos=0.3}, two heads, from=2-3, to=4-3, crossing over]
		\arrow["{\varphi_0}"{description, pos=0.3}, two heads, from=2-7, to=4-7, crossing over]
		\arrow["{\varphi_1}"{description, pos=0.3}, two heads, from=2-5, to=4-5, crossing over]
	\end{tikzcd}\]
	Then the induced map $\pi_\infty: E_\infty \to B_\infty$ between the limit types is a two-sided cartesian fibration, and the projection squares constitute two-sided cartesian functors.
	
	Furthermore, $\pi_\infty: E_\infty \to B_\infty$ together with the projection squares satisfies the universal property of a sequential limit~\wrt~to cones of two-sided cartesian functors.
\end{prop}

\begin{proof}
	Using the closure properties from this chapter, we can argue along the lines of~\Cite[Proposition~5.3.12]{BW21}.\footnote{I thank Ulrik Buchholtz for initially suggesting this proof in~\cite{BW21} because it circumvents dealing with unwieldy coherence data that would occur in different presentations of the sequential limit.} This means, the limit fibration, again is constructed via the pullback
	\[\begin{tikzcd}
		{E_\infty} & {} & {\prod_{n:\mathbb N} E_{2n}} & {} \\
		& {\prod_{n:\mathbb N} E_{2n+1}} & {} & {\prod_{n:\mathbb N} E_{2n}} & {} \\
		{A_\infty \times B_\infty} && {\mathllap{\prod_{n:\mathbb N} A_{2n}} \times \prod_{n:\mathbb N} B_{2n}} \\
		& {\prod_{n:\mathbb N} A_{2n+1} \times \prod_{n:\mathbb N} B_{2n+1}} && {\mathllap{\prod_{n:\mathbb N} A_n} \times \prod_{n:\mathbb N} B_n} & {}
		\arrow[dashed, two heads, from=1-1, to=3-1]
		\arrow[two heads, from=1-3, to=3-3]
		\arrow[from=3-1, to=4-2]
		\arrow[shift right=2, dashed, from=1-1, to=2-2]
		\arrow[shorten >=41pt, from=4-2, to=4-4]
		\arrow[shorten >=45pt, from=3-1, to=3-3]
		\arrow[from=3-3, to=4-4]
		\arrow[two heads, from=2-4, to=4-4]
		\arrow[from=1-1, to=1-3]
		\arrow[from=1-3, to=2-4]
		\arrow["\lrcorner"{anchor=center, pos=0.125, rotate=45},draw=none, from=1-1, to=4-4, shift left=3]
		\arrow["\lrcorner"{anchor=center, pos=0.125, rotate=45}, draw=none, from=3-1, to=4-4]
		\arrow[from=2-2, to=2-4, crossing over]
		\arrow[two heads, from=2-2, to=4-2, crossing over]
	\end{tikzcd}\]
	and is two-sided cartesian due to~\Cref{prop:2s-cart-fun-pb-cones}. From~\Cite[Proposition~5.3.12]{BW21} and its dual we get that the projection squares are two-sided cartesian functors. The universal property is established using~\Cite[Proposition~5.3.13]{BW21} and its dual.
\end{proof}

\begin{corollary}[Sequential limit cones in a slice are cocartesian functors]\label{prop:cocart-fun-seqlim-sl}
	Consider an inverse diagram of two-sided cartesian fibrations as below where all of the connecting squares are two-sided cartesian functors:
	\[\begin{tikzcd}
		&& \cdots && {E_\infty} \\
		\cdots & {E_2} && {E_1} && {E_0} \\
		& \cdots \\
		& {} && {A \times B}
		\arrow["{\kappa_0}"{description}, dashed, from=1-5, to=2-6]
		\arrow["{g_1}"{description, pos=0.7}, from=2-2, to=2-4]
		\arrow["{\kappa_1}"{description}, dashed, from=1-5, to=2-4]
		\arrow["{\kappa_2}"{description}, dashed, from=1-5, to=2-2]
		\arrow["{\varphi_1}"{description}, from=2-4, to=4-4]
		\arrow["{\varphi_2}"{description}, from=2-2, to=4-4]
		\arrow["{\varphi_0}"{description}, from=2-6, to=4-4]
		\arrow[dashed, from=1-5, to=4-4, "\varphi_\infty"]
		\arrow[from=2-1, to=2-2]
		\arrow["{g_0}"{description, pos=0.7}, from=2-4, to=2-6, crossing over]
	\end{tikzcd}\]
	Then the induced map $\pi_\infty: E_\infty \to A \times B$ between the limit types is a two-sided cartesian fibration, and the projection squares constitute two-sided cartesian functors.
	
	Furthermore, $\pi_\infty: E_\infty \fibarr A \times B$ together with the projection squares satisfies the universal property of a sequential limit, relativized to the basis $A \times B$.
\end{corollary}

\begin{proof}
	The sequential limit of a diagram of identity maps is the object itself, \eg~
	\[ \seqlim_n \pair{B}{\id_B} \simeq \sum_{\sigma:\N \to B} \prod_{n,k:\N} \sigma(n) = \sigma(n+k) \simeq B.\]
	Thus, the claim follows from~\Cref{prop:cocart-fun-seqlim}.
\end{proof}

\subsection{Cotensors}

\begin{prop}[Cocartesian fibrations are cotensored over maps/shape inclusions]\label{prop:2scart-fun-cotensor-maps}
	Let $P: A \to B \to \UU$ be a two-sided cartesian family with associated projection $\varphi \defeq \pair{\xi}{\pi}:E \fibarr A \times B$. For any type map or shape inclusion $j:Y \to X$, the maps $\varphi^X$ and $\varphi^Y$ are two-sided cartesian fibrations, and moreover the square
	\[\begin{tikzcd}
		{E^X} && {E^Y} \\
		{A^X \times B^X} && {A^Y \times B^Y}
		\arrow["{\varphi^X}"', two heads, from=1-1, to=2-1]
		\arrow[from=2-1, to=2-3]
		\arrow[from=1-1, to=1-3]
		\arrow["{\varphi^Y}", two heads, from=1-3, to=2-3]
	\end{tikzcd}\]
	is a two-sided cartesian functor.
\end{prop}

\begin{proof}
	By closedness under products, the maps $\varphi^X$, $\varphi^Y$ are two-sided cartesian fibrations. From~\cite[Proposition~5.3.15]{BW21}, we know that the square formed by the composites
	\[\begin{tikzcd}
		{E^X} && {E^Y} \\
		{A^X \times B^X} && {A^Y \times B^Y} \\
		{A^X} && {A^Y}
		\arrow["{\varphi^X}"', two heads, from=1-1, to=2-1]
		\arrow[from=2-1, to=2-3]
		\arrow[from=1-1, to=1-3]
		\arrow["{\varphi^Y}", two heads, from=1-3, to=2-3]
		\arrow[two heads, from=2-1, to=3-1]
		\arrow[two heads, from=2-3, to=3-3]
		\arrow["{\xi^X}"{description}, curve={height=40pt}, two heads, from=1-1, to=3-1]
		\arrow["{\xi^Y}"{description}, curve={height=-40pt}, two heads, from=1-3, to=3-3]
		\arrow[from=3-1, to=3-3]
	\end{tikzcd}\]
	is a cocartesian functor (since by precondition $\xi:E \fibarr A$ is a cocartesian fibration). The cartesian case over $B$ works the same.
\end{proof}

\begin{corollary}[Cocartesian fibrations in a slice are cotensored over maps/shape inclusions]\label{prop:2scart-fun-cotensor-maps-sl}
	Let $P: A \to B \to \UU$ be a two-sided cartesian family with associated projection $\varphi \defeq \pair{\xi}{\pi}:E \fibarr A \times B$. For any type map or shape inclusion $j:Y \to X$, the maps $X \iexp \varphi$ and $Y \iexp \varphi$ are two-sided cartesian fibrations, and moreover the triangle
	\[\begin{tikzcd}
		{X \iexp E} && {Y \iexp E} \\
		& {A \times B}
		\arrow["j \iexp \varphi", from=1-1, to=1-3]
		\arrow["{X \iexp \varphi}"', two heads, from=1-1, to=2-2]
		\arrow["{Y \iexp \varphi}", two heads, from=1-3, to=2-2]
	\end{tikzcd}\]
	is a two-sided cartesian functor.
\end{corollary}

\begin{prop}[Cocartesian functors are closed under Leibniz cotensors]\label{prop:2scart-fun-leibniz}
	Let $j: Y \to X$ be a type map or shape inclusion. Then, given two-sided cartesian fibrations $\psi: F \fibarr A \times B$, $\varphi: E \fibarr C \times D$, and a cocartesian functor
	\[\begin{tikzcd}
		F && E \\
		{A \times B} && {C \times D}
		\arrow["\psi"', two heads, from=1-1, to=2-1]
		\arrow["{\pair{k}{m}}"', from=2-1, to=2-3]
		\arrow["\mu", from=1-1, to=1-3]
		\arrow["\varphi", two heads, from=1-3, to=2-3]
	\end{tikzcd}\]
	the square induced between the Leibniz cotensors
	\[\begin{tikzcd}
		{F^X} && {F^Y\times_{E^Y} E^X} \\
		{(A \times B)^X} && {(A \times B)^X \times_{(A \times B)^Y} (C \times D)^X}
		\arrow[two heads, from=1-3, to=2-3]
		\arrow["{j \cotens \mu}", from=1-1, to=1-3]
		\arrow["{j \cotens \langle k,m \rangle }"', from=2-1, to=2-3]
		\arrow["{\psi^X}"', two heads, from=1-1, to=2-1]
	\end{tikzcd}\]
	is a cocartesian functor.
\end{prop}

\begin{proof}
	This works, again, analogously to~\cite[Proposition~5.3.16]{BW21}, using~\Cref{prop:2scart-fun-cotensor-maps}, and then~\Cref{prop:2s-cart-fun-pb-cones}.
\end{proof}

\begin{corollary}[Cocartesian functors in a slice are closed under Leibniz cotensors]\label{prop:2scart-fun-leibniz-sl}
	Let $j: Y \to X$ be a type map or shape inclusion. Then, given two-sided cartesian fibrations $\psi: F \fibarr A \times B$, $\varphi: E \fibarr A \times B$, and a two-sided cartesian functor
	\[\begin{tikzcd}
		F && E \\
		& {A \times B}
		\arrow["\kappa", from=1-1, to=1-3]
		\arrow["\psi"', two heads, from=1-1, to=2-2]
		\arrow["\varphi", two heads, from=1-3, to=2-2]
	\end{tikzcd}\]
	the square induced between the Leibniz cotensors
	\[\begin{tikzcd}
		{X \iexp F} &&&& {Y \iexp F \times_{Y \iexp E} X \iexp F} \\
		&& {A \times B}
		\arrow["{j \cotens_{A \times B} \kappa}", from=1-1, to=1-5]
		\arrow[two heads, from=1-1, to=2-3]
		\arrow[two heads, from=1-5, to=2-3]
	\end{tikzcd}\]
	is a two-sided functor.
\end{corollary}

In sum, we obtain a synthetic analogue of the cosmological closure properties of two-sided cartesian fibrations, \wrt~varying as well as a fixed base:
\begin{theorem}[(Sliced) cosmological closure properties of two-sided cartesian families]\label{thm:2scart-cosm-closure}
	Over Rezk bases, it holds that:
	
	Two-sided cartesian families are closed under composition, dependent products, pullback along arbitrary maps, and cotensoring with maps/shape inclusions. Families corresponding to equivalences or terminal projections are always cocartesian.
	
	Between two-sided cartesian families over Rezk bases, it holds that:
	Two-sided cartesian functors are closed under (both horizontal and vertical) composition, dependent products, pullback, sequential limits,\footnote{all three objectwise limit notions satisfying the expected universal properties \wrt~to cocartesian functors} and Leibniz cotensors.
	
	Fibered equivalences and fibered functors into the identity of $\unit$ are always cocartesian.
	
	Furthermore, all of this is analogously true~\wrt~two-sided cartesian families over the same base and applying sliced versions of the constructions,
\end{theorem}

%% file: two-sided-yon.tex
\subsection{Two-sided cartesian sections}

\begin{definition}[Two-sided cartesian sections]
Let $P:A \times B \to \UU$ be a two-sided family with associated cocartesian fibration $\xi: E \fibarr A$ and cartesian fibration $\pi: E \fibarr B$, resp.

A section $\sigma: \prod_{\substack{a:A \\ b:B}} P(a,b)$ is \emph{two-sided cartesian} if it maps pairs $\pair{u}{\id_b}$ to $\xi$-cocartesian sections and $\pair{\id_a}{v}$ to $\pi$-cartesian sections, \ie:~for all arrows $u:\Delta^1 \to A$, $v:\Delta^1 \to B$ and elements $a:A$, $b:B$ the dependent arrow $\sigma(u,\id_b) : \prod_{t:\Delta^1} P(u(t),b)$ is $\xi$-cocartesian while $\sigma(\id_a,v): \prod_{t:\Delta^1} P(a,v(t))$ is $\pi$-cartesian.

Note that this yields a proposition, and the (sub-)type of such sections is denoted by
\[ \prod_{\substack{a:A\\b:B}}^\tscart P(a,b) \cofibarr  \prod_{\substack{a:A\\b:B}} P(a,b) . \]
\end{definition}

Of central importance will be the following map. We fix a $P:A \times B \to \UU$ a two-sided family, and assume $a:A$ to be initial and $b:B$ to be terminal. We then define
\[ \yon: P(a,b) \to \prod_{A \times B} P, \quad \yon \defeq \lambda d,x,y. (\emptyset_x)_!((!_y)^* d). \]
Note that by two-sidedness of $P$ we have a path
\[ \yon(d)(x,y) = (\emptyset_x)_!((!_y)^* b) = (!_y)^*((\emptyset_x)_! d). \]

In the following, we first show that $\yon$ is, in fact, valued in two-sided cartesian sections. We then conclude that it is a quasi-inverse of the evaluation map, constituting a quasi-equivalence:
\[\begin{tikzcd}
	{\prod_{A \times B}^\tscart P} && {P(a,b)}
	\arrow[""{name=0, anchor=center, inner sep=0}, "{\ev_{\langle a,b \rangle}}"', curve={height=12pt}, from=1-1, to=1-3]
	\arrow[""{name=1, anchor=center, inner sep=0}, "\yon"', curve={height=12pt}, from=1-3, to=1-1]
	\arrow["\simeq"{description}, Rightarrow, draw=none, from=1, to=0]
\end{tikzcd}\]
Finally, the Yoneda Lemmas will follow as instances from this.

\begin{proposition}
Let $P:A \times B \to \UU$ be a two-sided family. Assume $a:A$ is initial and $b:B$ is terminal. Then for all $d:P(a,b)$, the section $\yon\,d:\prod_{A \times B} P$ is two-sided cartesian. 
\end{proposition}

\begin{proof}
This is an extension of \cite[Proposition~7.1.3]{BW21} to the two-sided case. We write $E \defeq \totalty{P}$ and fix an element $d:P(a,b)$. Let $\xi:E \fibarr A$ denote the associated cocartesian fibration. Again, we will only establish one of the two dual parts of the statement, namely that $\yon d(u,\id_y)$ is $\xi$-cocartesian for any $u:\Delta^1 \to A$ and $y:B$

From this, we define the map that yields the cartesian transport along the terminal maps in $B$, \ie
\[ \tau \jdeq \lambda x,y.(!_y)^*d : A \times B \to E. \]
Next, consider the family of cocartesian lifts over the initial maps in $A$, starting at the points given by $\tau$. This is realized by the $2$-cell $\tau:\hom_{A \times B \to E}(\tau\,d, \yon \,d)$ defined by
\[ \tau(x,y) \defeq \xi_!(\emptyset_x,(!_y)^*\,d): (!_y)^*d \cocartarr_{\pair{\emptyset_x}{y}} (\emptyset_x)_!(!_y)^*d. \]
The action of the $2$-cell $\chi$ on a pair $\pair{u}{\id_y}$ for $u:\Delta^1 \to A$ and $y:B$ is given by the following dependent square:
\[\begin{tikzcd}
	E & {\tau d(x,y)} && {\tau d(x',y)} \\
	& {\yon d(x,y)} && {\yon d(x',y)} \\
	{A \times B} & a && a \\
	& x && {x'} \\
	& y && y \\
	& y && y
	\arrow["{\id_{(!_y)^*d}}", Rightarrow, no head, from=1-2, to=1-4]
	\arrow[from=1-2, to=2-2, cocart]
	\arrow["{\yon d(u,\id_y)}"', from=2-2, to=2-4]
	\arrow[from=1-4, to=2-4, cocart]
	\arrow["{\id_a}", Rightarrow, no head, from=3-2, to=3-4]
	\arrow["u"', from=4-2, to=4-4]
	\arrow[dashed, from=3-4, to=4-4, "\emptyset_x"]
	\arrow[dashed, from=3-2, to=4-2, "\emptyset_{x'}", swap]
	\arrow["{\id_y}"', Rightarrow, no head, from=5-2, to=6-2]
	\arrow["{\id_y}", Rightarrow, no head, from=5-2, to=5-4]
	\arrow["{\id_y}"', Rightarrow, no head, from=6-2, to=6-4]
	\arrow[two heads, from=1-1, to=3-1]
	\arrow["{\id_y}", Rightarrow, no head, from=5-4, to=6-4]
\end{tikzcd}\]
By right cancelation of cocartesian arrows, $\yon d(u,\id_y)$ is cocartesian, too.
\end{proof}

We need one more lemma before we are ready to prove the main theorem of this subsection, which in turn will specialize to the desired versions of the Yoneda Lemma. The lemma gives canonical identities in the presence of inital and terminal elements, resp., in the base types.

\begin{lemma}[Coherence of terminal transport with two-sided cartesian sections]\label{lem:coh-ttransp-2scart}
Let $A$ and $B$ types with an initial element $a:A$ and terminal element $b:B$. Furthermore, consider a two-sided family $P:A \times B \to \UU$ with associated cocartesian fibration $\xi: E \fibarr A$ and cartesian fibration $\pi: E \fibarr B$, resp. Given a section $\sigma:\prod_{\substack{a:A\\b:B}}^\tscart P$, for any $x:A$, $y:B$  there are identifications
\[ (!_y)^*\sigma(a,b) = \sigma(a,y) \quad \text{and} \quad (\emptyset_x)_!\sigma(a,b) = \sigma(x,b). \]
\end{lemma}

\begin{proof}
We only treat the first named case since the second named one is completely dual.

Consider on the one hand the cartesian lift of the terminal map $!_y: y \to b$ \wrt~$\sigma(a,b)$ (over the identity $\id_a: a = a$), \ie~the dependent arrow $f \defeq (!_y)^*\sigma(a,b) \cartarr \sigma(a,b)$. On the other hand, consider the action of $\sigma$ on the pair $\pair{\id_a}{!_y}$, namely $h\defeq \sigma(\id_a,!_y) : \sigma(a,y) \cartarr \sigma(a,b)$, which is a cartesian arrow since the section $\sigma$ is two-sided cartesian. Then the mediating induced arrow $g \defeq \tyfill_h(f): \sigma(a,y) \cartarr (!_y)^*\sigma(a,b)$ is cartesian by left cancelation. But since it also $\pi$-vertical, lying over $\id_y$, it is an isomorphism, \cf~\Cref{fig:coh-2ssec}.
\end{proof}
\begin{figure}
\[\begin{tikzcd}
	& {\sigma(a,y)} \\
	E & {(!_y)^*\sigma(a,b)} && {\sigma(a,b)} \\
	& a && a \\
	{A \times B} & y && b
	\arrow["{!_y}", from=4-2, to=4-4]
	\arrow[Rightarrow, no head, from=3-2, to=3-4]
	\arrow["f"', cart, from=2-2, to=2-4]
	\arrow["g"', dashed, "\simeq", from=1-2, to=2-2]
	\arrow["h", cart, from=1-2, to=2-4]
	\arrow[two heads, from=2-1, to=4-1]
\end{tikzcd}\]
\caption{Coherence of terminal transport with two-sided cartesian sections}\label{fig:coh-2ssec}
\end{figure}

In analogy with \cite[Proposition~7.1.4]{BW21}, \cite[Theorem~5.7.18]{RV21}, \cite[Theorem~9.7]{RS17}, the map $\yon: P(a,b) \to \prod_{A \times B}^{\tscart} P$ mediates an equivalence between these two types:
\begin{proposition}\label{prop:fiber-init-term-2s}
Let $A,B$ be Rezk types with an initial element $a:A$ and a terminal element $b:B$. For a two-sided family $P:A \times B \to \UU$, evaluation at $\pair{a}{b}$ given by
\[ \ev_{\pair{a}{b}}: \Big( \prod_{A \times B}^\tscart P\Big) \to P(a,b) \]
is an equivalence.
\end{proposition}

\begin{proof}
We show that for the map $\yon$ as defined above we have identifications $\ev_{\pair{a}{b}} \circ \yon = \id_{P(a,b)}$ and $\yon \circ \ev_{\pair{a}{b}} = \id_{\prod_{A \times B}P}$. The first case is easy: the initial map into the initial element $a$ is just the identity, whose cocartesian lift is an identity as well, and the same holds analogously for the terminal element $b$, \ie
\[ \yon(d)(\sigma)(a,b) = (\emptyset_a)_!(!_b)^*(d) = (\id_a)_!(\id_b)^*(d) = d.\]
For the other round-trip, we have to give an identification
\[ (\yon \circ \ev_{\pair{a}{b}})(\sigma)(x,y) = \sigma(x,y) \]
Fix elements $x:A$ and $y:B$.
 Note that by \Cref{lem:coh-ttransp-2scart} there is a path
 \[ q:\sigma(a,y) =_{P(a,y)} (!_y)^*(\sigma(a,b)).\]
Since $\sigma$ is two-sided cartesian, we obtain the following dependent square:
\[\begin{tikzcd}
	E && {\sigma(a,y)} && {\sigma(x,y)} \\
	&& {(!_y)^*\sigma(a,b)} && {\yon d(\sigma(a,b))(x,y)} \\
	{A \times B} && a && x \\
	&& a && x \\
	&& y && y \\
	&& y && y
	\arrow["{\sigma(\emptyset_x, \id_y)}", from=1-3, to=1-5, cocart]
	\arrow["q"', Rightarrow, no head, from=1-3, to=2-3]
	\arrow["{\xi_!(\emptyset_x,(!_y)^* \sigma(a,b))}"', from=2-3, to=2-5, cocart]
	\arrow["{\emptyset_x}", from=3-3, to=3-5]
	\arrow["{\id_a}"', Rightarrow, no head, from=3-3, to=4-3]
	\arrow["{\id_x}", Rightarrow, no head, from=3-5, to=4-5]
	\arrow["{\id_y}"', Rightarrow, no head, from=5-3, to=6-3]
	\arrow["{\id_y}", Rightarrow, no head, from=5-3, to=5-5]
	\arrow["{\id_y}"', Rightarrow, no head, from=6-3, to=6-5]
	\arrow["{\id_y}", Rightarrow, no head, from=5-5, to=6-5]
	\arrow[two heads, from=1-1, to=3-1]
	\arrow["{\emptyset_x}"', from=4-3, to=4-5]
	\arrow[dashed, "g", from=1-5, to=2-5]
\end{tikzcd}\]
Since the filler $g$ is cocartesian by right cancelation, and vertical at the same time it is an isomorphism, hence an identity $\sigma(x,y) = \yon d(\sigma(a,b))(x,y)$.
\end{proof}

\subsection{Dependent and absolute two-sided Yoneda Lemma}

Following~\cite[Theorem~7.2.3]{BW21}, we obtain the \emph{dependent Yoneda Lemma for two-sided families} by~\Cref{prop:fiber-init-term-2s}, and again this will in turn imply the absolute version.

\begin{theorem}[Dependent Yoneda Lemma for two-sided families]\label{thm:dep-yon-2s}
Let $Q: \comma{a}{A} \times \comma{B}{b} \to \UU$ be a two-sided family over Rezk types $A$ and $B$. For any $a:A$ and $b:$, the evaluation map
\[ \ev_{\id_{\pair{a}{b}}} : \Big( \prod_{\comma{a}{A} \times \comma{B}{b}}^\tscart Q\Big) \rightarrow Q(\id_a,\id_b) \]
is an equivalence.
\end{theorem}

\begin{proof}
Recall that by~\cite[Lemma~8.9]{RS17}, the identity map $\id_a$ is an initial object of the comma type $\comma{a}{A}$, while analogously the identity map $\id_b$ is a terminal object of the cocomma type $\comma{B}{b}$. Thus, the claim follows as an instance of~\Cref{prop:fiber-init-term-2s}. 
\end{proof}

\begin{theorem}[Absolute Yoneda Lemma for two-sided families, \protect{\cite[Theorem~7.3.2]{RV21}}]\label{thm:abs-yon-2s}
	Let $P: A \times B \to \UU$ be a two-sided family over Rezk types $A$ and $B$. For any $a:A$ and $b:B$, the evaluation map
	\[ \ev_{\pair{\id_a}{\id_b}} : \Big( \prod_{\substack{u:\comma{a}{A} \\ v:\comma{B}{b}}}^\tscart P(\partial_1 \, u, \partial_0 \, v)\Big) \rightarrow P(a,b) \]
	is an equivalence.
\end{theorem}

\begin{proof}
	The claim follows by setting $Q \jdeq \pair{\partial_1}{\partial_0}^* P : \comma{a}{A} \times \comma{B}{b} \to \UU$ in~\Cref{thm:dep-yon-2s}. In particular, $Q$ is a two-sided fibration again by pullback stability.	
\end{proof}

%% file: two-sided-disc-fam.tex
\subsection{Definition and characterization}

\begin{definition}[Two-sided discrete families, \protect{\cite[Definition~8.28]{RS17}}]
	Let $P:A \to B \to \UU$ be a two-variable family over Rezk types $A$ and $B$. Then $P$ is a \emph{two-sided discrete family} if for all $a:A$, $b:B$ the family $P_b:A \to \UU$ is covariant and $P^a:B \to \UU$ is contravariant.
\end{definition}

\begin{proposition}[Two-sided discrete families as discrete objects, \cf~\protect{\cite[Prop.~7.2.4]{RV21}}]\label{prop:2s-disc}
	Given $P:A \to B \to \UU$ over Rezk types, the following are equivalent:
	\begin{enumerate}
		\item\label{it:2s-disc-i} The family $P$ is two-sided discrete.
		\item\label{it:2s-disc-ii} The family $P$ is cocartesian on the left and cartesian on the right. Additionally, every bifiber $P(a,b)$ is discrete, for $a:A$, $b:B$.
	\end{enumerate}
\end{proposition}

\begin{proof}
	\begin{description}
		\item[$\ref{it:2s-disc-ii} \implies \ref{it:2s-disc-i}$:] 
		By \Cref{prop:char-fib-cocart-left}, \Cref{it:coc-left-iii}, there is a fibered adjunction which pulls back as follows by~\cite[Proposition~B.2.3]{BW21}:\footnote{In contrast to the current version of~\cite[Proposition~B.2.3]{BW21} one only needs the fibrations involved to be isoinner, and not cocartesian, which in any case becomes clear from the given proof.}
		\[\begin{tikzcd}
			{E_b} &&& E \\
			& {\xi_b\downarrow A} &&& \comma{\xi}{A} \\
			{A \times \mathbf{1}} &&& {A \times B}
			\arrow[from=1-1, to=1-4, shorten <=22pt]
			\arrow[""{name=0, anchor=center, inner sep=0}, dashed, from=1-1, to=2-2]
			\arrow[two heads, from=1-1, to=3-1]
			\arrow[two heads, from=2-2, to=3-1]
			\arrow["{\id_A \times b}"{description}, from=3-1, to=3-4]
			\arrow[two heads, from=1-4, to=3-4]
			\arrow[two heads, from=2-5, to=3-4]
			\arrow[""{name=1, anchor=center, inner sep=0}, "\iota"{description}, from=1-4, to=2-5]
			\arrow[""{name=2, anchor=center, inner sep=0}, "\tau"{description}, curve={height=18pt}, dashed, from=2-5, to=1-4]
			\arrow[""{name=3, anchor=center, inner sep=0}, curve={height=12pt}, dashed, from=2-2, to=1-1, curve={height=18pt}]
			\arrow["\lrcorner"{anchor=center, pos=0.125}, draw=none, from=2-2, to=3-4]
			\arrow["\lrcorner"{anchor=center, pos=0.125}, shift right=5, draw=none, from=1-1, to=3-4]
			\arrow["\dashv"{anchor=center, rotate=-137}, draw=none, from=3, to=0]
			\arrow["\dashv"{anchor=center, rotate=-133}, draw=none, from=2, to=1]
			\arrow[from=2-2, to=2-5, crossing over]
		\end{tikzcd}\]
		This means exactly that $P^a$ is covariant. The analogous reasoning establishes the claim for $P_b$ being covariant since $P$ is cartesian on the right, for any $b:B$.
		Now, since any $P(a,b)$ is discrete, and the fibers of $P_b$ are given by $P(a,b)$ for any $a:A$, we obtain that $P^a:B \to \UU$ is a cocartesian family with discrete fibers, which is equivalent to $P^a$ being covariant by~\cite[Corollary~6.1.4]{BW21}.
		\item[$\ref{it:2s-disc-i} \implies \ref{it:2s-disc-ii}$:] The fibrations $P_a$ and $P^b$, resp., being contra- and covariant, resp., imply that all bifibers $P(a,b)$ are discrete, for all $a:A$, $b:B$.
		
		Furthermore, for any $b:B$, the family $P^b: A \to \UU$ being cocartesian means that $P^B :A \to \UU$ is cocartesian, and in addition all $P^B$-cocartesian lifts are $P_A$-vertical, \ie~lie over an identity in $B$.\footnote{This could have also been used as a more direct argument to prove ``$\ref{it:2s-disc-ii} \implies \ref{it:2s-disc-i}$'' as well.}
	\end{description}
\end{proof}

\begin{corollary}[Co-/cart.~arrows and two-sided cart.~functors, \cf~\protect{\cite[Lemma~7.4.3]{RV21}}]
	In a two-sided discrete family $P:A \to B \to \UU$ an arrow is $P_B$-cocartesian if and only if it is $P^A$-vertical. Similarly, an arrow is $P^A$-cartesian if and only if it is $P_B$-vertical.
	In particular, any two-sided discrete cartesian family is two-sided cartesian, and any fibered functor between two-sided discrete families is two-sided cartesian.
\end{corollary}

\begin{proof}
	The first statement follows from the inspection in the proof of~\Cref{prop:2s-disc}, together with~\cite[Proposition~6.1.5]{BW21}.
	
	This also establishes that any two-sided discrete cartesian family is, in fact, two-sided cartesian: the commutation condition~\Cref{prop:comm-lifts} is readily verified, because for dependent arrows being vertical (in the respective appropriate sense) is already sufficient for being co-/cartesian, resp.
	
	By naturality, this implies that any fibered functor between two-sided discrete families is two-sided cartesian.
\end{proof}

%% file: univ-ttmt.tex
In \cite{Shu19}, Shulman provided the final missing part to establish the far-reaching conjecture that any Grothendieck $\inftyone$-topos gives rise to a model of homotopy type theory. Indeed, he proved that any such higher topos can be presented by a model structure in which provides enough univalent universes, validating propositional resizing, and which are \emph{strictly} closed under all desired type formers.\footnote{with more general kinds of higher inductive types being work in progress} This had previously been established for certain classes of special cases.~\cite{ShuInv,ShuReedy,ShuInvEI,CisUniv} More generally, work connected to ``internal languages'' of higher categories is found in~\cite{AW05,GepnerKock,KLHtypTT,KapulkinSzumiloIntLang,KapulkinLCCQcat}.

Given a type-theoretic model topos $\mathscr E$, the internal presheaf category $\sE := [{\Simplex}^{\Op}, \mathscr E]$ of simplicial objects in $\mathscr E$ is again a type-theoretic model topos.\footnote{with the injective model structure~\cite[Corollary~8.29]{Shu19}} These are the ``standard models'' for simplicial homotopy type theory. Inside any such $\sE$ we find a copy of the TTMT of spaces $\mathscr S$, which is embedded fibrantly and spatially-discretely via the constant diagram functor
\[\begin{tikzcd}
	{\mathbf{Set}} && {\mathscr E} & \rightsquigarrow & {\mathscr S} && {\sE  \mathrlap{=\mathscr E^{\Simplex^{\Op}}}}
	\arrow[""{name=0, anchor=center, inner sep=0}, "{\Delta_\mathscr E}", curve={height=-12pt}, hook, from=1-1, to=1-3]
	\arrow[""{name=1, anchor=center, inner sep=0}, "{\Gamma_{\mathscr E}}", curve={height=-12pt}, from=1-3, to=1-1]
	\arrow[""{name=2, anchor=center, inner sep=0}, "{\Delta_{\sE}}", curve={height=-12pt}, hook, from=1-5, to=1-7]
	\arrow[""{name=3, anchor=center, inner sep=0}, "{\Gamma_{\sE}}", curve={height=-12pt}, from=1-7, to=1-5]
	\arrow["\dashv"{anchor=center, rotate=-90}, draw=none, from=2, to=3]
	\arrow["\dashv"{anchor=center, rotate=-90}, draw=none, from=0, to=1]
\end{tikzcd}\]
\ie~the map given by
\[ \Delta_{s\mathscr E}: \mathscr S \hookrightarrow \mathscr E, \quad X \mapsto \Big([n] \mapsto \coprod_{X_n} \mathbf 1\Big).\]

We adapt the coherence construction used in \cite{KL18,LW15,Awo18,StrRep,StrSimp,LS-HIT}, based on \cite{VVTySys,VVPi}, so that in the model extension types \`{a} la \cite{RS17}, going to back to earlier work by Lumsdaine and Shulman, can be chosen in a way that is strictly stable under pullback.

Our presentation is inspired by \cite{KL18,Awo18} and also \cite{LW15} in method and style. We define the type formers and the data required from the rules by performing the corresponding constructions in suitable \emph{generic contexts}. Substitution then corresponds to precomposition with the reindexing map, and applying type formers corresponds to postcomposition with an ensuing map between universes. This implies strict substitution stability on the nose, in particular avoiding choices of pullbacks (and dealing with the coherences that would ensue).

This is done with respect to just a single universe as~\eg~in \cite{Awo18,StrSimp,StrRep}. Because extension types are similar to $\Pi$-types, the splitting is also analogous to the method for $\Pi$-types from the aforementioned sources.\footnote{Since the same ideas have successfully generalized from \emph{global} to \emph{local universes} \cite{LW15,Shu19}, we believe the same generalization would also work for strict extension types (which we do not consider in the present text, however). In particular, a formulation in (an appropriate extension of) Shulman's setting~\cite[Appendix~A]{Shu19} should be possible.}

We briefly recall Shulman's setting and results that we are building on, although we will rarely need to be explicit about the machinery under the hood from here on.

A \emph{type-theoretic model topos (TTMT)} $\mathscr E$ is given by a Grothendieck $1$-topos $\mathscr E$, together with the structure of a right proper simplicial Cisinski model category, which is in addition simplicially cartesian closed (\ie~reindexing preserves simplicial copowers) and is also equipped with an appropriate notion of ``type-theoretic'' fibration.\footnote{in Shulman's terminology, a \emph{locally presentable} and \emph{relatively acyclic} \emph{notion of fibration structure}, satisfying in addition a ``fullness'' condition} An immediate example is given by the prime model of homotopy type theory, the type-theoretic model topos of spaces $\mathscr S$, but there are many more examples. In particular, the class of TTMTs is closed under typical operations, such as slicing and small products, as well as passage to diagram categories and internal localizations. Moreover, any Grothendieck--Rezk--Lurie-$\inftyone$-topos can be presented by a type-theoretic model topos. As a consequence, the study of Rezk types in simplicial HoTT can be understood as a synthetic version of internal $\infty$-category theory in an arbitrary given $\inftyone$-topos.

Crucially, any TTMT hosts enough\footnote{bounded from below by cardinal} strict fibrant univalent universes~\cite[Theorem~5.22, Theorem~11.2]{Shu19}.\footnote{as usual, in a strong enough meta-theory such as ZFC with inaccessibles} Abstractly from the given fibration structure, the universes can be constructed using a version of the small object argument~\cite[Theorem~5.9]{Shu19}. Then a splitting method due to Voevodsky~\cite{VVTySys,KL18} akin to Giraud's~\emph{left adjoint splitting}~\cite{Gir74,LW15,StrRep} is applied to obtain \emph{strict} type formers for these universes.

%% file: coh-ext.tex
\subsection{General ideas}

We adapt the coherence construction presented in \cite{KL18,LW15,Awo18}, based on \cite{VVTySys,VVPi}, so that in the model extension types \`{a} la \cite{RS17} can be chosen in a way that is strictly stable under pullback.

Our presentation is close to \cite{KL18,Awo18} in style. We define the type formers and the data required from the rules by performing the corresponding constructions in suitable \emph{generic contexts}. Substitution then corresponds to precomposition with the reindexing map. Applying type formers/operations corresponds to postcomposition with an ensuing map between universes. This implies strict substitution stability on the nose.

We work \wrt~just a single universe, as is also done in \cite{Awo18}. The extension types of \cite[Figure~4]{RS17} are both syntactically and semantically similar to $\Pi$-types: Intuitively, the extension type $\exten{\psi}{A}{\varphi}{a}$ consists of sections $b:\prod_\psi A$ that on the subshape $\varphi \cofibarr \psi$ coincide with the given section $a$ in the sense of \emph{judgmental} equality. Hence, they can be seen as a kind of ``$\Pi$-types with side conditions'' (for functions whose domain is a tope, \emph{a priori} a \emph{pre}-type).

Since in the case of ordinary $\Pi$-types the splitting method has been generalized from \emph{global} to \emph{local universes} \cite{LW15,Shu19}, we believe the same should also work for strict extension types (which we do not consider in the present text, however). This has also been claimed in~\cite[Appendix~A.2]{RS17}.

In particular, a formulation in (an appropriate extension of) Shulman's setting~\cite[Appendix~A]{Shu19} should then be possible.

In contrast to most of~\cite[Appendix~A]{RS17} we will work completely internally to the topos, disregarding the extra ``non-fibrant'' structure made explicit in a \emph{comprehension category with shapes}~\cite[Definition~A.5]{RS17}.

This is justified, because shapes are reflected into types by a newly added rule, \cf~\Cref{ssec:fib-shapes}. In the model, this is validated because the simplicial shapes turn up in the model as (spatially-discrete) fibrant objects, \cf~\Cref{sec:sdiag-models}.

In sum, this allows us to focus the presentation on the splitting of the extension types, without having to deal with also splitting the extra layers, introducing universes for ``cofibrations'' or similar.\footnote{Even though all of this should be possible if desired, \cf~also the remarks about strict stability in~\cite[Appendix~A.2]{RS17}.}

We will work in a classical meta-theory with global choice and universes. Sometimes we use the internal extensional type theory (ETT) of the topos for an easy-to-parse denotation of the objects presenting generic contexts. This is along the lines of the presentation in~\cite{Awo18} and also~\cite{KL18,LW15}.

As in~\cite[Theorem~A.16]{RS17}, in the (injective) model structure of the ttmt presenting $\sE$ we interpret the extension types as ordinary $1$-categorical pullbacks of Leibniz cotensors. Since the model structure at hand is type-theoretic~\cite[Definition~2.12]{ShuInv} the Leibniz cotensors are fibrations, so are all their pullbacks, which in particular includes the (dependent) extension types.\footnote{In the case of non-fibrant shapes, the Leibniz cotensor is still a fibration because the model structure is cartesian monoidal, \cf~\cite[Section~1, Section~A.2]{RS17}.}

Specifically, in the non-dependent case the extension type $\ndexten{\psi}{A}{\varphi}{a}$ is given simply by (ordinary, $1$-categorical) pullback:
\[\begin{tikzcd}
	{\ndexten{\psi}{A}{\varphi}{a}} && {A^\psi} \\
	\unit && {A^\varphi}
	\arrow[from=1-1, to=2-1]
	\arrow["a"', from=2-1, to=2-3]
	\arrow[from=1-1, to=1-3]
	\arrow[from=1-3, to=2-3]
	\arrow["\lrcorner"{anchor=center, pos=0.125}, draw=none, from=1-1, to=2-3]
\end{tikzcd}\]
For the dependent case the construction will have to be relativized as we will see later on.

We want to show that the interpretation of the extension types can be chosen in a strictly pullback-stable way. In particular, after~\cite[Definition~A.10, Theorem~A.16]{RS17} given a type family $A \fibarr \Gamma \times \psi$ and a partial section $a:\Gamma \times \varphi \to_{\Gamma \times \psi} A$ substitution of the extension type\footnote{Here and in the following, notation such as $\Pi(A,B)$ or $\ccexten{\varphi}{A}{\varphi}{a}$ instead of $\prod_A B$, $\exten{\psi}{A}{\varphi}{a}$  designates the (strictly stable) type formers (to be) defined in the interpretation, in line with~\cite{RS17,KL18,KL18,Awo18}.} $\ccexten{\varphi}{A}{\varphi}{a}$ is only considered along (type) context morphisms $\sigma: \Delta \to \Gamma$, leaving the shape inclusion $j: \varphi \cofibarr \psi$ fixed. Hence, we will understand the type former of the extension type as a \emph{family} of type formers $\ExtAt{j}$, given by
\[ \ExtAt{j}(A,a) := \ccexten{\psi}{A}{\varphi}{a}\]
indexed by the shape inclusions $j$.\footnote{Note that, \emph{a posteriori} this also yields stability \wrt~to reindexing along cube context morphisms $\tau:J \to I$ pulling back shape inclusions $j:\varphi \to_I \psi$: By their defining universal property, which only involves a condition on reindexings of the type context $\Gamma$, the (pseudo-stable) extension types from~\cite[Definition~A.10]{RS17} are determined uniquely up to isomorphism. The splitting then yields uniqueness up to equality. Then, considering reindexings along shape maps does not lead out of this class, so the choice remains strictly stable.}

\begin{theorem}[Strict stability of extension types]\label{thm:coh-ext}
	Let $\E$ be a type-theoretic model topos, and consider the interpretation of simplicial homotopy type theory \`{a} la \cite{RS17,Shu19} in the type-theoretic model topos $\sE$. Then, a fixed splitting of the standard type formers induced by the universal fibration $\pi: \widetilde{\UU} \to \UU$ for small fibrations yields an interpretation of the extension types $\ExtAt{j}$ (where $j:\varphi \to_I \psi$ is a fixed shape inclusion relative to some cube $I$) that is strictly stable under pullback. This means, for families $A: \Gamma \times \psi \to \UU$ and partial sections $a:\Gamma \times \varphi \to_{\Gamma \times \psi} A$ we have: 
	\begin{align*}
		\sigma^*\big(\ExtAt{j}(A,a)\big) & = \ExtAt{j}(\sigma^*A,\sigma^*a) \\
		\sigma^*\big(\LambdaAt{j}_{A,a}(b)\big) & = \LambdaAt{j}_{\sigma^*A,\sigma^*a}(\sigma^*b) \\
		\sigma^*\big(\appAt{j}_{A,a}(f,s)\big) & = \appAt{j}_{\sigma^*A,\sigma^*a}(\sigma^*f,\sigma^*s)
	\end{align*}
\end{theorem}

The rest of this section lays out the splitting of the extension type formers, assuming all the necessary and the pre-established logical and model-categorical present in the background.

\subsection{Global universe splitting}

Depending on the precise setup---in particular with extra shape layers---one might have to use~\eg~the local universe method~\cite{LW15} to strictify the \emph{pseudo}-stable tope logic beforehand~\cite[Definition~A.7, Remark~A.8]{RS17}. But we will neglect this here, as mentioned before, since we are staying inside the given model structure on the topos, so it is enough just to a consider a splitting of this.

A family in simplicial type theory is modeled by a composition
\[ A \fibarr \Gamma \fibarr \Phi \cofibarr \Xi, \]
where $\Phi \cofibarr \Xi$ plays the role of a shape inclusion. Let $\pi:\totalty{\UU} \to \UU$ be the type-theoretic universe in $\sE$ (strictly \`{a} la Tarski, closed under all standard type formers) which classifies small Reedy fibrations, which exists by~\cite[Corollary~8.29]{Shu19}. This means in particular we are considering the injective model structure on internal presheaves, which in the case of the base $\Simplex$ coincides with the Reedy model structure, \cf~also \cite{CisUniv,ShuReedy}.

We just have to split \wrt~the universe $\pi:\totalty{\UU} \to \UU$. This is done after Streicher~\cite[Appendix~C]{StrRep} and Voevodsky~\cite{VVTySys,VVPi} in the following way:

Using meta-theoretic global choice, for each family $A : \Gamma \to \UU$, we select a distinguished square
\[\begin{tikzcd}
	{\Gamma.A} && {\widetilde{\UU}} \\
	\Gamma && \UU
	\arrow["{p_A}"', from=1-1, to=2-1]
	\arrow["A"', from=2-1, to=2-3]
	\arrow["{q_A}", from=1-1, to=1-3]
	\arrow["\pi", from=1-3, to=2-3]
	\arrow["\lrcorner"{anchor=center, pos=0.125}, draw=none, from=1-1, to=2-3]
\end{tikzcd}\]

A map from $B: \Delta \to \UU$ and $A: \Gamma: \Gamma \to \UU$ is given by a square:
\[\begin{tikzcd}
	{\Delta.B} && {\Gamma.A} \\
	\Delta && \Gamma
	\arrow["{p_B}"', from=1-1, to=2-1]
	\arrow["\sigma"', from=2-1, to=2-3]
	\arrow["{q_\sigma}", from=1-1, to=1-3]
	\arrow["{p_A}", from=1-3, to=2-3]
\end{tikzcd}\]
which is split cartesian if and only if $B = A \circ \sigma$ and $q_B = q_A \circ q_\sigma$:
\[\begin{tikzcd}
	{\Delta.B} & {} & {\Gamma.A} && {\widetilde{\UU}} \\
	\Delta && \Gamma && \UU
	\arrow["{p_B}", from=1-1, to=2-1]
	\arrow["\sigma", from=2-1, to=2-3]
	\arrow["A", from=2-3, to=2-5]
	\arrow["{q_\sigma}", dashed, from=1-1, to=1-3]
	\arrow["{p_A}", from=1-3, to=2-3]
	\arrow["{q_A}", from=1-3, to=1-5]
	\arrow["\pi", from=1-5, to=2-5]
	\arrow["\lrcorner"{anchor=center, pos=0.125}, draw=none, from=1-3, to=2-5]
	\arrow["{q_B}", curve={height=-18pt}, from=1-1, to=1-5]
	\arrow["B"', curve={height=18pt}, from=2-1, to=2-5]
	\arrow["\lrcorner"{anchor=center, pos=0.125}, shift left=5, draw=none, from=1-2, to=2-5]
\end{tikzcd}\]
This then necessarily implies that the left-hand side is a pullback as well.
Since the identites involved are strict equalities between objects and morphisms involving distinguished squares, this models substitution up to equality on the nose by defining $\sigma^*A := A[\sigma] := A \circ \sigma$.

When defining the generic contexts \`{a} la~\cite{KL18}, we often define them in the internal extensional type theory (ETT) of the topos, as done in~\cite{Awo18,LW15}. In particular, given a map $f:B \to A$ the endofunctor defined via the adjoint triple
\[ \sum_f \dashv f^* \dashv \prod_f\]
as
\[ P_f(X) = \sum_{A \to 1} \Big(\prod_f (B^*(X))\Big)  \]
in the internal language ETT reads as
\[ P_f(X) = \sum_{a:A} X^{B_a},\]
where $B_a = f^*a$.

Moreover, for $A: \Gamma \to \UU$ we allow ourselves to abbreviate $A := \Gamma.A = \widetilde{\UU}_A = \pi^*A$, as commonly done.

For the universal fibration $\pi: \widetilde{\UU} \to \UU$, the total object denotes the generic context
\[ \widetilde{\UU} = \sum_{a:\UU} A =  \sem{A:\UU, a:A}.\]

We show how to define a strictly stable choice of extension types \`{a} la~\cite[Proposition~2.4]{Awo18}. To illustrate the analogy, we give a brief recollection how to define the generic contexts for the formation and introduction rule.

\subsection{Recollection: Strictly stable $\Pi$-types}

The generic context for the $\Pi$-formation rule is given, in the internal language, by the type
\[ \UU^\Pi := \sum_{A:\UU} \UU^A = \sem{A:\UU, B:A \to \UU},\]
and the generic context for the introduction rule is given by
\[ \UU^\lambdaAt{j} := \sum_{A: \UU } \sum_{B:A \to \UU} \prod_{a:A} B(a) = \sem{A:\UU, B:A \to \UU, t:(a:A) \to B(a)},\]
which captures the \emph{generic dependent term}.

The idea is that a map $\pair{A}{B}:\Gamma \to \UU^\Pi$ precisely captures the input data for the $\Pi$-type former.

The universe then admits $\Pi$-types (\cf~\emph{$\Pi$-structure}, \cf~\cite[Definition~1.4.2, Theorem~1.4.15]{KL18}) if there exist maps $\Pi: \UU^\Pi \to \UU$, $\lambda: \UU^\lambda \to \UU$ 
making the square
\[\begin{tikzcd}
	{\UU^\lambda} && {\widetilde{\UU}} \\
	{\UU^\Pi} && {\UU}
	\arrow[from=1-1, to=2-1]
	\arrow["\Pi"', from=2-1, to=2-3]
	\arrow["\lambda", from=1-1, to=1-3]
	\arrow[from=1-3, to=2-3]
\end{tikzcd}\]
commute, and moreover rendering as a pullback.
These maps implement the formation and introduction rule, hence we take the maps
\[ \Pi:\UU^\Pi \to \UU, ~ \Pi(A,B) := \prod_A B, \quad  \lambda:\UU^\Pi \to \UU, ~ \lambda(A,B,t) :=\Big \langle \prod_A(B), t \Big \rangle\]
induced from the structure of the ambient category as a type-theoretic model category.\footnote{Concretely, for fibrations $p_A: \Gamma.A \to \Gamma$, $p_B: \Gamma.A.B \to \Gamma.A$ we define $\Pi(A,B) := \prod_{p_A}(p_B)$ via the pushforward functor $\prod_{p_A} \vdash (p_A)^*$ which in this setting preserves fibrations.}

In particular, for the formation rule, strict stability under pullback follows from strict commutation of the diagram
\[\begin{tikzcd}
	&& \Delta \\
	\\
	\Gamma && {\UU^\Pi} && \UU
	\arrow["{\pair{A}{B}}"{description}, from=3-1, to=3-3]
	\arrow["\Pi"{description}, from=3-3, to=3-5]
	\arrow["\sigma"', from=1-3, to=3-1]
	\arrow["{\prod_{\sigma^*A}\sigma^*B}", from=1-3, to=3-5]
	\arrow["{\pair{\sigma^*A}{\sigma^*B}}"{description}, from=1-3, to=3-3]
	\arrow["{\prod_A B}"{description}, curve={height=24pt}, from=3-1, to=3-5]
\end{tikzcd}\]
as elaborated in~\cite{Awo18}.

We now treat the extension types from \cite{RS17} in a similar fashion.

\subsection{Strictly stable extension types: Generic contexts}
Recall the rules from~\cite[Figure~4]{RS17}.

To form an extension type, we start with a context fibered-over-shapes
\[ \Gamma \fibarr \Phi \cofibarr \Xi\]
and a \emph{separate}\footnote{The semantic reason for this is explained right before~\cite[Theorem~A.17]{RS17}. \Cf~also the explanation about the rules at the beginning of~\cite[Section~2.2]{RS17}} shape inclusion
\[ \psi \cofibarr \varphi \cofibarr I.\]
The extension type ist then formed for a pair $\pair{A}{a}$ of a type and a partial section
\begin{equation}\label{eq:input-ext-form}
	A: \Gamma \times \psi \fibarr \Phi \times \psi \cofibarr \Xi \times I, \quad a: \Gamma \times \varphi \to_{\Gamma \times \psi} A
\end{equation}
In fact, the defining universal property for the extension types \cite[Definition~A.10]{RS17} asks for substitution (pseudo-)stability of $A$ and $a$ along morphisms $\sigma:\Delta \to \Gamma$. Hence, we in fact consider a \emph{family} of type formers $(\ExtAt{j})_{j:\varphi \cofibarr_I \psi}$ indexed externally by the shape inclusions $j:\varphi \to \psi$, so that
\[ (\ExtAt{j})(A,a) = \ccexten{\psi}{A}{\varphi}{a}.\]
Thus, the input data to form the extension type $\ExtAt{j}$ should be represented using a suitable generic context $\UU^{\ExtAt{j}}$ as a morphism
\[ \pair{A}{a}: \Gamma \to \UU^{\ExtAt{j}} \]
with $\pair{A}{a}$ as in~\eqref{eq:input-ext-form}.\footnote{Note that we are suppress the further structure $\Gamma \fibarr \Phi \cofibarr \Xi$ here, in line with~\cite[Deftinition~A.10]{RS17}.}
In analogy to the case of $\Pi$-types, we form the generic contexts by
\begin{align}
	\UU^{\ExtAt{j}} &:= \sum_{A:\UU^\psi} j^*A, \\
	\UU^{\lambdaAt{j}} & := \sum_{A:\UU^\psi} \sum_{a:j^*A} a^*(A^\psi).
\end{align}

\subsection{Strictly stable extension types: Formation and introduction}

The formation and introduction rule are, top to bottom, Rule 1 and 2, resp., in~\cite[Figure~4]{RS17}.

Then, at stage $\Gamma$ the object $\UU^{\LambdaAt{j}}$ consists of triples $\angled{A,a,b}$ as in
\[\begin{tikzcd}
	&& A \\
	\\
	{\Gamma \times \varphi} && {\Gamma \times \psi}
	\arrow["{\Gamma \times j}"', from=3-1, to=3-3]
	\arrow[two heads, from=1-3, to=3-3]
	\arrow["a", dotted, from=3-1, to=1-3]
	\arrow["b"', curve={height=18pt}, dotted, from=3-3, to=1-3]
\end{tikzcd}\]
\ie~strictly $a = b \circ (\Gamma \times j)$. Considering the projection $\UU^{\LambdaAt{j}} \to \UU^{\ExtAt{j}}$ the fiber at an instance $\pair{A}{a}$ is exactly the semantic extension type
\[ \ExtAt{j}(A,a):=\ccexten{\psi}{A}{\varphi}{a} \]
defined by the \emph{split}\footnote{\ie~$\ccexten{\psi}{A}{\varphi}{a}$ and its projection to $\Gamma$ have been \emph{chosen}} cartesian square, after~\cite[Theorem~A.16]{RS17}:
\[\begin{tikzcd}
	\ccexten{\psi}{A}{\varphi}{a}  && {A^\psi} \\
	\\
	\Gamma && {A^\varphi \times_{(\Gamma \times \psi)^\varphi} (\Gamma \times \psi)^\psi}
	\arrow["{\pair{a}{\eta} }", from=3-1, to=3-3]
	\arrow["\lrcorner"{anchor=center, pos=0.125}, draw=none, from=1-1, to=3-3]
	\arrow[two heads, from=1-1, to=3-1]
	\arrow[two heads, from=1-3, to=3-3]
	\arrow[from=1-1, to=1-3]
\end{tikzcd}\]
In particular, the right vertical map is a fibration,\footnote{If shapes are taken to be fibrant, this already follows from type-theoretic-ness of the model structure. Otherwise one would have to use that the model structure is cartesian monoidal as in~\cite[Lemma~A.4]{RS17}.} so the object $\ExtAt{j}(A,a)$ is fibrant.

Note that in the extensional type theory of the ambient topos one could describe the extension type also as the $\Sigma$-type
\[ \sum_{b:\psi \to A} (b\circ j \jdeq a),\]
where $(-_1) \jdeq (-_2)$ stands for the extensional identity type.

Now, in analogy to \cite[Theorem~1.4.15]{KV18}, we define the map $\ExtAt{j}:\UU^{\lambdaAt{j}} \to \widetilde{\UU}$ by the universal property of the extension type. The generic context $\UU^{\lambdaAt{j}}$ consists of $\angled{A,a,b}$~s.t.
\[\begin{tikzcd}
	&& {A^\psi} \\
	\\
	\Gamma && {A^\varphi \times_{{\Gamma \times \psi}^\varphi} (\Gamma \times \psi)^\psi}
	\arrow["b", from=3-1, to=1-3]
	\arrow["{\pair{a}{\eta} }"', from=3-1, to=3-3]
	\arrow[two heads, from=1-3, to=3-3]
\end{tikzcd}\]
where $\eta: \Gamma \to (\Gamma \times \psi)^\psi$ denotes the transpose of the identity of $\Gamma \times \psi$. Now, we define the generic $\lambda$-term of the extension type as the gap map $(\LambdaAt{j})(b) := \widehat{b}$ of the pullback:
\[\begin{tikzcd}
	\Gamma \\
	&& \ccexten{\psi}{A}{\varphi}{a} && {A^\psi} \\
	\\
	&& \Gamma && {A^\varphi \times_{{\Gamma \times \psi}^\varphi} (\Gamma \times \psi)^\psi}
	\arrow[two heads, from=2-5, to=4-5]
	\arrow[from=2-3, to=2-5]
	\arrow[two heads, from=2-3, to=4-3]
	\arrow["{\pair{a}{\eta} }", from=4-3, to=4-5]
	\arrow["b", curve={height=-18pt}, from=1-1, to=2-5]
	\arrow["{\widehat{b}}"{description}, dashed, from=1-1, to=2-3]
	\arrow[curve={height=18pt}, Rightarrow, no head, from=1-1, to=4-3]
	\arrow["\lrcorner"{anchor=center, pos=0.125}, draw=none, from=2-3, to=4-5]
\end{tikzcd}\]
Then, the square
\[\begin{tikzcd}
	{\UU^\LambdaAt{j}} &&& \widetilde{\UU} \\
	\\
	{\UU^\ExtAt{j}} &&& \UU
	\arrow[from=1-1, to=3-1]
	\arrow["{\ExtAt{j}}"', from=3-1, to=3-4]
	\arrow["{\LambdaAt{j}}", from=1-1, to=1-4]
	\arrow[from=1-4, to=3-4]
	\arrow["\lrcorner"{anchor=center, pos=0.125}, draw=none, from=1-1, to=3-4]
	\arrow["\lrcorner"{anchor=center, pos=0.125}, draw=none, from=1-1, to=3-4]
\end{tikzcd}\]
being a pullback precisely captures the universal property as dicussed in~\cite[Definition~A.10, Theorem~A.16]{RS17}. In the terminology of \cite[Theorem~2.3.4]{KL18} this means that the projection $\UU^{\LambdaAt{j}} \to \UU^{\ExtAt{j}}$ is the \emph{universal dependent extension type} over $\UU^{\ExtAt{j}}$. In particular, in presence of the splitting it is a \emph{chosen} small fibration. Thus, as illustrated \eg~in \cite[Remark~2.6]{Awo18} we obtain strict pullback stability: everything in
\[\begin{tikzcd}
	&& \Delta \\
	\\
	\Gamma && {\UU^{\ExtAt{j}}} && \UU
	\arrow["{\pair{A}{a}}"{description}, from=3-1, to=3-3]
	\arrow["{\ExtAt{j}}"{description}, from=3-3, to=3-5]
	\arrow["\sigma"', from=1-3, to=3-1]
	\arrow["{(\ExtAt{j})(\sigma^*A,\sigma^*a)}", from=1-3, to=3-5]
	\arrow["{\pair{\sigma^*A}{\sigma^*a}}"{description}, from=1-3, to=3-3]
	\arrow["{(\ExtAt{j})(A,a)}"{description}, curve={height=24pt}, from=3-1, to=3-5]
\end{tikzcd}\]
commutes strictly on the nose because we have split the model structure from the get-go, yielding as desired
\[ \ExtAt{j}(A,a) \circ \sigma = \ExtAt{j}(A\sigma, a\sigma). \]

Similarly, and using that the generic lifts $b$ are given by gap maps of (strict) pullbacks we obtain
\[ \lambdaAt{j}(A,a,b) \circ \sigma = \ExtAt{j}(A\sigma, a\sigma, b\sigma). \]

\subsection{Strictly stable extension types: Elimination and computation}

The elimination and computation rule are Rule 3 and 4, resp., in~\cite[Figure~4]{RS17}.

To interpet the elimination rule, note first that the exponential $A^\psi$ comes with a (chosen) evaluation map:
\[\begin{tikzcd}
	{A^\psi \times_\Gamma (\Gamma \times \psi)} && A \\
	& \Gamma
	\arrow["{\mathrm{ev}}", from=1-1, to=1-3]
	\arrow[from=1-1, to=2-2]
	\arrow[from=1-3, to=2-2]
\end{tikzcd}\]
By pulling this back along the (chosen) map from the extension type, we obtain a map
\[ \mathrm{ev}: \ExtAt{j}(A,a) \times_\Gamma (\Gamma \times \psi) \to_\Gamma A,\]
fibered over $\Gamma$. Now,we want to define application $\appAt{j}$ of a function $f:\Gamma \to_\Gamma \ExtAt{j}(A,a)$ to a section $s: \Gamma \to_\Gamma \Gamma \times \psi$. Analogously to~\cite[Theorem~1.4.15]{KL18}, we take the composition:
\[\begin{tikzcd}
	\Gamma && {\ExtAt{j}(A,a) \times_\Gamma (\Gamma \times \psi)} && A \\
	&& \Gamma
	\arrow[from=1-3, to=2-3]
	\arrow["{\langle f,s \rangle}", from=1-1, to=1-3]
	\arrow["{\mathrm{ev}}", from=1-3, to=1-5]
	\arrow[Rightarrow, no head, from=1-1, to=2-3]
	\arrow[from=1-5, to=2-3]
	\arrow["{\appAt{j}(f,s)}", curve={height=-40pt}, from=1-1, to=1-5]
\end{tikzcd}\]
This validates the rule, since $\appAt{j}(f,s):\Gamma \to_\Gamma A$, as desired.
For substitution along $\sigma: \Delta \to \Gamma$, consider the following diagram involving chosen fibrations and split cartesian squares:
\[\begin{tikzcd}
	{\ExtAt{j} (\sigma^*A,\sigma^*a) \mathrlap{\times_\Delta(\Delta \times \psi)}} &&& {\ExtAt{j} (A,a) \times_\Gamma(\Gamma \times \psi)} \\
	& {\sigma^*A} &&& A \\
	\Delta &&& \Gamma & {}
	\arrow[from=1-1, to=3-1]
	\arrow["{\appAt{j}(\sigma^*f,\sigma^*s)}"{pos=0.65}, shorten <=4pt, dashed, from=1-1, to=2-2,]
	\arrow[shorten <=46pt, from=1-1, to=1-4]
	\arrow[from=2-2, to=3-1]
	\arrow["{\pair{\sigma^*f}{\sigma^*a}}", curve={height=-18pt}, dotted, from=3-1, to=1-1]
	\arrow[from=3-1, to=3-4, ]
	\arrow[from=1-4, to=3-4]
	\arrow["{\appAt{j}(f,s)}"{description}, from=1-4, to=2-5]
	\arrow[from=3-4, to=2-5]
	\arrow["\lrcorner"{anchor=center, pos=0.125}, shift left=2, draw=none, from=1-1, to=3-4]
	\arrow["\lrcorner"{anchor=center, pos=0.125}, draw=none, from=2-2, to=3-4]
	\arrow["{\pair{f}{a}}"{description, pos=0.7}, curve={height=-18pt}, dotted, from=3-4, to=1-4]
	\arrow[from=2-2, to=2-5, crossing over]
\end{tikzcd}\]
This uniquely defines the map $\appAt{j}(\sigma^*f,\sigma^*s)$, and by construction
\[ \sigma^*\appAt{j}(f,s) = \appAt{j}(\sigma^*f,\sigma^*s). \]
Furthermore the desired computation rule holds, saying that for a term $s$ in the smaller tope $\varphi$, one judgmentally has $\appAt{j}(f,s) \jdeq a[s]$. This is established by the commutation of the diagram:
\[\begin{tikzcd}
	&& A \\
	\\
	{\Gamma \times \varphi} && {\Gamma \times \psi} \\
	& \Gamma
	\arrow["{\Gamma \times j}"{description}, from=3-1, to=3-3]
	\arrow[from=1-3, to=3-3]
	\arrow["a"{description}, from=3-1, to=1-3]
	\arrow["s"{description}, from=4-2, to=3-1]
	\arrow[from=4-2, to=3-3]
	\arrow["f"{description}, curve={height=-12pt}, dotted, from=3-3, to=1-3]
	\arrow["{\appAt{j}(f,s)}"', curve={height=70pt}, dotted, from=4-2, to=1-3]
\end{tikzcd}\]

The $\beta$- and $\eta$-rule (\cf~\cite[Figure~4]{RS17}) can be proven to hold similarl (Rule 5 and 6, resp., in~\cite[Figure~4]{RS17}).

%% file: conclusion.tex
\subsection{Synthetic cocartesian fibrations}

We have developed, also in previous joint work with Ulrik Buchholtz~\cite{BW21}, a theory of co-/cartesian fibrations of synthetic $\inftyone$-categories in simplicial HoTT due to~\cite{RS17}. This generalizes concepts and results from~\emph{op.~cit.} to the non-discrete setting.

Our account follows~\cite[Chapter~5]{RV21} and indeed includes characterization theorems for cocartesian arrows, fibrations, and functors, in terms of left adjoint right inverse \emph{aka} Chevalley conditions. These serve to prove several closure properties, resembling the axioms of an $\infty$-cosmos. Thanks to the Chevalley criteria, a portion of these generalizes to Leibniz cotensors \wrt~to \emph{arbitrary} type maps or shape inclusions.\footnote{One should also compare this to the $\infty$-cosmoses of LARI and RARI adjunctions in~\cite{RV21}.}

\subsection{Synthetic bicartesian fibrations and Moens' Theorem}

We have given a basic account of Beck--Chevalley bifibrations, leading up to Moens' Theorem, translating Streicher's developments and proofs~\cite[Chapter~15]{streicher2020fibered} from the analytical $1$-dimensional to the synthetic $\inftyone$-categorical setting. This constitutes an application of our synthetic theory to prove a technically more involved theorem of fibered category theory \`{a} la B\'{e}nabou. To our knowledge, Moens' Theorem for $\inftyone$-categories has not been considered previously (in the analytical setting).

\subsection{Synthetic two-sided cartesian fibrations}

We have introduced synthetic two-sided cartesian fibrations, again in the spirit of~\cite[Chapter~5]{RV21}. This includes a systematic development of the notion, proving characterizations and closure properties in a modular way. Along the way, we have considered various (auxiliary) of ``fibered'' or ``sliced fibrations''. Our study leads up to a synthetic version of Riehl--Verity's two-sided Yoneda Lemma.

One might argue that this treatment is unwieldy at times, and might be more elegantly done in a setting with the appropriate categorical universes at hand, \cf~\Cref{sec:outlook}. However, it seems to be a general principle of simplicial HoTT that to be order to be build up the universe inside the theory in the first place, one has to understand the respective fibrations before.

\subsection{Semantics}

We have proved a coherence theorem for the extension types of~\cite{RS17}, relative to the interpretation of simplicial HoTT in any $\infty$-topos of simplicial objects,~\cite[Section~8]{Shu19}. The method is the well-known ``global universe splitting'' related to Giraud's left adjoint splitting, used previously in works on the semantics of HoTT, such as~\cite{KL18,LW15,Awo18}. This was hinted at but left out in~\cite{RS17}. The coherence construction added on top of the groundwork from~\cite[Appendix~A]{RS17} and\cite[Section~8]{Shu19} implies that indeed simplicial HoTT has models in simplicial objects internal to an arbitrary Grothendieck--Rezk--Lurie $\infty$-topos.

\subsection{sHoTT as a synthetic language for internal $\inftyone$-categories}

Consequently, the Rezk types get interpreted as Rezk objects in the respective $\inftyone$-topos. In particular, those also form an $\infty$-cosmos, capturing the formal category theory of (internal) $\inftyone$-categories. The cosmological notions in \cite{RV21} are formulated in terms of certain  constructions and notions, such as (relative) adjunctions/absolute lifting diagrams, and (fibered) equivalences of modules. Importantly, these notions are invariant under equivalences of $\infty$-cosmoses, hence are suitable to capture the respective internal formal $\infty$-category. Now, by means of the preestablished standard interpretation of sHoTT, one sees that our internal notions translate exactly to the desired analytical counterparts. \Eg, cocartesian fibrations are expressed through a LARI condition, and (LARI) adjunctions can be defined as a fibered equivalence of modules/hom objects. This also matches up with previous model-categorical investigations, \eg~in~\cite{rasekh2021cartesian}.

Therefore, sHoTT provides an expressive and convenient tool for reasoning about internal $\inftyone$-categories synthetically, at least with regard to the fibrational theory. It notably profits from the relations to $\infty$-cosmos theory, both for the internal and external theory.

However, there are also principle difficulties and obstructions at play. Homotopy-invariance is forced,~\ie~``discrete'' constructions are forbidden. In the face of countermodels, a native development of categorical universes and the ensuing Grothendieck construction does not seem possible. In general, it seems rather tricky in this theory to define concrete Rezk types as opposed to abstract ones, constructed out of anonymous given ones.

Also notably, the elementary feature of \emph{opposite categories} seems subtle to implement.

We present a few of our perspectives to remedy these shortcomings in the next section.

%% file: outlook.tex
\subsection{Discrete two-sided fibrations}

Natural follow-up work for this thesis includes a dedicated discussion of discrete two-sided fibrations, \emph{aka}~$\infty$-profunctors or (bi-)modules. We expect to be able to prove an analogous characterization as discrete fibrations of fibrations. This should, in the synthetic setting, imply the expected cosmological closure properties, and also additional operations, known from the \emph{virtual equipment} of modules due to~\cite{RV21}. Finally, we make some considerations on the closure properties of synthetic $\infty$-distributors in particular. Externally, the distributors internal to an arbitrary $\infty$-cosmos form a kind of double-category, encompassing the (formal) $\infty$-category theory of the $\infty$-cosmos.\footnote{This even has an associated internal language \`{a} la Makkai's FOLDS~\cite[Section~11.2]{RV21}.} This philosophy is a cornerstone of $\infty$-cosmos theory and the Model Independence Theorem. Then, also together with a first consideration of relative adjunctions in this thesis, possible applications could notably include a theory of Kan extensions. It would also be conceivable to develop, a calculus of modules in sHoTT, after~\cite{RVKan}, even though for this purpose working in one categorical dimension up might be more suitable.

\subsection{Categorical universes and modalities}
The next step in developing this synthetic theory of fibered $(\infty,1)$-categories is the treatment of categorical \emph{universes}, or \emph{object classifiers}, and the analogue of the \emph{completion operation} introduced by Rezk in \cite{rez01}, which completes a general Segal type to a Rezk type. In unpublished joint work with Ulrik Buchholtz we have shown that there are models in which simply restricting the universe to the canonical subtype of small $\infty$-groupoids, or small Segal types, resp., does not yield a Segal type again, cf.~ \cite{BW18a,BW21}. The goal is to establish a hierarchy of categorical universes within simplicial type theory, as previously done in various settings, notably by Ayala--Francis \cite{AFfib} and Rasekh \cite{Ras18model}.

We expect the following construction to work. In the first step, using techniques from \emph{Cubical Homotopy Type Theory} \cite{LOPS18,CCHM2018} we can, for suitable synthetic notions $F$ of fibration, define a type classifying small fibrations of flavor $F$. This requires a further extension of sHoTT, which is however justified by the semantics in simplicial spaces. Mainly, one internalizes the structure of $\sSpace$ as a \emph{cohesive $(\infty,1)$-topos} over $\Space$, by adding new type formers (\emph{modalities}) \cite{Shu18} for (synthetic) higher versions of Lawvere's \emph{axiomatic cohesion} \cite{LawvereCohesion}. It also requires to embed the interpreting $\infty$-topos of simplicial spaces into the $\infty$-topos of \emph{cubical spaces}, as discussed (for the $1$-topos case) in my joint note \cite{SW21} with Thomas Streicher, building upon work by Sattler \cite{Sat18} and Kapulkin--Voevodsky \cite{KV18}.

An analogous method has also been used in~\cite{WL19} ultimately constructing universes satisfying directed univalence, but with a different approach to the structural analysis along the way.

In our setting, we expect connections to a synthetic version of~\emph{flat} or \emph{$\inftyone$-Conduch\'{e}} fibrations.\footnote{I am grateful to David Ayala, Aaron Mazel-Gee, Emily Riehl, and Jay Shah for pointing me to the flat fibrations and discussions revolving around them.}

One should note that despite extending sHoTT even further, the additions still preserve the intrinsic character of the theory because the cohesive structure is present in the standard model $\sSpace$. Ultimately, we are aiming for a description of the categorical structure of the universes, and a version of the $\infty$-Grothendieck construction (\aka~straightening/unstraightening) in terms of the cohesive structure. This should be contrasted to \eg~the works of Lurie, Riehl--Verity, and Rezk which at times import the machinery of simplicially enriched categories, or Cisinski--Nguyen, who use model structures on \emph{marked} simplicial sets.

This has been joint ongoing work with Ulrik Buchholtz~\cite{BW18c,BW19}.

Combining this with the theory of two-sided cartesian fibrations from this thesis, one should also aim for synthetic higher categorical universes of spans.

\subsection{Opposites, twisted arrows, and classical Yoneda Lemma}

It is natural to ask for an operation giving for any Rezk type $A$ its \emph{opposite} Rezk type $A^{\Op}$. It has turned to be formally very delicate to introduce opposites into a dependent type theory. However, based on \emph{multimodal type theory} due to Licata--Riley--Shulman \cite{Licata19} and Gratzer--Kavvos--Nuyts--Birkedal \cite{GKNB20,MDTT} we have outlined rules and a semantics, more generally, capturing fibered modalities induced by operations on simplices. This in particular includes a modal operator yielding the type $\tw_A$ of twisted arrows.

Combining the work on fibrations and universes with the new modal extension is expected to yield as a prime application the more ``traditional'' Yoneda Lemma and Yoneda embedding $A \to (A^{\Op} \to \mathrm{Space})$, using the twisted arrow fibrations, after Kazhdan--Varshavsky~\cite{KV14}. This has been joint ongoing work with Ulrik Buchholtz~\cite{BW18b}.

\subsection{Synthetic higher algebra}

In the longer run, with most of the theory described in the previous paragraph established, we can turn to more advanced topics in synthetic fibered $(\infty,1)$-category theory. Specifically, we have in mind symmetric monoidal $(\infty,1)$-categories and $\infty$-operads, both of which can be defined in fibrational language, as done \eg~by Lurie \cite{Lur17}.

Another direction is to extend our notions of synthetic $\infty$-categories and cocartesian fibrations to an equivariant setting, hence providing an analogue to Barwick \etal's program \cite{barwick2016} of equivariant or \emph{parametrized $\infty$-category theory}. The parts of higher group theory already developed in HoTT to this date will be crucial here, notably \cite{BDR18,BBDG}. 

\subsection{Synthetic $(\infty,2)$- and $(\infty,n)$-categories and fibrations}

In light of the remarkable recent developments of two-dimensional higher category theory, we might hope to extend the syntax and semantics to the $\inftytwo$-level, working~\eg~in $2$-fold complete Segal spaces or $\Theta_2$-spaces. Since then objects of our theory would then be $\inftytwo$-categories we might hope to develop some parts of $\infty$-comos theory inside the theory. However, to get a reasonable grip on this, it seems instructive to achieve the aforementioned constructions of categorical universes first in the $\inftyone$-dimensional case. \Eg,~a two-dimensional theory should support a general pasting principle (both strictly operationally, and internally for the respective $\inftytwo$-categorical universe types) \`{a} la~\cite{infty2pasting}, and for this purpose directed univalence seems indispensable. Another interesting direction would be ``more natively'' develop a calculus of modules and other pieces of formal $\infty$-category theory in this two-dimensional setting. Thus, there are several reasons to be interested specifically in a version of sHoTT for synthetic $\inftytwo$-categories. As shown in the work of the notions of cartesian fibrations also become more intricate. We hope that our ``shape-independent'' account of LARI fibrations and LARI cells could provide useful here, \cf~\cite{BW21} and~\Cref{ssec:lari-cells}.

More generally, as suggested by~\cite{RS17} another direction of study would be a generalization to a type theory of $(\infty,n)$-categories, \eg~inside $\Theta_n$-spaces using an appropriate shape theory, \cf~also the recent work by Rasekh on fibrations of $(\infty,n)$-categories~\cite{RasYonDSp}.

In the light of close connections of simplicial type theory with cubical type theory, and the recent analytic results about cubical notions of weak higher categories, one could also hope for mutual new insights and connections between the analytic and type-theoretic side.

\subsection{Higher topos theory}

In \cite{streicher2020fibered,StrFVGM} Streicher lays out an analysis of geometric morphisms in terms of fibered $1$-category theory, recalling and extending work by B\'{e}nabou, Moens, and Jibladze. One of the central results is Jibladze's Theorem which says that locally small, cocomplete fibered toposes over a fixed base (elementary) topos $\mathscr B$ are given up to equivalence as \emph{Artin gluings} $F^* \partial_1: \mathscr E/F \to \mathscr B$ of some geometric morphism $F: \mathscr B \to \E$ (identified with its inverse image part).

The notion of fibered topos considered here implies that the reindexings are \emph{logical} functors, \ie~preserve the subobject classifiers.

It would be intriguing to analyze this and related results for higher toposes. First, generalizing this notion of fibered topos naturally exhibits the homotopy level as a parameter, since in a higher topos we are dealing with \emph{object} classifiers for arbitrary $n$-types, for $-1 \leq n \leq \infty$ (with $n=-1$ recovering the subobject classifier, which is the universe of propositions).

Second, it could be worthwile to consider more general notions than sheaf $(\infty,1)$-toposes, such as the proposed \emph{elementary $(\infty,1)$-toposes} after Shulman \cite{Shu18} and Rasekh \cite{RasPhd,RasEHT}.

A first step in this direction has been taken in this thesis. We have established a form of Moens' Lemma for synthetic $\inftyone$-categories, hence it holds in the intended models as well. This constitutes one of several examples of how the synthetic theory can be a helpful tool for the analytical theory as well, \cf~\eg~\cite{ABFJ}.

Furthermore, our treatment of BCC and Moens fibrations suggests that some more fibered category theory \`{a} la~\cite{streicher2020fibered} can be developed inside our synthetic setting, including~\eg~geometric fibrations. This could nicely complement the analytical treatment.

\subsection{Other type theories and implementations}

A different but related overall approach to higher categories in type theory is given by \emph{2-level type theory (2LTT)}~\cite{2ltt,CapriottiPhD}, based on Voevodsky's Higher~Type~System. In contrast, 2LTT seems closer to a (possible) foundational theory, whereas simplicial HoTT could be seen rather as a domain-specific language (DSL). Thus, it would be natural to provide an interpretation of sHoTT in a suitable 2-level type theory, also opening up possibilities of software-implementation.

On the other hand, there is an ongoing development of a new proof-assistant called \texttt{rzk} by Nikolai Kudasov~\cite{KudRzk} that supports simplicial type theory. Since we are seeking to extend our type theory by (multi-)modalities, which also (in part) have implementations, it would be most desirable to (eventually) achieve a high degree of modularity on both the theory as well as the practical implementation of the type theories at play.

Another approach, as sketched by Buchholtz~\cite{BucHoTTEST}, is to investigate the connections between sHoTT and ``HoTT with a simplicial interval'' as a less strict, more intensional replacement. One could ask how far along one gets with the latter, probably in the presence of additional induction principles but with no strict extension types.

%% file: rel-adj.tex
\section{Relative adjunctions}\label{ssec:reladj}

Taking up a suggestion by Emily Riehl, we provide here a brief treatment of \emph{relative adjunctions} in the sense of Ulmer~\cite{UlmDense}, \cf~also~\cite{MasarykFormal}. The purpose is to provide a more formal account to cocartesian arrows, or more generally, LARI cells. As a payoff, we will see that the Chevalley condition defining the LARI cells implies the Chevalley condition for LARI fibrations in the sense of~\cite{BW21}, and likewise for LARI functors. 

\subsection{Definition and characterization}

\begin{definition}[Transposing relative adjunction]
 Let $A,B,C$ be Rezk types and $(g: C \to A \leftarrow B: f)$ a cospan. A \emph{(transposing) left relative adjunction} consists of a functor $\ell:C \to B$ together with a fibered equivalence
 \[ \big(\comma{\ell}{B} \equiv_{C \times B} \comma{g}{f}\big) \simeq \prod_{\substack{c:C \\ b:B}} \hom_B(\ell\,c, b) \simeq \hom_A(g\,c, f\,b). \]
 Given such data, we call $\ell$ a \emph{(transposing) left adjoint of $f$ relative to $g$} or \emph{(transposing) $g$-left adjoint of $f$}.
\end{definition}

In case $C \jdeq A$ and $g \jdeq \id_A$ one obtains the usual notion of (transposing) adjunction. There also exists a relative analogue of the units. We might occasionally drop the predicate ``left'' in our discussion since we will only consider the left case. But beware that relative adjunctions are a genuinely asymmetric notion.

\begin{definition}[Relative adjunction via units]\label{def:reladj-units}
	Let $A,B,C$ be Rezk types and $(g: C \to A \leftarrow B: f)$ a cospan. A \emph{(transposing) left relative adjunction} consists of a functor $\ell:C \to B$ together with a natural transformation $\eta: g \Rightarrow_{C \to A} f\ell$, called \emph{relative unit}, such that transposition map
	\[ \Phi_\eta \defeq \lambda b,c,k.fk \circ \eta_c : \comma{\ell}{B} \to_{C \times B} \comma{g}{f} \]
	is a fiberwise equivalence.  
\end{definition}

By the characterizations about type-theoretic weak equivalences, \Cref{def:reladj-units} translates to:
\begin{align}\label{eq:reladj-units}
	\prod_{\substack{c:C \\b:B}} \prod_{m:gc \to_A fb} \isContr\Big( \sum_{k:\ell\,c \to_B b} \Phi_\eta(k) = m \Big)
\end{align}
Diagrammatically, this can be depicted as follows, demonstrating once more the generalization from the usual notion of adjunction:
\[\begin{tikzcd}
	{g\,c} && {f\,b} & b \\
	{f\,\ell c} &&& {\ell\,c}
	\arrow["{\forall\,m}", from=1-1, to=1-3]
	\arrow["{\eta_c}"', from=1-1, to=2-1]
	\arrow["{f\,k}"', from=2-1, to=1-3]
	\arrow["{\exists! \,k}", dashed, from=2-4, to=1-4]
\end{tikzcd}\]
It turns out that also in the synthetic setting we recover the equivalence of relative adjunctions with \emph{absolute left lifting diagrams (ALLD)}, whose universal property in terms of pasting diagrams can be (informally or analytically) visualized as:
A lax diagram
\[\begin{tikzcd}
	&& B \\
	C && A
	\arrow[""{name=0, anchor=center, inner sep=0}, "g"', from=2-1, to=2-3]
	\arrow["f", from=1-3, to=2-3]
	\arrow[""{name=1, anchor=center, inner sep=0}, "\ell", from=2-1, to=1-3]
	\arrow["\eta"', shorten <=2pt, shorten >=2pt, Rightarrow, from=0, to=1]
\end{tikzcd}\]
is an absolute lifting diagram if and only if any given lax square on the left factors uniquely as a pasting diagram as demonstrated below left:
\[\begin{tikzcd}
	X && B & {} & X && B \\
	C && A & {} & C && A
	\arrow["\gamma"', from=1-1, to=2-1]
	\arrow["g"', from=2-1, to=2-3]
	\arrow["\beta", from=1-1, to=1-3]
	\arrow["f", from=1-3, to=2-3]
	\arrow["{=}"{description}, draw=none, from=1-4, to=2-4]
	\arrow["\gamma"', from=1-5, to=2-5]
	\arrow[""{name=0, anchor=center, inner sep=0}, "g"', from=2-5, to=2-7]
	\arrow[""{name=1, anchor=center, inner sep=0}, "\beta", from=1-5, to=1-7]
	\arrow["f", from=1-7, to=2-7]
	\arrow["{\forall\,\mu}"{description}, shorten <=10pt, shorten >=10pt, Rightarrow, from=2-1, to=1-3]
	\arrow[""{name=2, anchor=center, inner sep=0}, from=2-5, to=1-7]
	\arrow["{\exists!\,\mu'}", shorten <=2pt, shorten >=2pt, Rightarrow, from=2, to=1]
	\arrow["\eta", shift right=5, shorten <=2pt, shorten >=2pt, Rightarrow, from=0, to=2]
\end{tikzcd}\]
Accordingly we define this type-theoretically\footnote{For a first discussion about lax squares and pasting diagrams in sHoTT~\cf~\cite[Appendix~A]{BW21}. We do currently not have a systematic account to these. \Eg~certainly at some point a pasting theorem \`{a} la~\cite{infty2pasting} would be most desirable. This would presumably require a categorical universes validating a directed univalence principle, and possibly also modalities from cohesion.} as follows:
\begin{definition}[Absolute left lifting diagram]
	A diagram
	\[\begin{tikzcd}
		&& B \\
		C && A
		\arrow[""{name=0, anchor=center, inner sep=0}, "g"', from=2-1, to=2-3]
		\arrow["f", from=1-3, to=2-3]
		\arrow[""{name=1, anchor=center, inner sep=0}, "\ell", from=2-1, to=1-3]
		\arrow["\eta"', shorten <=2pt, shorten >=2pt, Rightarrow, from=0, to=1]
	\end{tikzcd}\]
	is an \emph{absolute lifting diagram (ALLD)} if the following proposition is satisfied:
	\[ \isALLD_{\ell,f,g}(\eta) \defeq \prod_{\substack{X:\UU \\ \beta:X \to B \\ \gamma: X \to C}} \prod_{\mu:g\gamma \Rightarrow f\beta} \isContr \Big(\sum_{\mu': \ell \gamma \Rightarrow \beta} \prod_{x:X} (f\mu_x' \circ \eta_{\gamma\,x} =_{g\gamma\,x \to f\beta\,x} \mu_x)  \Big)\] 
\end{definition}
Note that one can infer the data $\angled{\ell,f,g}$ from $\eta$ alone. We might also speak of $\eta$ as an \emph{absolute left lifting cell}.

Diagrammatically, the demanded identity of morphisms reads:
\[\begin{tikzcd}
	{g(\gamma\,x)} && {f\ell(\gamma\,x)} && {f(\beta\,x)}
	\arrow["{\eta_{\gamma\,x}}", from=1-1, to=1-3]
	\arrow["{f\mu_x'}", from=1-3, to=1-5]
	\arrow[""{name=0, anchor=center, inner sep=0}, "{\mu_x}"', curve={height=24pt}, from=1-1, to=1-5]
	\arrow[shorten >=2pt, Rightarrow, no head, from=1-3, to=0]
\end{tikzcd}\]

The above definitions can be dualized to obtain relative \emph{right} adjoints and absolute \emph{right} lifting diagrams. Because of the inherent asymmetry some analogies to the case of ordinary adjunctions are missing (such as the presence of \emph{both} units and counits). However, we can still provide a characterization and some closure results.

\begin{theorem}[Characterizations of relative left adjunctions, \cf~\protect{\cite[Thm.~3.5.8/3]{RV21}}, \protect{\cite[Thm.~11.23]{RS17}}, \protect{\cite[Thm.~B.1.4]{BW21}}]\label{thm:reladj-char}
	Let $A,B,C$ be Rezk types and $(g: C \to A \leftarrow B: f)$ a cospan. Then the following types are equivalent propositions:
	\begin{enumerate}
		\item\label{it:reladj-transp} The type $\sum_{\ell:C \to B} \comma{\ell}{B} \equiv_{C \times B} \comma{g}{f}$ of (transposing) $g$-left adjoints of $f$.
		\item\label{it:reladj-unit} The type $\sum_{\ell:C \to B} \sum_{\eta:g \Rightarrow f\ell} \isEquiv\big( \Phi_\eta \big)$ with $\Phi$ as in~\Cref{def:reladj-units}.
		\item\label{it:reladj-alld} The type $\sum_{\ell:C \to B} \sum_{\eta:g \Rightarrow f\ell} \isALLD_{\ell,f,g}(\eta)$ of completions of the cospan consisting of $f$ and $g$ to an ALLD.
	\end{enumerate}
\end{theorem}

\begin{proof}
	In parts we can work analogously as in the proof of \cite[Theorem~11.23]{RS17}. In particular, an equivalence between the types from~\Cref{it:reladj-transp} and~\Cref{it:reladj-unit} follows\footnote{The classical version is due to~\cite[Lemma~2.7]{UlmDense}, and it works by the analogous argument.} just as in \emph{loc.~cit.} by using the (covariant discrete) Yoneda Lemma~\cite[Section~9, and (11.9)]{RS17}. Next, analogously as in the proof of \cite[Theorem~11.23]{RS17} one also shows that, given $\ell:C \to B$, the type $\sum_{\eta:g \Rightarrow f\ell} \isEquiv\big( \Phi_\eta \big)$ is a proposition.
	
	We now fix $\ell:C \to B$ and $\eta:g \Rightarrow f\ell$. Recall~\eqref{eq:reladj-units}. The direction from~\Cref{it:reladj-alld} to~\Cref{it:reladj-unit} follows by setting $X \jdeq \unit$. Conversely, we see that we get from~\Cref{it:reladj-unit} to~\Cref{it:reladj-alld} by ``reindexing'' Condition~\eqref{eq:reladj-units} along any given span $(\gamma : C \leftarrow X \rightarrow B:\beta)$.\footnote{More precisely, we use the fact that, given a family of propositions $P:A \to \Prop$, there is an equivalence $\Phi:\prod_{a:A} P(a) \simeq \prod_{\substack{X:\UU \\ \alpha:X \to A}} \prod_{x:X} P(\alpha,x):\Psi$. We can take $\Phi(\sigma) \defeq \lambda X,\alpha,x.\sigma(\alpha\,x)$ and $\Psi(\tau) \defeq \lambda a.\tau(\unit)(a)(\pt)$.}
\end{proof}

\begin{corollary}\label{cor:unique-left-rel-adj}
Given a cospan $(g:C \to A \leftarrow B:f)$, if both $\ell, \ell':C \to B$ are left adjoints to $f$ relative to $g$, then there is an identity $\ell = \ell'$:
\[\begin{tikzcd}
	&& B \\
	C && A
	\arrow[""{name=0, anchor=center, inner sep=0}, "g"', from=2-1, to=2-3]
	\arrow[""{name=1, anchor=center, inner sep=0}, "f", from=1-3, to=2-3]
	\arrow[""{name=2, anchor=center, inner sep=0}, "\ell"{description}, curve={height=6pt}, from=2-1, to=1-3]
	\arrow[""{name=3, anchor=center, inner sep=0}, "{\ell'}"{description}, curve={height=-12pt}, from=2-1, to=1-3]
	\arrow[shorten <=3pt, shorten >=3pt, Rightarrow, no head, from=2, to=3]
	\arrow["\eta"{description}, shorten <=6pt, shorten >=6pt, Rightarrow, from=0, to=1]
\end{tikzcd}\]
\end{corollary}

We write $\relAdj{\ell}{g}{f}$ if it exists.

\begin{definition}[Relative LARI adjunction]
	A relative left adjunction is called \emph{relative LARI adjunction} if its relative unit is invertible.
\end{definition}

\section{LARI cells, fibrations, and functors}\label{ssec:lari-stuff}

\subsection{LARI cells}\label{ssec:lari-cells}

Let $j: \Phi \hookrightarrow \Psi$ be a shape inclusion. Let $B$ be a Rezk type and $P:B \to \UU$ be an isoinner family. For its unstraightening $\pi: E \fibarr B$ the diagram induced by expoentiation is given through: 
\[\begin{tikzcd}
	{\mathllap{E^\Psi \equiv \sum_{\pair{u}{f}:\Phi^E}} \sum_{v:\ndexten{\Psi}{B}{\Phi}{u}} \exten{\Psi}{v^*P}{\Phi}{f}} &&&& {E^\Phi \simeq \sum_{u:\Phi \to B} \mathrlap{\prod_\Phi u^*P}} \\
	\\
	{\mathllap{B^\Psi \equiv \sum_{u:\Phi \to B}} \ndexten{\Psi}{B}{\Phi}{u}} &&&& {B^\Phi}
	\arrow[from=1-1, to=3-1]
	\arrow[from=3-1, to=3-5]
	\arrow[from=1-1, to=1-5]
	\arrow[from=1-5, to=3-5]
\end{tikzcd}\]
An element $v:\Psi \to B$ is to be understood as \emph{$\Psi$-shaped cell} (or \emph{diagram}) in the type $B$. A section $g:\prod_{t:\Psi} P(v(t))$ is a \emph{dependent $\Psi$-shaped cell (over $v$)} in the family $P$.

\begin{definition}[$j$-LARI cell]
Let $g: \Psi \to E$ be a $\Psi$-shaped cell in $E$, lying over $\angled{u,v,f}$ with $u:\Phi \to B$, $v:\Psi \to B$, and $f:\Phi \to E$ (both the latter lying over $u$). We call $g$ a \emph{$j$-LARI cell} if the ensuing canonical commutative diagram\footnote{By some slight abuse of notation $g$ really stands for the whole tuple $\angled{u,v,f,g}$, and the homotopy is reflexivity. This is a valid reduction due to fibrant replacement.}
\[ 
\begin{tikzcd}
	&& {E^\Psi} \\
	\unit && {B^\Psi \times_{B^\Phi} E^\Psi}
	\arrow["g", from=2-1, to=1-3]
	\arrow[""{name=0, anchor=center, inner sep=0}, "{\angled{u,v,f}}"', from=2-1, to=2-3]
	\arrow["{\pi' \mathrlap{\defeq j \cotens \pi}}", from=1-3, to=2-3]
	\arrow[shorten <=15pt, shorten >=15pt, Rightarrow, no head, from=1-3, to=0]
\end{tikzcd}
\]
is an absolute left lifting diagram,~\ie~there is a relative adjunction as encoded by the fibered equivalence
\[ \comma{g}{E^\Psi} \simeq_{E^\Psi} \comma{\angled{u,v,f}}{\pi'},\]
or, equivalently by~\Cref{thm:reladj-char}
\begin{align}\label{eq:lari-cell}
	\isLARICell_{j}^P(g) \defeq \isEquiv\big( \Phi_{\refl} \big)
\end{align}
\end{definition}
where the transposition map $\Phi_\refl$ simply projects the data of a morphism $\beta$ in $E^\Psi$ onto its part in $E^\Phi \times_{B^\Phi} B^\Psi$.

Unpacking this, after contracting away redundant data, yields
\begin{align}\label{eq:lari-cell-explicit} \isLARICell_{j}(g) \equiv \prod_{\substack{\angled{r,w,k,m}:\Psi \to E \\ \alpha:\angled{u,v,f} \to \angled{r,w,k}}} \isContr\Big(\exten{\pair{t}{s}:\Delta^1 \times \Psi} {P\big(\alpha_1(t,s)\big)} {b_1 \poprod j}{[\pair{g}{m},\alpha_2]}\Big),
\end{align}
where we denote by $b_1:\partial \Delta^1 \hookrightarrow \Delta^1$ the boundary inclusion, and $\alpha \jdeq \angled{\alpha_1, \alpha_2, \alpha_3}$ consists of morphisms:
\begin{align*}
	\alpha_1 & : \hom_{\Phi \to B}(u,r), \\
	\alpha_2 & : \ndexten{(\Delta^1 \times \Psi)}{B}{b_1 \poprod j}{[\pair{v}{w},\alpha_1]}, \\
	\alpha_3 & : \prod_{\pair{t}{s}:\Delta^1 \times \Phi} P\big(\alpha_1\,t\,s \big)
\end{align*}
Here, $(-_1)\poprod(-_2)$ denotes the pushout product,~\Cref{def:po-prod}. Intuitively, this means that the given data $\angled{\alpha,g}$ can be uniquely lifted as indicated in~\Cref{fig:lari-cell}.

\begin{figure}
\[\begin{tikzcd}
	& \cdot &&& \cdot \\
	& {} &&& {} \\
	E &&&&&& \Phi & {(\cdot} & {\cdot)} \\
	& \cdot &&& \cdot \\
	& \cdot &&& \cdot && \Psi & {(\cdot} &&& {\cdot)} \\
	B & {} &&& {} \\
	\\
	& \cdot &&& \cdot \\
	{} \\
	& {} &&& {}
	\arrow[from=5-2, to=5-5]
	\arrow[from=1-2, to=1-5]
	\arrow[dashed, from=4-2, to=4-5]
	\arrow[""{name=0, anchor=center, inner sep=0}, "k"{pos=0.3}, from=1-5, to=2-5]
	\arrow[""{name=1, anchor=center, inner sep=0}, "r"{pos=0.3}, from=5-5, to=6-5]
	\arrow[from=8-2, to=8-5]
	\arrow[""{name=2, anchor=center, inner sep=0}, "f"'{pos=0.3}, from=1-2, to=2-2]
	\arrow[shift left=4, from=2-2, to=2-5]
	\arrow["j"', hook, from=3-7, to=5-7]
	\arrow[from=3-8, to=3-9]
	\arrow[two heads, from=3-1, to=6-1]
	\arrow[from=5-8, to=5-11]
	\arrow[""{name=3, anchor=center, inner sep=0}, "v"'{pos=0.3}, from=5-2, to=6-2]
	\arrow[shift left=4, from=6-2, to=6-5]
	\arrow["\jdeq"{description}, draw=none, from=3-7, to=3-8]
	\arrow["\jdeq"{description}, draw=none, from=5-7, to=5-8]
	\arrow[""{name=4, anchor=center, inner sep=0}, "g"'{pos=0.6}, shorten <=5pt, from=2, to=4-2]
	\arrow[""{name=5, anchor=center, inner sep=0}, "m"{pos=0.6}, shorten <=5pt, from=0, to=4-5]
	\arrow[""{name=6, anchor=center, inner sep=0}, "w"{pos=0.6}, shorten <=5pt, from=1, to=8-5]
	\arrow["{\forall\,\alpha_3}"{description}, shift left=3, shorten <=19pt, shorten >=19pt, Rightarrow, from=2, to=0]
	\arrow["{\forall\, \alpha_1}"{description}, shift left=3, shorten <=19pt, shorten >=19pt, Rightarrow, from=3, to=1]
	\arrow[""{name=7, anchor=center, inner sep=0}, "v"'{pos=0.6}, shorten <=5pt, from=3, to=8-2]
	\arrow["{\forall\, \alpha_2}"{description}, shift right=4, shorten <=19pt, shorten >=19pt, Rightarrow, from=7, to=6]
	\arrow["{\exists! \, \beta}"{description}, shift right=4, shorten <=19pt, shorten >=19pt, Rightarrow, dashed, from=4, to=5]
\end{tikzcd}\]
	\caption{Universal property of $j$-LARI cells (schematic illustration)}
	\label{fig:lari-cell}
\end{figure}
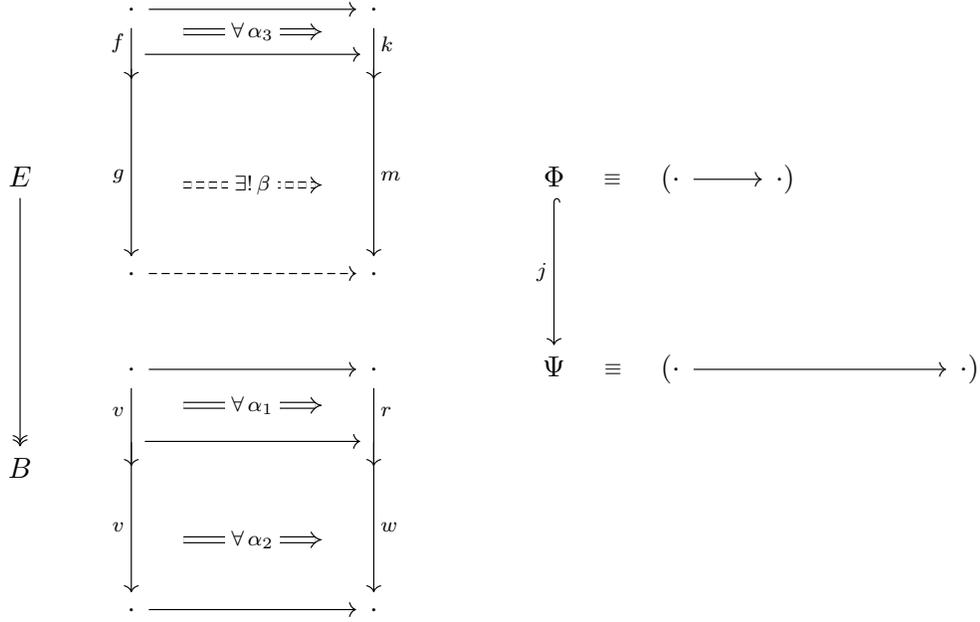

\subsection{LARI fibrations}

\begin{definition}[Enough $j$-LARI lifts]\label{def:enough-lari-lifts}
Let $P:B \to \UU$ be an isoinner family over a Rezk type $B$, and $j:\Phi \hookrightarrow \Psi$ a $j$-LARI lift. Then $P$ is said to \emph{have all} or \emph{enough $j$-LARI lifts} if and only if the type
\begin{align}\label{eq:enough-lari-cell}
\prod_{\angled{u,v,f}:E^\Phi \times_{B^\Phi} B^\Psi} \sum_{g:\exten{\Psi}{v^*P}{\Phi}{f}} \isLARICell_{j}(g)
\end{align}
is inhabited.
\end{definition}

By~\Cref{thm:reladj-char}, in fact $\sum_{g:\ldots} \isLARICell_{j}(g)$ is a proposition. Given $\angled{u,v,f}$, we denote the arrow part from the center of contraction of this type, occuring in~(\ref{eq:enough-lari-cell}), as
\[ P_!(u,v,f) \defeq g_{u,v,f} :\exten{\Psi}{v^*P}{\Phi}{f}, \]
generalizing from cocartesian families.\footnote{We could also add the inclusion $j:\Phi \hookrightarrow \Psi$ as an annotation, but this is not necessary here since we will only deal with such inclusions one at a time.}
Similarly, for $g\jdeq P_!(u,v,f)$ we denote the ensuing ``filling'' data from~\Cref{eq:lari-cell-explicit} by
\[ \tyfill_g(\alpha,m) : \exten{\pair{t}{s}:\Delta^1 \times \Psi} {P\big(\alpha_1(t,s)\big)} {b_1 \poprod j}{[\pair{g}{m},\alpha_2]}. \]

\begin{theorem}[$j$-LARI families via enough $j$-LARI lifts]\label{thm:lari-fams-lifting}
	Let $B$ be a Rezk type, $P: B \to \UU$ be an isoinner family, and denote by $\pi: E \to B$ the associated projection map. Then $P$ has enough $j$-LARI lifts if and only if it is a $j$-LARI family,~\ie~the Leibniz cotensor map $\pi' \defeq i_0 \cotens \pi: E^\Psi \to E^\Phi \to_{B^\Phi} B^\Psi$ has a left adjoint right inverse:
	\[\begin{tikzcd}
	{E^{\Psi}} & {} \\
	&& {E^\Phi \times_{B^\Phi} B^\Psi} && {E^\Phi} \\
	&& {B^{\Psi}} && {B^\Phi}
	\arrow[two heads, from=2-3, to=3-3]
	\arrow["{B^j}"', from=3-3, to=3-5]
	\arrow[from=2-3, to=2-5]
	\arrow["{\pi^\Phi}", two heads, from=2-5, to=3-5]
	\arrow["\lrcorner"{anchor=center, pos=0.125}, draw=none, from=2-3, to=3-5]
	\arrow["{E^j}", shift left=2, curve={height=-18pt}, from=1-1, to=2-5]
	\arrow[""{name=0, anchor=center, inner sep=0}, "\chi"', curve={height=12pt}, dotted, from=2-3, to=1-1]
	\arrow[""{name=1, anchor=center, inner sep=0}, "{\pi'}"', curve={height=12pt}, from=1-1, to=2-3]
	\arrow["{\pi^\Psi}"', shift right=2, curve={height=18pt}, two heads, from=1-1, to=3-3]
	\arrow["\dashv"{anchor=center, rotate=-117}, draw=none, from=0, to=1]
	\end{tikzcd}\]
\end{theorem}

\begin{proof}
	Assume $P:B \to \UU$ is an isoinner family with enough $j$-LARI lifts. The gap map can be taken as the strict projection
	\[ \pi' \defeq \lambda u,v,f,g.\angled{u,v,f}: E^{\Psi} \to E^\Phi \times_{B^\Phi} B^\Psi.\]
	For the candidate LARI we take the map that produces the $j$-LARI lift, \ie
	\[ \chi \defeq \lambda u,v,f.\angled{u,v,f,P_!(u,v,f)} : E^\Phi \times_{B^\Phi} B^\Psi \to E^\Psi.\]
	This is by definition a (strict) section of $\pi'$.
	
	For $\angled{u,v,f}:E^\Phi \times_{B^\Phi} B^\Psi$ and $\angled{r,w,k,m} : E^{\Psi}$ we define the maps
	\[\begin{tikzcd}
		{\hom_{E^\Psi}(\chi(u,v,f), \angled{r,w,k,m})} && {\hom_{E^\Phi \times_{B^\Phi} B^\Psi}(\angled{u,v,f},\angled{r,w,k})}.
		\arrow["F_{\angled{u,v,f},\angled{r,w,k,m}}", shift left=2, from=1-1, to=1-3]
		\arrow["G_{\angled{u,v,f},\angled{r,w,k,m}}", shift left=2, from=1-3, to=1-1]
	\end{tikzcd}\]
	defined by\footnote{We decompose morphisms in $E^\Psi$ as pairs $\pair{\alpha}{\beta}$ where $\alpha$ denotes the part in $E^\Phi \times_{B^\Phi} B^\Psi$ and $\beta$ is the given $\Psi$-shaped cell in $P$ lying over.}
	\[ F_{\angled{u,v,f},\angled{r,w,k,m}}(\alpha,\beta) \defeq \alpha, \quad G_{\angled{u,v,f},\angled{r,w,k,m}}(\gamma) \defeq \pair{\gamma}{\tyfill_{P_!(u,v,f)}(\gamma,m)} \]
	are quasi-inverse to one another.\footnote{For brevity, we shall henceforth leave the fixed parameters $\angled{u,v,f},\angled{r,w,k,m}$ implicit.}
	
	Clearly, $G$ is a section of the projection $F$ since for a morphism $\gamma$ we find
	\[ F(G(\gamma)) = F(\gamma,\tyfill_{P_!(u,v,f)}(\gamma,m)) = \gamma.\]

	For a morphism $\pair{\alpha}{\beta}$ in the opposing transposing morphism space we obtain
	\[ G(F(\alpha,\beta)) = G(\alpha) = \pair{\alpha}{\tyfill_{P_!(u,v,f)}(\alpha,m)}\]
	where we obtain an identification $\beta = \tyfill_{P_!(u,v,f)}(\alpha,m)$ (over $\refl_\alpha$) because of the universal property~\Cref{eq:lari-cell-explicit}. 
	
	This suffices to show that $\chi \dashv \pi'$ is a LARI adjunction as claimed.
	
	Conversely, suppose $\chi$ is a given LARI of $\pi'$, \wlogg~a strict section. This gives, for any data $\angled{u,v,f}:E^\Phi \times_{B^\Phi} B^\Psi$ we obtain a (strictly) commutative triangle:
	\[\begin{tikzcd}
		&& {E^\Psi} \\
		\unit && {B^\Psi \times_{B^\Phi} E^\Psi}
		\arrow["{\chi(u,v,f)}", from=2-1, to=1-3]
		\arrow[""{name=0, anchor=center, inner sep=0}, "{\angled{u,v,f}}"', from=2-1, to=2-3]
		\arrow["{\pi' }", from=1-3, to=2-3]
		\arrow[shorten <=15pt, shorten >=15pt, Rightarrow, no head, from=1-3, to=0]
	\end{tikzcd}\]
	
	Moreover, for all $\angled{r,w,k,m} : E^\Psi$, the map
	\[\begin{tikzcd}
		{\hom_{E^\Psi}(\chi(u,v,f),\langle r,w,k,m\rangle)} && {\hom_{E^\Phi \times_{B^\Phi} B^\Psi}(\langle u,v,f\rangle,\langle r,w,k\rangle)}
		\arrow["F_{\angled{r,w,k,m}}", from=1-1, to=1-3]
	\end{tikzcd}\]
	defined by
	\[ F(\alpha,\beta) \jdeq \alpha\]
	is an equivalence. Finally, contractibility of the fibers amounts to the universal property~\Cref{eq:lari-cell-explicit}, but this precisely means that $\chi(u,v,f)$ is a $j$-LARI cell.
\end{proof}

\subsection{LARI functors}

Let $j:\Phi \hookrightarrow \Psi$ be a shape inclusion.
\begin{definition}[$j$-LARI functors]\label{def:lari-fun}
	Over Rezk types $A$ and $B$, resp., consider $j$-LARI families $Q:A \to \UU$ and $P:E\to \UU$, resp., with
	\[ \xi \defeq \Un_A(Q): F \fibarr A, \quad \pi \defeq \Un_E(P): E \fibarr B. \]
	Assume there is a fibered functor from $\xi$ to $\pi$ given by a commutative square:
	\[\begin{tikzcd}
		F && E \\
		A && B
		\arrow["\xi"', two heads, from=1-1, to=2-1]
		\arrow["\alpha"', from=2-1, to=2-3]
		\arrow["\pi", two heads, from=1-3, to=2-3]
		\arrow["\varphi", from=1-1, to=1-3]
	\end{tikzcd}\]
	This defines a \emph{$j$-LARI functor} if and only if the following proposition is satisfied:
	\[ \prod_{m:\Psi \to F} \isLARICell_{j}^Q(m) \to \isLARICell_{j}^P(\varphi\,m)\]
\end{definition}

\begin{prop}[Naturality of cocartesian liftings]\label{prop:nat-larifun}
	Let $B$ be a Rezk type, and $P:B \to \UU$, $Q:C \to \UU$, resp. be $j$-LARI families with associated fibrations $\xi:F \fibarr A$ and $\xi:E \fibarr B$, resp. Then the proposition that $\pair{\alpha}{\varphi}$ be a $j$-LARI functor from $\xi$ to $\pi$ is logically equivalent to $\pair{\alpha}{\varphi}$ commuting with cocartesian lifts: \ie~for any $\angled{r,w,k,m} : F^\Phi \times_{A^\Phi} A^\Psi$ there exists an identification\footnote{suppressing the ``lower'' data which can be taken to consist of identities anyway} of $\Psi$-cells
	\[ \varphi\big(Q_!(r,w,k)\big) = P_!(\alpha \,r, \alpha\,w, \varphi\,k, \varphi\,d). \]
\end{prop}

\begin{proof}
	Since $\varphi$ is a $j$-LARI functor by assumption $g \defeq \varphi\big(Q_!(r,w,k)\big)$ is a $j$-LARi cell,~\ie~$\relAdj{g}{\angled{u,v,f}}{\pi'}$. But also for the cell $g' \defeq P_!(\alpha \,r, \alpha\,w, \varphi\,k, \varphi d)$ by construction we have \\ $\relAdj{g'}{\angled{u,v,f}}{\pi'}$, hence there is a homotopy $g=g'$, by uniqueness of relative left adjoints.
	
	Conversely, since any $j$-LARI arrow in a family occurs as a $j$-LARI lift (of the data given by projection and restriction), the assumed identifications yield the desired implication.
\end{proof}

\begin{theorem}[Chevalley criterion for $j$-LARI functors, \cf~\protect{\cite[Theorem~5.3.19]{BW21}}, {\protect\cite[Theorem~5.3.4]{RV21}}]\label{thm:char-lari-fun}
	Given data as in~\Cref{def:lari-fun}, the following are equivalent:
	\begin{enumerate}
		\item The fiberwise map $\pair{\alpha}{\varphi}$ is a $j$-LARI functor.
		\item The mate of the induced canonical natural isomorphism is invertible, too:
		\[\begin{tikzcd}
			{F^{\Psi}} && {E^\Psi} & {F^\Psi} && {E^{\Delta^1}} & {} \\
			{F^\Phi \times_{A^\Phi} A^\Psi} && {E^\Phi \times_{B^\Phi} B^\Psi} & {F^\Phi \times_{A^\Phi} A^\Psi} && {E^\Phi \times_{B^\Phi} B^\Psi} & {}
			\arrow["{\xi'}"', from=1-1, to=2-1]
			\arrow["{\varphi^\Phi \times_{\alpha^\Phi} \alpha^\Psi}"', from=2-1, to=2-3]
			\arrow["{\varphi^{\Delta^1}}", from=1-1, to=1-3]
			\arrow[""{name=0, anchor=center, inner sep=0}, "{r'}", from=1-3, to=2-3]
			\arrow["{=}", shorten <=9pt, shorten >=9pt, Rightarrow, from=2-1, to=1-3]
			\arrow[""{name=1, anchor=center, inner sep=0}, "\ell", from=2-4, to=1-4]
			\arrow["{\varphi^\Phi \times_{\alpha^\Phi} \alpha^\Psi}"', from=2-4, to=2-6]
			\arrow["{\varphi^{\Delta^1}}", from=1-4, to=1-6]
			\arrow["{\ell'}"', from=2-6, to=1-6]
			\arrow["{=}"', shorten <=12pt, shorten >=23pt, Rightarrow, from=2-6, to=1-4]
			\arrow["\rightsquigarrow", Rightarrow, draw=none, from=0, to=1]
		\end{tikzcd}\]
	\end{enumerate}
\end{theorem}

\begin{proof}
	The counit of the adjunction exhibiting $\pi:E \fibarr B$ as a $j$-LARI fibration, at stage $\angled{u,v,f,g}:E^\Psi$, can be taken to be
	\begin{align}\label{eq:counit-lari-fib}
		\varepsilon_{u,v,f,g} \jdeq \angled{\id_u,\id_v,\id_f,\tyfill_{P_!(u,v,f)}(\id_u,\id_v,\id_f,g)},
	\end{align}
	as one sees by the usual construction from the transposing map, \cf~the proof of~\Cref{thm:lari-fams-lifting}. Now, from the proof of~\cite[Propositon~A.1.2]{BW21} \footnote{\cf~\cite[Theorem~5.3.19]{BW21} for the cocartesian case} we see that the pasting cell constructed from the diagram
	\[\begin{tikzcd}
		{F^\Phi \times_{A^\Phi} A^\Psi} & {F^\Psi} && {E^\Psi} \\
		& {F^\Phi \times_{A^\Phi} A^\Psi} && {E^\Phi \times_{B^\Phi} B^\Psi} && {E^\Psi}
		\arrow["\mu", from=1-1, to=1-2]
		\arrow["{\xi'}", from=1-2, to=2-2]
		\arrow[""{name=0, anchor=center, inner sep=0}, Rightarrow, no head, from=1-1, to=2-2]
		\arrow["{\varphi^\Phi \times_{\alpha^\Phi} \alpha^\Psi}"', from=2-2, to=2-4]
		\arrow["{\varphi^\Psi}", from=1-2, to=1-4]
		\arrow["{\pi'}"', from=1-4, to=2-4]
		\arrow[shorten <=18pt, shorten >=18pt, Rightarrow, no head, from=2-2, to=1-4]
		\arrow["\chi"', from=2-4, to=2-6]
		\arrow[""{name=1, anchor=center, inner sep=0}, Rightarrow, no head, from=1-4, to=2-6]
		\arrow[shorten <=4pt, shorten >=4pt, Rightarrow, no head, from=0, to=1-2]
		\arrow[shorten <=4pt, shorten >=4pt, Rightarrow, no head, from=2-4, to=1]
	\end{tikzcd}\]
	is given, at stage $\angled{r,w,k,m}:F^\Psi$ by
	\begin{align*} 
		& \widetilde{\varepsilon}_{r,w,k,m} = \varepsilon_{\alpha\,r,\alpha\,w,\varphi\,k,\varphi(Q_!(r,w,k))} \\
		& = \angled{\id_{\alpha\,r},\id_{\alpha\,w},\id_{\varphi\,k},\tyfill_{P_!(\alpha\,r,\alpha\,w,\alpha\,k)}(\id_{\alpha\,r},\id_{\alpha\,w},\id_{\varphi\,k},\varphi Q_!(r,w,k))}.
	\end{align*}
	This collapses to just the comparison map $P_!(\alpha\,r,\alpha\,w,\alpha\,k) \to \varphi Q_!(r,w,k)$ given by filling. Now by~\Cref{prop:nat-larifun} the invertibility of this is equivalent to $\pair{\alpha}{\varphi}$ being a $j$-LARI functor. 
\end{proof}

%% file: fib-adj.tex
\section{Fibered equivalences}

We state some expected and useful closure properties of fibered equivalences.

\begin{lemma}[Right properness]\label{lem:rprop}
	Pullbacks of weak equivalences are weak equivalences again, \ie~given a pullback diagram
	\[\begin{tikzcd}
		{C \times_A B} && B \\
		C && A
		\arrow["\simeq" swap, from=1-3, to=2-3, "k"]
		\arrow["j", from=2-1, to=2-3]
		\arrow["\simeq", from=1-1, to=2-1, "k'" swap]
		\arrow[from=1-1, to=1-3]
		\arrow["\lrcorner"{anchor=center, pos=0.125}, draw=none, from=1-1, to=2-3]
	\end{tikzcd}\]
	then, as indicated if the right vertical map is a weak equivalence, then so is the left hand one.	
\end{lemma}

\begin{proof}
	Denote by $P: A \to \UU$ the straightening of $k:B \to A$, so that $B \simeq \sum_{a:A} P\,a$, and $D \simeq \sum_{c:C} P(j\,c)$. The map $k:B \to A$ (identified with its fibrant replacement) is a weak equivalence if and only if $\prod_{a:A} \isContr(P\,a)$. This implies $\prod_{c:C} P(j\,c)$ which is equivalent to $k' = j^*k$ being a weak equivalence, as desired.
\end{proof}

\begin{proposition}[Homotopy invariance of homotopy pullbacks]\label{prop:htopy-inv-of-pb}
	Given a map between cospans of types
	\[\begin{tikzcd}
		C && A && B \\
		{C'} && {A'} && {B'}
		\arrow["{\simeq~ r}"', from=1-1, to=2-1]
		\arrow["{g'}"', from=2-1, to=2-3]
		\arrow["g", from=1-1, to=1-3]
		\arrow["{\simeq~p}"', from=1-3, to=2-3]
		\arrow["f"', from=1-5, to=1-3]
		\arrow["{f'}", from=2-5, to=2-3]
		\arrow["{\simeq~q}", from=1-5, to=2-5]
	\end{tikzcd}\]
	where the vertical arrows are weak euivalences the induced map
	\[ B \times_A C \to B' \times_{A'} C'\]
	is an equivalence as well.
\end{proposition}

\begin{proof}
	By right properness and $2$-out-of-$3$ the mediating map $C \to P \defeq C' \times_{A'} A$ is an equivalence as can be seen from the diagram:
	\[\begin{tikzcd}
		C &&& A \\
		& P \\
		{C'} &&& {A'}
		\arrow["\simeq"{description}, from=1-1, to=3-1]
		\arrow[from=3-1, to=3-4]
		\arrow[from=1-1, to=1-4]
		\arrow["\simeq", from=1-4, to=3-4]
		\arrow["\simeq"{description}, dashed, from=1-1, to=2-2]
		\arrow[from=2-2, to=1-4]
		\arrow["\simeq"{description}, from=2-2, to=3-1]
		\arrow["\lrcorner"{anchor=center, pos=0.125}, draw=none, from=2-2, to=3-4]
	\end{tikzcd}\]
	This gives the following cube
	\[\begin{tikzcd}
		{C \times_A B} && B \\
		& C && A \\
		{C' \times_{A'} B'} && {B'} \\
		& {C'} && {A'}
		\arrow[from=1-1, to=3-1]
		\arrow[from=3-3, to=4-4]
		\arrow[from=3-1, to=4-2]
		\arrow[from=1-1, to=2-2]
		\arrow[from=1-3, to=2-4]
		\arrow["\lrcorner"{anchor=center, pos=0.125, rotate=45}, draw=none, from=1-1, to=2-4]
		\arrow["\lrcorner"{anchor=center, pos=0.125}, draw=none, from=2-2, to=4-4]
		\arrow["\lrcorner"{anchor=center, pos=0.125, rotate=45}, draw=none, from=3-1, to=4-4]
		\arrow[from=4-2, to=4-4]
		\arrow[from=3-1, to=3-3]
		\arrow["\simeq"{description, pos=0.3}, from=1-3, to=3-3]
		\arrow["\simeq"{description, pos=0.3}, from=2-4, to=4-4]
		\arrow[from=1-1, to=1-3]
		\arrow["\simeq"{description, pos=0.3}, from=2-2, to=4-2, crossing over]
		\arrow[from=2-2, to=2-4, crossing over]
	\end{tikzcd}\]
	and by the Pullback Lemma we know that the back face is a pullback, too. Then again, by right properness
	\[B \times_A C \to B' \times_{A'} C' \]
	is an equivalence.
\end{proof}

\begin{proposition}
Given a fibered equivalence as below
\[\begin{tikzcd}
	F && E \\
	& G \\
	& B
	\arrow[two heads, from=1-1, to=2-2]
	\arrow[two heads, from=1-3, to=2-2]
	\arrow[two heads, from=2-2, to=3-2]
	\arrow[two heads, from=1-1, to=3-2]
	\arrow[two heads, from=1-3, to=3-2]
	\arrow["\simeq" swap, "\varphi", from=1-1, to=1-3]
\end{tikzcd}\]
and a map $k:A \to B$ the fibered equivalence $\varphi$ pulls back as shown below:
	\[\begin{tikzcd}
		{F'} &&&&& F \\
		&&& {E'} &&&&& E \\
		& {G'} &&&&& {G} \\
		\\
		A &&&&& B
		\arrow["\varphi", "\simeq" swap, from=1-6, to=2-9]
		\arrow[two heads, from=2-9, to=5-6]
		\arrow[two heads, from=1-6, to=3-7]
		\arrow[two heads, from=2-9, to=3-7]
		\arrow[from=1-1, to=1-6]
		\arrow["\varphi'", "\simeq" swap, dashed, from=1-1, to=2-4, shorten >=8pt]
		\arrow[two heads, from=2-4, to=3-2]
		\arrow[two heads, from=1-1, to=5-1]
		\arrow["k", from=5-1, to=5-6]
		\arrow[two heads, from=3-2, to=5-1]
		\arrow[two heads, from=3-7, to=5-6]
		\arrow["\lrcorner"{anchor=center, pos=0.125}, draw=none, from=3-2, to=5-6]
		\arrow["\lrcorner"{anchor=center, pos=0.125}, draw=none, from=2-4, to=5-6]
		\arrow["\lrcorner"{anchor=center, pos=0.125}, draw=none, from=1-1, to=5-6]
		\arrow[two heads, from=1-6, to=5-6]
		\arrow[two heads, from=1-1, to=3-2]
		\arrow[from=3-2, to=3-7, crossing over]
		\arrow[two heads, from=2-4, to=5-1, crossing over]
		\arrow[from=2-4, to=2-9, crossing over]
		\end{tikzcd}
		\]
\end{proposition}

\begin{proof}
	By fibrant replacement, we can consider families $R:B \to \UU$, $P: F \jdeq \totalty{R} \to \UU$, $Q : E \jdeq \totalty{P} \to \UU$, with $F \jdeq \totalty{Q}$. The fibered equivalence $\varphi$ is given by a family of equivalences
	\[ \varphi: \prod_{\substack{b:B \\ x:R\,b}} \big( \sum_{e:P\,b\,x} Q \, b \, x \, e \big) \stackrel{\simeq}{\longrightarrow} P\,b\,x. \]
	The induced family
	\[ \varphi' \defeq \lambda a,x.\varphi(k\,a,x): \prod_{\substack{b:B \\ x:R\,k(a)}} \big( \sum_{e:P\,k(a)\,x} Q \, k(a) \, x \, e \big) \stackrel{\simeq}{\longrightarrow} P\,k(a)\,x\]
	also constitutes a fibered equivalence. Commutation of all the diagrams is clear since, after fibrant replacement, all the vertical maps are given by projections.
\end{proof}

\begin{lemma}[Closedness of fibered equivalences under dependent products]\label{lem:fib-eq-closed-pi}
	Let $I$ be a type. Suppose given a family $B:I\to \UU$ and indexed families $P,Q:\prod_{i:I} B_i \to \UU$ together with a fiberwise equivalence $\varphi: \prod_{i:I} \prod_{b:B} P_i\,b \stackrel{\simeq}{\longrightarrow} Q_i\,b$. Then the map
	\[ \prod_{i:I} \varphi_i : \big(\prod_{i:I} \totalty{P_i} \big) \longrightarrow_{\prod_{i:I} B_i} \big(\prod_{i:I} \totalty{Q_i} \big) \]
	induced by taking the dependent product over $I$ is a fiberwise equivalence, too.
\end{lemma}

\begin{proof}
	For $i:I$, we fibrantly replace the given fiberwise equivalence $\varphi_i$ by projections, giving rise to (strictly) commutative diagrams:
	\[\begin{tikzcd}
		{\sum_{b:B} \sum_{x:Q_i(b)} P_i(b,x) \simeq E_i} && {F_i \simeq \sum_{b:B} P_i(b)} \\
		& {B_i}
		\arrow["{\pi_i}"', two heads, from=1-1, to=2-2]
		\arrow["{\xi_i}", two heads, from=1-3, to=2-2]
		\arrow["{\varphi_i}", from=1-1, to=1-3, "\simeq" swap]
	\end{tikzcd}\]
	Now, $\varphi$ being a fiberwise equivalence is equivalent to
	\begin{align*}
		\prod_{i:I} \isEquiv(\varphi_i)	&\simeq  \prod_{i:I} \prod_{\substack{b:B_i \\ x:Q_i(b)}} \isContr(P_i(b,x)) \\
		&	\simeq   \prod_{\beta:\prod_{i:I} B_i} \prod_{\sigma:\prod_{i:I} \beta^* P_i} \prod_{i:I} \isContr\left( P_i(\beta(i), \sigma(i))\right).
	\end{align*}
	By (weak) function extensionality,\footnote{\cf~\cite[Theorem~13.1.4(ii)]{RijIntro}, or the discussion at the beginning of~\cite[Section~4.4.]{RS17}} this implies
	\begin{align*}
		&\prod_{\beta:\prod_{i:I} B_i} \prod_{\sigma:\prod_{i:I} \beta^*Q_i} \isContr \left( \prod_{i:I} \pair{\beta}{\sigma}^* P_i\right) \\ 
		\simeq & \isEquiv\left( \prod_{i:I} \varphi_i \right)
	\end{align*}
	wich yields the desired statement. Note, that the latter equivalence follows by fibrant replacement of the diagram obtained by applying $\prod_{i:I}(-)$:
	\[\begin{tikzcd}
		{\sum_{\substack{\beta:\prod_{i:I} B_i \\ \sigma:\prod_{i:I} \beta^*Q_i}} \prod_{i:I} \langle\beta,\sigma\rangle^*P_i \simeq \prod_{i:I} E_i} && {\mathllap{\prod_{i:I} F_i \simeq} \sum_{\beta:\prod_{j:I} B_i} \prod_{i:I} \beta^*Q_i} \\
		& {\prod_{i:I} B_i}
		\arrow["{\prod_{i:I} \pi_i}"', from=1-1, to=2-2]
		\arrow["{\prod_{i:I} \varphi_i}"{pos=0.3}, shorten >=45pt, from=1-1, to=1-3]
		\arrow["{\prod_{i:I} \xi_i}", from=1-3, to=2-2]
	\end{tikzcd}\]
\end{proof}

\begin{lemma}[Closedness of fibered equivalences under sliced products]\label{prop:fib-eqs-sliced-pi}
	Given indexed families $P,Q:I \to B \to \UU$ and a family of fibered equivalences $\prod_{i:I} \prod_{b:B} P_i \,b  \stackrel{\simeq}{\longrightarrow} Q_i \,b$. Then the induced fibered functor
	\[ \times_{i:I}^B \varphi_i : \prod_{i:I} \prod_{b:B} \times_{i:I}^B P_i \,b \longrightarrow \times_{i:I}^B Q_i\,b \]
	between the sliced products over $B$ is also a fibered equivalence.
\end{lemma}

\begin{proof}
	As usual, denote for $i:$ by $\pi_i \defeq \Un_B(P_i) : E_i \to B$ and $\xi_i \defeq \Un_B(Q_i) : F_i \to B$ the unstraightenings of the given fibered families, giving rise to a (strict) diagram:
	\[\begin{tikzcd}
		{E_i} && {F_i} \\
		& B
		\arrow["{\pi_i}"', two heads, from=1-1, to=2-2]
		\arrow["{\xi_i}", two heads, from=1-3, to=2-2]
		\arrow["{\varphi_i}", from=1-1, to=1-3]
	\end{tikzcd}\]
	Since weak equivalences are closed under taking dependent products, the induced fibered map $\prod_{i:I} \varphi_i: \prod_{i:I} E_i \to_{I \to B} \prod_{i:I} F_i$ is also a weak equivalence, and by right properness \Cref{lem:rprop} the desired mediating map is as well:
	\[\begin{tikzcd}
		{\times^B_{i:I} E_i} && {\prod_{i:I} E_i} \\
		& {\times^B_{i:I} F_i} && {\prod_{i:I} F_i} \\
		& B && {B^I}
		\arrow["{\prod_{i:I} \varphi_i}", from=1-3, to=2-4]
		\arrow[from=1-1, to=1-3]
		\arrow["{\times_{i:I}^B \varphi_i}"{pos=0.2} description, dashed, from=1-1, to=2-2]
		\arrow[two heads, from=2-4, to=3-4]
		\arrow["\cst", from=3-2, to=3-4]
		\arrow[two heads, from=2-2, to=3-2]
		\arrow[curve={height=6pt}, two heads, from=1-1, to=3-2]
		\arrow["\lrcorner"{anchor=center, pos=0.125}, draw=none, from=2-2, to=3-4]
		\arrow["\lrcorner"{anchor=center, pos=0.125}, draw=none, from=1-1, to=3-4]
		\arrow[curve={height=6pt}, two heads, from=1-3, to=3-4]
		\arrow[from=2-2, to=2-4, crossing over]
	\end{tikzcd}\]
\end{proof}

\begin{proposition}
	For an indexing type $I$ and a base Rezk type $B$, families of fibered equivalences between Rezk types over $B$ are closed under taking sliced products, \ie: Given a family of isoinner fibrations over $B$ together with a fibered equivalence as below left, the induced maps on the right make up a fibered equivalence as well:
	\[\begin{tikzcd}
		{F_i} && {E_i} & {\times_{i:I}^B F_i} && {\times_{i:I}^B E_i} \\
		& {G_i} & {} & {} & {\times_{i:I}^B G_i} \\
		& B &&& B
		\arrow[""{name=0, anchor=center, inner sep=0}, "{\varphi_i}"', curve={height=6pt}, from=1-1, to=1-3]
		\arrow[""{name=1, anchor=center, inner sep=0}, "{\psi_i}"', curve={height=12pt}, from=1-3, to=1-1]
		\arrow[""{name=2, anchor=center, inner sep=0}, "{\times_{i:I}^B \varphi_i}"', curve={height=6pt}, from=1-4, to=1-6]
		\arrow[""{name=3, anchor=center, inner sep=0}, "{\times_{i:I}^B \psi_i}"', curve={height=12pt}, from=1-6, to=1-4]
		\arrow[two heads, from=1-1, to=2-2]
		\arrow[two heads, from=1-3, to=2-2]
		\arrow[curve={height=12pt}, two heads, from=1-1, to=3-2]
		\arrow[curve={height=-12pt}, two heads, from=1-3, to=3-2]
		\arrow[two heads, from=2-2, to=3-2]
		\arrow[curve={height=12pt}, two heads, from=1-4, to=3-5]
		\arrow[curve={height=-12pt}, two heads, from=1-6, to=3-5]
		\arrow[two heads, from=2-5, to=3-5]
		\arrow[two heads, from=1-6, to=2-5]
		\arrow[two heads, from=1-4, to=2-5]
		\arrow[squiggly, from=2-3, to=2-4]
		\arrow["\dashv"{anchor=center, rotate=-90}, draw=none, from=1, to=0]
		\arrow["\dashv"{anchor=center, rotate=-90}, draw=none, from=3, to=2]
	\end{tikzcd}\]
\end{proposition}

\begin{proof}
	This is a consequence of~\Cref{lem:fib-eq-closed-pi,prop:fib-eqs-sliced-pi}, considering the following diagram:
	\[\begin{tikzcd}
		{\times^B_{i:I} F_i} &&&&& {\prod_{i:I} E_i} \\
		& {} & {\times^B_{i:I} F_i} &&&&& {\prod_{i:I} F_i} \\
		& {\times^B_{i:I} G_i} &&&&& {\prod_{i:I} G_i} \\
		B &&&&& {B^I}
		\arrow[from=1-1, to=1-6]
		\arrow["\simeq"{description}, from=1-1, to=2-3]
		\arrow[two heads, from=1-1, to=4-1]
		\arrow[from=4-1, to=4-6]
		\arrow[two heads, from=1-6, to=4-6]
		\arrow[dashed, two heads, from=1-1, to=3-2]
		\arrow[two heads, from=3-2, to=4-1]
		\arrow[dashed, two heads, from=2-3, to=3-2]
		\arrow[two heads, from=1-6, to=3-7]
		\arrow["\simeq"{description}, from=1-6, to=2-8]
		\arrow[two heads, from=2-8, to=3-7]
		\arrow[curve={height=-22pt}, two heads, from=2-8, to=4-6]
		\arrow[two heads, from=3-7, to=4-6]
		\arrow["\lrcorner"{anchor=center, pos=0.125}, draw=none, from=2-3, to=3-7]
		\arrow["\lrcorner"{anchor=center, pos=0.125}, draw=none, from=3-2, to=4-6]
		\arrow["\lrcorner"{anchor=center, pos=0.125}, draw=none, from=1-1, to=2-8]
		\arrow[from=3-2, to=3-7]
		\arrow[from=2-3, to=2-8, crossing over]
		\arrow[curve={height=-22pt}, two heads, from=2-3, to=4-1, crossing over]
	\end{tikzcd}\]
\end{proof}

\section{Fibered (LARI) adjunctions}

Building on previous work \cite[Section~11]{RS17} and \cite[Appendix~B]{BW21} we provide a characterization of fibered LARI adjunctions along similar lines.

\begin{theorem}[Characterizations of fibered adjunctions, cf.~{\protect\cite[Theorem~11.23]{RS17}, \cite[Theorem~B.1.4]{BW21}}]\label{thm:char-fib-adj}
	Let $B$ be a Rezk type. For $P,Q:B \to \UU$ isoinner families we write $\pi \defeq \Un_B(P) : E \defeq \totalty{P} \fibarr B$ and $\xi\defeq \Un_B(Q):F \jdeq \totalty{Q} \fibarr B$. Given a fibered functor $\varphi: E \to_B F$ such that (strictly)
	\[\begin{tikzcd}
		E && F \\
		& B
		\arrow["\varphi", from=1-1, to=1-3]
		\arrow["\pi"', two heads, from=1-1, to=2-2]
		\arrow["\xi", two heads, from=1-3, to=2-2]
	\end{tikzcd}\]
	the following are equivalent propositions:
	\begin{enumerate}
		\item\label{it:fib-ladj-vert-i} The type of \emph{fibered left adjoints} of $\varphi$, \ie~fibered functors $\psi$ which are ordinary (transposing) left adjoints of $\varphi$ whose unit, moreover, is vertical.
		\item\label{it:fib-ladj-vert-ii} The type of fibered functors $\psi:F \to_B E$ together with a vertical $2$-cell $\eta: \id_F \Rightarrow_B \varphi \psi$ s.t.~$\Phi_\eta \defeq \lambda u,d,e,k.\varphi_u \,k \circ \eta_d: \prod_{\substack{a,b:B \\ u:a \to b}} \prod_{d:Q\,a,~  e:P\,b} \big(\psi_a \,d \to^P_u e \big) \to \big( d \to^Q_u \varphi_b \, e)$ is a fiberwise equivalence. 
		\item\label{it:fib-ladj-sliced} The type of \emph{sliced} (or \emph{fiberwise}) \emph{left adjoints (over $B$)} to $\varphi$, \ie~fibered functors $\psi:F \to_B E$ together with a fibered equivalence $\relcomma{B}{\psi}{E} \simeq_{F \times_B E} \relcomma{B}{F}{\varphi}$.
		\item\label{it:fib-ladj-bi-diag-i} The type of \emph{bi-diagrammatic fibered} (or \emph{fiberwise}) \emph{left adjoints}, \ie~fibered functors $\psi$ together with:
					\begin{itemize}
						\item a vertical natural transformation $\eta : \id_F \Rightarrow_B \varphi \psi$
						\item two \emph{vertical} natural transformations $\varepsilon, \varepsilon': \psi \varphi \Rightarrow_B \id_E$
						\item homotopies\footnote{by Segal-ness, the witnesses for the triangle identities are actually unique up to contractibility} $\alpha : \varphi \varepsilon \circ \eta \varphi =_{E \to F} \id_\varphi, ~\beta : \varepsilon' \psi \circ \psi \eta =_{F \to E} \id_\psi$
					\end{itemize}
		\item\label{it:fib-ladj-bi-diag-ii} The type of fibered functors $\psi$ together with:
					\begin{itemize}
					\item a vertical natural transformation $\eta : \id_F \Rightarrow_B \varphi \psi$
					\item two natural transformations $\varepsilon, \varepsilon': \psi \varphi \Rightarrow \id_E$
					\item homotopies $\alpha : \varphi \varepsilon \circ \eta \varphi =_{E \to F} \id_\varphi, ~\beta : \varepsilon' \psi \circ \psi \eta =_{F \to E} \id_\psi$
				\end{itemize}
	\end{enumerate}
\end{theorem}

\begin{proof} ~ \\
	At first, we prove that, given a \emph{fixed} and \emph{fibered} functor $\psi:F \to_B E$ the respective witnessing data are propositions.\footnote{This justifies the ensuing list of \emph{logical} equivalences.}
	\begin{description}
		\item[$\ref{it:fib-ladj-vert-i} \iff \ref{it:fib-ladj-bi-diag-ii}$:] This follows from the equivalence between transposing left adjoint and bi-diagrammatic left adjoint data, \cf~\cite[Theorem~11.23]{RS17}.
		\item[$\ref{it:fib-ladj-bi-diag-i} \implies \ref{it:fib-ladj-bi-diag-ii}$:] This is clear since the latter is a weakening of the former. 
		\item[$\ref{it:fib-ladj-bi-diag-ii} \implies \ref{it:fib-ladj-bi-diag-i}$:] Denoting the base component of $\varepsilon$ by $v:\Delta^1 \to (B \to B) \simeq B \to (\Delta^1 \to B)$, projecting down from $\alpha$ via $\xi$ we obtain the identity $\xi \alpha : v \circ \id_B = \id_B$. Thus $\varepsilon$ is vertical, and similarly one argues for $\varepsilon'$.
		\item[$\ref{it:fib-ladj-sliced} \iff \ref{it:fib-ladj-bi-diag-i}$:] Given the fibered functor $\psi$, both lists of data witness that for every $b:B$ the components $\psi_b \dashv \varphi_b: P\,b \to Q\,b$ define an adjunction between the fibers, again by~\cite[Theorem~11.23]{RS17}.
		\item[$\ref{it:fib-ladj-vert-ii} \implies \ref{it:fib-ladj-sliced}$:] The latter is an instance of the former.
		\item[$\ref{it:fib-ladj-bi-diag-i} \implies \ref{it:fib-ladj-vert-ii}$:] Using naturality and the triangle identities, we show that the fiberwise conditions (vertical case) can be lifted to the case of arbitrary arrows in the base.\footnote{I thank Ulrik Buchholtz for pointing out the subsequent argument.}
		Consider the transposing maps:
		\begin{align*}
			\Phi \defeq \lambda k.\varphi_u\,k \circ \eta_d & : \prod_{\substack{a,b:B \\ u:a \to b}} \prod_{\substack{d:Q\,a \\  e:P\,b}} \big(\psi_a \,d \to^P_u e \big) \to \big( d \to^Q_u \varphi_b \, e \big) \\
			\Psi \defeq \lambda m.\varepsilon_e \circ \psi_u\,m  & : \prod_{\substack{a,b:B \\ u:a \to b}} \prod_{\substack{d:Q\,a \\ e:P\,b}}  \big( d \to^Q_u \varphi_b \, e) \to \big(\psi_a \,d \to^P_u e \big) 
		\end{align*}
	The first roundtrip yields:
	\[\begin{tikzcd}
		{\big(k:\psi_a\,d} & {e\big)} & {\big(\varphi_u\,k \circ \eta_d: d} & {\varphi_b\,e \big)} & {} \\
		&& {\big( \varepsilon_e \circ \psi_a(\varphi_u\,k \circ \eta_d):\psi_a\,d} & e
		\arrow["P", from=1-1, to=1-2]
		\arrow["\Phi", maps to, from=1-2, to=1-3]
		\arrow[""{name=0, anchor=center, inner sep=0}, "Q", from=1-3, to=1-4]
		\arrow["P", from=2-3, to=2-4]
		\arrow["\Psi"{description}, shorten <=6pt, maps to, from=0, to=2-3]
	\end{tikzcd}\]
	The result yields back $k$ using a triangle identity in the triangle on the left, and naturality of $\varepsilon$ in the square on the right:
\[\begin{tikzcd}
	{\psi_a\,d} && {(\psi \varphi \psi)_a\,d} & {} & {(\psi\varphi)_b\,e} \\
	&& {\psi_a\,d} && e
	\arrow["{\psi_a\eta_d}", from=1-1, to=1-3]
	\arrow[""{name=0, anchor=center, inner sep=0}, "{\id_{\psi_a\,d}}"', curve={height=12pt}, Rightarrow, no head, from=1-1, to=2-3]
	\arrow["{\psi_a\varepsilon_d}", from=1-3, to=2-3]
	\arrow["{(\psi\varphi)_u\,k}", from=1-3, to=1-5]
	\arrow["k"', from=2-3, to=2-5]
	\arrow["{\varepsilon_e}", from=1-5, to=2-5]
	\arrow[Rightarrow, no head, from=2-3, to=1-5]
	\arrow[shorten <=6pt, Rightarrow, no head, from=0, to=1-3]
\end{tikzcd}\]
In addition, we have also used naturality of $\varepsilon$ for $\psi_a \varepsilon_d \jdeq \varepsilon_{\psi_a\,d}$. An analogous argument proves the other roundtrip.
	\end{description}
We have proven, that relative to a fixed fibered functor $\psi: F \to_B E$ the different kinds of witnesses that this is a fibered left adjoint to $\varphi$ are equivalent propositions, giving rise to a predicate $\isFibLAdj_\varphi : (F \to_B E) \to \Prop$. What about the $\Sigma$-type $\FibLAdj_\varphi \defeq \sum_{\psi:F \to_B E} \isFibLAdj_\varphi(\psi)$ as a whole? E.g.~using the data from \cref{it:fib-ladj-sliced} (after conversion via~\cite[Theorem~11.23]{RS17}), said type is equivalent to
\begin{align*}
	 \FibLAdj_\varphi(\psi) & \simeq \sum_{\psi:\prod_{b:B}  P\,b \to Q\,b} \sum_{\eta:\prod_{b:B} \prod_{d:Q\,b} \hom_{Q\,b}(d,(\varphi\,\psi)_b\,d)} \prod_{\substack{b:B \\ d:Q\,b \\e:P\,b}} \isEquiv\big( \lambda k.\varphi_b(k) \circ \eta_d\big) \\ 
	& \simeq  \prod_{\substack{b:B \\ d:Q\,b}} \sum_{\psi_b:P\,b} \sum_{\eta_d:d \to_{Q\,b} \varphi_b(\psi_b d)} \prod_{e:P\,b} \isEquiv(\lambda k.\varphi_b(k) \circ \eta_d).
\end{align*}
Finally, one shows that this is indeed a proposition, completely analogously to the argument given in the proof of \cite[Theorem~11.23]{RS17} for the non-dependent case.
\end{proof}

\begin{definition}[Fibered (LARI) adjunction]\label{def:fib-lari-adj}
	Let $B$ be a Rezk type and $\pi: E \fibarr B$, $\xi:F \fibarr B$ be isoinner fibrations, with $P \defeq \St_B(\pi)$ and $Q \defeq \St_B(\xi)$. Given a fibered functor $\varphi: E \to_B F$, the data of a \emph{fibered left adjoint right inverse (LARI) adjunction} is given by
	\begin{itemize}
		\item a fibered functor $\psi: F \to_B E$,
		\item and an equivalence $\Phi: \relcomma{B}{\psi}{E} \simeq_{F \times_B E} \relcomma{B}{F}{\varphi}$ s.t.~the fibered unit
		\[ \eta_\Phi \defeq \lambda b,d.\Phi_{b,d,\psi_b \, d}(\id_{\psi_b\,d}):\prod_{b:B} \prod_{d:Q\,b} d \to_{Q\,b} (\varphi \psi)_b(d) \]
		is a componentwise homotopy.
	\end{itemize} 
Together, this defines the data of a \emph{fibered LARI adjunction}. Diagrammatically, we represent this by:
\[\begin{tikzcd}
	E && F \\
	& B
	\arrow["\pi"', two heads, from=1-1, to=2-2]
	\arrow["\xi", two heads, from=1-3, to=2-2]
	\arrow[""{name=0, anchor=center, inner sep=0}, "\varphi"', from=1-1, to=1-3]
	\arrow[""{name=1, anchor=center, inner sep=0}, "\psi"', curve={height=12pt}, dotted, from=1-3, to=1-1]
	\arrow["\dashv"{anchor=center, rotate=-90}, draw=none, from=1, to=0]
\end{tikzcd}\]
\end{definition}

In fact, as established in the previous works of~\cite[Section~11]{RS17} the unit of a coherent adjunction is determined uniquely up to homotopy. Hence, using the characterizations of a (coherent) LARI adjunction, the type of fibered LARI adjunctions in the above sense is equivalent to the type of LARI adjunctions which are also fibered adjunctions. This implies the validity of the familiar closure properties for this restricted class as well.

\section{Sliced commas and products}

We record here explicitly some closure properties involving sliced commas and products that are often used, especially in the treatise of two-sided fibrations and related notions.

\begin{proposition}[Dependent products of sliced commas]\label{prop:dep-prod-comm-sl-commas}
For a type $I$ and $i:I$, given fibered cospans
\[ \psi_i: F_i \to_{B_i} G_i \leftarrow_{B_i} E_i : \varphi\]
of Rezk types, taking the dependent product fiberwisely commutes with forming sliced comma types:
\[\begin{tikzcd}
	& {\psi_i \downarrow_{B_i} \varphi_i} \\
	{F_i} && {E_i} & {\prod_{i:I} (\psi_i \downarrow_{B_i} \varphi_i)} & {\left(\prod_{i:I} \psi_i \right) \downarrow_{\left(\prod_{i:I} B_i \right)} \left(\prod_{i:I}\varphi_i\right)} \\
	& {G_i} & {} & {} & {} \\
	& {B_i} &&& {\prod_{i:I} B_i}
	\arrow["{\psi_i}"{description}, two heads, from=2-1, to=3-2]
	\arrow["{\varphi_i}"{description}, two heads, from=2-3, to=3-2]
	\arrow[two heads, from=3-2, to=4-2]
	\arrow[curve={height=12pt}, two heads, from=2-1, to=4-2]
	\arrow[curve={height=-12pt}, two heads, from=2-3, to=4-2]
	\arrow[from=1-2, to=2-1]
	\arrow[from=1-2, to=2-3]
	\arrow[shorten <=25pt, shorten >=25pt, Rightarrow, from=2-1, to=2-3]
	\arrow[two heads, from=2-4, to=4-5]
	\arrow[squiggly, from=3-3, to=3-4]
	\arrow[two heads, from=2-5, to=4-5]
	\arrow[""{name=0, anchor=center, inner sep=0}, shift left=2, curve={height=-6pt}, from=2-5, to=2-4]
	\arrow[""{name=1, anchor=center, inner sep=0}, shift left=2, curve={height=-6pt}, from=2-4, to=2-5]
	\arrow["\simeq"{description}, Rightarrow, draw=none, from=1, to=0]
\end{tikzcd}\]
\end{proposition} 

\begin{proof}
We denote by $P_i, Q_i, R_i: B_i \to \UU$ the straightenings of the given maps $E_i \fibarr B_i$, $F_i \fibarr B_i$, and $G_i \fibarr B_i$, resp.
Using fibrant replacement, the sliced commas are computed as
\begin{align}\relcomma{B_i}{\psi_i}{\varphi_i} \simeq \sum_{b:B_i} \sum_{\substack{e:P_i(b) \\ d:Q_i(b)}} \big(\psi_i\big)_b(d) \longrightarrow_{R_i(b)} \big(\varphi_i\big)_b(e).
\end{align}\label{eq:fam-sliced-comma}
From this and the type-theoretic axiom of choice, we obtain as fibrant replacement for
\[ \prod_{i:I} (\psi_i \downarrow_{B_i} \varphi_i) \fibarr \prod_{i:I} B_i\]
the type
\begin{align*}
	\prod_{i:I} (\psi_i \downarrow_{B_i} \varphi_i)  & \stackrel{\text{(\ref{eq:fam-sliced-comma})} }{\simeq} \prod_{i:I}  \sum_{b:B_i}  \sum_{\substack{e:P_i(b) \\ d:Q_i(b)}} \big(\psi_i\big)_b(d) \longrightarrow_{R_i(b)} \big(\varphi_i\big)_b(e) \\
	&  \stackrel{\text{(AC)}}{\simeq} \sum_{\beta:\prod_{i:I} B_i} \prod_{i:I} \sum_{\substack{e:P_i(\beta(i)) \\ d:Q_i(\beta(i))}} \big(\psi_i\big)_{\beta(i)}(d) \longrightarrow_{R_i(\beta(i))} \big(\varphi_i\big)_{\beta(i)}(e) \\
	& \stackrel{\text{(\ref{eq:fam-sliced-comma})}}{\simeq} \sum_{\beta:\prod_{i:I} B_i} 
	\sum_{\substack{\sigma:\prod_{i:I} P_i(\beta(i)) \\	\tau:\prod_{i:I} Q_i(\beta(i))}} \big( \prod_{i:I} \psi(\tau) \big) \to_{ \big( \prod_{i:I} R_i(\beta(i))\big) }  \big( \prod_{i:I} \varphi_i(\sigma)\big) \\
	&  \stackrel{\text{(AC)}}{\simeq} \left(\prod_{i:I} \psi_i \right) \downarrow_{\left(\prod_{i:I} B_i \right)} \left(\prod_{i:I} \varphi_i \right)	
\end{align*}
This yields the desired fibered equivalence.
\end{proof}

\begin{corollary}[Products of commas in a slice]\label{lem:prod-comma-slice}
	Fix a base Rezk type $B$ be and an indexing type $I$. Given for $i:I$ an isoinner fibration $\pi_i:E_i \fibarr B$ consider a cospan of isoinner fibrations $\psi_i: F_i \to E_i \leftarrow G_i: \varphi$. Then we have a fibered equivalence:
	\[\begin{tikzcd}
		{ \times_{i:I}^B \big( \relcomma{B}{\psi_i}{\varphi_i}  \big)} && {\relcomma{B}{\left( \times_{i:I}^B \varphi_i \right)}{\left(\times_{i:I}^B \psi_i \right)}} \\
		{\prod_{i:I} F_i \times_B E_i} && {\left(\prod_{i:I} F_i\right) \times_{I \to B} \left(\prod_{i:I} E_i\right)} \\
		& B
		\arrow[two heads, from=1-1, to=2-1]
		\arrow["\simeq", from=1-1, to=1-3]
		\arrow["\simeq", from=2-1, to=2-3]
		\arrow[two heads, from=1-3, to=2-3]
		\arrow[two heads, from=2-1, to=3-2]
		\arrow[two heads, from=2-3, to=3-2]
	\end{tikzcd}\]
\end{corollary}

\begin{proposition}[Fibered (LARI) adjunctions are preserved by sliced products]\label{prop:fib-lari-pres-by-sl-prod}
	For an indexing type $I$ and a base Rezk type $B$, families of fibered (LARI) adjunctions between Rezk types over $B$ are closed under taking sliced products, \ie: Given a family of isoinner fibrations over $B$ together with a fibered (LARI) adjunction as below left, the 
	induced maps on the right make up a fibered (LARI) adjunction as well:
		\[\begin{tikzcd}
			{E_i} && {F_i} & {\times_{i:I}^B E_i} && {\times_{i:I}^B F_i} \\
			& B &&& B
			\arrow[two heads, from=1-3, to=2-2]
			\arrow[two heads, from=1-1, to=2-2]
			\arrow[""{name=0, anchor=center, inner sep=0}, "{\varphi_i}"{description}, curve={height=8pt}, from=1-1, to=1-3]
			\arrow[""{name=1, anchor=center, inner sep=0}, "{\psi_i}"{description}, curve={height=14pt}, from=1-3, to=1-1]
			\arrow[squiggly, from=1-3, to=1-4]
			\arrow[two heads, from=1-6, to=2-5]
			\arrow[two heads, from=1-4, to=2-5]
			\arrow[""{name=2, anchor=center, inner sep=0}, "{\times_{i:I}^B \varphi_i}"{description}, curve={height=8pt}, from=1-4, to=1-6]
			\arrow[""{name=3, anchor=center, inner sep=0}, "{\times_{i:I}^B \psi_i}"{description}, curve={height=14pt}, from=1-6, to=1-4]
			\arrow["\dashv"{anchor=center, rotate=-90}, draw=none, from=1, to=0]
			\arrow["\dashv"{anchor=center, rotate=-90}, draw=none, from=3, to=2]
		\end{tikzcd}\]
\end{proposition}

\begin{proof}
Given a family of fibered adjunctions as indicated amounts to a family of fibered equivalences, themselves fibered over $B$, for $i:I$:
\[\begin{tikzcd}
	{\psi_i \downarrow_B E_i} && {\varphi_i \downarrow_B E_i} \\
	& {F_i \times_B E_i} \\
	& B
	\arrow[two heads, from=1-1, to=2-2]
	\arrow[two heads, from=1-3, to=2-2]
	\arrow[two heads, from=1-1, to=3-2]
	\arrow[two heads, from=1-3, to=3-2]
	\arrow[""{name=0, anchor=center, inner sep=0}, curve={height=-12pt}, from=1-1, to=1-3]
	\arrow[""{name=1, anchor=center, inner sep=0}, curve={height=-6pt}, from=1-3, to=1-1]
	\arrow[two heads, from=2-2, to=3-2]
	\arrow["\simeq"{description}, draw=none, from=0, to=1]
\end{tikzcd}\]
Taking the dependent product over $i:I$ produces a fibered equivalence, itself fibered over $B^I$. Pullback along $\cst: B \to B^I$ yields the sliced products and again preserves the fibered equivalence:
\[\begin{tikzcd}
	{	{\times_{i:I}^B \psi_i \downarrow_B E_i}} &&& {\prod_{i:I} \psi_i \downarrow_B E_i } & {} \\
	& {{\times_{i:I}^B F_i \downarrow_B \varphi_i }} & {} && {\prod_{i:I} F_i \downarrow_B \varphi_i} & {} \\
	& {{\times_{i:I}^B F_i \times_B E_i}} &&& {\prod_{i:I} F_i \times_B E_i} \\
	B &&& {B^I}
	\arrow[two heads, from=1-1, to=4-1]
	\arrow[from=1-1, to=1-4]
	\arrow["\cst"', from=4-1, to=4-4]
	\arrow[two heads, from=1-1, to=3-2]
	\arrow[two heads, from=3-2, to=4-1]
	\arrow[two heads, from=3-5, to=4-4]
	\arrow["\simeq"{description}, dashed, from=1-1, to=2-2]
	\arrow[two heads, from=2-2, to=3-2]
	\arrow["\simeq"{description}, from=1-4, to=2-5]
	\arrow[two heads, from=2-5, to=3-5]
	\arrow[two heads, from=1-4, to=4-4]
	\arrow["\lrcorner"{anchor=center, pos=0.125}, draw=none, from=3-2, to=4-4]
	\arrow["\lrcorner"{anchor=center, pos=0.125}, draw=none, from=2-2, to=3-5]
	\arrow["\lrcorner"{anchor=center, pos=0.125}, draw=none, from=1-1, to=2-5]
	\arrow[two heads, from=1-4, to=3-5, crossing over]
	\arrow[from=3-2, to=3-5, crossing over]
	\arrow[from=2-2, to=2-5, crossing over]
\end{tikzcd}\]
Since sliced products canonically commute with both sliced commas and fiber products, this gives a fibered equivalence
\[\begin{tikzcd}
	{\big( \times_{i:I}^B \psi_i  \big) \downarrow_{I \to B} \big( \times_{i:I}^B  E_i \big)} && {\big(\times_{i:I}^B F_i \big) \downarrow_{I \to B} \big( \times_{i:I}^B  \varphi_i\big)} \\
	\\
	& {\big(\times_{i:I}^B F_i \big) \times_{I \to B} \big(\times_{i:I}^B E_i \big)}
	\arrow[""{name=0, anchor=center, inner sep=0}, curve={height=-12pt}, from=1-1, to=1-3]
	\arrow[""{name=1, anchor=center, inner sep=0}, curve={height=-12pt}, from=1-3, to=1-1]
	\arrow[two heads, from=1-1, to=3-2]
	\arrow[two heads, from=1-3, to=3-2]
	\arrow["\simeq"{description}, Rightarrow, draw=none, from=0, to=1]
\end{tikzcd}\]
which exactly yields the desired fibered adjunction of the sliced products.
\end{proof}

%% file: cv.tex
\begin{tabular}{rll}
	{\small \textbf{Oct~08--Apr~14}} & \emph{TU Darmstadt, Germany}& BSc.~Mathematics  \\
		& & (minor:~Philosophy) \\
	{\small \textbf{Dec~12}}		   & Thesis: \emph{Splitting the Classical} &   Advisor:~Prof.~Dr.~Thomas Streicher	\\
					   & \emph{Model Structure on Simplicial Sets} & \\ 
					   	& & \\
	{\small \textbf{Apr~14--Sep~16}} & \emph{TU Darmstadt, Germany} & MSc.~Mathematics  \\
		& & (minor:~Computer Science) \\
	{\small \textbf{Sep~16}}		   & Thesis: \emph{The Cubical Model} & Adv.:~Prof.~Dr.~Thomas Streicher\\
					   & \emph{of Type Theory}			&									\\	
					   	& & \\				
	{\small \textbf{Oct~16--Sep~21}} & \emph{TU Darmstadt, Germany}& Doctoral Cand. \& Scientific Assistant \\
			   & 					 & Dept.~of Mathematics, Logic Group.\\
			   	& & \\
	{\small \textbf{Oct~21}}		 	& Thesis: \emph{A Synthetic Perspective} & 	Adv.: 
	Prof.~Dr.~Thomas Streicher	\\
						& \emph{of $\inftyone$-Category Theory:} & 								\\
						& \emph{Fibrational and Semantic aspects} & 								\\
							& & \\
	{\small \textbf{Oct~21--Feb~22}} & \emph{University of Birmingham, UK} & Research Fellow \\
					   &									 & School of Computer Science,\\
					   &									 & Theory of Computation Group. \\
					   &									 & Adv.:~Dr.~Benedikt~Ahrens \\
					   							& & \\
	{\small \textbf{from Mar~22}} & \emph{Max Planck Institute for Mathematics,} & Postdoctoral Fellow \\
					   &	\emph{Bonn, Germany}								 & Adv.:~Dr.~Viktoriya~Ozornova
\end{tabular}